\documentclass{amsart}
\usepackage{amscd,amsmath,amssymb,dsfont,enumerate,graphicx,psfrag,ulem}
\newtheorem{thm}{Theorem}

\newtheorem{thr}{Theorem}[section]
\newtheorem{cor}[thr]{Corollary}

\newtheorem{lem}[thr]{Lemma}
\newtheorem{prop}[thr]{Proposition}
\theoremstyle{definition}
\newtheorem{defn}[thr]{Definition}

\theoremstyle{remark}
\newtheorem{rem}[thr]{Remark}

\numberwithin{equation}{section}
\newtheorem{example}{Example}

\newcommand{\abs}[1]{\left\vert#1\right\vert}
\newcommand{\set}[1]{\left\{#1\right\}}

\newcommand{\knotgroup}{\pi}
\newcommand{\Too}{\longrightarrow}
\newcommand{\M}{M}
\newcommand{\V}{V}

\newcommand{\fs}{\footnotesize}
\newcommand{\seq}{\sim}
\newcommand{\ass}{\stackrel{\textup{\tiny def}}{=}} 
\newcommand{\Hash}{\thinspace\begin{minipage}{8pt}\includegraphics[width=8pt]{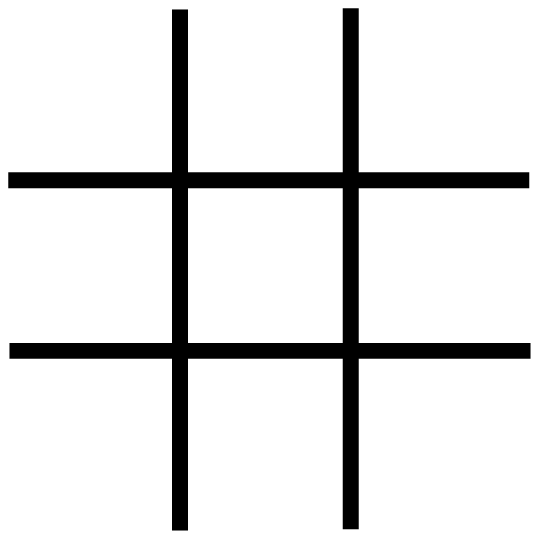}\end{minipage}\thinspace}

\newcommand{\lblY}[4]{\ \begin{minipage}{#4pt}\psfrag{a}[c]{\small$#1$}\psfrag{b}[c]{\small$#2$}\psfrag{c}[c]{\small$#3$}\includegraphics[width=#4pt]{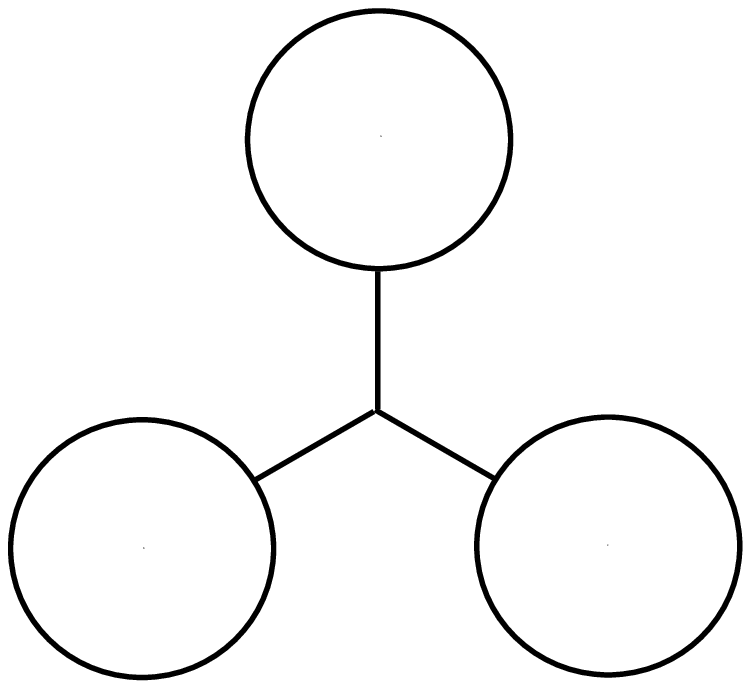}\end{minipage}\ }
\newcommand{\genusoneknot}[6]{\left(\rule{0pt}{18pt}\begin{bmatrix} #1 & #2\\ #3 & #4 \end{bmatrix}\raisebox{-5.5pt}{\huge{,}}\normalsize \begin{pmatrix} #5\\ #6 \end{pmatrix}\right)}
\def\ssa{{\text{=\rm \raisebox{0.03ex}{:}\vspace{-0.05ex}}}}
\def\co{\colon\thinspace}
\DeclareMathOperator{\Rank}{Rank}
\DeclareMathOperator{\Link}{Link}
\DeclareMathOperator{\Ab}{Ab}
\DeclareMathOperator{\Lk}{Lk}
\allowdisplaybreaks
\begin{document}

\title[Surgery Presentations for Coloured Knots]
{Surgery Presentations for Knots Coloured by Metabelian Groups}
\author{Daniel Moskovich}
\address{Department of Mathematics, University of Toronto, 40 St. George Street, Toronto, Ontario, Canada M5S 2E4}%
\email{ddmoskov@math.toronto.edu}%
\urladdr{http://www.sumamathematica.com/}

\thanks{The author would like to thank Tomotada Ohtsuki, Kazuo Habiro, Andrew Kricker, Julius Shaneson, Alexander Stoimenow, and Najmuddin Fakhruddin for helpful discussions, and also Charles Livingston, Kent Orr, Stefan Friedl and Steven Wallace for useful comments and for pointing out references. The bulk of this work was done with the support of a JSPS Research Fellowship for Young Scientists.}

\date{28th of December, 2010}

\begin{abstract}
A $G$--coloured knot $(K,\rho)$ is a knot $K$ together with a representation $\rho$ of its knot group onto $G$. Two $G$--coloured knots are said to be $\rho$--equivalent if they are related by surgery around $\pm1$--framed unknots in the kernels of their colourings. The induced local move is a $G$--coloured analogue of the crossing change. For certain families of metabelian groups $G$, we classify $G$--coloured knots up to $\rho$--equivalence. Our method involves passing to a problem about $G$--coloured analogues of Seifert matrices.
\end{abstract}
\maketitle

\section{Introduction}\label{S:Intro}

\subsection{Preamble}\label{SS:Results}

One of the fundamental facts in knot theory is that any knot can be untied by crossing changes, and that crossing changes are realized by surgery around $\pm1$--framed unknots. For $G$--coloured knots, where $G$ is a group, twist moves as in Figure \ref{F:FRMove} take the place of crossing changes, and these are realized by surgery around $\pm1$--framed unknots in the kernel of the $G$--colouring. Two $G$--coloured knots are said to be \emph{$\rho$--equivalent} if they are related, up to ambient isotopy, by a sequence of twist moves. How many $\rho$--equivalence classes of $G$--coloured knots are there? What distinguishes one from another?\par

In \cite{KM09}, Kricker and I considered the case of $G$ a dihedral group $D_{2n}=\mathcal{C}_2\ltimes\mathds{Z}/n\mathds{Z}$. We proved that the number of $\rho$--equivalence classes of $D_{2n}$-coloured knots is $n$. These are told apart by the \emph{coloured untying invariant}, an algebraic invariant of $\rho$--equivalence classes defined in terms of \emph{surface data} (see \cite{Mos06b}). Surface data is the analogue for a $G$--coloured knot of a Seifert matrix. Our proof was constructive, in the sense that it provided an explicit sequence of twist moves to relate each $D_{2n}$-coloured knot to a chosen representative of its $\rho$--equivalence class.\par

The purpose of this work is to expand the above result to knots coloured by a wider class of metabelian groups $G=\mathcal{C}_m\ltimes A$. We show that the results of \cite[Section 4]{KM09} extend to $G$--coloured knots for most metacyclic groups (Theorem \ref{T:metacyclic}), and for certain classes of metabelian groups with $\Rank(A)=2$ (Theorem \ref{T:R2M3Diag} and Theorem \ref{T:R2M3Non}). In particular, we classify $A_4$-coloured knots up to $\rho$--equivalence (Theorem \ref{T:A4Theorem}). In all cases, `the only obstruction to $\rho$--equivalence is the obvious one'. The obstruction to carrying out the same computations for metabelian groups with $\Rank(A)>2$ is identified by Theorem \ref{T:clasperprop}.\par

The starring role is played by the surface data. For a $G$--coloured knot, the surface data determines the $G$--colouring; moreover, the $S$--equivalence relation on Seifert matrices induces an $S$--equivalence relation on surface data (Section \ref{SS:S-equiv-matrix}). The relevant equivalence relation on $G$--coloured knots becomes \emph{$\bar\rho$--equivalence}, induced by a special kind of twist move called the \emph{null-twist} (Figure \ref{F:HosteMove}). To classify $G$--coloured knots up to $\rho$--equivalence, we first classify them up to $\bar\rho$--equivalence. When $\Rank(A)\leq 2$, two $G$-coloured knots with $S$--equivalent surface data must be $\bar\rho$--equivalent and therefore $\rho$--equivalent (Theorem \ref{T:clasperprop}). Thus, $\bar\rho$--equivalence classes are distinguished by invariants coming from surface data, which in turn have explicit linear algebraic formulae. Two such invariants are the \emph{surface untying invariant} (Section \ref{SS:su}) and the \emph{$S$--equivalence class of the colouring} (Section \ref{SS:sequiv}). To go further and to distinguish $\rho$--equivalence classes, we use the \emph{coloured untying invariant} (Section \ref{SS:cu}), also given in terms of surface data. To distinguish $\bar\rho$--equivalence classes when $\Rank{A}>2$, surface data alone turns out to be insufficient, and we must take into account also triple-linkage between bands (Section \ref{S:ClasperProof}).

\begin{figure}
\begin{minipage}{40pt}
\psfrag{a}[c]{\footnotesize$g_1$}\psfrag{b}[c]{\footnotesize$g_2$}\psfrag{c}[c]{\footnotesize$g_{r}$}
\includegraphics[height=135pt]{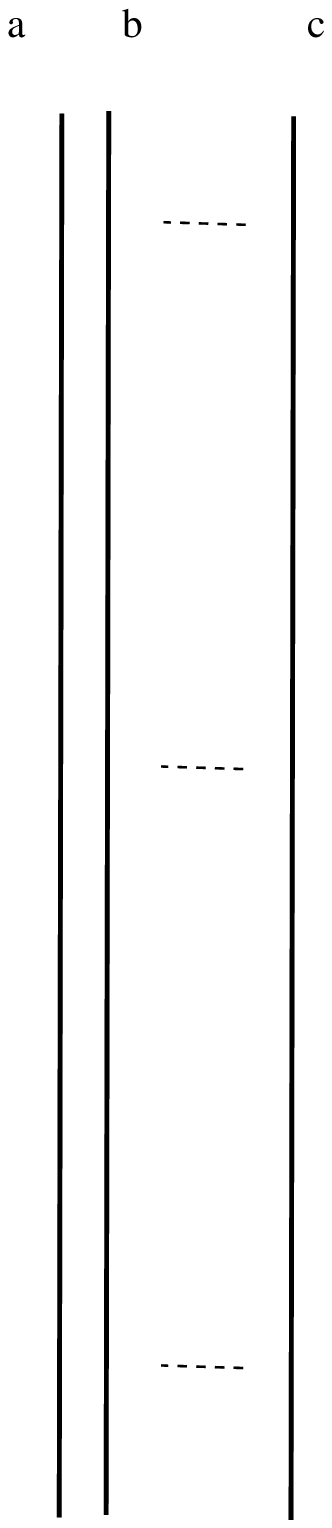}
\end{minipage}\ \ $\overset{\raisebox{2pt}{\scalebox{0.8}{\text{twist}}}}{\Longleftrightarrow}$\quad
\begin{minipage}{85pt}
\psfrag{a}[c]{\footnotesize$g_1$}\psfrag{b}[c]{\footnotesize$g_2$}\psfrag{c}[c]{\footnotesize$g_{r}$}
\psfrag{d}[c]{\footnotesize$g_1$}\psfrag{e}[c]{\footnotesize$g_2$}\psfrag{f}[c]{\footnotesize$g_{r}$}
\includegraphics[height=140pt]{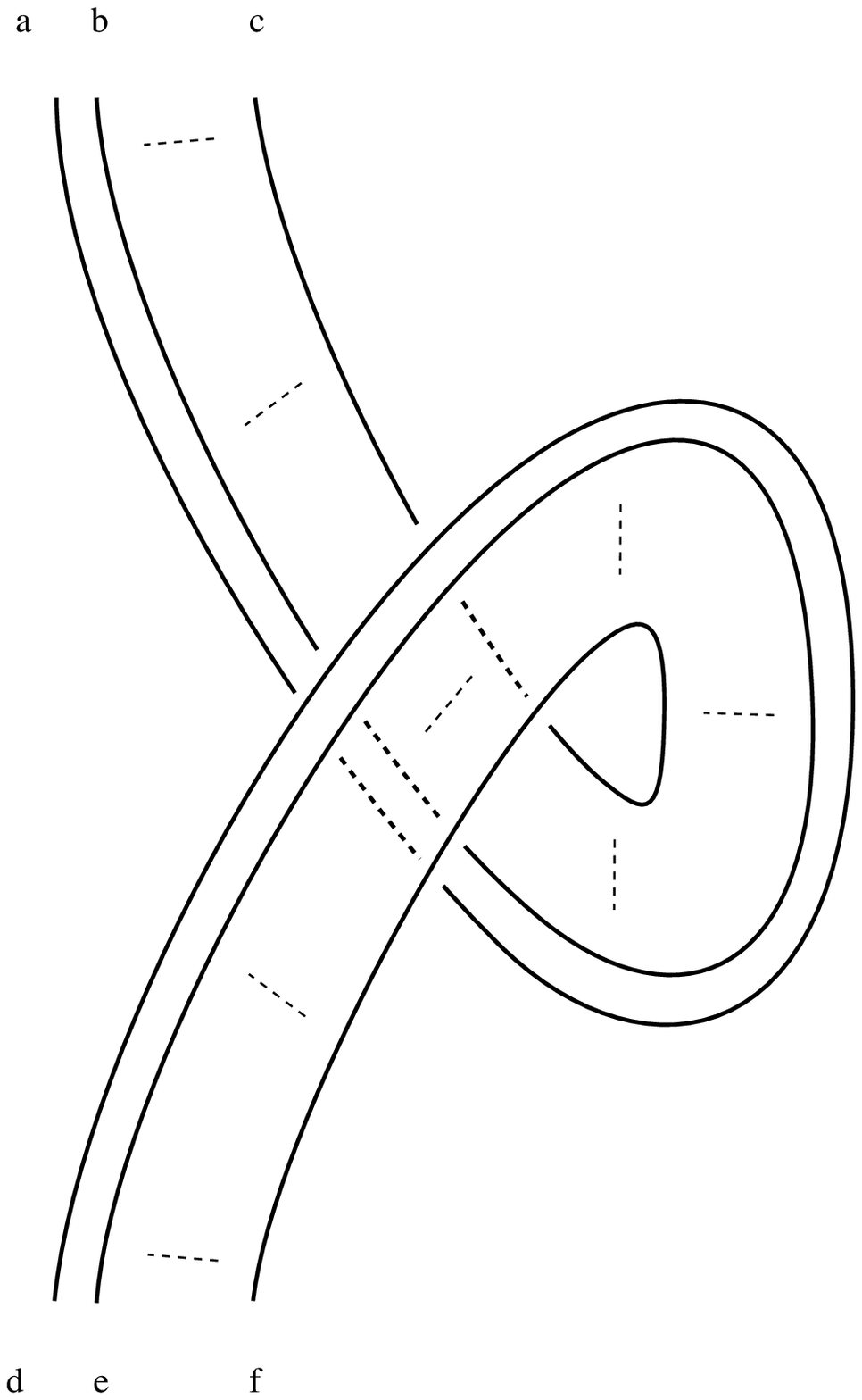}
\end{minipage}
\ \ \ $=$\ \ \
\begin{minipage}{60pt}
\psfrag{a}[c]{\footnotesize$g_1$}\psfrag{b}[c]{\footnotesize$g_2$}\psfrag{c}[c]{\footnotesize$g_{r}$}
\psfrag{d}[c]{\footnotesize$g_1$}\psfrag{e}[c]{\footnotesize$g_2$}\psfrag{f}[c]{\footnotesize$g_{r}$}
\psfrag{p}[c]{\footnotesize$2\pi$\ twist}
\includegraphics[height=140pt]{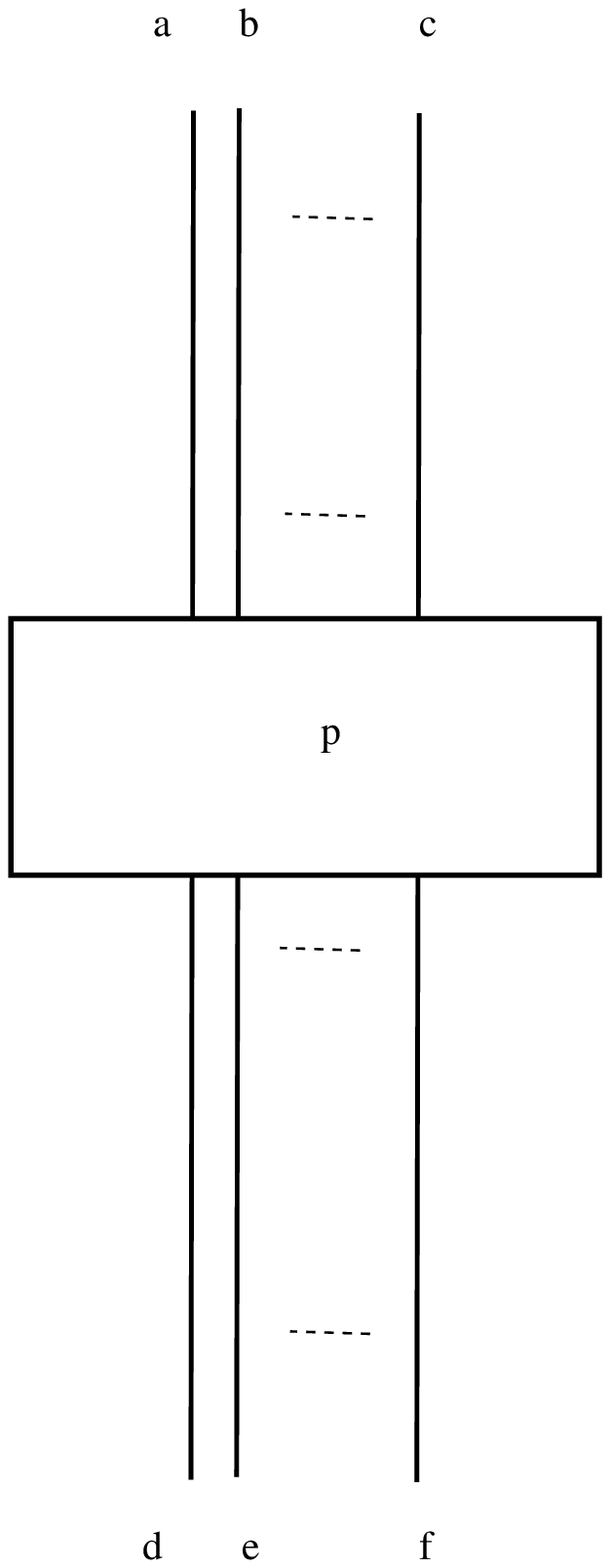}
\end{minipage}
\caption{\label{F:FRMove} This local move, called a \emph{twist move}, is defined whenever
$g_1^{\epsilon_1}g_2^{\epsilon_2}\cdots g_{r}^{\epsilon_r}\in G$ vanishes, where $\epsilon_i$ is $1$ if the strand is pointing up and $-1$ if it is pointing down.}
\end{figure}

\subsection{Technical Summary}\label{SS:Method}

Let $G= \mathcal{C}_m\ltimes_\phi A$ be a fixed metabelian group, where $\mathcal{C}_m=\left\langle\left.\rule{0pt}{9pt}t\thinspace\right|t^m=1\right\rangle$ is a cyclic group, and $A$ is an finitely generated abelian group. A \emph{$G$--coloured knot} is a pair $(K,\rho)$ of an oriented knot with basepoint $K\co S^1\hookrightarrow S^3$, together with a surjective homomorphism $\rho$ of the knot group of $K$ onto $G$. Such $G$--coloured knots were previously studied by Hartley \cite{Har79}. Two $G$--coloured knots are said to be \emph{$\rho$--equivalent} if they are related up to ambient isotopy by a finite sequence of twist moves. We bound the number of $\rho$--equivalence classes from above and from below. In favourable cases these bounds agree. In Section \ref{S:r12}, we classify $G$--coloured knots up to $\rho$--equivalence in all such favourable cases, when the rank of $A$ is at most $2$.\par

A key idea is to introduce various weaker equivalence relations. The $G$--colouring $\rho$ induces:

\begin{itemize}
\item An $A$--colouring $\bar\rho$ of a Seifert surface exterior $E(F)$.
\item For $\tilde{G}= \mathcal{C}_0\ltimes_\phi A$, and $\tilde{G}$--colouring $\hat\rho$ of $K$.
\item An $A$--colouring $\tilde\rho$ of the $m$--fold branched cyclic cover $C_m(K)$.
\end{itemize}

Each of these colourings in turn induces an equivalence relation on $G$--coloured knots, which we call $\bar\rho$--equivalence, $\hat\rho$--equivalence, and $\tilde\rho$--equivalence correspondingly. Chief among these is $\bar\rho$--equivalence. Two (rigid) knots are \emph{tube equivalent} if they possess tube equivalent Seifert surfaces (Definition \ref{D:tubequiv}). Two $G$--coloured knots are $\bar\rho$--equivalent if they are related up to tube equivalence by null-twists (see Figure \ref{F:HosteMove}). As $\bar\rho$--equivalence is defined with respect to a colouring of a Seifert surface by an abelian group, its study is amenable to linear algebraic techniques. Our main effort is to classify $G$--coloured knots up to $\bar\rho$--equivalence. Such a classification leads to a classification of $G$--coloured knots up to $\rho$--equivalence if either all of the equivalence relations happen to coincide (as is the case for some metabelian groups in Section \ref{S:r12}), or if $G$ is simple enough that the remaining work can be done by hand (as for the case $G=A_4$ in Section \ref{S:A4}).

\begin{rem}
In a different context, the twist move is called the Fenn--Rourke move, and the null-twist is called the Hoste move (see \textit{e.g.} \cite{Hab06}).
\end{rem}

Both a twist moves and a null-twist come from integral Dehn surgery, and the trace of such surgery a special kind of bordism (Proposition \ref{P:bordbarrho}). Therefore the order of the appropriately defined bordism group gives an upper bound on the number of possible $\rho$--equivalence classes of $G$--coloured knots. This upper bound was studied by Litherland and Wallace \cite{LiWal08} following work of Cochran, Gerges, and Orr \cite{CGO01}. Their result was that the number of $\rho$--equivalence classes of $G$--coloured knots is bounded above by the product of orders of certain homology groups. We tighten this upper bound by considering instead the $\bar\rho$--equivalence relation. We find that the order of $H_3(A;\mathds{Z})$ is an upper bound for the number of $\bar\rho$--equivalence classes (Corollary \ref{C:Wallacebound}).\par

\begin{figure}
\begin{minipage}{40pt}
\psfrag{a}[c]{\footnotesize$g_1$}\psfrag{b}[c]{\footnotesize$g_2$}\psfrag{c}[c]{}\psfrag{d}[c]{\footnotesize$g_{2r}$}
\includegraphics[height=135pt]{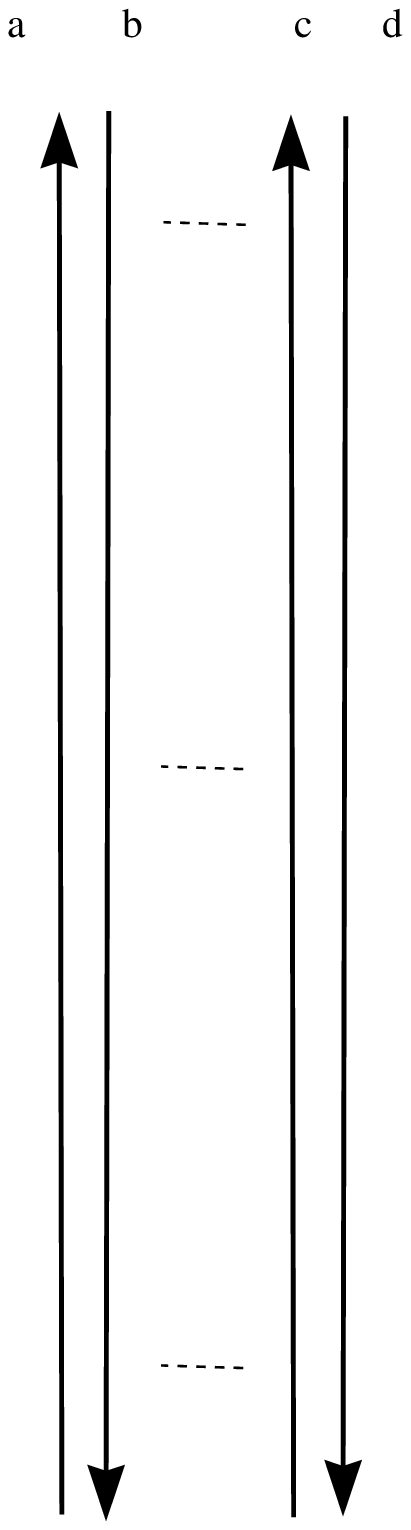}
\end{minipage}\ \ $\overset{\raisebox{2pt}{\scalebox{0.8}{\text{null-twist}}}}{\Longleftrightarrow}$\quad
\begin{minipage}{85pt}
\psfrag{a}[c]{\footnotesize$g_1$}\psfrag{b}[c]{\footnotesize$g_2$}\psfrag{c}[c]{}\psfrag{d}[c]{\footnotesize$g_{2r}$}
\psfrag{e}[c]{\footnotesize$g_1$}\psfrag{f}[c]{\footnotesize$g_2$}\psfrag{g}[c]{}\psfrag{h}[c]{\footnotesize$g_{2r}$}
\includegraphics[height=140pt]{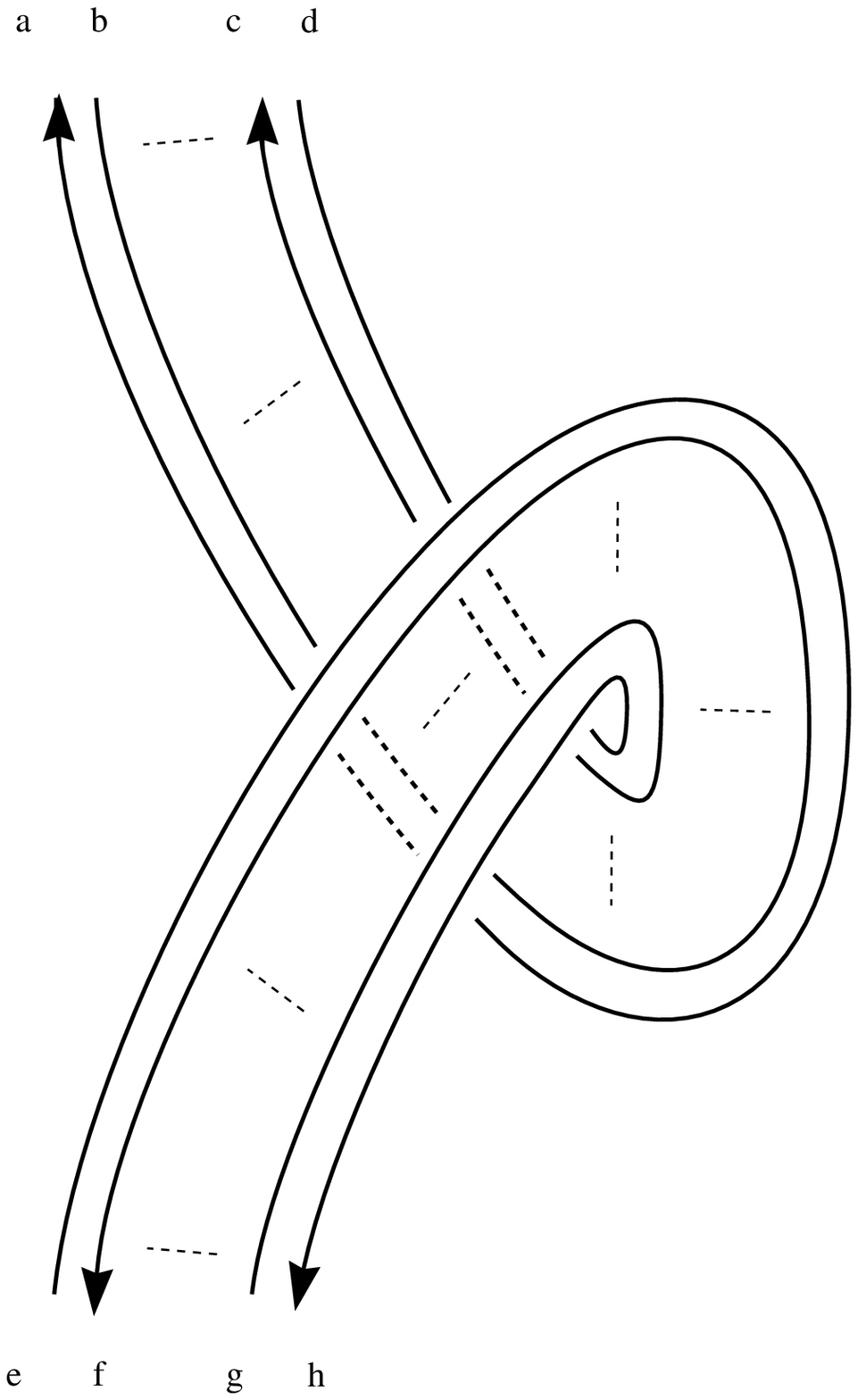}
\end{minipage}
\ \ \ $=$\ \ \
\begin{minipage}{60pt}
\psfrag{a}[c]{\footnotesize$g_1$}\psfrag{b}[c]{\footnotesize$g_2$}\psfrag{c}[c]{}\psfrag{d}[c]{\footnotesize$g_{2r}$}
\psfrag{e}[c]{\footnotesize$g_1$}\psfrag{f}[c]{\footnotesize$g_2$}\psfrag{g}[c]{}\psfrag{h}[c]{\footnotesize$g_{2r}$}
\psfrag{p}[c]{\footnotesize$2\pi$\
twist}
\includegraphics[height=140pt]{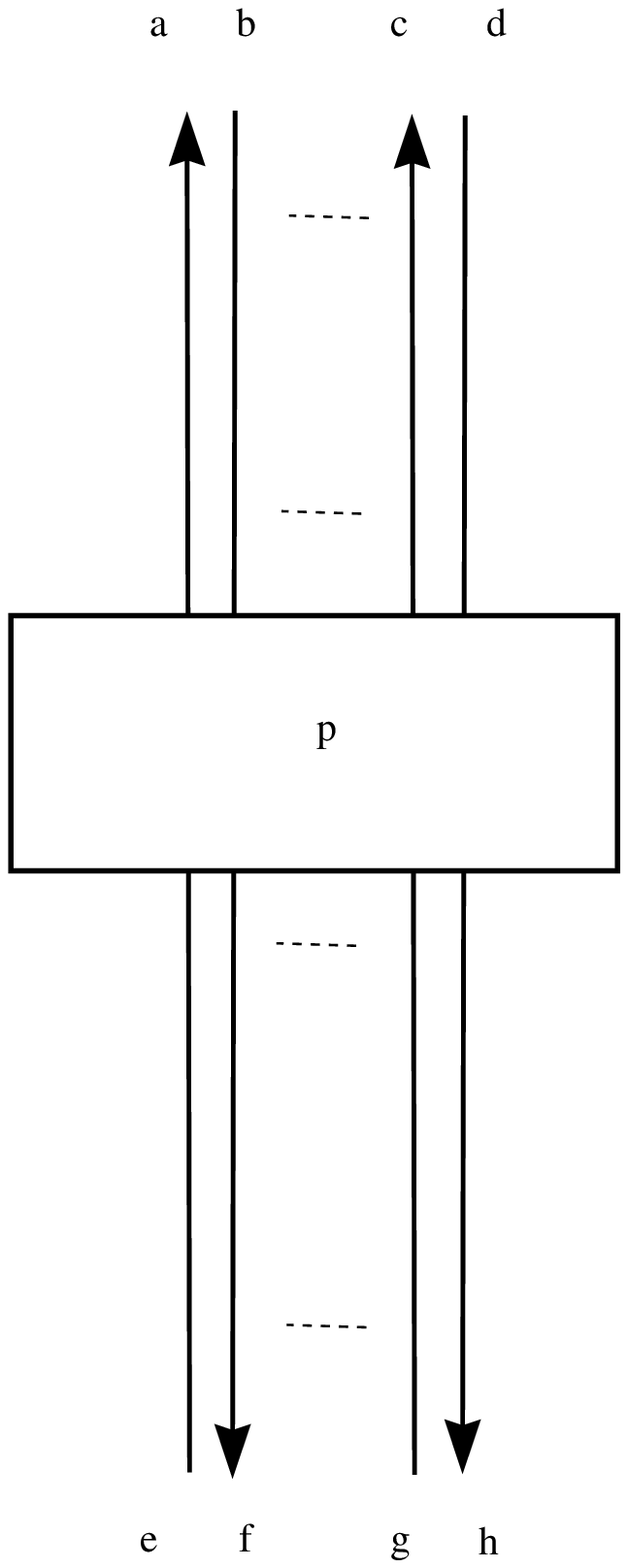}
\end{minipage}
\caption{\label{F:HosteMove} This local move, called a null-twist, is defined whenever
$g_1g_2^{-1}g_3g_4^{-1}\cdots g_{2r-1}g_{2r}^{-1}\in G$ vanishes.}
\end{figure}

For lower bound calculations, the goal is to compile the longest possible list of non--$\rho$--equivalent $G$--coloured knots. Recall \cite[Definition 3]{KM09}.

\begin{defn}
A \emph{complete set of base-knots} for a group $G$ is a set $\Psi$ of $G$--coloured knots $(K_i,\rho_i)$, no two of which are $\rho$--equivalent, such that any $G$--coloured knot $(K,\rho)$ is $\rho$--equivalent to some $(K_i,\rho_i)\in\Psi$. A element of $\Psi$ is called a \emph{base-knot} (the term imitates `base-point').
\end{defn}

 We remark that for the applications outlined in Section \ref{SS:Motivation}, base-knots should be chosen to be as ``nice'' as possible, in that they should be unknotting number $1$ knots whose irregular $G$--covers we know how to present explicitly.\par

 The method of this paper consists of transforming the geometric-topology problem of finding a complete set of base-knots into a problem in linear algebra over a commutative ring, and then solving that problem for the relevant commutative rings. I arrived at this approach by thinking hard about the band-sliding algorithm in \cite[Section 4]{KM09} until I understood the underlying algebraic mechanism that makes it work.\par

Choose a Seifert surface $F$ for $K$ and a basis $x_1,\ldots,x_{2g}$ for $H_1(F)$, which induces an associated basis $\xi_1,\ldots,\xi_{2g}$ for $H_1(E(F))$. The $G$--colouring $\rho$ restricts to an $A$--colouring $\bar\rho\co H_1(E(F))\to A$ (Section \ref{SS:ASeif}). We obtain a Seifert matrix $\M$ for $K$ and a \emph{colouring vector} $\V\in A^{2g}$, whose entries are the $\bar\rho$--images of the $\xi_i$'s. Such a pair $(\M,\V)$ is called \emph{surface data} for $(K,\rho)$. Surface data is the analogue for $G$--coloured knots of a Seifert matrix (Section \ref{S:surfacedata}). In particular, it makes sense to discuss \emph{$S$--equivalence} of surface data (Section \ref{SS:S-equiv-matrix}); and moreover, when $\Rank(A)\leq 2$, $S$--equivalence of surface data implies $\bar\rho$--equivalence of $G$--coloured knots (Theorem \ref{T:clasperprop}). The implication is that rather than working with twist-moves on $G$--coloured knots, we may instead work with the induced equivalence relation on surface data. Matrices are simpler mathematical objects that knots, and for `simple enough' groups $G$ the induced problem solves itself.\par

 To distinguish between $\bar\rho$--equivalence classes, we identify two $\bar\rho$--equivalence invariants coming from the surface data. The first of these, given in Section \ref{SS:su}, is an element of $A$ which is a version of the coloured untying invariant of \cite[Section 6]{Mos06b}, which we call the \emph{surface untying invariant}. It may be interpreted as a linking number of push-offs of curves naturally associated to the map $\bar\rho$. The second, which we call the \emph{$S$--equivalence class of the colouring}, is an element of $A\wedge A$ coming from the $S$--equivalence class of the surface data. These two invariants suffice to distinguish the base-knots presented in Sections \ref{S:r12} and \ref{S:A4} up to $\bar\rho$--equivalence. An extension of the coloured untying invariant (Section \ref{SS:cu}) is then used to distinguish these base-knots up to $\rho$--equivalence.\par

For a metacyclic group for which $2(\phi^{-3}-\mathrm{id})$ is invertible, two $G$--coloured knots are $\bar\rho$--equivalent if and only if they are $\rho$--equivalent, thus no extra work is required. Conversely, for $G=A_4$ the group of symmetries of an oriented tetrahedron, two $G$--coloured knots may even be ambient isotopic without being $\bar\rho$--equivalent! For this group, which is the smallest metabelian group with $\Rank(A)>1$ and is also a finite subgroup of $SO(3)$ and therefore interesting, we conclude the paper by showing `by hand' that the lower bound is sharp, \textit{i.e.} that the coloured untying invariant is a complete invariant of $\rho$--equivalence classes for $A_4$-coloured knots.\par

When $\Rank(A)>2$, an additional $\bigwedge^3 A$--valued obstruction to $\bar\rho$--equivalence emerges from triple-linkage between bands of the Seifert surface. This obstruction, which we call the \emph{$Y$--obstruction}, is the topic of Section \ref{S:ClasperProof}, where in Theorem \ref{T:clasperprop} we prove that two $S$--equivalent knots are $\bar\rho$--equivalent if and only if their $Y$--obstruction vanishes. Triple-linkage between bands detects information one step below the Alexander module in the derived series of the knot group \cite{Tur83,Tur84}.\par

The moral is that $\bar\rho$--equivalence is a useful equivalence relation to consider on $G$--coloured knots, because of its relationship to $S$--equivalence, and the fact that it is generated by a local move. Conceptually, it is a similar idea to null--equivalence \cite{GR04} and to $H_1$-bordism \cite{CM00}.\par

With $\Lk=0$ and \sout{Inn} short-hands for ``admit only null-twists'' and ``admit only tube equivalence'', the following summarizes the equivalence relations which this papers considers, and how they relate to one another.

\begin{equation}\label{E:equivrels}
\psfrag{a}[r]{$\rho$--equivalence}\psfrag{b}[c]{$\hat\rho$--equivalence}\psfrag{c}[c]{$\tilde\rho$--equivalence}\psfrag{d}[l]{$\bar\rho$--equivalence}
\psfrag{r}[c]{\fs$\Lk=0$}\psfrag{s}[c]{\fs \sout{Inn}}\psfrag{t}[c]{\fs \sout{Inn}}\psfrag{u}[c]{\fs$\Lk=0$}
\psfrag{1}[c]{$cu$,$s$}\psfrag{2}[c]{$\Omega$}\psfrag{3}[c]{$su$}
\begin{minipage}{300pt}\centering\includegraphics[width=260pt]{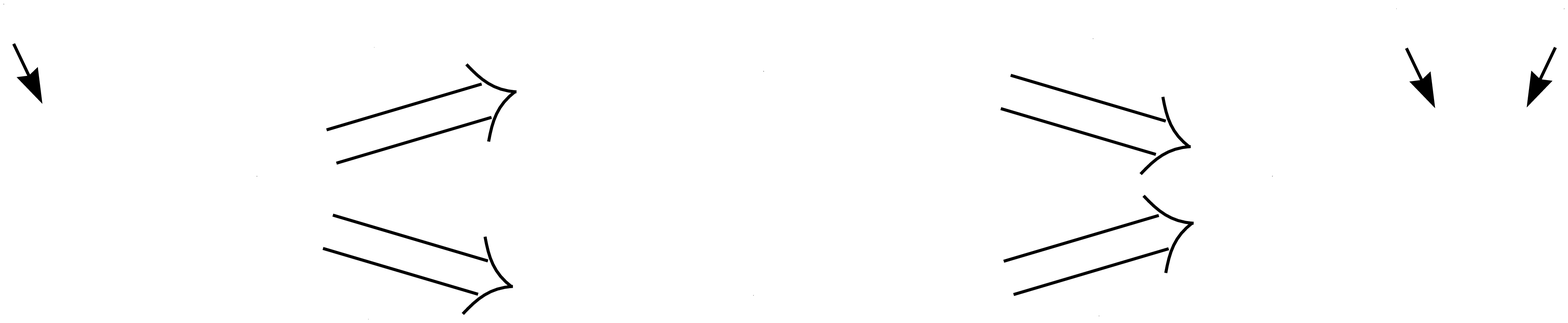}\end{minipage}
\end{equation}

If we would have used \emph{equivariant} homology and bordism, with respect to the action of $\mathcal{C}_m$ on $A$, then we could have pushed the bordism upper bound $\Omega$, the surface untying invariant $\mathrm{su}$, and the $S$--equivalence class $s$ of the colouring, all `one step to the left', so as to try to classify $G$--coloured knots up to $\hat\rho$--equivalence. 

\subsection{My motivation for studying $\rho$--equivalence}\label{SS:Motivation}

My motivation for studying $\rho$--equivalence is to construct quantum topological invariants associated to formal perturbative expansions around \emph{non-trivial} flat connections. Building on the results in this paper, I plan to mimic Garoufalidis and Kricker's construction of a rational Kontsevich invariant of a knot \cite{GK04} in the $G$--coloured setting. The $1$--loop part of the Garoufalidis--Kricker theory determines the Alexander polynomial, while the $2$--loop part contains the Casson invariant of cyclic branched coverings of a knot. Studying $G$--coloured analogues of the rational Kontsevich invariant might provide an avenue to attack the Volume Conjecture, by interpreting hyperbolic volume as $L^2$-torsion \cite[Theorem 4.3]{Lu02}, which has a formula in terms of Jacobians of the Fox matrix \cite[Theorem 4.9]{Lu02} and which should be closely related to the $1$--loop parts of our prospective invariants. This would seem to me to be a natural perturbative approach to proving conjectures about semiclassical limits of quantum invariants, because in physics the fundamental object is Witten's invariant rather than the LMO invariant--- the path integral over \emph{all} $SU(2)$--connections, as opposed to its perturbative expansion close to the trivial $SU(2)$--connection.\par

The LMO invariant and the rational Kontsevich invariant are built out of a surgery presentation for a knot, in the complement of a standard unknot (see \textit{e.g.} \cite[Chapter 10]{Oht02}). The analogue for $G$--coloured knots is a surgery presentation in the complement of a base-knot and in the kernel of its colouring. We will show in future work that, for sufficiently nice base-knots (the complete sets of base-knots in this paper are indeed `sufficiently nice'), a Kirby theorem-like result holds for such presentations, allowing us to prove invariance for quantum invariants coming from surgery. Thus, such surgery presentations provides a solid foundation on which to construct $G$--coloured rational Kontsevich invariants.\par

Invariants of $G$--coloured knots have proven useful in knot theory in that they detect information beyond $\pi/\pi^{\prime\prime}$. Classically, Reidemeister used the linking matrix of a knot's dihedral covering link to distinguish knots with the same Alexander polynomial (\cite{Rei29}, see also \textit{e.g.} \cite{Per74}). More recently, twisted Alexander polynomials have been receiving a lot of attention, particularly in the context of knot concordance (see \textit{e.g.} \cite{FV10}). For the groups in question, I hope and expect that these will be related to the ``$1$--loop part'' of the theory, which might lead in the direction of the Volume Conjecture. On the next level, Cappell and Shaneson \cite{CS75,CS84} found a formula for the Rokhlin invariant of a dihedral branched covering space, which provides an obstruction to a knot being ribbon. Presumably this will be related to the ``$2$--loop part'' of the theory.\par

An unrelated motivation is the study of faithful $G$--actions on a closed oriented connected smooth $3$--manifold $M$ by diffeomorphisms. The question is whether there exists a bordism $W$ and a handle decomposition of $W$ as $M_G\times I$ with $2$--handles attached, for some fixed standard $3$--manifold $M_G$, such that the $G$--action on $M$ extends to a smooth faithful $G$--action on $W$. If $G$ happens to be a finite subgroup of $SO(3)$, this is equivalent to the existence of a surgery presentation $L\subset S^3$ for $M$ which is invariant under the standard action of $G$ on $S^3$. This would imply that an invariant of $3$--manifolds which admits a surgery presentation must take on some symmetric form for such manifolds, as discussed by Przytycki and Sokolov \cite{PrzS01}. This was proven for cyclic groups in \cite{Sak01} following \cite{PrzS01}, and for free actions of dihedral groups in \cite{KM09}. In the same vein, the results of this paper will be used, in future work, to prove the above claim also for certain $A_4$ actions.

\subsection{Comparison with the literature}\label{SS:Compare}

The results of this paper generalize the results of my joint paper with Andrew Kricker \cite[Section 4]{KM09}, based in turn on \cite{Mos06b}, to a wider class of metabelian groups. The main innovation in our methodology is that \cite{KM09} works with knot diagrams, while we work with surface data.\par
Our bordism argument is based on \cite{LiWal08} and on Steven Wallace's thesis \cite{Wal08}.\par
The results of this paper imply that, for certain metabelian groups $G$, any $G$--coloured knot $(K,\rho)$ has a surgery presentation in the complement of a base-knot for any of our complete sets of base-knots, and that the components of that surgery presentation lie in $\ker\rho$. Such a surgery presentation of $(K,\rho)$ may be lifted to a surgery presentation of irregular covering spaces associated to $(K,\rho)$, containing embedded covering links. This construction was carried out for $D_{2n}$-coloured knots in \cite{KM09}. For the groups we consider, we defer the explicit construction of such surgery presentations to future work.\par
If our base-knots all have unknotting number $1$ then we can prove a Kirby Theorem-like result for surgery presentations of $(K,\rho)$, which we can then use to construct new invariants of a $G$--coloured knots and of their covering spaces and covering links. Thus, our approach is well-suited to \emph{constructing} invariants. On the other hand, if we wanted to \emph{calculate} known invariants, then generalizing the surgery presentations of David Schorow's thesis \cite{SchorowPhD}, based on the explicit bordism constructed by Cappell and Shaneson \cite{CS84}, looks promising to me. His surgery presentation is constructed directly from a $G$--coloured knot diagram, without first having to reduce it to a base-knot by twist moves.

\subsection{Why this generality?}\label{SS:Generality}

In this paper, $\rho$--equivalence is studied by applying linear algebra to surface data. In particular, we need a Seifert surface in order to define surface data. The widest class of topological objects with Seifert matrices is homology boundary links in integral homology spheres \cite{Ko87}. With effort, the results of this paper should extend to that setting.\par
The methods in this paper are largely linear algebraic, and linear algebra can only be performed over a commutative ring. For $G$ metabelian, a $G$--colouring of a knot $(K,\rho)$ induces an $A$--colouring $\bar\rho$ of a Seifert surface complement, which allows us to encode $\rho$ as a colouring vector. If $G$ were not metabelian, the colouring would no longer correspond to a vector, and we would need more than linear algebra to bound from below the number of $\bar\rho$--equivalence classes.\par
If $A$ were not finitely generated, then $\bar\rho$ would not be surjective, and the arguments of Section \ref{S:ClasperProof} and of Section \ref{S:untyinginvariants} would fail.

\subsection{Contents of this paper}\label{SS:contents}

In Section \ref{S:Prelims} we recall the concept of a $G$--coloured knot and we establish conventions and notation. In Section \ref{S:surfacedata} we define surface data and prove that it satisfies analogous properties to the Seifert matrix. In particular, it admits an $S$--equivalence relation. In Section \ref{S:rhoequiv} we define the various flavours of $\rho$--equivalence, and show their relation with relative bordism and how they are generated by local moves. In Section \ref{S:ClasperProof} we prove Theorem \ref{T:clasperprop}, relating $S$--equivalence with $\bar\rho$--equivalence. In Section \ref{S:untyinginvariants} we identify invariants of $\rho$--equivalence classes and of $\bar\rho$--equivalence classes in terms of homology and surface data. In Section \ref{S:r12} we apply the results of the previous sections, matching upper and lower bounds, to classify $G$--coloured knots up to $\bar\rho$--equivalence and up to $\rho$--equivalence, for families of metabelian groups with $\Rank(A)\leq 2$. In Section \ref{S:A4} we go beyond the algebraic techniques of earlier sections, and beginning from the $\bar\rho$--equivalence classification of $A_4$-coloured knots, we work `by hand' to classify $A_4$-coloured knots up to $\rho$--equivalence. The paper concludes by listing some open problems in Section \ref{S:conclusion}.

\section{Preliminaries}\label{S:Prelims}

\subsection{The metabelian group $G$}\label{SS:metagroup}

A metabelian homomorph $G$ of a knot group is finitely generated, of weight one \cite{GA75,Joh80}, and is isomorphic to a semi-direct product $\mathcal{C}_m\ltimes_\phi A$ where $\mathcal{C}_m=\left\langle\left.\rule{0pt}{9pt}t\thinspace\right|t^m=1\right\rangle$ is a (possibly infinite) cyclic group, and $A$ is an finitely generated abelian group. The above notation means that the conjugation action of $\mathcal{C}_m$ on $A$ is $t^{-1}at=\phi(a)$. Write $A$ additively, and write conjugation by $t$ as left multiplication, using a dot, while we don't write the dot for multiplication in $G$, so that $t\cdot a$ stands for $t^{-1}at$.

\begin{example}
Dihedral groups are metabelian homomorphs of knot groups. They have presentation
\[D_{2n}\ass\ \left\langle t,s\left|\rule{0pt}{9.5pt}\ t^{2}=s^{n}=1,\ tst=s^{-1}\right.\right\rangle.\]
\end{example}

\begin{example}
 The alternating group of order $4$ is another metabelian homomorph of knot groups, with presentation
 \[A_4\ass \ \left\langle t,s_1,s_2\left|\rule{0pt}{9.5pt}\ t^3=s_1^2=s_2^2=1,\ t^2s_1t=s_2,\ t^2s_2t=s_1s_2\right.\right\rangle.\]
\end{example}

\subsection{$G$--coloured knots}\label{SS:ColKnots}

We adopt conventions that facilitate concrete discussion. None of our results depend essentially on these conventions.\par
In this paper, every $n$--sphere comes equipped with a fixed parametrization
\[\left\{(x_1,\ldots,x_{n+1})\in\mathds{R}^{n+1}\left|\rule{0pt}{9.5pt}\ x_1^2+\cdots+x_{n+1}^2=1\right.\right\}\to S^n\]
and each disk with a fixed parametrization $[-1,1]^{\times n}\to D^n$.\par

A knot  is an embedding $K\co S^1\hookrightarrow S^3$ together with the orientation induced by the counter-clockwise orientation of $S^1$, and a basepoint $K|_{(0,1)}$. We parameterize a tubular neighbourhood of a knot $K$ as $N(K)\co D^2\times S_1\hookrightarrow S^3$ such that
$N(K)\left(\{(0,0)\}\times \{(x,y)\}\right)= K(x,y)$, and $\Link(K,\ell)=0$, where $\ell$ denotes $N(K)\left(\{(1,1)\}\times
S^1\right)$. Thus $K$ comes equipped with a distinguished meridian $\mu\ass N(K)\left(\partial D^2\times \{(0,0)\}\right)$ and with a canonical
longitude $\ell$.\par

The \emph{knot group} is $\pi\simeq \pi_1 E(K)$. A \emph{$G$--coloured knot} is a knot $K\subset S^3$ together with a surjective homomorphism $\rho\co\pi\twoheadrightarrow G$. We draw $G$--coloured knots by labeling arcs in a knot diagram by $\rho$--images of corresponding Wirtinger generators.\par

Because Wirtinger generators of a knot are all related by conjugation, they all map to elements of the same coset $t^a A$, where $a\neq 0$ because $\rho$ is surjective. By convention, set $a$ to be $1$, so that all Wirtinger generators map to elements of $t A$.

\begin{rem}
Our coloured knots are called \emph{based coloured knots} in \cite{LiWal08}.
\end{rem}

\begin{lem}\label{L:inneriso}
Consider $G$--colourings $\rho_{1,2}\co \pi\twoheadrightarrow G$ of a knot $K$. If there exists an inner automorphism $\psi$ of $G$ such that $\rho_1(x)=\psi(\rho_2(x))$ for all $x\in \pi$, then $(K,\rho_{1,2})$ are ambient isotopic.
\end{lem}

\begin{proof}
 We summarize the argument in \cite[Page 678]{Mos06b} and \cite[Lemma 14]{KM09}. Because $\pi$ is normally generated by $\mu$, the group $G$ is normally generated by $\rho(\mu)$, so conjugation by any $g\in G$ corresponds to some composition of conjugations by labels of arcs of some knot diagram $D$ for $K$. For each such arc $\alpha$ in turn, create a kink in $\alpha$ by a Reidemeister \textrm{I} move, shrink the rest of the knot to lie inside a small ball, drag the knot through the kink (the effect is to conjugate the labels of all arcs in $D$ by the label of $\alpha$), and get rid of the kink by another Reidemeister \textrm{I} move. This sequence of Reidemeister moves brings us back to $D$, and its combined effect will have been to realize the action of $\psi$ on $\rho_1$ by ambient isotopy.
\end{proof}

\begin{example}\label{Ex:cyclic}
The degenerate case of a $G$--coloured knot is a $\mathcal{C}_n$-coloured knot. Any knot is canonically $\mathcal{C}_n$-coloured by the mod $n$ linking pairing, which with our conventions sends all of its meridians to $t$. Thus the set of $\mathcal{C}_n$--coloured knots is in bijective correspondence with
the set of knots.
\end{example}

\begin{example}\label{Ex:dihedral}
The simplest non-degenerate case of a $G$--coloured knot is a knot coloured by a dihedral group. Each Wirtinger generator is mapped to an element of the form $ts^i\in D_{2n}$, which depends only on $i\in\mathds{Z}/n\mathds{Z}\triangleleft D_{2n}$. Therefore a $D_{2n}$-colouring is encapsulated by a labeling of arcs of a knot diagram by elements in $\mathds{Z}/n\mathds{Z}$.
Such a knot diagram, labeled by integers or with colours standing in for those integers, was called an $n$--coloured knot by Fox, and this is the genesis of the term `coloured knots' \cite{Fox62}. There is no need to orient the knot diagram, because a $\rho$--image of a Wirtinger generator is its own inverse. See Figure \ref{F:Fan}.
\end{example}

\begin{figure}
\psfrag{a}[c]{$4$}\psfrag{b}[c]{$3$}\psfrag{c}[c]{$2$}\psfrag{d}[c]{$0$}\psfrag{e}[c]{$1$}
\includegraphics[width=170pt]{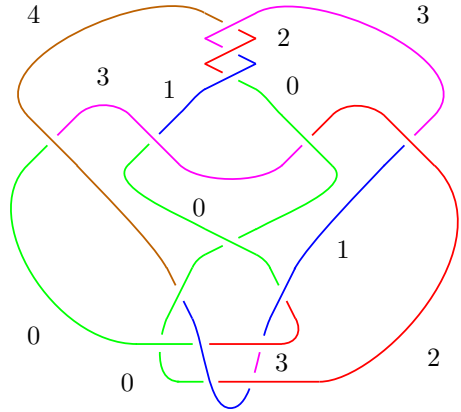}
\caption{\label{F:Fan} A $5$-coloured knot in the sense of Fox. To recover a $D_{10}$-coloured knot replace each label $i\in\mathds{Z}/n\mathds{Z}$ by $ts^i$.}
\end{figure}

\begin{example}\label{Ex:A_4}
The simplest example of a $G$--coloured knot for $G$ not metacyclic is a knot coloured by the alternating group. Each Wirtinger generator gets mapped to one of $\{t,ts_1,ts_2,ts_1s_2\}$. See Figure \ref{F:A4tref}.
\end{example}

\begin{figure}
\psfrag{a}{$t$}\psfrag{b}{$ts_1$}\psfrag{c}[c]{$ts_2$}
\includegraphics[width=90pt]{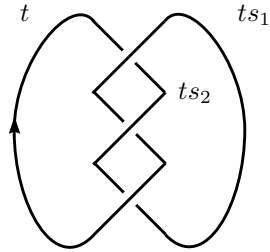}
\caption{\label{F:A4tref} An $A_4$-coloured trefoil.}
\end{figure}

\section{Surface data}\label{S:surfacedata}

Let $G= \mathcal{C}_m\ltimes_\phi A$ be a fixed metabelian homomorph of a knot group.\par

In this section we define and explore \emph{surface data}. Surface data is an analogue for $G$--coloured knots of the Seifert matrix. In particular, it admits an $S$--equivalence relation (Section \ref{SS:S-equiv-matrix}).\par

We fix some linear algebra notation for the rest of the paper. The transpose of a matrix $M$ is denoted $M^{\thinspace T}$. We write
both column vectors and row vectors as rows, but we separate row vector elements with commas and column vector elements with
semicolons. Thus $\left(v_1;\,\ldots;v_n\right)$ denotes $\left(\begin{smallmatrix}v_1\\\vdots\\v_n\end{smallmatrix}\right)$.
The number $0$ denotes a zero matrix, whose size depends on its context. The direct sum of matrices $M\oplus N$ is
$\left(\begin{smallmatrix}M & 0\\ 0 & N\end{smallmatrix}\right)$. We denote the $n\times n$ unit matrix by $I_n$. We use square brackets
for matrices over $\mathds{Z}$, and round brackets for matrices over $A$.

\subsection{$A$-coloured Seifert surfaces and covering spaces}\label{SS:ASeif}

Let $(K,\rho)$ be a $G$--coloured knot, and let $F$ be a Seifert surface for $K$. For us, a Seifert surface comes equipped with a basepoint on its boundary, an
orientation (right-hand convention), and a fixed parametrization, for instance as a zero mean curvature ``soap bubble'' surface with the parameterized knot $K$ as its boundary. Let $E(F)$ denote the exterior of $F$, which inherits a basepoint $\star_F$ from $F$ by pushing off along the positive normal.\par

Let $C_{m}(K)$ be the $m$--fold branched covering space of $K$, obtained from $E(F)$ via the standard cut-and-paste construction
(see \textit{e.g.} \cite[Chapter 5C]{Rol90}). By convention $C_{0}(K)\ass C_\infty(K)$.\par

 In this section we characterize the homomorphism $\bar\rho\co H_1\left(E(F)\right)\twoheadrightarrow A$ which arises from the restriction of $\rho$ to the complement of $F$, and the homomorphism $\tilde\rho\co H_1\left(C_m(K)\right)\twoheadrightarrow A$. This section generalizes \cite[Section 4.1.1]{KM09}, to which the reader is referred for details.\par

Write $\knotgroup$ as a semidirect product $\mathds{Z}\ltimes\knotgroup^{\prime}$. The abelianization map
$\Ab\co\ \knotgroup\twoheadrightarrow\mathcal{C}_0$ is given by $\Ab(x)=t^{\Link(x,K)}$, where
$\Link(x,K)$ equals the algebraic intersection number of $x$ with $F$. Any based loop $x$ in the complement of $F$ does
not intersect $F$. So the image of the map $\iota_*\co \pi_1 E(F)\rightarrow \knotgroup$ induced by the inclusion $\iota\co\ E(F) \hookrightarrow E(K)$ lies in $\knotgroup^\prime$. Additionally, the group $G$ factors as $G=\rho(\mathds{Z})\ltimes \rho(\knotgroup^\prime)$ with $\rho(\mathds{Z})=\mathcal{C}_m$ and
$\rho(\knotgroup^\prime)=A$ (see for instance \cite[Proposition 14.2]{BZ03}). Combining these facts tells us that the image of $\rho\circ \iota_\ast$ is contained in $A$, and we obtain a map $\rho^{(1)}\co \pi_1 E(F)\twoheadrightarrow A$. Apply the abelianization map to the domain and to the range of $\rho^{(1)}$ to obtain a map $\bar\rho\co H_1 \left(E(F)\right)\twoheadrightarrow A$, which we call the \emph{restriction of $\rho$ to the complement of $F$}.\par

In another direction, for $G\ass\mathcal{C}_m\ltimes_\phi A$ a metabelian homomorph of a knot group, a $G$--colouring $\rho$ of a knot $K$ factors as follows (see \emph{e.g.} \cite[Proposition 14.3]{BZ03}):

\begin{equation}
\begin{CD}
\rho\co \pi= \mathds{Z}\ltimes_\tau \pi^\prime @>\beta_n>> \mathcal{C}_m\ltimes_{\psi^\prime}H_1(C_m(K)) @>\rho^\prime>> G\\
 @. @VVV  @VVV\\
 \phantom{a} @. H_1(C_m(K)) @>\tilde{\rho}>> A
\end{CD}
\end{equation}

We will call $\tilde{\rho}$ the \emph{lift of $\rho$ to $C_m(K)$}.\par

The relationship between $\tilde\rho$ and $\bar\rho$ is as follows. Given a choice of $A$--coloured Seifert surface $(F,\bar\rho)$, construct $\mathrm{pr}\co C_{m}(K)\twoheadrightarrow E(K)$ by gluing together copies $R_0,\ldots,R_{m-1}$ of $E(F)$. A basis $\set{x_1,\ldots,x_{2g}}$ for $H_1(F)$ lifts to a generating set $\set{t^i\cdot x_1,\ldots,t^i\cdot x_{2g}}_{0\leq i\leq m-1}$ for $H_1(C_m)$. Choose indexes such that $t^i\cdot x_j\in R_i$ for all $i=0,\ldots m-1$ and $j=1,\ldots,2g$. This corresponds to a choice of a lift to $C_m(K)$ of $\star_F$. Then $\tilde\rho|_{R_0}=\bar\rho$. Conversely, given a choice of lift of $\star_F$, $\tilde\rho$ is recovered from $\bar\rho$ by setting $\tilde{\rho}(t^i\cdot x_j)\ass \phi^i \rho(x_j)$.\par

The discussion above is summarized by the commutative diagram below:

\begin{equation}
\begin{minipage}{300pt}\centering
\psfrag{a}[c]{$\pi$}\psfrag{b}[c]{$\pi^\prime$}\psfrag{c}[c]{$H_1(C_m(K))$}\psfrag{d}[c]{$G$}\psfrag{e}[c]{$\pi^\prime E(F)$}\psfrag{f}[c]{$H_1(E(F))$}
\psfrag{g}[c]{$A$}\psfrag{r}[c]{\fs$\rho$}\psfrag{s}[c]{\fs$\tilde\rho$}\psfrag{t}[c]{\fs$\bar\rho$}\psfrag{p}[r]{\fs$\mathrm{pr}_\ast$}
\includegraphics[width=210pt]{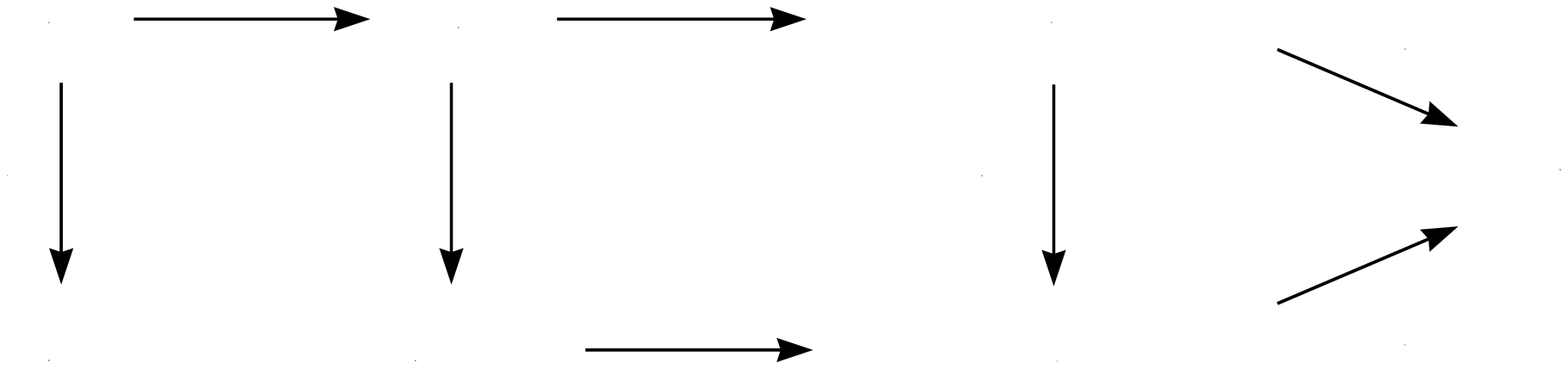}
\end{minipage}
\end{equation}

\noindent \rule{0pt}{12pt} Conditions for an $A$--colouring of $F$ to arise as a restriction of a knot colouring are given in Proposition \ref{P:HNN}, and conditions for an $A$--colouring of $C_m(K)$ to arise as a lift of a knot colouring are given in Proposition \ref{P:HNNlift}.\par

\begin{rem}\label{R:tubecomment}
Two Seifert surfaces of a knot are tube equivalent, \emph{i.e.} ambient isotopic up to addition or removal of tubes. See \textit{e.g.} \cite{BFK98,Lev65,Ric71}. However, two $A$--coloured Seifert surfaces of a $G$--coloured knot are only tube equivalent up to inner automorphism of the colouring as in Lemma \ref{L:inneriso}.
\end{rem}

\subsection{Definition of surface data}

\begin{defn}
A \emph{marked Seifert surface} for a knot $K$ is a Seifert surface $F$ for $K$, together with a choice of basis for $H_1(F)$.
\end{defn}

Let $(F,\bar\rho)$ be an $A$--coloured Seifert surface for a $G$--coloured knot $(K,\rho)$. A choice of basis $\set{x_1,\ldots,x_{2g}}$ for $H_1(F)$ induces an
\emph{associated basis} $\set{\xi_1,\ldots,\xi_{2g}}$ for $H_1\left(E(F)\right)$ which is uniquely characterized by the
condition that $\Link(x_i,\xi_j)=\delta_{ij}$ (see \textit{e.g.} \cite[Definition 13.2]{BZ03}). Let $\tau^{\pm}\co F\to
E(F)$ be the \emph{push-off maps} which take $x\in F$ to $(x,\pm1)\in F\times \{\pm 1\}\subset E(F)$. The group $A$ is abelian, and is therefore a $\mathds{Z}$--module in a unique way.

\begin{defn}
A pair $(\M,\V)$ is called \emph{surface data} for $(K,\rho)$ with respect to a marked Seifert surface $\left(F,\set{x_1,\ldots,x_{2g}}\right)$ for $K$ if:
\begin{itemize}
\item $\M=(\M_{ij})$ is the \emph{Seifert matrix} of $K$ defined by the equation
\begin{equation}\label{E:Seifertmatrix}\tau^-_\ast(x_i)=\sum_{j=1}^{2g}\M_{ij}\xi_j.\end{equation}
\item $\V$, called the \emph{colouring vector} of $(K,\rho)$ with respect to $\set{x_1,\ldots,x_{2g}}$, is defined by the equation
\begin{equation}
\V\ass \left(v_1;\,\ldots;v_{2g}\right)\ass \left(\bar\rho(\xi_1);\,\ldots;\bar\rho(\xi_{2g})\right)\in A^{2g}.
\end{equation}
\end{itemize}
Conversely, a pair $(\M,\V)$ is called \emph{surface data} if there exists a $G$--coloured knot $(K,\rho)$ and a marked Seifert surface $\left(F,\set{x_1,\ldots,x_{2g}}\right)$ for $K$ with respect to which $(\M,\V)$ is the surface data of $(K,\rho)$.
\end{defn}

The following is a direct generalization of \cite[Proposition 8]{KM09}.

\begin{prop}\label{P:HNN}[Proof in Section {\ref{SS:LinAlgProofs}}]
Let $K$ be an oriented knot with marked Seifert surface $\left(F,\set{x_1,\ldots,x_{2g}}\right)$. Corresponding to this data, there are bijections
between three sets:
\begin{enumerate}
\item The set of epimorphisms $\left\{\rho\co\knotgroup\twoheadrightarrow G\right\}$ with $\rho(\mu)=t$.
\item The set of epimorphisms $\left\{\psi\co H_1\left(E(F)\right)\twoheadrightarrow
A\right\}$ satisfying the condition that
$\psi\left(\tau^+_*(a)\right)=t\cdot\psi\left(\tau_*^-(a)\right)$ for
all $a\in H_1(F)$.
\item The set of vectors $\left\{\V\ass\left(v_1;\,\ldots;v_{2g}\right)\in A^{2g}\right\}$ satisfying:
    \begin{enumerate}
    \item\label{I:vigen}The elements of the set $\{v_1,\ldots,v_{2g}\}$ together generate $A$.
    \item The identity $\M^{\thinspace T}\thinspace\V=\M\thinspace t\cdot \V$ holds in $A^{2g}$.
    \end{enumerate}
\end{enumerate}
\end{prop}

A corollary is a simple necessary condition, which appears to be new, for a knot to be $G$--colourable.

\begin{cor}\label{C:rankbound}
If twice the genus of a knot $K$ is less than $\Rank(A)$, then there cannot exist a
surjective homomorphism $\rho\co\pi\twoheadrightarrow G$.
\end{cor}

For $A$--coloured covering spaces we have:

\begin{prop}\label{P:HNNlift}
Let $K$ be an oriented knot equipped with a marked Seifert surface $\left(F,\set{x_1,\ldots,x_{2g}}\right)$. Corresponding to this data, there are bijections between three sets:
\begin{enumerate}
\item The set of epimorphisms $\left\{\rho\co\knotgroup\twoheadrightarrow G\right\}$ with $\rho(\mu)=t$.
\item The set of epimorphisms $\left\{\psi\co H_1\left(E(F)\right)\twoheadrightarrow
A\right\}$ satisfying the condition that
$\psi(\tau(z))=t\cdot \psi(z)$ for
all $a\in H_1(F)$.
\item The set of vectors $\left\{\V\ass\left(v_1;\,\ldots;v_{2g}\right)\in A^{2g}\right\}$
satisfying:
    \begin{enumerate}
    \item The elements of the set $\{v_1,\ldots,v_{2g}\}$ together generate $A$.
    \item The vector $P\thinspace\V$ vanishes in $A^{2g}$, where $P$ is a presentation matrix for $H_1(C_m(K))$ as a $\mathcal{C}_m$-module.
    \end{enumerate}
\end{enumerate}
\end{prop}

This is the analogue of Proposition \ref{P:HNN} for lifts of $G$--colourings and it is proved in the same way \textit{mutatis mutandis}.

\subsection{$S$--equivalence}\label{SS:S-equiv-matrix}

Recall that two Seifert surfaces are \emph{tube equivalent} if they are ambient isotopic up to addition and removal of tubes. Tube equivalence is weaker than ambient isotopy, because we allow only ambient isotopy which preserves a Seifert surface (although we don't care which one).

\begin{defn}\label{D:tubequiv}
 Two $G$--coloured knots $(K_{1,2},\rho_{1,2})$ are \emph{tube equivalent} if there exist tube equivalent $A$--coloured Seifert surfaces $(F_{1,2},\rho_{1,2})$ for $(K_{1,2},\rho_{1,2})$ correspondingly.
\end{defn}

In this section, two ambient isotopic knots are considered the same, and two tube equivalent $G$--coloured knots are considered the same.\par

Two matrices $\M_{1,2}$ are \emph{$S$--equivalent} if there exists a knot $K$ and a choice $\left(F_{1,2},\set{x^{1,2}_1,\ldots,x^{1,2}_{2g_{1,2}}}\right)$ of marked Seifert surfaces for $K$, such that the Seifert matrix of $K$ with respect to $\left(F_1,\set{x^1_1,\ldots,x^1_{2g_{1,2}}}\right)$ is $\M_1$, and the
Seifert matrix with respect to $\left(F_2,\set{x^2_1,\ldots,x^2_{2g_{1,2}}}\right)$ is $\M_2$ (this is equivalent to the more standard definition of
$S$--equivalence via moves on Seifert matrices \cite{Mur65,Ric71,Tro62}, as may be seen from \cite[Proposition 4.2]{GG08}). Two knots $K_{1,2}$ are \emph{$S$--equivalent} if they share the same Seifert matrix $\M$ with respect to some choice of marked Seifert surfaces $\left(F_{1,2},\set{x^{1,2}_1,\ldots,x^{1,2}_{2g_{1,2}}}\right)$ correspondingly \cite{GG08,NS03}. This is a well-defined equivalence relation on knots modulo ambient isotopy.\par

These definitions extend to the $G$--coloured context.

\begin{defn}\label{D:Sequiv}\hfill
\begin{itemize}
\item Two surface data $(\M_1,\V_1)$ and $(\M_2,\V_2)$ are said to be \emph{$S$--equivalent} if there exists a $G$--coloured knot
$(K,\rho)$ together with a choice of marked Seifert surfaces  $\left(F_{1,2},\set{x^{1,2}_1,\ldots,x^{1,2}_{2g_{1,2}}}\right)$ for $K$, such that the surface data of $(K,\rho)$ with respect to $(F_1,\set{x^1_1,\ldots,x^2_{2g_1}})$ is $(\M_1,\V_1)$, and the surface data with respect to  $(F_2,\set{x^2_1,\ldots,x^2_{2g_2}})$ is $(\M_2,\V_2)$.
\item Two $G$--coloured knots $(K_{1,2},\rho_{1,2})$ are \emph{$S$--equivalent} if there exist Seifert surfaces $F_{1,2}$ for $K_{1,2}$
correspondingly, and bases for their first homology, with respect to which the surface data of $(K_{1},\rho_{1})$ is $S$--equivalent to the surface data of
$(K_{2},\rho_{2})$.
\end{itemize}
\end{defn}

$S$--equivalence is a well-defined equivalence relation on $G$--coloured knots modulo tube equivalence, by Naik and Stanford's proof \cite{NS03}, which is fleshed out in \cite{GG08}.

\begin{rem}\label{R:tubecom2}
$S$--equivalence would not be well-defined on $G$--coloured knots modulo ambient isotopy, because $A$--coloured Seifert surfaces corresponding to ambient isotopic $G$--coloured knots might not be tube equivalent. See Remark \ref{R:tubecomment}.
\end{rem}

Our definition of $S$--equivalence on surface data coincides with a definition in terms of moves on matrices.

\begin{prop}\label{P:Sequiv}
Two surface data are $S$--equivalent if and only if they are related a finite
sequence of the following moves and their inverses:
\begin{description}
\item[$\Lambda_1$]\[
(\M,\V)\ \mapsto (U^{\thinspace T}\thinspace \M U,U^{-1}\V)
\]\noindent where $U$ is an integral square matrix such that $\det U=\pm1$ (such a matrix is said to be \emph{unimodular}).
\item[$\Lambda_2$]
\begin{multline*}
(\M,\V)\ \mapsto \left(\thinspace\begin{aligned}
\begin{bmatrix}
\ & \ & \ & c_1 & 0\\
\ & \M & \ & \vdots & \vdots\\
\ & \ & \ & c_{2g} & 0\\
c_1 & \cdots & c_{2g} & 0 & -1\\
0 & \cdots & 0 & 0 & 0
\end{bmatrix}
\end{aligned}\ \scalebox{1.5}{,}\ \left(\begin{matrix}v_1\\ \vdots\\ v_{2g}\\ 0\\ \frac{t-1}{t}\cdot\left(\sum_{i=1}^{2g}c_i v_i\right)
\end{matrix}\right)\right)\ \mathrm{or}\\
\left(\thinspace\begin{aligned}
\begin{bmatrix}
\ & \ & \ & c_1 & 0\\
\ & \M^{\thinspace T} & \ & \vdots & \vdots\\
\ & \ & \ & c_{2g} & 0\\
c_1 & \cdots & c_{2g} & 0 & 0\\
0 & \cdots & 0 & 1 & 0
\end{bmatrix}
\end{aligned}\ \scalebox{1.5}{,}\ \left(\begin{matrix} v_1\\ \vdots\\
v_{2g}\\ 0
\\ (t-1)\cdot\left(\sum_{i=1}^{2g}c_i v_i\right)
\end{matrix}\right)\right)\end{multline*}
\noindent with $c_1,\ldots,c_{2g}$ arbitrary integers.
\end{description}
\end{prop}

\begin{proof}
If $(\M_1,\V_1)$ and $(\M_2,\V_2)$ are related by a $\Lambda_1$-move, and if $(K,\rho)$ is a $G$--coloured knot with surface data $(\M_1,\V_1)$ with respect to a choice of Seifert surface $F$ for $K$ and some choice of basis $x_1,\ldots,x_{2g}$ for $H_1(F)$, then the action of $U$ on $H_1(F)$ induces a new basis
$y_1,\ldots,y_{2g}$ for $H_1(F)$, such that the surface data for $(K,\rho)$ with respect to $(F,\set{y_1,\ldots,y_{2g}})$ is $(\M_2,\V_2)$.\par

If $(\M_2,\V_2)$ is obtained from $(\M_1,\V_1)$ by a $\Lambda_2$-move, and if $(\M_1,\V_1)$ is surface data for a $G$--coloured knot $(K,\rho)$ with respect to a choice $(F,\set{x_1,\ldots,x_{2g}})$ of marked Seifert surface, then $(\M_2,\V_2)$ arises as surface data for $(K,\rho)$ with respect to a Seifert surface $F^{\prime}=F\cup\{\text{$1$--handle}\}$ and a basis $\set{x_1,\ldots,x_{2g},x_1^{\text{new}},x_2^{\text{new}}}$ for $H_1(F^{\prime})$ as
follows:

\begin{equation}\label{E:genincrease}
 \begin{minipage}{100pt}
    \includegraphics[width=100pt]{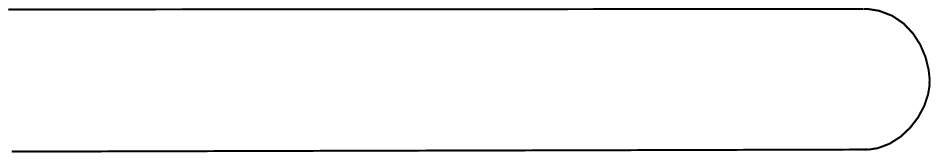}
    \end{minipage}\quad\ \ \ \overset{\text{stabilize}}{\begin{minipage}{18pt}\includegraphics[width=18pt]{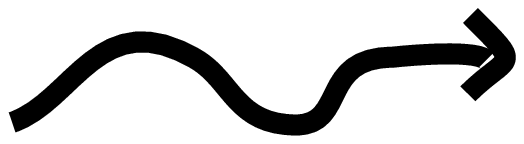}\end{minipage}}\quad
    \left\{\ \begin{array}{c}
     \raisebox{20pt}{\begin{minipage}{130pt}
     \psfrag{a}[c]{\tiny$x_1^{\text{new}}$}\psfrag{b}[c]{\tiny$x_2^{\text{new}}$}
    \includegraphics[width=130pt]{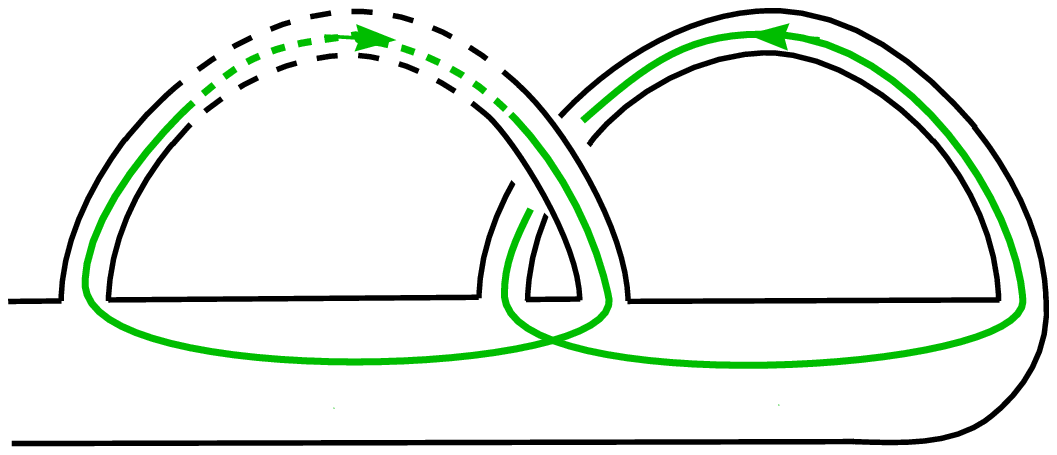}
    \end{minipage}}\\[0.6cm]
    \raisebox{20pt}{\begin{minipage}{130pt}
     \psfrag{a}[c]{\tiny$x_1^{\text{new}}$}\psfrag{b}[c]{\tiny$x_2^{\text{new}}$}
    \includegraphics[width=130pt]{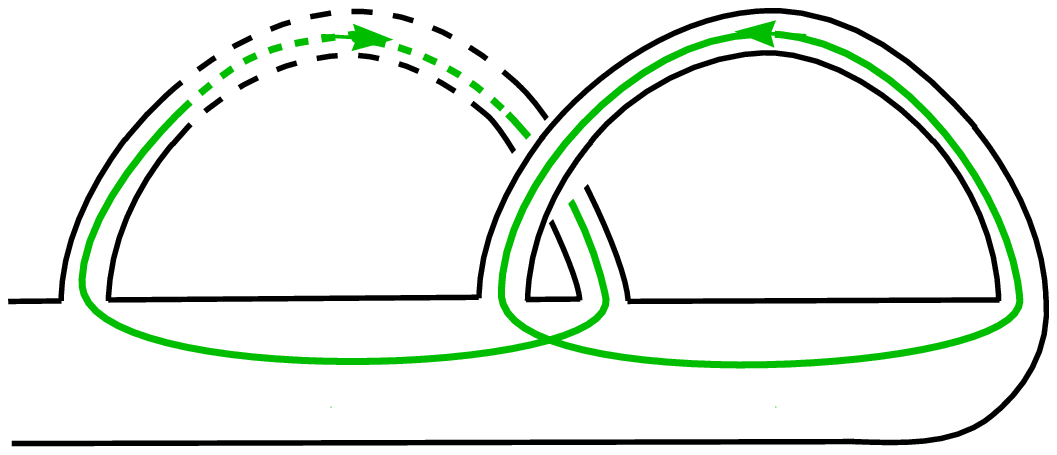}
    \end{minipage}}
    \end{array}\right.
\end{equation}

Conversely, let $(\M_1,\V_1)$ and $(\M_2,\V_2)$ be surface data for a $G$--coloured knot $(K,\rho)$ with respect to choices $(F_1,\set{x_1,\ldots,x_{2g}})$ and $(F_2,\set{y_1,\ldots,y_{2g}})$ of marked Seifert surfaces. Then, in particular, $\M_1$ and $\M_2$ are related by a finite sequence of the following moves and their inverses:

\begin{description}
\item[$\Lambda_1$]\[
\M\ \mapsto U^{\thinspace T}\thinspace \M U
\]\noindent for $U$ a unimodular matrix.
\item[$\Lambda_2$]
\begin{multline}
\M\ \mapsto
\begin{bmatrix}
\ & \ & \ & c_1 & 0\\
\ & \M & \ & \vdots & \vdots\\
\ & \ & \ & c_{2g} & 0\\
c_1 & \cdots & c_{2g} & 0 & -1\\
0 & \cdots & 0 & 0 & 0
\end{bmatrix}
\ \mathrm{or}\
\begin{bmatrix}
\ & \ & \ & c_1 & 0\\
\ & \M & \ & \vdots & \vdots\\
\ & \ & \ & c_{2g} & 0\\
c_1 & \cdots & c_{2g} & 0 & 0\\
0 & \cdots & 0 & 1 & 0
\end{bmatrix}
\end{multline}
\noindent with $c_1,\ldots,c_{2g}$ arbitrary integers.
\end{description}

For a proof, see \textit{e.g.} \cite[Theorem 5.4.1]{Mur96} or \cite[Theorem 2.3]{Ric71}. The $\Lambda_1$-move corresponds to a change of basis for $H_1(F)$, which induces the move $\V\mapsto U^{-1}\V$ on the colouring vector. The $\Lambda_2$-move corresponds to a $1$--handle attachment. Let $(v_1;\,\ldots;v_{2g};x;y)$ be the corresponding colouring vector. By the argument of \cite[Page 1371]{KM09}, for any colouring data $(\M,\V)$, the equation
$\M^{\thinspace T}\thinspace\V=\M\thinspace t\cdot \V\in A^{2g}$ holds.
Therefore:

\begin{multline}\label{E:colvecform}
\begin{bmatrix}
\ & \ & \ & c_1 & 0\\
\ & \M & \ & \vdots & \vdots\\
\ & \ & \ & c_{2g} & 0\\
c_1 & \cdots & c_{2g} & 0 & -1\\
0 & \cdots & 0 & 0 & 0
\end{bmatrix}\cdot  \left(\begin{matrix}t\cdot v_1\\ \vdots\\ t\cdot v_{2g}\\ t\cdot x\\
t\cdot y
\end{matrix}\right)-\begin{bmatrix}
\ & \ & \ & c_1 & 0\\
\ & \M^{\thinspace T} & \ & \vdots & \vdots\\
\ & \ & \ & c_{2g} & 0\\
c_1 & \cdots & c_{2g} & 0 & 0\\
0 & \cdots & 0 & -1 & 0
\end{bmatrix}\cdot
\left(
\begin{matrix}v_1\\ \vdots\\ v_{2g}\\ x\\ y
\end{matrix}
\right)\\%
\rule{0pt}{17pt}=
\left(
\begin{matrix}\M\thinspace t\cdot \V-\M^{\thinspace T}\V+\left(\sum_{i=1}^{2g}c_i\right)(t-1)\cdot x\\\rule{0pt}{16pt}
(t-1)\cdot\left(\sum_{i=1}^{2g}c_i\thinspace v_i\right)-t\cdot y\\ x
\end{matrix}
\right)\ =
\left(\begin{matrix}0\\\vdots\\0\\ 0\\ 0\end{matrix}\right).
\end{multline}

The bottom row tells us that $x= 0$, while the second lowest row tells us that $y=\frac{t-1}{t}\cdot\left(\sum_{i=1}^{2g}c_iv_i\right)$ as
required. The remaining case is proved in the same way,\textit{mutatis mutandis}.
\end{proof}

Over an integral domain, any Seifert matrix is $S$--equivalent to a non-singular matrix or to zero \cite{Lev70,Tro62}.

\begin{prop}\label{P:TrotterProp}
If $A$ is isomorphic to a vector space over an integral domain, then for any surface data $(\M,\V)$, there exists surface data
$(\M^{\prime},\V^{\prime})$ which is $S$--equivalent to $(\M,\V)$,such that the matrix $\M^{\prime}$ is non-singular.
\end{prop}

\begin{proof}
The argument of \cite[pages 484--485]{Tro62} shows that over an integral domain, any singular Seifert matrix is related by $\Lambda_1$-moves to a Seifert matrix of the form

\begin{equation}
\begin{bmatrix}
\ & \ & \ & c_1 & 0\\
\ & \M & \ & \vdots & \vdots\\
\ & \ & \ & c_{2g} & 0\\
c_1 & \cdots & c_{2g} & 0 & 0\\
0 & \cdots & 0 & 1 & 0
\end{bmatrix}.
\end{equation}

Corresponding to this Seifert matrix, by Equation \ref{E:colvecform}, the colouring vector is of the form $\left(v_1;\,\ldots;v_{2g-2};0;(t-1)\cdot\left(\sum_{i=1}^{2g}c_iv_i\right)\right)$. As $v_1,\ldots,v_{2g}$ generate $A$ as a $\mathcal{C}_m$-module, this
implies that $g>2$, and we may obtain a smaller matrix $\M^\prime$ such that $(\M^\prime,\left(v_1;\,\ldots;v_{2g-2}\right))$ is $S$--equivalent to $(\M,\V)$ by an inverse $\Lambda_2$-move. Continue until a nonsingular matrix is reached.
\end{proof}

\subsection{Proof of Proposition {\ref{P:HNN}} and of Corollary
{\ref{C:rankbound}}}\label{SS:LinAlgProofs}

\begin{proof}[Proof of Proposition {\ref{P:HNN}}]
Note first that $\nu$ normally generates $\pi$, therefore $\rho(\nu)$ normally generates $G$, and so by an inner automorphism we may set $\rho(\mu)=t$.\par

The argument of \cite[Proof of Proposition 8]{KM09} shows that there
is a bijective correspondence between three sets:

\begin{enumerate}
\item{The set of epimorphisms $\left\{\rho\co\knotgroup\twoheadrightarrow G\right\}$ with $\rho(\mu)=t$.}
\item{The set of maps $\left\{\psi\co H_1\left(E(F)\right)\twoheadrightarrow
A\right\}$ satisfying two conditions:
\begin{enumerate}
\item The image of $\psi$ generates $A$ as a $\mathcal{C}_m$-module.
\item For every $a\in H_1(F)$, we have
$\psi\left(\tau^+_*(a)\right)=\psi\left(\tau_*^-(a)\right)^{-1}$.
\end{enumerate}}
\item{\rule{0pt}{11pt} The set of vectors $\left\{\V\ass\,\left(v_1;\,\ldots;v_{2g}\right)\in
A^{2g}\right\}$ satisfying:
\begin{enumerate}
\item{The elements of the set $\left\{t^k\cdot v_1,\ldots,t^k\cdot v_{2g}\right\}_{k\in\mathds{Z}}$ together generate
$A$.}
\item{
\begin{equation*}
\M\thinspace t\cdot\V=\M^{\thinspace T}\V\in A^{2g}.
\end{equation*}}
\end{enumerate}}
\end{enumerate}

Note that our choice of distinguished meridian for $K$ means that we don't have to mod out the first set by an equivalence relation. Let $I_\V\subseteq A$ denote the ideal generated by $\left\{v_1,\ldots,v_{2g}\right\}$. It remains to prove that $I_\V$ equals $A$. Equation \ref{E:Seifertmatrix} implies that

\begin{multline}\label{E:8.8Thm}
\bar\rho\left([\mu]^{-i}\tau_\ast^+(x_1;\,\ldots;x_{2g})[\mu]^{i}\right)=\M^{\thinspace T}\thinspace t^i\cdot\V\\= \M\thinspace
t^{i+1}\cdot\V=\bar\rho\left([\mu]^{-i-1}\tau_\ast^-(x_1;\,\ldots;x_{2g})[\mu]^{i+1}\right).
\end{multline}
 Without the limitation of generality ,take $i=0\in\mathds{Z}$.\par

Because $A$ is finitely generated, it may be given the structure of a principal ideal ring. It then follows from the Chinese remainder
theorem that any solution to

\begin{equation}\label{E:MWMV}
\M\thinspace W=\M^{\thinspace T}\V
\end{equation}

\noindent must restrict to a solution of \ref{E:MWMV} over each Sylow subgroup of $A$, and if $A$ is infinite, over the integers (we would like $W$ to become $t\cdot\V$). We may therefore restrict to the case that $A$ is of the form $\mathcal{C}_{q}^{r}$ with $q$ prime or zero. The goal is to show that $W$ is unique. The ideal $I_\V$, defined as the ideal generated by the entries of $\V$, equals $A$ if and only if, for any surface date $(\M^{\prime},\V^{\prime})$ which is $S$--equivalent to $(\M,\V)$, we have $I_{\V^{\prime}}=A$. If $A$ is isomorphic to a vector space over the integers, by Proposition \ref{P:TrotterProp}, $\M$ must be $S$--equivalent to a non-singular Seifert matrix. This implies that $W$, which we know exists, is uniquely determined by Equation \ref{E:MWMV}.\par

Next, if $A$ is an abelian $p$--group, then the quotient $A/\Phi(A)$ is an elementary abelian group, where $\Phi(A)$ denotes the Frattini subgroup of $A$ (see \textit{e.g.} \cite[Section 10.4]{Hal76}). The group $A/\Phi(A)$ is isomorphic to a vector space over an integral domain (a field in fact), and we may uniquely solve Equation \ref{E:MWMV} over $A/\Phi(A)$ to give $W=M^{-1}\M^{\thinspace T}\V$.
The proposition is thus proven over an abelian $p$--group. We are finished, because by the Burnside Basis Theorem (see \textit{e.g.}
\cite[Theorem 12.2.1]{Hal76}), any lift of a solution to Equation \ref{E:MWMV} whose entries generate $A$ will be a vector in $A^{2g}$
whose entries generate $A$.
\end{proof}

Recall that a square integral matrix $P$ is said to be \emph{unimodular} if $\det P=\pm 1$, and two matrices $\M_{1,2} $ are said to be \emph{unimodular congruent} if $P^{\,T}\,\M_1 P=\M_2$ for some unimodular $P$.

\begin{proof}[Proof of Corollary {\ref{C:rankbound}}]

Because $S\ass \M^{\thinspace T}-\M$ is unimodular congruent to
$\left[\begin{smallmatrix}0 & 1\\ -1 &
0\end{smallmatrix}\right]^{\oplus g}$ (see e.g. \cite{BZ03}[Proposition 8.7]), it is invertible over any commutative ring.  Rewrite
\begin{equation}
\M\thinspace t\cdot \V=\M^{\thinspace T}\V
\end{equation}
\noindent as
\begin{equation}\label{E:MVSV}
\M \left(t-1\right)\cdot\thinspace \V= S\thinspace \V.
\end{equation}
\noindent by subtracting $\M\thinspace\V$ from both sides of the equation. Left multiply both sides by $S^{-1}$ to obtain
\begin{equation}\label{E:SMVV}
S^{-1} \thinspace \M \thinspace \left(t-1\right)\cdot\thinspace \V=\V.
\end{equation}

Because $S$ is invertible and because $(t-1)$ induces an automorphism of $A$ (see \textit{e.g.} \cite[Proposition 14.2]{BZ03}), it follows that $\Rank(\M)$ is bounded below by $\Rank(\V)$, which in turn equals the minimal number of elements in a generating set for $A$ by Proposition
\ref{P:HNN}.
\end{proof}

\section{Surgery equivalence relations between $G$--coloured knots}\label{S:rhoequiv}

In Section \ref{SS:rhoequiv}, we define equivalence relations on $G$--coloured knots whose study is the focus of this paper. The relationship between these was described in Section \ref{SS:Method}. The $\bar\rho$--equivalence relation is put into the context of a big construction (relative bordism) by Proposition \ref{P:bordbarrho}.

\subsection{The equivalence relations}\label{SS:rhoequiv}

Recall the twist move and the null-twist from Section \ref{SS:Results}, Figures \ref{F:FRMove} and \ref{F:HosteMove}, and recall tube equivalence of $G$--coloured knots from Definition \ref{D:tubequiv}. Recall also the restriction $\bar\rho$ and the lift $\tilde\rho$ of the $G$--colouring $\rho$. Consider the infinite cyclic covering

\begin{equation}\tilde{G}\ass \mathcal{C}_0\ltimes_{\tilde{\phi}} A\overset{p}{\twoheadrightarrow} \mathcal{C}_m\ltimes_\phi  A=G,\end{equation}

\noindent with $p(t^ia)\ass t^{i\bmod m}a$ for all $a\in A$. The $G$--colouring $\rho$ of $K$ pulls back to a $\tilde{G}$--colouring $\hat\rho$ of $K$, which we call the \emph{colift of $\rho$ to $\tilde{G}$}.\par

Define the following equivalence relations on the set of $G$--coloured knots.

\begin{defn}Two $G$--coloured knots $(K_{1,2},\rho_{1,2})$ are said to be:
\begin{itemize}
\item \emph{$\rho$--equivalent} if they are related up to ambient isotopy by twist moves.
\item \emph{$\hat\rho$--equivalent} if they are related up to ambient isotopy by null-twists.
\item \emph{$\bar\rho$--equivalent} if they are related up to tube equivalence by null-twists.
\item \emph{$\tilde\rho$--equivalent} if they are related up to tube equivalence by twist moves.
\end{itemize}
\end{defn}

The justification for these names is as follows. A null-twist respects a $\tilde{G}$--colouring such as $\hat\rho$, as does ambient isotopy. It may be realized as a twist moves between bands of some Seifert surface by the tubing construction, and therefore it respects an $A$--colouring of the complement of a Seifert surface, such as $\bar\rho$. A twist move respects an $A$--colouring of $C_m(K)$ such as $\tilde\rho$. Forgetting the $\mathcal{C}_m$-module structure on both sides, $\tilde\rho$ descends to a homomorphism from $H_1(C_m(K))$ onto $A$, which we call $\check{\tilde{\rho}}$, and which is preserved by tube equivalence but not by ambient isotopy of $K$. In fact $\tilde\rho$--equivalence is what we should be calling $\check{\tilde{\rho}}$--equivalence.

\subsection{Relative bordism}\label{SS:relbord}

In this section we work in the smooth category, and write the unit interval as $I\ass [0,1]$. References for this section are Conner--Floyd \cite{CF64} and Cochran--Gerges--Orr \cite{CGO01}.

\begin{defn}
Consider two compact oriented $n$--manifolds $M_{1,2}$, whose boundaries $\partial M_{1,2}$ are compact 
oriented $(n-1)$--manifolds. Fix a subgroup $H\subseteq G$, and let $f_{1,2}\co M_{1,2} \twoheadrightarrow K(G,1)$  be a pair of smooth maps which map $\partial M_{1,2}$ onto $K(H,1)$. The pairs $(M_1,f_1)$ and $(M_2, f_2)$ are said to be \emph{$(G,H)$--relative bordant} of there exists a compact oriented $n$--manifold $N$ called a \emph{connecting manifold}, a compact oriented $(n+1)$--manifold $W$, and a smooth map $F\co W\twoheadrightarrow G$ such that:
\begin{itemize}
\item $\partial N=\partial M_1\cup -\partial M_2$ and $N\cap M_{1,2}=\partial M_{1,2}$ and $\partial W=\left(M_1\amalg M_2\right)\bigcup_{\partial N} N$.
\item $F\!\mid_{M_{1,2}}=f_{1,2}$ and $F$ maps $N$ onto $K(H,1)$.
\end{itemize}
We call $(W,F)$ a relative bordism between $(M_1,f_1)$ and $(M_2, f_2)$. The $n$th $(G,H)$--relative bordism group is denoted
$\Omega_n(G,H)$. See Figure \ref{F:bordism}. 
\end{defn}

\begin{figure}
\psfrag{a}[c]{$f_1$}\psfrag{b}[c]{$F$}\psfrag{c}[c]{$f_2$}\psfrag{X}[l]{\small$((K(G,1),K(H,1))$}
\psfrag{M}{$M_1$}\psfrag{N}{$M_2$} \psfrag{O}{$N$}\psfrag{P}{$W$}
\includegraphics[width=5.5cm]{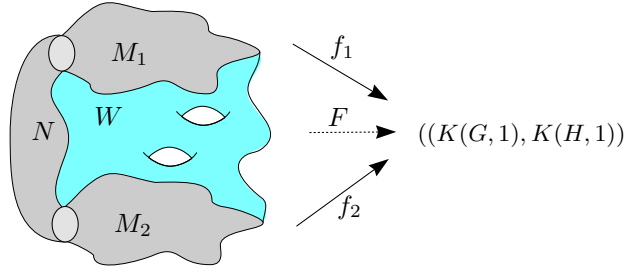}
\caption{\label{F:bordism}Relative bordism}
\end{figure}

Relative bordism of knots is defined as relative bordism of knot complements. Namely, a $G$--colouring $\rho\co \pi\twoheadrightarrow G$ induces a smooth map $f\co E(K)\to K(G,1)$ such that $\partial E(K)\subseteq K(H,1)$, where $H\ass \langle\rho(\mu),\rho(\ell)\rangle$ is the $\rho$--image of the peripheral subgroup of $\pi=\pi_1 E(K)$. For $G$ metabelian, the $\rho$--image of the longitude is trivial, and the $\rho$--image of the distinguished meridian is a generator of $\mathcal{C}_m\simeq\Ab\, G$. This motivates the following definition.

\begin{defn}\label{D:bordismrels}
Two $G$--coloured knots $(K_{1,2},\rho_{1,2})$ are:
\begin{itemize}
\item \emph{$\rho$--bordant} if there exists a $(G,\mathcal{C}_m)$--relative bordism $(W,F)$ between them, with $F\!\mid_{E(K_{1,2})}=f_{1,2}$ smooth maps induced by $\rho_{1,2}$ correspondingly.
\item \emph{$\hat\rho$--bordant} if there exists a $(\tilde{G},\mathcal{C}_0)$--relative bordism $(W,F)$ between them, with $F\!\mid_{E(K_{1,2})}=f_{1,2}$ smooth maps induced by $\hat\rho_{1,2}$ correspondingly.
\item \emph{$\bar\rho$--bordant} if there exists an $(G,\mathcal{C}_m)$--relative bordism $(W,F)$ between them, and Seifert surfaces $F_{1,2}$ for $K_{1,2}$ correspondingly, with $F\!\mid_{E(F_{1,2})}=f_{1,2}$ smooth maps induced by $\bar\rho_{1,2}$ correspondingly.
\item \emph{$\tilde\rho$--bordant} if there exists a $(G,\mathcal{C}_m)$--relative bordism $(W,F)$ between them, with $F\!\mid_{E(K_{1,2})}=f_{1,2}$ smooth maps induced by $\check{\tilde{\rho}}_{1,2}$ correspondingly.
\end{itemize}
\end{defn}

\begin{example}
Two $\mathcal{C}_n$-coloured knots are $\Link$-bordant if and only if they are bordant.
\end{example}

\subsection{Surgery}\label{SS:surgery}

Given an $n$--manifold $X$ and an embedding $\varphi\co S^{n-i}\times D^{i}\subset X$ with $1\leq i\leq n$, we may form a new $n$--manifold
\begin{equation}
X^{\prime}\ass\ \left(X-\mathrm{int}\,\mathrm{im}\varphi\right)\cup_{\varphi\mid_{S^{n-i}\times S^{i-1}}}\left(D^{n-i+1}\times S^{i-1}\right)\end{equation}
\noindent by cutting out $S^{n-i}\times D^{i}$ and gluing in $D^{n-i+1}\times S^{i-1}$. This process is called \emph{$i$-handle attachment}. In this paper, \emph{surgery} means $2$--handle attachment to a $3$--manifold (so by ``surgery'' we mean ``integral Dehn surgery''). The \emph{trace} of an $i$--handle attachment is the bordism
\begin{equation}
W^{\prime}\ass\ \left(X\times I\right)\cup_{S^{n-i}\times D^{i}\times\{1\}}\left(D^{n-i+1}\times D^{i}\right).
\end{equation}
\noindent Such a bordism is called \emph{elementary}. In the case of surgery, call $\varphi(S^1)$ with its induced framing a \emph{surgery component}, and call its image in the trace of the surgery the \emph{attaching curve} for the $2$--handle $D^2\times D^2\subset W^{\prime}$. By the Pontryagin construction, $X^\prime$ depends only on the attaching curve.\par
By the fundamental theorem of Morse theory every bordism has a handle decomposition, and therefore can be represented as a union of elementary bordisms. To remind the reader, given a bordism $W$ between $n$-manifolds $M_{1,2}$, a handle decomposition is a diffeomorphism from $W$ to a $4$--manifold obtained by attaching handles to the cylinder $M_{1}\times I$, where the handles may be assumed to be attached in disjoint \emph{times slices} of the form $M_{1}\times [h,h+\epsilon]$.\par

We pass to the relative setting.

\begin{defn}
A \emph{surgery description of $(M_2,f_2)$ in $(M_1,f_1)$} is a relative bordism $(W,F)$ between $(M_1,f_1)$ and $(M_2, f_2)$ such that $W$ is homeomorphic to the cylinder $M_1\times I$ with $2$--handles attached, and $F$ is an extension of $f_1$ over the cylinder and over the $2$--handles.
\end{defn}

\begin{example}
Any $\mathcal{C}_n$-coloured knot has a surgery description in the complement of the $\mathcal{C}_n$-coloured unknot. This is a special case of the Lickorish--Wallace Theorem, that every $3$--manifold has a surgery description, which in the bordism setting follows from the result of Rokhlin that the bordism group of $3$--manifolds is trivial (\cite{Rok51}, see also \cite{Rou85} for a pretty proof).
\end{example}

Each bordism equivalence relation in Definition \ref{D:bordismrels} has a corresponding surgery equivalence relation.

\begin{defn}\label{D:surgrels}
Let $\psi\in\{\rho,\hat\rho,\bar\rho,\tilde\rho\}$. Two $G$--coloured knots $(K_{1,2},\rho_{1,2})$ are \emph{$\psi$--surgery equivalent} if there is a $\psi$--bordism $(W,F)$ between them such that $W$ is homeomorphic to the cylinder $E(K_1)\times I$ with $2$--handles attached.
\end{defn}

\begin{rem}
In the language of \cite{KM09}, two $G$--coloured knots in $S^3$ are related by surgery in $\ker\rho$ if and only if they are $\rho$--surgery equivalent.
\end{rem}

\subsection{Relationships between equivalence relations}\label{SS:rhorels}

The following is the main proposition of Section \ref{S:rhoequiv}.

\begin{prop}\label{P:bordbarrho}
The following conditions are equivalent:

\begin{enumerate}
\item $\bar\rho$--bordism.
\item $\bar\rho$--surgery equivalence.
\item $\bar\rho$--equivalence.
\end{enumerate}
\end{prop}

\begin{proof}
\hfill
\begin{description}
\item[$1\Rightarrow 2$]
We mimic the arguments of \cite[Section 4.3]{LiWal08} and \cite[Proof of Theorem 4.2]{CGO01} (see either source for details). Let $(W,F)$ be a $\bar\rho$--bordism between $(K_{1,2},\rho_{1,2})$.  Forgetting Seifert surfaces, in particular $(W,F)$ is a $\hat\rho$--equivalence. The boundary of the connecting manifold $N\subset W$ consists of two disjoint copies of $T^2$. The closed $3$--manifold $N\cup_{T^2\sqcup T^2}\left(T^{2}\times I\right)$ is an element of $\Omega_3(\mathcal{C}_0)\simeq \set{1}$. Therefore there exists a $\hat\rho$--bordism $W^\prime$ between $(K_{1,2},\rho_{1,2})$ with connecting manifold $T^2\times I$. Take a smooth handle decomposition of $W^\prime$ relative to the boundary as $\left(E(K_1)\times I\right)\cup\{\text{$2$--handles}\}$ by the standard argument (see \textit{e.g.} \cite[Section 5.4]{GS91}). This gives rise to a $\hat\rho$--surgery equivalence $(W^\prime,F^\prime)$. Choose Seifert surfaces $F_{1,2}$ for $K_{1,2}$ correspondingly. The induced restriction $\bar\rho_{2}^\prime$ of $\rho_2$ is related to $\bar\rho_2$ by an inner automorphism of $G$. Therefore $(K_{2},\bar\rho_2)$ and $(K_2,\bar\rho_2^\prime)$ are related by ambient isotopy (Lemma \ref{L:inneriso}), realized by a second $\hat\rho$--surgery equivalence $(W^{\prime\prime},F^{\prime\prime})$ with connecting manifold $T^2\times I$. Thus, \begin{equation}(W_{\text{srg}},F_{\text{srg}})\ass\left(W^\prime\cup_{E(K_2)}W^{\prime\prime},F^\prime\cup_{\bar\rho_2^\prime} F^{\prime\prime}\right)\end{equation}
\noindent becomes a $\bar\rho$--surgery equivalence between $(K_{1,2},\rho_{1,2})$.

\item[$2\Rightarrow 3$] We imitate the argument of \cite[Proof of Theorem 1.1]{LiWal08} and \cite[Proof of Theorem 4.2]{CGO01}.
``Filling in'' the connecting manifold $T^2\times I$ with a solid torus times an interval turns $W_{\text{srg}}$ into a surgery description of $S^3$. The Kirby Theorem implies that a surgery description of $S^3$ can be transformed to a $\pm1$--framed unlink by blow-ups and handle-slides, changing the handle decomposition of $W_{\text{srg}}$. Writing the unlink as $L\ass L_1\cup\cdots\cup L_\nu$, slide each $L_i$ (an attaching circle for a $2$--handle) to the time-slice $E(K_1)\times \left[\frac{i-1}{\nu},\frac{i}{\nu}\right]$. This induces a decomposition of $W_{\text{srg}}$ as a union of elementary $\bar\rho$--bordisms
    \begin{equation}W_{\text{srg}}=\bigcup_{i=1}^\nu E(K_i)\times \left[\frac{i-1}{\nu},\frac{i}{\nu}\right]\cup_{L_i} H_i.\end{equation}
    For $i=1$, the $G$--colouring $\rho_1$ induces $f_{1}\co E(K_1)\times \left[0,\frac{1}{\nu}\right]\to G$ which extends over the $2$--handle $H_1$. Therefore $L_1$ represents an element in $\ker\rho$. We may represent $L_1$ as an unknot which rings $2r$ strands in $K_1$ by pushing $L_1$ down to $E(K_1)\times\set{0}$ (note that $\Link(K_1,L_1)=0$). Thus, surgery around $L_1$ is a null-twist. The same argument show that surgeries around $L_2,\ldots,L_\nu$ are all null-twists.

\item[$3\Rightarrow 1$]
 Figure \ref{F:surg-real}, and tubing, shows how to realize a null-twist as an (elementary) $\bar\rho$--bordism.

\begin{figure}
\begin{minipage}{65pt}
\psfrag{a}[c]{\footnotesize$g_1$}\psfrag{b}[c]{\footnotesize$g_2$}\psfrag{c}[c]{}\psfrag{d}[c]{\footnotesize$g_{2r}$}
\psfrag{e}[c]{\footnotesize$g_1$}\psfrag{f}[c]{\footnotesize$g_2$}\psfrag{g}[c]{}\psfrag{h}[c]{\footnotesize$g_{2r}$}
\includegraphics[height=140pt]{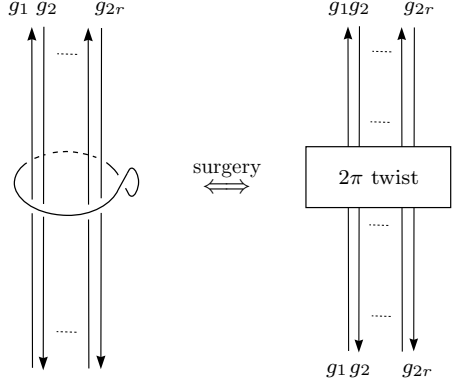}
\end{minipage}
 $\overset{\raisebox{2pt}{\scalebox{0.8}{\text{surgery}}}}{\Longleftrightarrow}$\quad \ \
\begin{minipage}{60pt}
\psfrag{a}[c]{\footnotesize$g_1$}\psfrag{b}[c]{\footnotesize$g_2$}\psfrag{c}[c]{}\psfrag{d}[c]{\footnotesize$g_{2r}$}
\psfrag{e}[c]{\footnotesize$g_1$}\psfrag{f}[c]{\footnotesize$g_2$}\psfrag{g}[c]{}\psfrag{h}[c]{\footnotesize$g_{2r}$}
\psfrag{p}[c]{\footnotesize$2\pi$\
twist}
\includegraphics[height=140pt]{HosteMove-3}
\end{minipage}
\caption{\label{F:surg-real}Realizing a null-twist by surgery.}
\end{figure}

\end{description}
\end{proof}

\begin{rem}\label{R:LW2}
Litherland and Wallace conjectured the analogue of Proposition \ref{P:bordbarrho}, replacing $\bar\rho$ by $\rho$.
\end{rem}

The above proposition helps us to understand $\bar\rho$--equivalence in two ways. First, it puts it in the framework of relative bordism, which is a ``bigger construction'', by showing that every $\bar\rho$--bordism can be `upgraded' to a surgery presentation. Relative bordism can be calculated homologically, because, for $i\leq 3$, the group $\Omega_i(G,H)$ is isomorphic to the relative homology group $H_i(G,H)$ (see {\textit{e.g} \cite[{Theorem \textrm{IV}.7.37}]{Rud08}}). This leads to an upper bound of $\abs{H_3(G,\mathds{Z})}$ for the number of $\bar\rho$--equivalence classes. We calculate $H_3(G,\mathds{Z})$ by first applying the Lyndon--Hochschild--Serre spectral sequence (\textit{e.g.} \cite[{Chapter \textrm{VII}, Section 6}]{Bro82}) to identify it with $H_0(\mathds{Z};H_3(A;\mathds{Z}))\simeq H_3(A;\mathds{Z})$ and calculate the latter following Cartan \cite{Car54}. Summarizing:

\begin{cor}\label{C:Wallacebound}
The number of $\bar\rho$--equivalence classes is bounded above by the order of $\abs{H_3(A;\mathds{Z})}$.
\end{cor}

The local-move description of $\bar\rho$--equivalence is a ``small construction'' which is good for making explicit calculations.

\begin{rem}\label{R:LiWal}
The above argument, applied in the paper of Litherland and Wallace \cite{LiWal08}, would have led to a sharp upper bound of $n$ instead of $2n$ for the number of $\rho$--equivalence classes of $D_{2n}$--coloured knots. Two $\bar\rho$--equivalent $G$--coloured knots are $\rho$--equivalent, and $n$ is an upper bound for the number of $\bar\rho$--equivalence classes by the above homological calculation.
\end{rem}

\begin{rem}\label{R:tightenbordism}
The complex $(K(G,1),S^1)$ has a $\mathds{Z}$--action by conjugation by $t$, corresponding to ambient isotopy of the knot as in the proof of Lemma \ref{L:inneriso}. Equivariant bordisms with respect to this action would correspond to $\hat{\rho}$--equivalence, and so would lead to a tighter upper-bound on the number of $\rho$--equivalence classes.
\end{rem}

\section{An algebraic characterization of $\bar\rho$--equivalence}\label{S:ClasperProof}

The finitely generated abelian group $A$ is given the structure of a principle ideal ring, which by abuse of notation we also call $A$.

\subsection{Result statement}
 A celebrated result of Naik and Stanford states that the $\Delta$--move generates $S$--equivalence \cite{NS03}. Translated into the language of claspers (recalled in Section \ref{SS:clasper-review}), this is equivalent to saying that for any $S$--equivalent knots $K_{1,2}$ there exists a Seifert surface $F_1$ for $K_1$ and a set of $Y$--claspers $C=\set{Y_1,\ldots,Y_k}$ in the complement of $F_1$, such that surgery around $C$ gives $K_2$. In the $G$--coloured context, leaves $A^i_{1,2,3}$ of clasper $Y_i$ come equipped with colours $a^i_{1,2,3}\in A$ correspondingly, and we can associate to $(K_{1,2},\rho_{1,2})$ the sum of their triple wedge products in $\bigwedge^{ 3} A$--- the \emph{$Y$--obstruction}
$Y((K_1,\bar\rho_1),(K_{2},\bar\rho_2))$. The $Y$--obstruction is independent of the choices made in its construction. The goal of this section is to prove the following theorem.

\begin{thm}\label{T:clasperprop}
Two $S$--equivalent $G$--coloured knots $(K_{1,2},\rho_{1,2})$ are $\bar\rho$--equivalent if and only if their $Y$--obstruction vanishes.
\end{thm}

In the special case $\Rank(A)\leq 2$, the group $\bigwedge^{ 3} A$ vanishes, and Theorem \ref{T:clasperprop} becomes that $S$--equivalence implies $\bar\rho$--equivalence. We sketch a proof of this (simpler) claim for $\Rank(A)= 2$, as the rank $1$ case follows from analogous arguments. This offers a shortcut through this section for the reader interested only in such groups. Let $s_{1,2}$ be generators of $A$. Engineer a band projection for $F_1$ by Section \ref{SSS:Shorten} so that entries in the corresponding colouring vector are all elements of the set $\set{0,\pm s_{1},\pm s_{2}}$. Any $\Delta$--move between bands is then realized by null-twists, by the proofs of Lemmas \ref{L:(0,a,b)} and \ref{L:(a,a,b)}.

\subsection{Review of clasper calculus}\label{SS:clasper-review}
One use of clasper calculus is to provide a graphical language to prove theorems of the form ``two objects in class $C$ are related by a finite sequence of local moves $M$ if and only if they share homological information $I$''. Examples of such theorems are in \cite{GR04,Mas03,Mat87,MN89}. Theorem \ref{T:clasperprop} is of such form. Our definitions follow \cite[Section 2]{Hab00}, but are simplified because we require only a small segment of clasper calculus. Conventions which differ from those of Habiro are written in \textbf{bold} font.

A \emph{basic clasper} is defined to be a union of three oriented embedded objects $C\ass A_1\cup A_2\cup E \subset S^3$ with $A_{1,2}$ \textbf{zero-framed unknots bounding disjoint discs} and $E$ an oriented $\frac{1}{2}\mathds{Z}$--framed line segment such that  $E\cap A_{1,2}$ are a pair of points in $S^3$. Framing $\frac{1}{2}$ and $-\frac{1}{2}$ on $E$ are graphically represented as \begin{minipage}{30pt}\includegraphics[width=30pt]{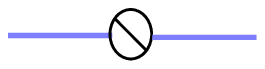}\end{minipage} and
\begin{minipage}{30pt}\includegraphics[width=30pt]{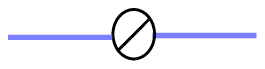}\end{minipage} correspondingly. Unknots $A_1$ and $A_2$ are called \emph{leaves} of $C$, while $E$ is called the \emph{edge} of $C$. Basic claspers provide a graphical notation for linkage as in Figure \ref{F:Habiro1}.

\begin{figure}
\begin{minipage}{75pt}
\psfrag{a}[l]{\small$A_1$}\psfrag{b}[l]{\small$E$}\psfrag{c}[l]{\small$A_2$}\psfrag{x}[c]{\small$X$}
\includegraphics[width=75pt]{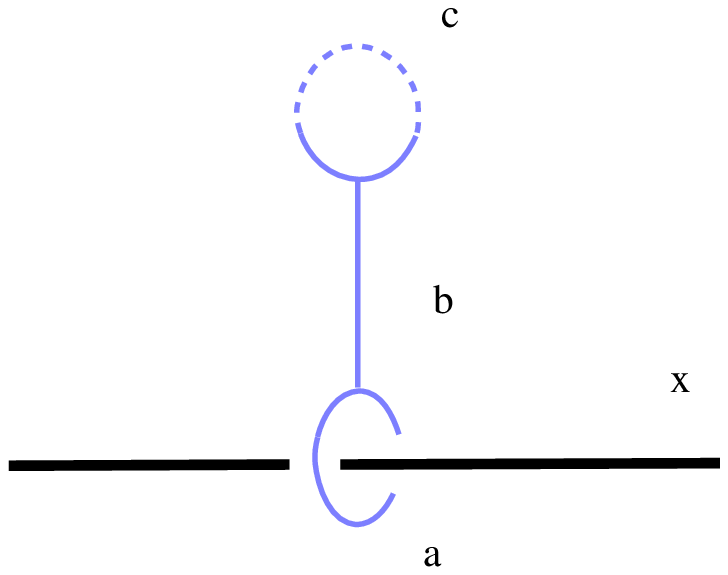}
\end{minipage}\quad\ $\overset{\raisebox{2pt}{\scalebox{0.8}{\text{surgery}}}}{\Longleftrightarrow}$
\raisebox{4pt}{\begin{minipage}{75pt} \psfrag{x}[c]{\small$X$}
\includegraphics[width=75pt]{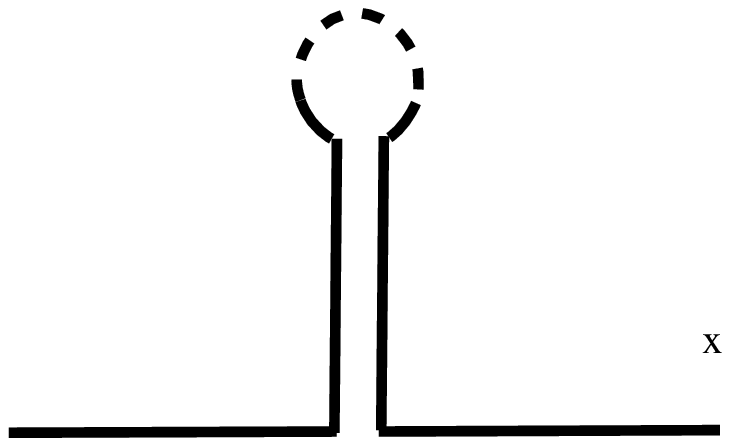}
\end{minipage}}
$\overset{\raisebox{2pt}{\scalebox{0.8}{\text{surgery}}}}{\Longleftrightarrow}$
\begin{minipage}{75pt}
\psfrag{a}[l]{\small$L_{1}$}\psfrag{b}[l]{\small$L_2$}\psfrag{x}[c]{\small$X$}
\includegraphics[width=75pt]{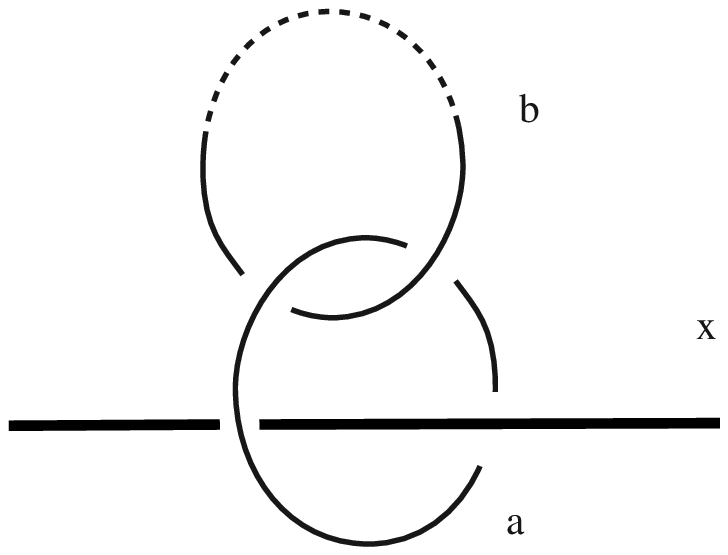}
\end{minipage}
\caption{\label{F:Habiro1}How to interpret a basic clasper, both directly on the diagram, and as surgery on a $0$--framed Hopf link.
 The thick line $X$ is some collection of arcs of a knot, and of clasper edges.}
\end{figure}

A \emph{clasper} $C\ass\mathbf{A}\cup G \subset S^3$ is a collection $\mathbf{A}\ass\ A_1\cup\ldots\cup A_k$ of \textbf{zero-framed unknots bounding disjoint discs} together with an oriented embedded uni-trivalent graph $G$ whose trivalent vertices are \textbf{oriented counterclockwise} and each of whose edges is half-integer framed, such that $\mathbf{A}\cap G$ equals the set of $1$--valent vertices of $G$ in $S^3$, and each leaf $A_i\subset \mathbf{A}$ meets $G$ at a single point $l_i\in \mathbf{A}\cap G$. Thus, a \emph{simple clasper} is a clasper with two leaves.\par

\begin{figure}
\begin{minipage}{75pt}
\includegraphics[width=75pt]{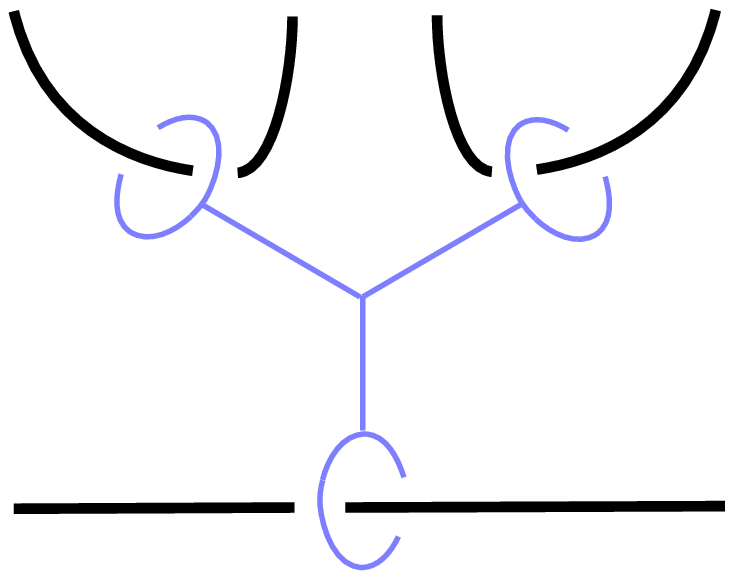}
\end{minipage}\quad\quad
$\overset{\raisebox{2pt}{\scalebox{0.8}{\text{surgery}}}}{\Longleftrightarrow}$\quad\
\begin{minipage}{75pt}
\includegraphics[width=75pt]{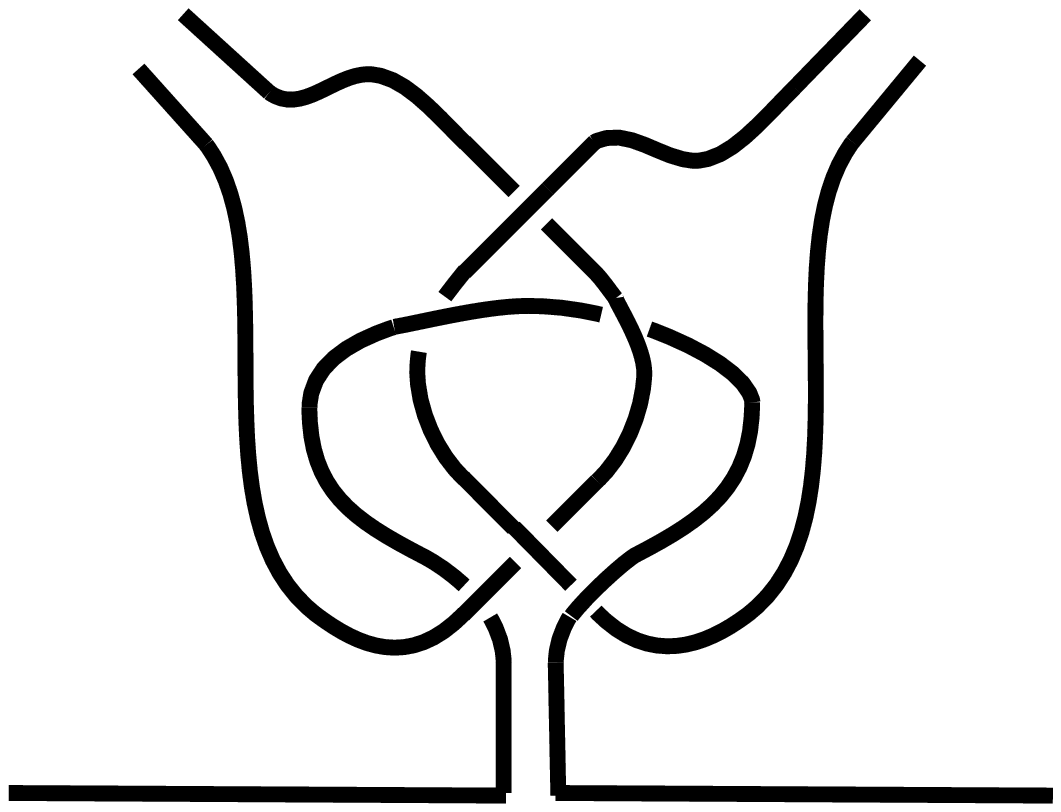}
\end{minipage}
\caption{\label{F:YClasper} How to interpret a $Y$--clasper.}
\end{figure}

Another useful class of claspers is \emph{$Y$--claspers}, interpreted in Figure \ref{F:YClasper}. Boxes are a useful graphical shorthand, as described in Figure \ref{F:Habiro2}.

\begin{figure}
\begin{minipage}{54.8pt}
\psfrag{a}[l]{\small$A_1$}\psfrag{b}[l]{\small$E$}\psfrag{c}[l]{\small$A_2$}\psfrag{x}[c]{\small$X$}
\includegraphics[height=50pt]{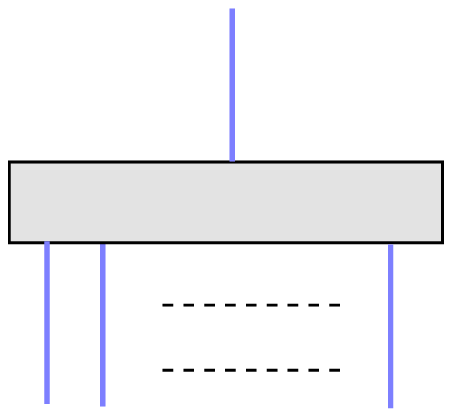}
\end{minipage}
\quad$\Longleftrightarrow$\quad
\begin{minipage}{75pt}
\psfrag{a}[l]{\small$A_1$}\psfrag{b}[l]{\small$E$}\psfrag{c}[l]{\small$A_2$}\psfrag{x}[c]{\small$X$}
\includegraphics[height=50pt]{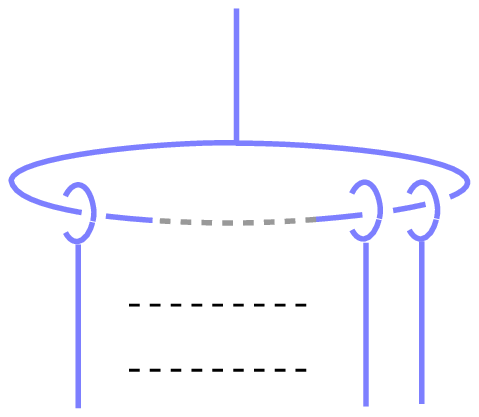}
\end{minipage}
\caption{\label{F:Habiro2}The box notation, as in \cite[Figure 44]{Hab00}.}
\end{figure}

\begin{figure}
\begin{minipage}{75pt}
\includegraphics[height=75pt]{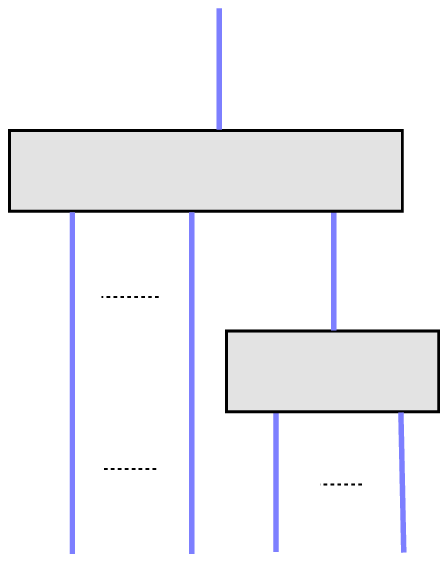}
\end{minipage}
\ \ $\Longleftrightarrow$\quad\ \
\begin{minipage}{75pt}
\includegraphics[height=75pt]{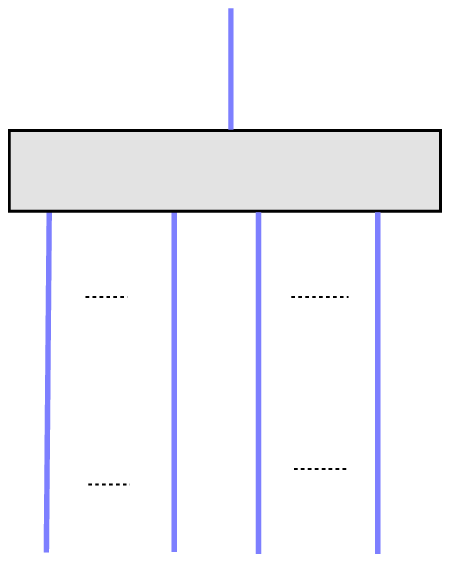}
\end{minipage}
\ \ $\Longleftrightarrow$\quad\ \
\begin{minipage}{75pt}
\includegraphics[height=75pt]{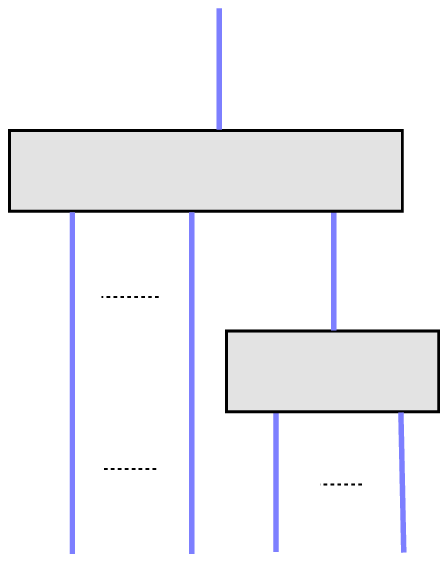}
\end{minipage}
 \caption{\label{F:unitebox}The unite-box move as in \cite[Figure 37]{Hab00}.}
\end{figure}

We make repeated use of Habiro's twelve moves \cite[Page 14--15]{Hab00}\footnote{For easy reference, the reader might want to print out \cite{MosXX}.}, to which we add an additional \emph{unite-box move} described in Figure \ref{F:unitebox}.

\subsection{Review of $\Delta$--Moves}

The following proposition describes four equivalent ways to define the $\Delta$--move. It is well-known, but the author could find no reference for it in the literature.

\begin{prop}
The following local moves are equivalent:
\begin{subequations}\label{E:Delta}
\begin{equation}
\begin{minipage}{80pt}
\includegraphics[width=70pt]{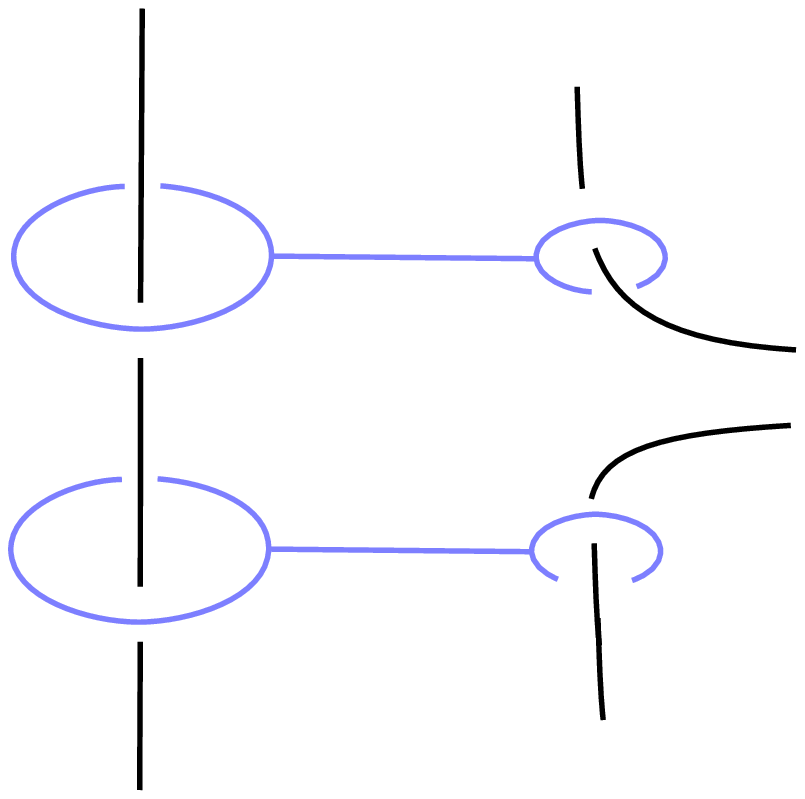}
\end{minipage}\quad\overset{\raisebox{2pt}{\scalebox{0.8}{$\Delta_1$}}}{\Longleftrightarrow}\quad
\begin{minipage}{80pt}
\includegraphics[width=70pt]{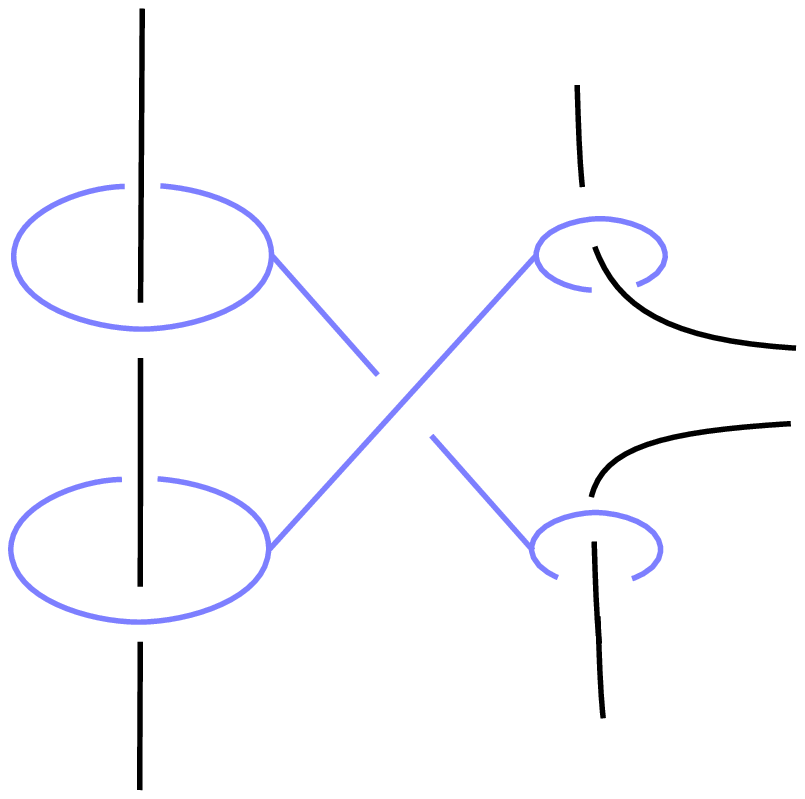}
\end{minipage}
\end{equation}
\begin{equation}
\psfrag{a}[c]{}\psfrag{b}[c]{}\psfrag{c}[c]{}
\begin{minipage}{80pt}
\includegraphics[width=75pt]{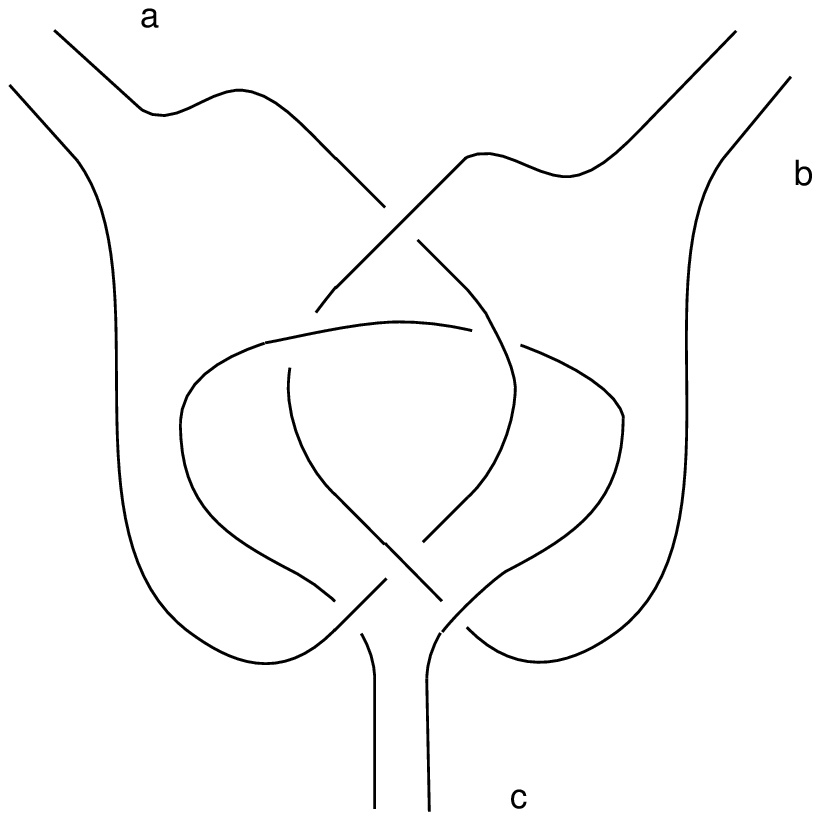}
\end{minipage}\quad\overset{\raisebox{2pt}{\scalebox{0.8}{$\Delta_2$}}}{\Longleftrightarrow}\quad
\begin{minipage}{80pt}
\includegraphics[width=75pt]{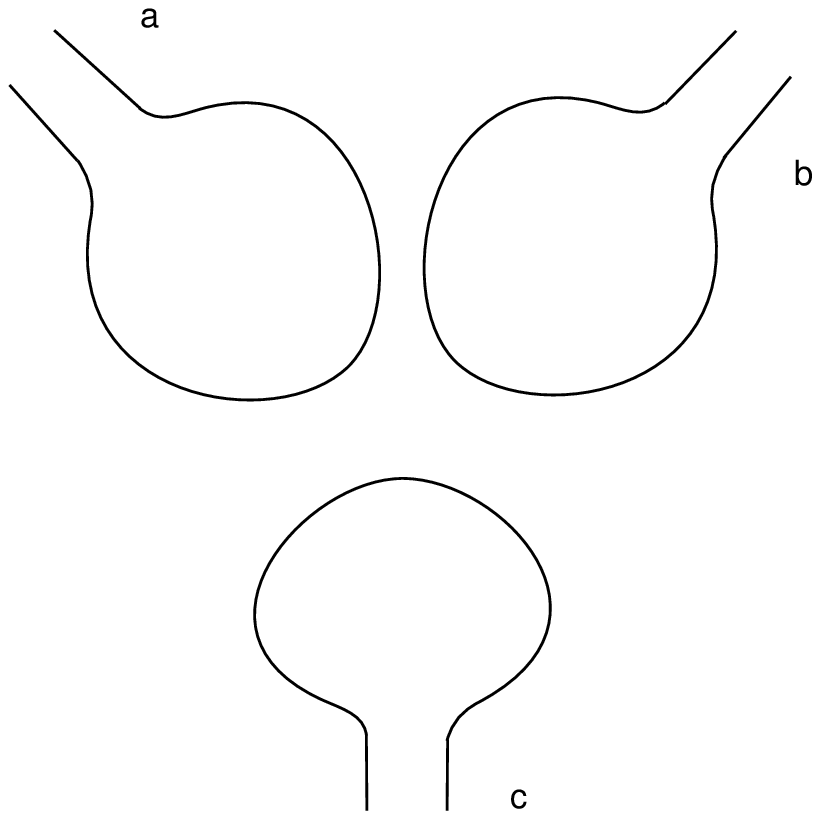}
\end{minipage}
\end{equation}
\begin{equation}
\psfrag{a}[c]{}\psfrag{b}[c]{}\psfrag{c}[c]{}
\begin{minipage}{80pt}
\includegraphics[width=75pt]{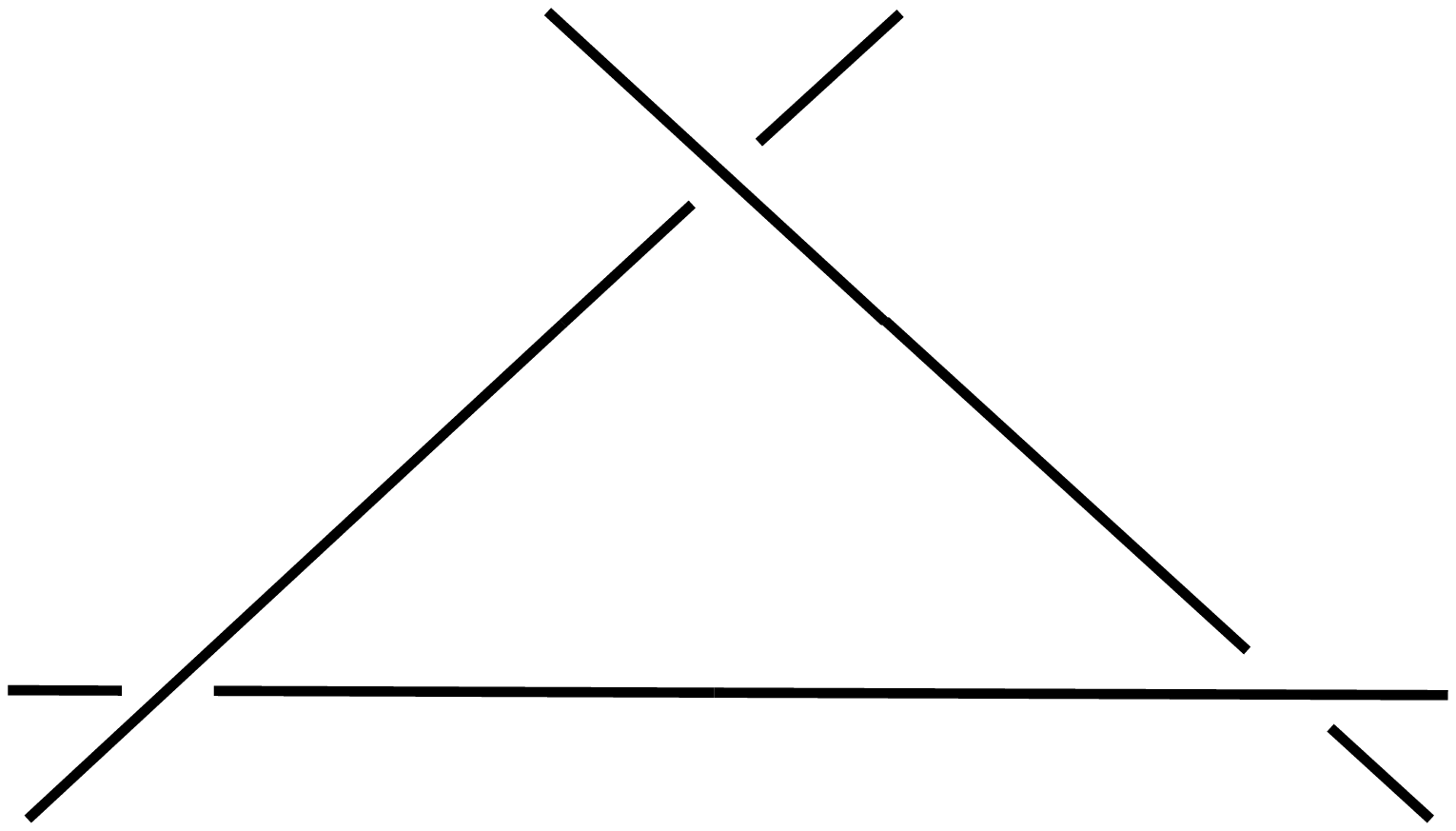}
\end{minipage}\quad\overset{\raisebox{2pt}{\scalebox{0.8}{$\Delta_3$}}}{\Longleftrightarrow}\quad
\begin{minipage}{80pt}
\includegraphics[width=75pt]{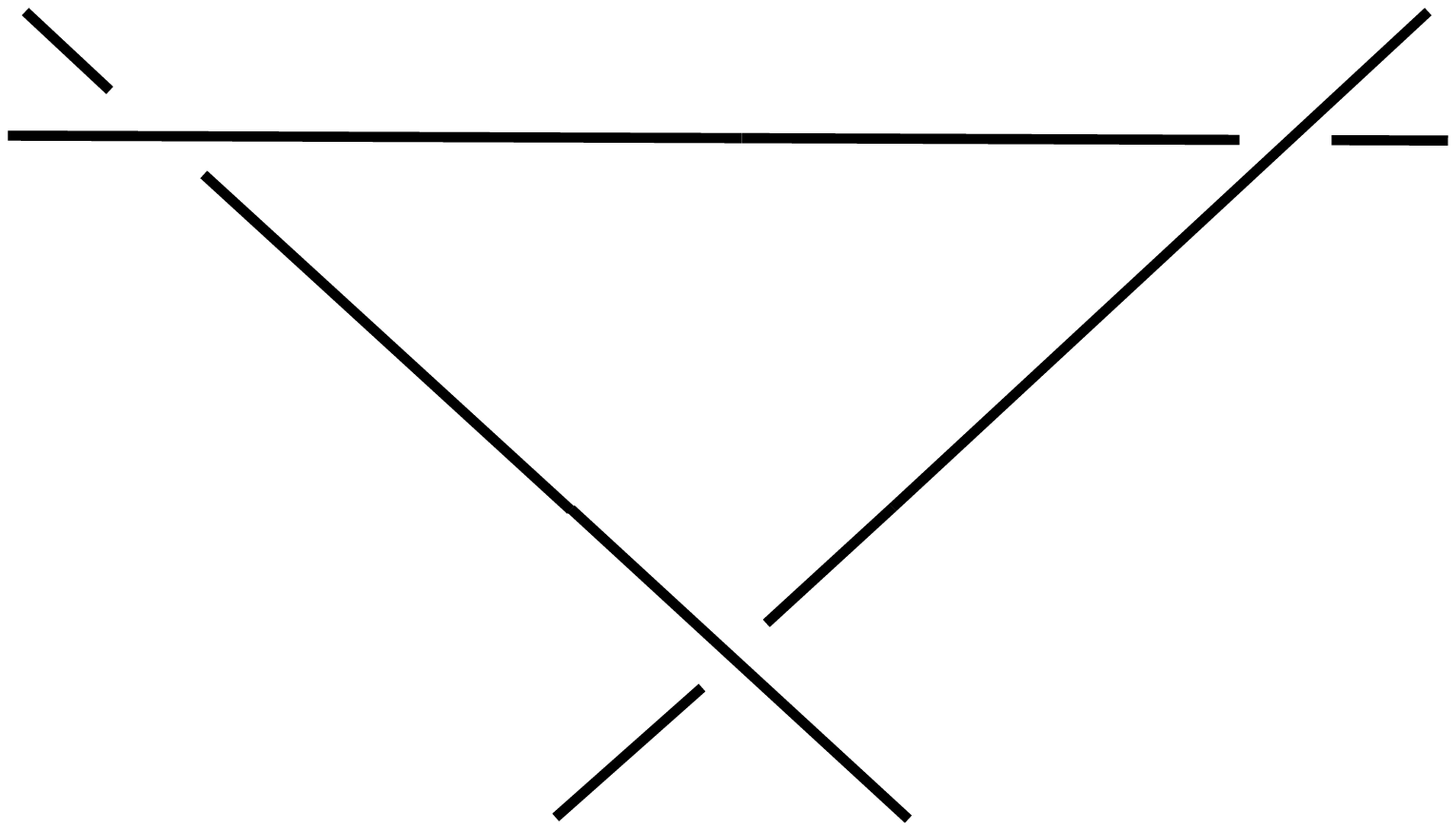}
\end{minipage}
\end{equation}
\begin{equation}
\begin{minipage}{80pt}
\includegraphics[width=80pt]{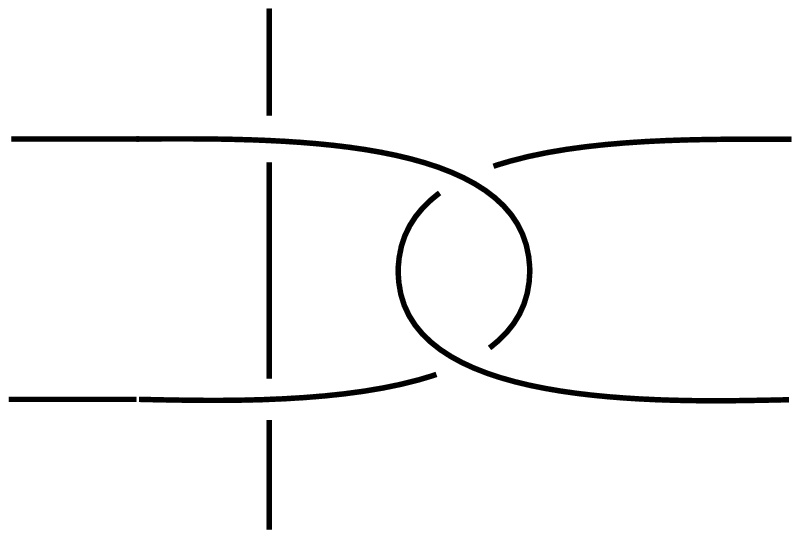}
\end{minipage}\quad\overset{\raisebox{2pt}{\scalebox{0.8}{$\Delta_4$}}}{\Longleftrightarrow}\quad
\begin{minipage}{80pt}
\includegraphics[width=80pt]{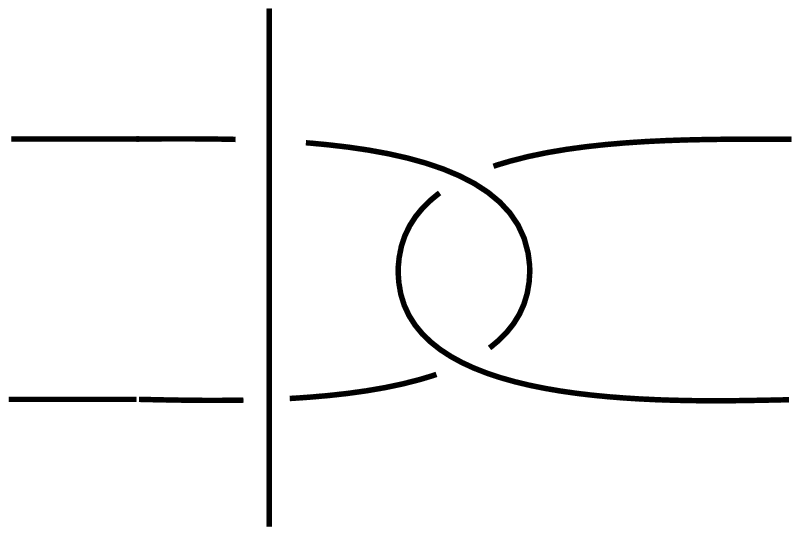}
\end{minipage}
\end{equation}
\end{subequations}
Define the \emph{$\Delta$--move} to be any of the above.
\end{prop}

\begin{proof}\hfill
\begin{description}
\item[$\Delta_1\Rightarrow \Delta_2$]
\begin{multline*}
\psfrag{a}[c]{}\psfrag{b}[c]{}\psfrag{c}[c]{}
\begin{minipage}{80pt}
\includegraphics[width=75pt]{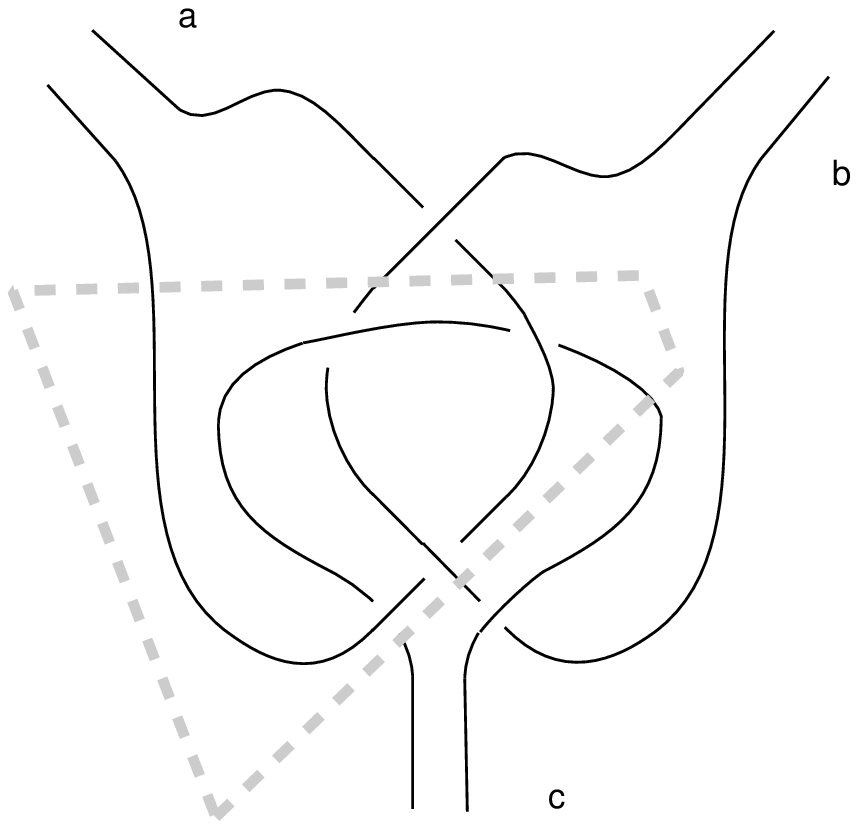}
\end{minipage}\ \ \ \overset{\text{zoom in}}{\begin{minipage}{30pt}\includegraphics[width=30pt]{fluffyarrow}\end{minipage}}\quad
\begin{minipage}{90pt}
\includegraphics[width=90pt]{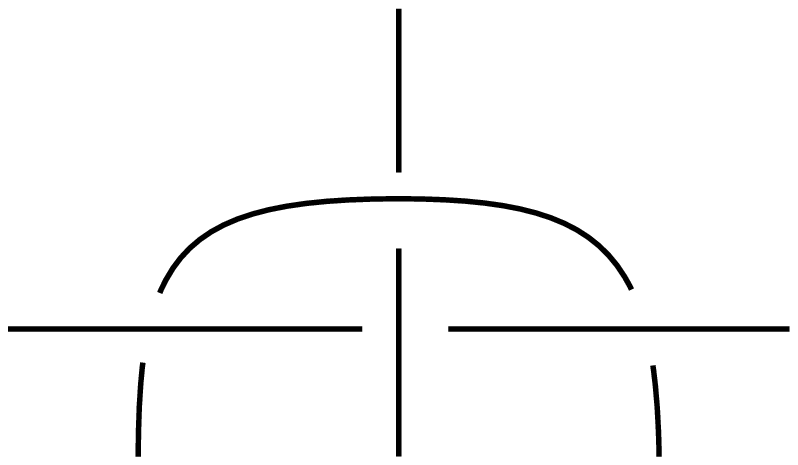}
\end{minipage}
\ \ \ \overset{\raisebox{2pt}{\scalebox{0.8}{\text{surgery}}}}{\Longleftrightarrow}\ \
\begin{minipage}{90pt}
\includegraphics[width=90pt]{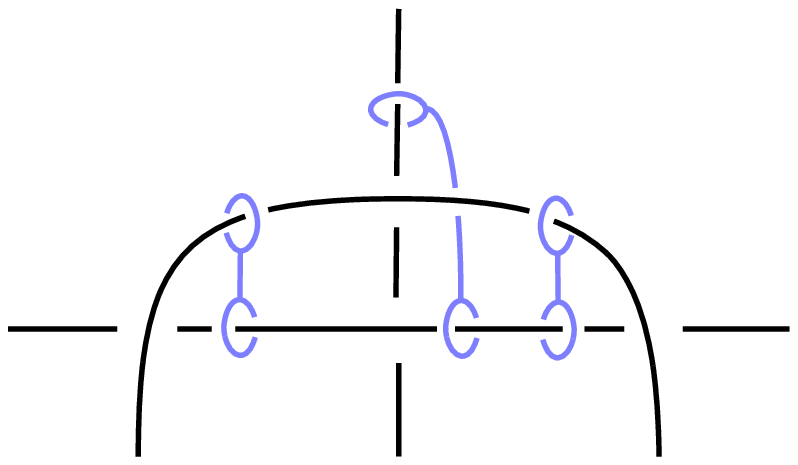}
\end{minipage}\\
\overset{\raisebox{2pt}{\scalebox{0.8}{$\Delta_1$}}}{\Longleftrightarrow}\
\begin{minipage}{85pt}
\includegraphics[width=85pt]{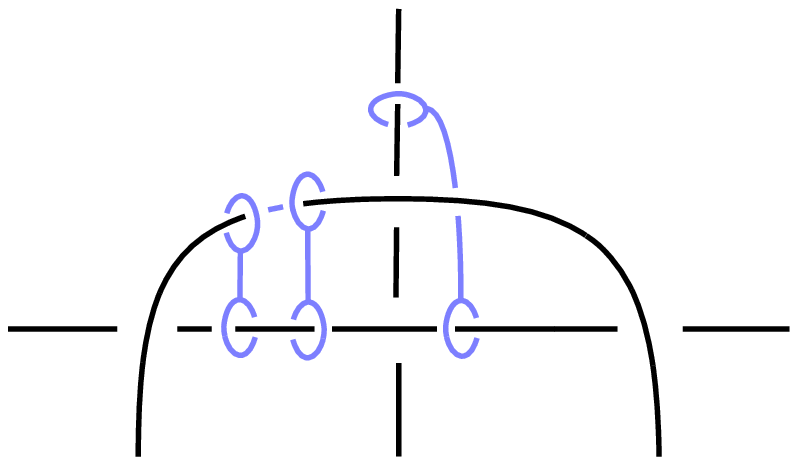}
\end{minipage}\ \ \overset{\raisebox{2pt}{\scalebox{0.8}{\text{surgery}}}}{\Longleftrightarrow}\
\begin{minipage}{85pt}
\includegraphics[width=85pt]{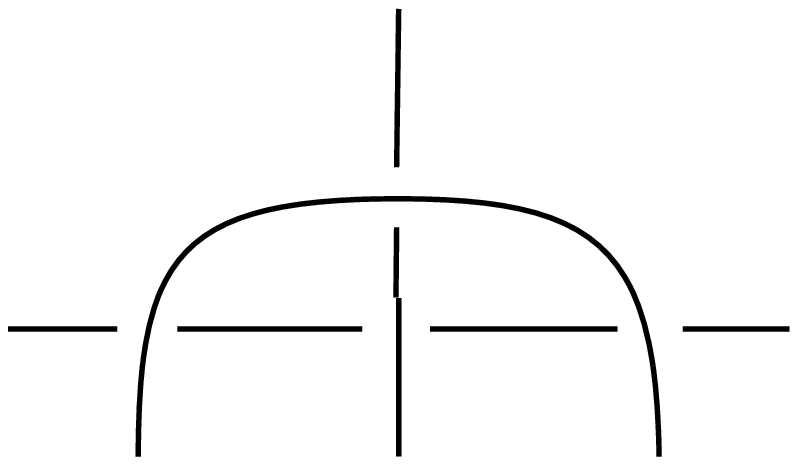}
\end{minipage} \ \ \overset{\text{zoom out}}{\begin{minipage}{30pt}\includegraphics[width=30pt]{fluffyarrow}\end{minipage}}\ \
\begin{minipage}{75pt}
\psfrag{a}[c]{}\psfrag{b}[c]{}\psfrag{c}[c]{}
\includegraphics[width=70.5pt]{BorrBands-4}
\end{minipage}
\end{multline*}
\item[$\Delta_2\Rightarrow \Delta_3$]
$$\psfrag{a}[c]{}\psfrag{b}[c]{}\psfrag{c}[c]{}
\begin{minipage}{80pt}
\includegraphics[width=75pt]{DeltaY-1}
\end{minipage}\quad\overset{\raisebox{2pt}{\scalebox{0.8}{$\Delta_2$}}}{\Longleftrightarrow}\quad
\begin{minipage}{80pt}
\includegraphics[width=75pt]{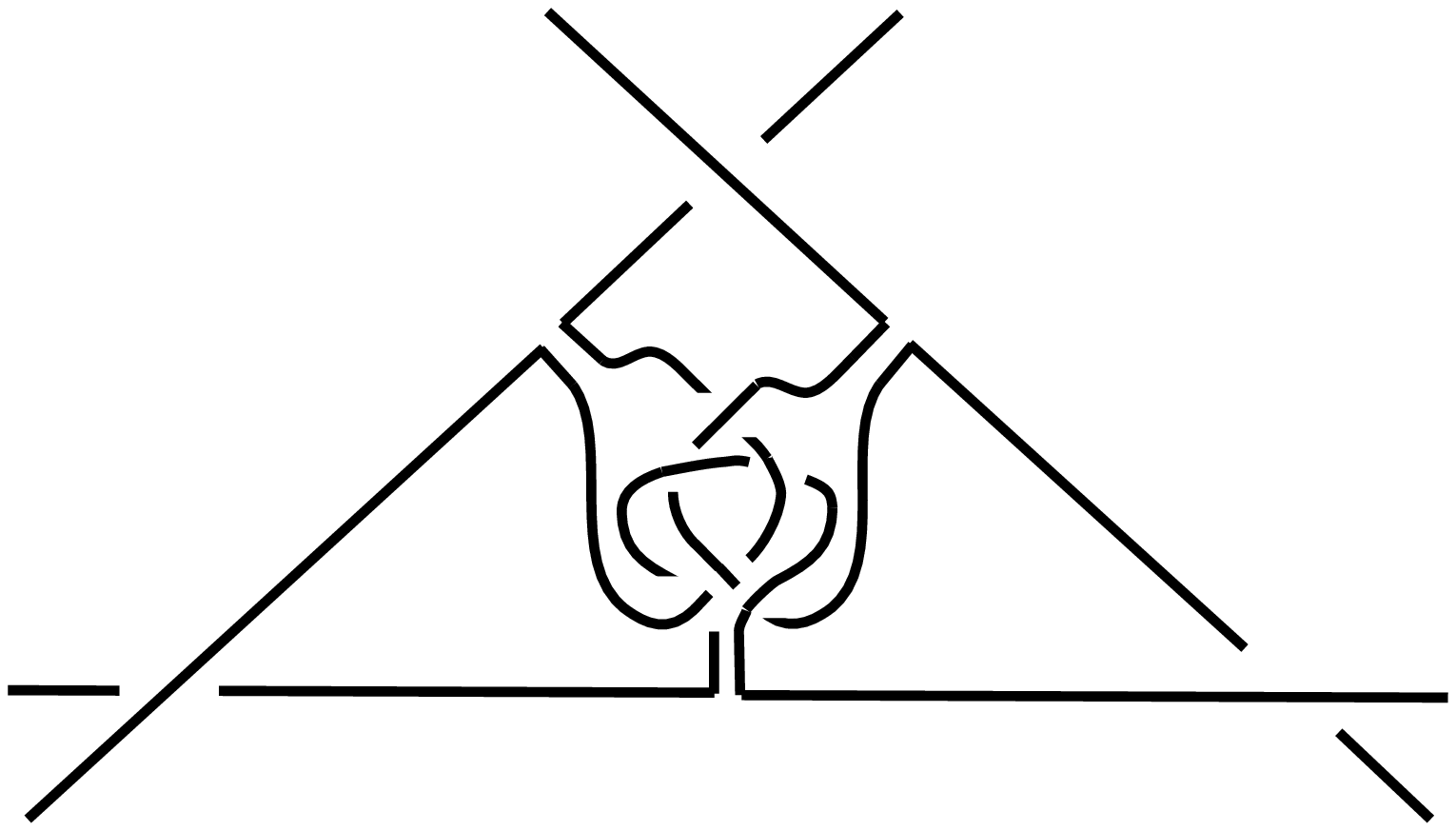}
\end{minipage}
\quad\overset{\raisebox{2pt}{\scalebox{0.8}{\text{isotopy}}}}{\Longleftrightarrow}\quad
\begin{minipage}{80pt}
\includegraphics[width=75pt]{DeltaY-3}
\end{minipage}$$
\item[$\Delta_3\Rightarrow \Delta_4$]
\begin{multline*}
\begin{minipage}{80pt}
\includegraphics[width=80pt]{dblpass-1}
\end{minipage}\quad\overset{\raisebox{2pt}{\scalebox{0.8}{\text{isotopy}}}}{\Longleftrightarrow}\quad
\begin{minipage}{80pt}
\includegraphics[width=80pt]{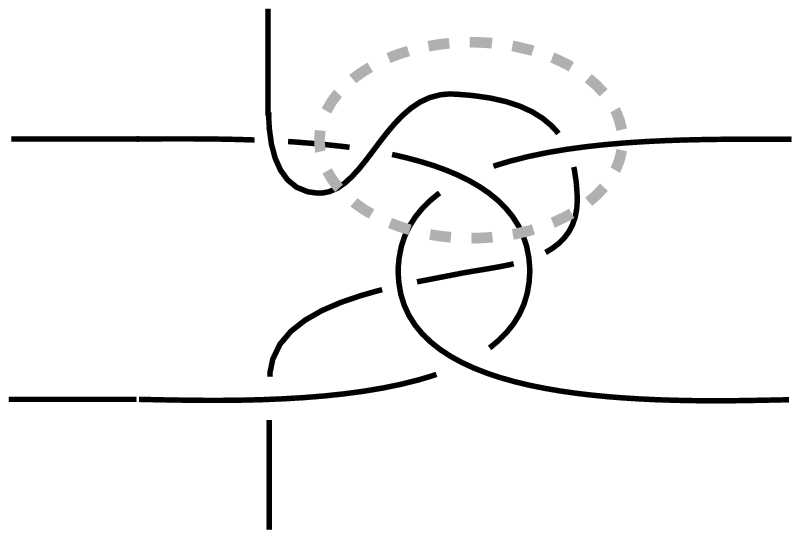}
\end{minipage}\\
\overset{\raisebox{2pt}{\scalebox{0.8}{$\Delta_3$}}}{\Longleftrightarrow}\quad
\begin{minipage}{80pt}
\includegraphics[width=80pt]{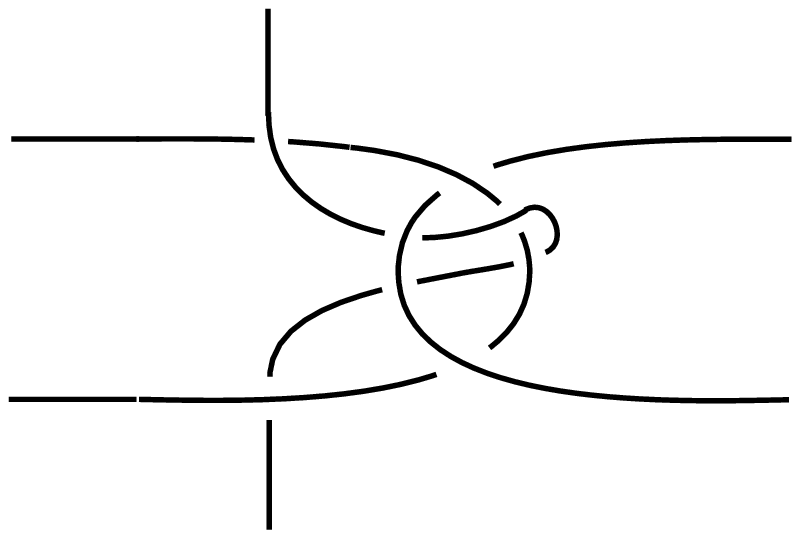}
\end{minipage}\quad\overset{\raisebox{2pt}{\scalebox{0.8}{\text{isotopy}}}}{\Longleftrightarrow}\quad
\begin{minipage}{80pt}
\includegraphics[width=80pt]{dblpass-f}
\end{minipage}
\end{multline*}
\item[$\Delta_4\Rightarrow \Delta_1$]\begin{multline*}
\begin{minipage}{80pt}
\includegraphics[width=70pt]{clasporder-1}
\end{minipage}\quad\overset{\raisebox{2pt}{\scalebox{0.8}{\text{surgery}}}}{\Longleftrightarrow}\quad
\begin{minipage}{80pt}
\includegraphics[width=70pt]{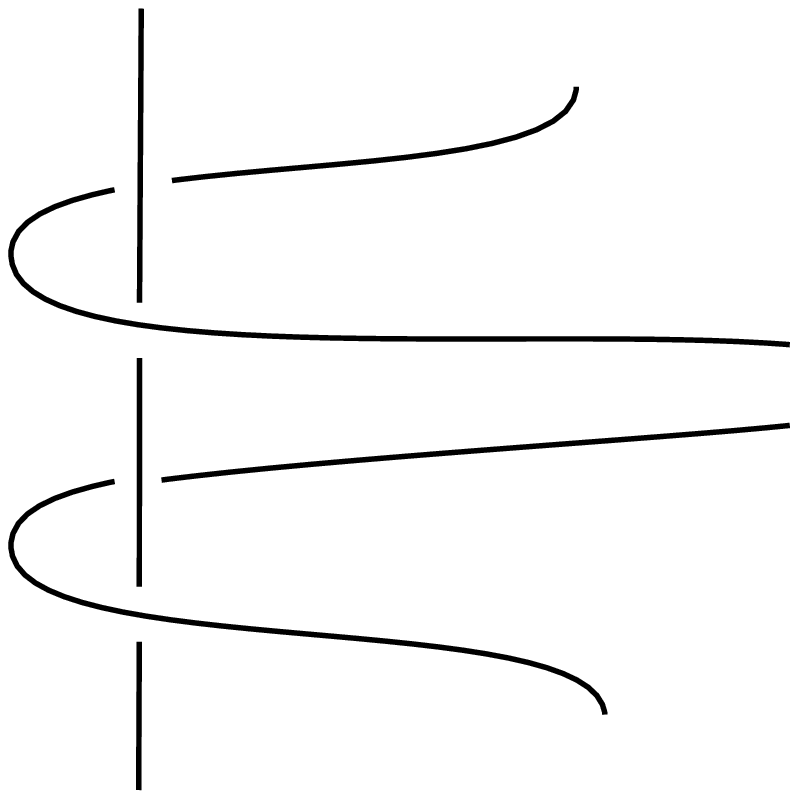}
\end{minipage}\quad\overset{\raisebox{2pt}{\scalebox{0.8}{\text{isotopy}}}}{\Longleftrightarrow}\quad
\begin{minipage}{80pt}
\includegraphics[width=70pt]{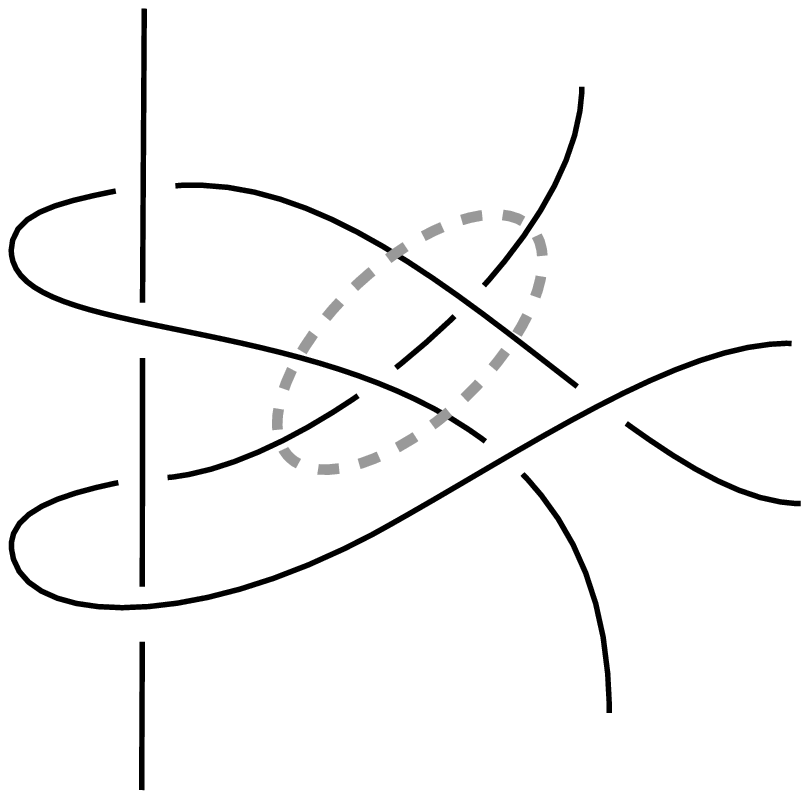}
\end{minipage}\\
\overset{\raisebox{2pt}{\scalebox{0.8}{$\Delta_4$}}}{\Longleftrightarrow}\quad
\begin{minipage}{80pt}
\includegraphics[width=70pt]{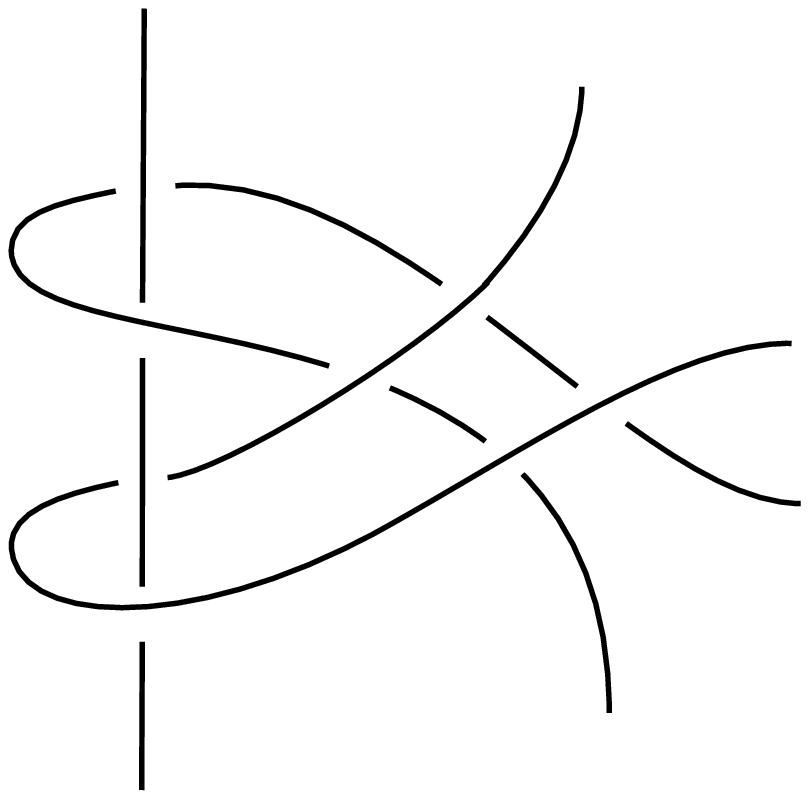}
\end{minipage}
\quad\overset{\raisebox{2pt}{\scalebox{0.8}{\text{surgery}}}}{\Longleftrightarrow}\quad
\begin{minipage}{80pt}
\includegraphics[width=70pt]{clasporder-2}
\end{minipage}
\end{multline*}
\end{description}
\end{proof}

\subsection{The space $\mathcal{C}$ of $A$--coloured $Y$--claspers}

A $Y$--clasper with leaves $A_{1,2,3}$ in the complement of an $A$--coloured Seifert surface is \emph{coloured $(a_1,a_2,a_3)\in A^3$} if $\bar\rho(A_{1,2,3})=a_{1,2,3}$ correspondingly (recall that the trivalent vertex and the leaves are oriented counterclockwise). Write the set of $(a_1,a_2,a_3)$--coloured $Y$--claspers in $A$--coloured Seifert surface complements as $\lblY{a_1}{a_2}{a_3}{32}$. Inserting a half-twist in an edge corresponds to inverting the colour of the leaf adjacent to that edge. We may formally add (sets of) coloured claspers over $\mathds{N}$ by taking their disjoint union: $\lblY{a_1}{a_2}{a_3}{32}+\lblY{b_1}{b_2}{b_3}{32}$ denotes the set of pairs of claspers in $A$--coloured Seifert surface complements, one of which is coloured $(a_1,a_2,a_3)$, and the other $(b_1,b_2,b_3)$.  The identity element is the empty $Y$--clasper, \textit{i.e.} nothing at all, written as $0\in\mathcal{C}$. This monoid of formal sums is denoted $\mathcal{C}$.\par

We write

\begin{equation}\sum_{i=1}^{N_1}n_i\lblY{a_i^1}{b^1_i}{c^1_i}{43}\,\sim_{\bar\rho}\,\sum_{i=1}^{N_2}n_i\lblY{a_i^2}{b^2_i}{c^2_i}{43}\end{equation}

\noindent if any $A$--coloured Seifert surface $(F,\bar\rho)$ is $\bar\rho$--equivalent to any $A$--coloured Seifert surface $(F^\prime,\bar\rho^\prime)$ obtained from $(F,\bar\rho)$ through a finite sequence of $Y$--clasper surgeries, deletion of an element in $\sum_{i=1}^{N_1}n_i\lblY{a_i^1}{b^1_i}{c^1_i}{43}$, and insertion of  an element in $\sum_{i=1}^{N_2}n_i\lblY{a_i^2}{b^2_i}{c^2_i}{43}$, and also the converse.

Define a homomorphism

\begin{equation}
\begin{aligned}
\Phi\co \mathcal{C}\qquad &\Too\qquad \bigwedge\nolimits^{ 3} A\\
\sum_{i=1}^{k}n_i\lblY{a_1^i}{a_2^i}{a_3^i}{43} &\mapsto\ \ \sum_{i=1}^k n_i\left(a_1^i\wedge a_2^i\wedge a_3^i\right).
\end{aligned}
\end{equation}

By abuse of terminology, $\Phi(C)$ means $\Phi$ of its class in $\mathcal{C}$.\par

\begin{prop}[Proof in Section {\ref{SS:ClasperProof}}]\label{P:Y-barrho}
The relation $\sim_{\bar\rho}$ is an equivalence relation, and $\mathcal{C}/\sim_{\bar\rho}$ is an abelian group.
The map $\Phi$ descends to an isomorphism of abelian groups
\begin{equation}\hat{\Phi}\co\ \mathcal{C}/\sim_{\bar\rho}\quad \Too\quad \bigwedge\nolimits^3 A.\end{equation}
\end{prop}

\subsection{The $Y$--obstruction}

If for two $G$--coloured knots $(K_{1,2},\rho_{1,2})$ there exists a Seifert surface $F_{1}$ for $K_{1}$ and a set of $Y$--claspers $C\in E(F_1)$ such that surgery on $C$ gives $(K_2,\rho_2)$, then the \emph{$Y$--obstruction of  $(K_{1,2},\rho_{1,2})$} is defined to be

\begin{equation}Y((K_1,\bar\rho_1),(K_{2},\bar\rho_2))\ass \Phi(C).\end{equation}

\begin{lem}
The $Y$--obstruction $Y((K_1,\bar\rho_1),(K_{2},\bar\rho_2))$ does not depend on the choice of $Y$--clasper $C$ in its definition.
\end{lem}

\begin{proof}
If surgery around $C_1\subset E(F_1)$ and surgery around $C_2\subset E(F_1)$ both give $(F_2,\bar\rho_2)$, then surgery around $C_1\cup \bar{C}_2$ gives back $(F_1,\bar\rho_1)$, where $\bar{C}_2$ is the result of inserting a half twist in one edge of each $Y$--clasper in $C_2$. But by \cite[Lemma 3.2]{Mas03} (see also \cite[Section 4.3]{Tur84}), $[A_1^i]\wedge [A_2^i]\wedge [A_3^i]=0\in \bigwedge^3 H_1(E(F))$, where $[A^i_{1,2,3}]$ are homology classes representing leaves of $Y$--claspers in $C_1\cup \bar{C}_2$. A-fortiori $\Phi(C_1\cup \bar{C}_2)=0$.
\end{proof}

In the remainder of this section we prove that the $Y$--obstruction is independent of the choice of Seifert surface used in its construction.\par

\begin{defn}
A \emph{weak band projection} of a knot $K$ is a Seifert surface $F$ for $K$ and a projection of an identification
\[D^2\cup B_1\cup\cdots \cup B_{2g}\to\ F\]
where $D^2$ and each $B_i$ is a disk. Moreover, we require $B_i\cap B_j=\emptyset$ for $i\neq j$. We write $\partial B_i\ssa\, \alpha_i\gamma_i\beta_i\gamma_i^{\prime -1}$ with $D^{2}\cap B_i\ssa\, \alpha_i\cup\beta_i$. A weak band projection is called a \emph{band projection} (see \textit{e.g.} \cite[Chapter 8B]{BZ03}) if
\[
\partial D^2\ssa\ \alpha_1\delta_1\beta_2^{-1}\delta_2\beta_1^{-1}\delta_3\alpha_2\delta_4\cdots\alpha_{2g-1}\delta_{4g-3}\beta_{2g}^{-1}\delta_{4g-2}\beta_{2g-1}^{-1}\delta_{4g-1}\alpha_{2g}\delta_{4g}.
\]
\noindent Note that the bands of a weak band projection are oriented, and that it induces a basis for $H_1(F)$, and therefore also for $H_1(E(F))$. See Figure \ref{F:bandproj}.
\end{defn}

\begin{figure}
\psfrag{1}[c]{$B_1$}\psfrag{2}[c]{$B_2$}\psfrag{3}[c]{$B_3$}\psfrag{4}[c]{$B_{4}$}\psfrag{5}[c]{$B_{2g-1}$}\psfrag{6}[c]{$B_{2g}$}
\psfrag{D}[c]{$D^2$}\psfrag{d}[c]{\Large$\cdots$}\psfrag{x}[c]{$x_1$}\psfrag{y}[c]{$x_2$}\psfrag{c}[c]{$\xi_1$}\psfrag{e}[c]{$\xi_2$}
\includegraphics[width=4.75in]{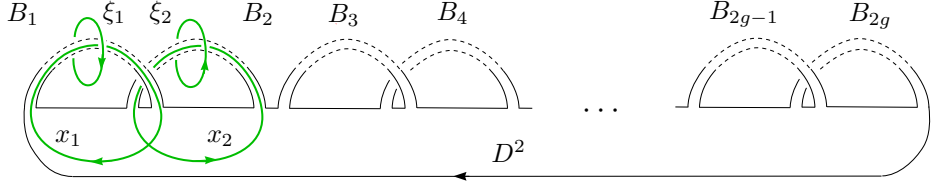}
\caption{\label{F:bandproj} A band projection of a knot.}
\end{figure}

Any ambient isotopy of $F$ can be realized by a sequence of band slides for any weak band projection of $F$ (see \textit{e.g.} \cite{MosXX}). A dual basis element $\xi_i\in H_1(E(F))$ is associated to each band, and to it an entry $v_i=\bar\rho(\xi_i)$ of the colouring vector. If all orientations are counterclockwise (other cases are analogous), the band-slide of $B_1$ over $B_2$ is realized by the following local picture.

         \begin{equation}\label{E:pileY}
                \begin{minipage}{100pt}
                \psfrag{A}[c]{\fs$a$}\psfrag{B}[c]{\fs$b$}
                \includegraphics[width=100pt]{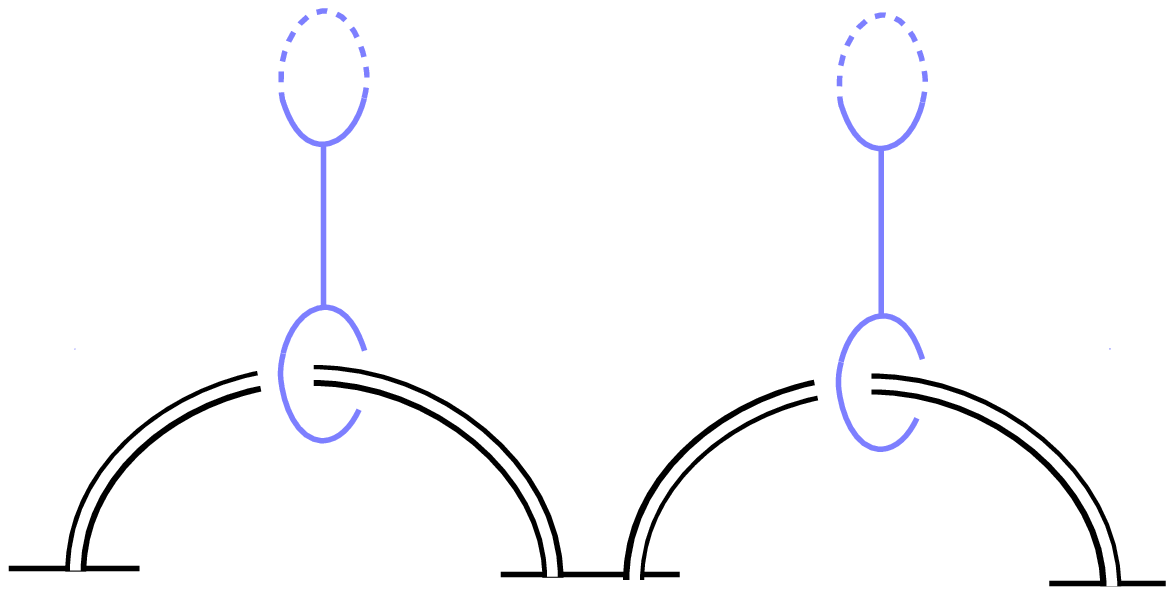}
                \end{minipage}\quad \overset{\raisebox{2pt}{\scalebox{0.8}{\text{isotopy}}}}{\Longleftrightarrow}\quad
                \begin{minipage}{130pt}
                \psfrag{A}[c]{\fs$a$}\psfrag{B}[l]{\fs$b-a$}
                \includegraphics[width=100pt]{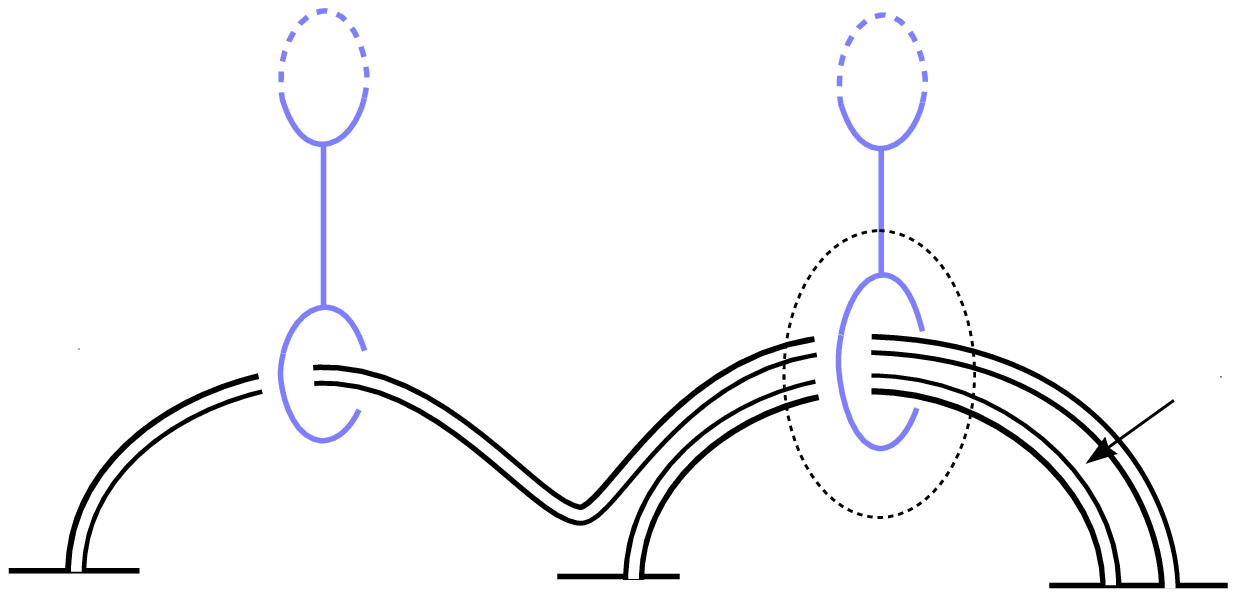}
                \end{minipage}.
            \end{equation}

Zoom in:

\begin{equation}\label{E:pile}
                \psfrag{B}[c]{$B_1$}\psfrag{C}[c]{$B_2$}
                \raisebox{10pt}{\begin{minipage}{60pt}
                \includegraphics[width=55pt]{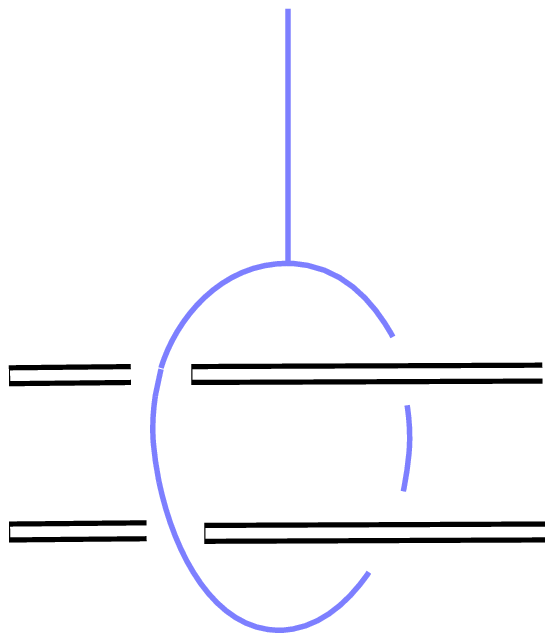}
                \end{minipage}}
                \quad
                \overset{\raisebox{2pt}{\scalebox{0.8}{\text{Move 8}}}}{\Longleftrightarrow}\quad
                \ \raisebox{10pt}{\begin{minipage}{60pt}
                \includegraphics[width=60pt]{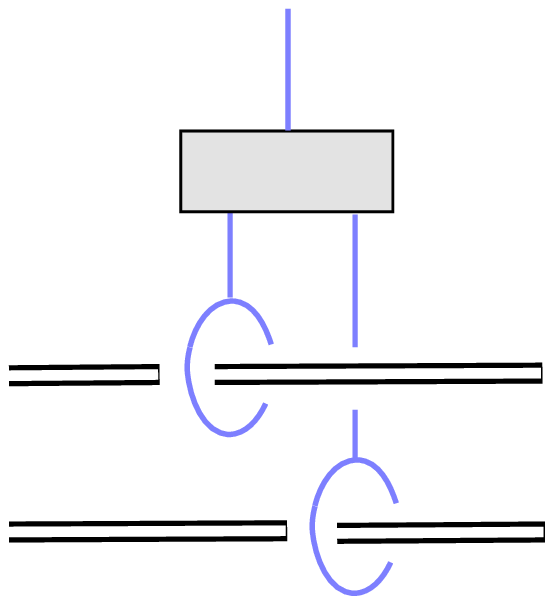}
                \end{minipage}}
                \quad\ \overset{\raisebox{2pt}{\scalebox{0.8}{\text{unzip}}}}{\Longleftrightarrow}\quad
                \ \raisebox{10pt}{\begin{minipage}{60pt}
                \includegraphics[width=60pt]{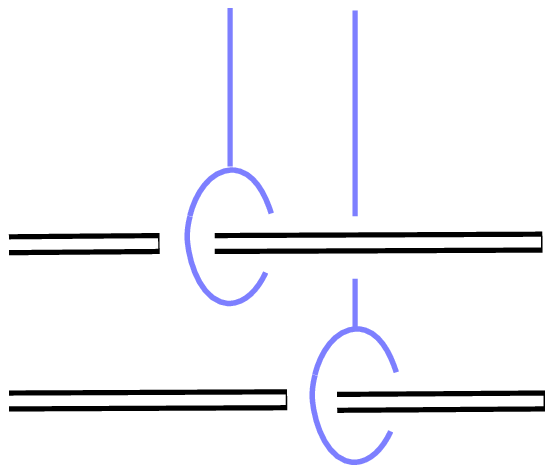}
                \end{minipage}}
\end{equation}
\noindent where `unzip' means \cite[Definition 3.12]{Hab00}.\par

For each $Y$--clasper in $\lblY{b}{c}{d}{27.5}$ whose leaf clasped $B_2$, we now have two $Y$--claspers in $\lblY{a}{c}{d}{27.5}$ and in
 $\begin{minipage}{30pt}\psfrag{a}[c]{\small$b-a$}\psfrag{b}[c]{\small$c$}\psfrag{c}[c]{\small$d$}\includegraphics[width=30pt]{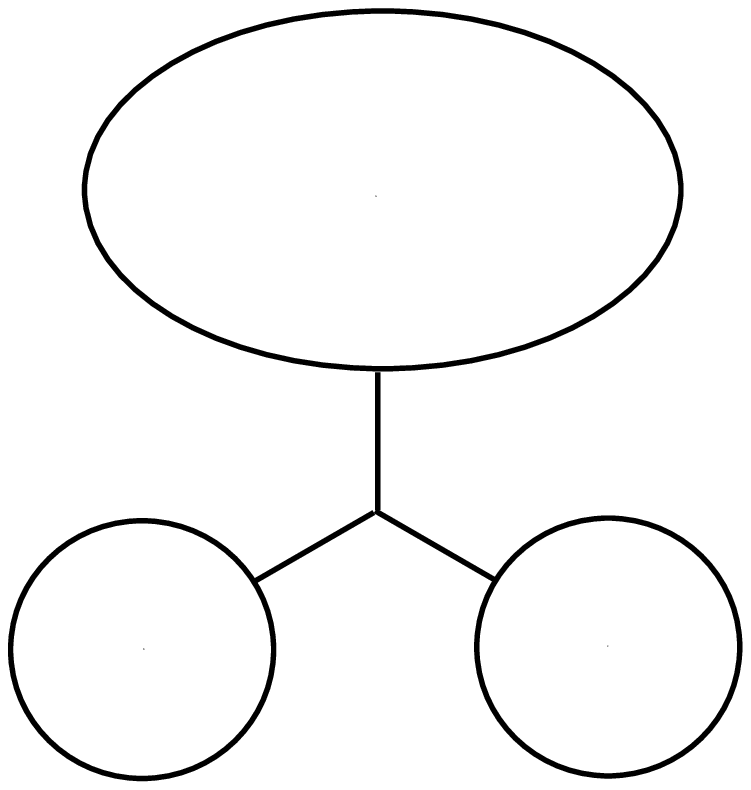}\end{minipage}$ correspondingly. The $\Phi$--image is unchanged.\par

We next show that the $Y$--obstruction is invariant under stabilization. A $1$--handle attachment to $F$ locally looks, up to reflection, as in Figure \ref{F:genincrease}. The only possible contributions to the $Y$--obstruction come from linkage with $B_1^{\text{new}}$. But the loop which rings around $B_1^{\text{new}}$ is in $\ker\bar\rho$, and so any $Y$ clasper which clasps  $B_1^{\text{new}}$ is in $\ker\Phi$.

\begin{figure}
\psfrag{a}[c]{\fs$B_1^{\text{new}}$}\psfrag{b}[c]{\fs$B_2^{\text{new}}$}\psfrag{y}[c]{$0$}
\begin{minipage}{115pt}
    \includegraphics[width=115pt]{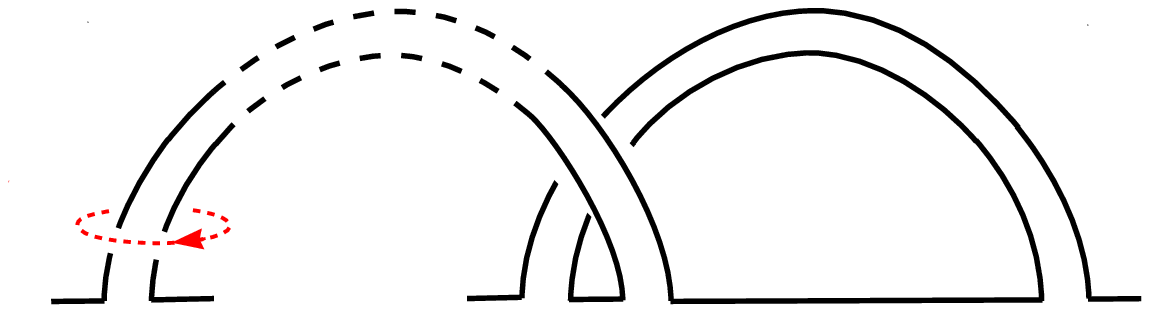}
\end{minipage}
\qquad\quad\quad
\begin{minipage}{115pt}
    \includegraphics[width=115pt]{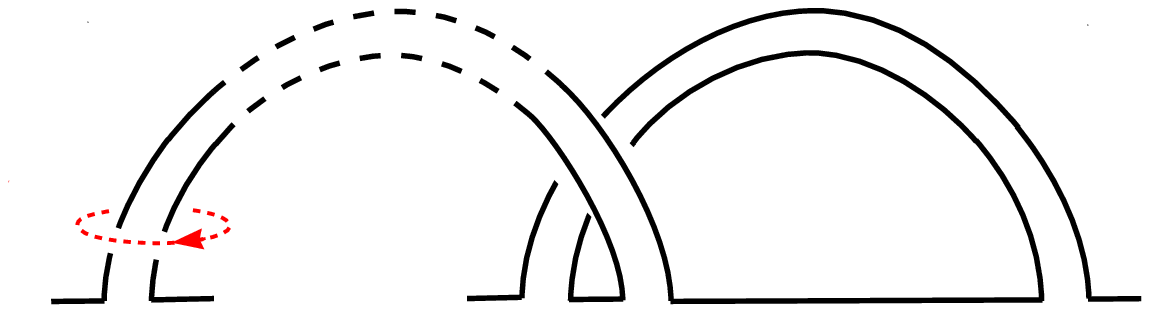}
    \end{minipage}
\caption{\label{F:genincrease} Local pictures of $1$--handle attachments to a Seifert surface.}
\end{figure}

\subsection{Null-twists don't change the $Y$--obstruction}\label{SS:Y-nulltwist}

Let $(K_{1,2},\rho_{1,2})$ be a pair of $S$--equivalent $G$--coloured knots which are related by a sequence of null-twists. The goal of this section is to show that $Y((K_1,\bar\rho_1),(K_2,\bar\rho_2))$ vanishes. Let $F_{1,2}$ be Seifert surfaces for $K_{1,2}$ correspondingly. By the tubing construction, we assume the null-twists to be between bands of $F_1$. As in Section \ref{SS:S-equiv-matrix}, we may assume without the limitation of generality that there exist bases $\set{x^{1,2}_1,\ldots,x^{1,2}_{2g}}$ for $H_1(F_{1,2})$ correspondingly, which give rise to identical Seifert matrices. In this section, each time we stabilize $F_1$ we automatically stabilize $F_2$ in the same way, and each time we change the basis of $H_1(F_1)$ we automatically change the basis of $H_1(F_2)$ in the same way. The colouring vectors with respect to $\left(F_{1,2},\set{x^{1,2}_1,\ldots,x^{1,2}_{2g}}\right)$ also coincide because null-twists don't change the colouring vector.

Define a \emph{$Y_0$-move} to be a set of $\Delta$--moves realized as surgery around a set of $Y$--claspers in $\ker\Phi$.\par

The proof consists of three steps. First, for a chosen basis $\mathcal{B}$ of $H_1(F_1)$, we arrange by tube-equivalence for all non-zero entries in the colouring vector to be elements of $\mathcal{B}$, up to sign. Next, gather the null-twists together into a local picture by $Y_0$-moves. Finally, trivialize this local picture by $Y_0$-moves.

\subsubsection{Step 1: Shorten Words}\label{SSS:Shorten}

The goal of this section is to present an algorithm to generate the following output from the following input.

\begin{description}
\item[Input] A band projection of a Seifert surface, together with an ordered basis $\mathcal{B}\ass \set{b_1,\ldots,b_r}$ for $A$.
\item[Output] A band projection of a Seifert surface, with every non-zero entry of the corresponding colouring vector in $\mathcal{B}$, up to sign.
\end{description}

Carry out the procedure as follows. Let $\bar{\mathcal{B}}$ denote $\set{\pm b_1,\ldots,\pm b_r}$. Write the word length of an element $a\in A$ with respect to $\bar{\mathcal{B}}$ as $w_{\bar{\mathcal{B}}}(a)$. Denote by $\mathcal{V}$ the set of colouring vectors coming from band projections. A colouring vector $\V\in\mathcal{V}$ has a partition into pairs $\set{(v_{2i-1},v_{2i})}_{1\leq i\leq g}$. Define a partial order $\prec$ on $\mathcal{\V}$ by ordering its elements first by the lexicographical partial order by word lengths of their entries $w_{\bar{\mathcal{B}}}(v_i)$, and then by the lexicographical partial order by total word-lengths of their pairs $w_{\bar{\mathcal{B}}}(v_{2i-1})+w_{\bar{\mathcal{B}}}(v_{2i})$ .
If $w_{\bar{\mathcal{B}}}(v_i)\leq 1$ for $i=1,\ldots,2g$ then we are done. Otherwise there exists an entry in the colouring vector, which we assume without limitation of generality is $v_{2g}$, such that $w_{\bar{\mathcal{B}}}(v_{2g})\geq w_{\bar{\mathcal{B}}}(v_j)$ for $j=1,\ldots,2g$, and $w_{\bar{\mathcal{B}}}(v_{2g})> 1$.
Choose an element $b\in\bar{\mathcal{B}}$ such that $w_{\bar{\mathcal{B}}}(v_{2g}+b)<w_{\bar{\mathcal{B}}}(v_{2g})$.\par

Recall that $ta$ (as opposed to $t\cdot a$) simply means ``left-multiply $a$ by $t$''. Because $\bar\rho$ is surjective, there exists an oriented based loop $C\in \pi$ bounding a disc $D$ with $\bar\rho(C)= \frac{t}{t-1}\cdot b$. Form a cylinder $Z\in E(F)$ with $\partial Z=(C\times [0,1])\cup (D\times\{0,1\})$. One may imagine a bunch of bands passing through a pipe $C\times [0,1]$. Stabilize $F$ by adding bands $B_{1,2}^{\text{new}}$ where $B_1^{\text{new}}$ links $Z$, immediately to the right of $B_{2g}$ the band corresponding to $v_{2g}$.

\begin{equation}\label{E:stabilizecol}
\psfrag{e}[c]{\fs$ta$}
\begin{minipage}{70pt}
\psfrag{Z}[c]{$Z$}\psfrag{C}[c]{$C$}
\includegraphics[height=45pt]{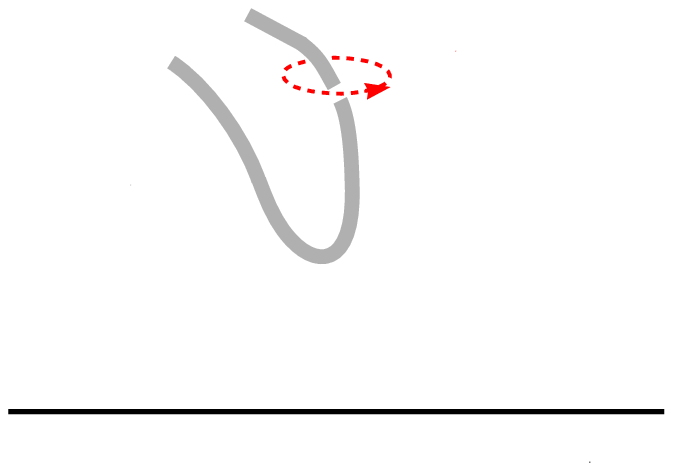}
\end{minipage} \overset{\raisebox{2pt}{\scalebox{0.8}{\text{isotopy}}}}{\Longleftrightarrow}\
\begin{minipage}{105pt}
\psfrag{a}[c]{\fs$ta$}\psfrag{d}[c]{\fs$tatbt^{-1}$}
\includegraphics[height=45pt]{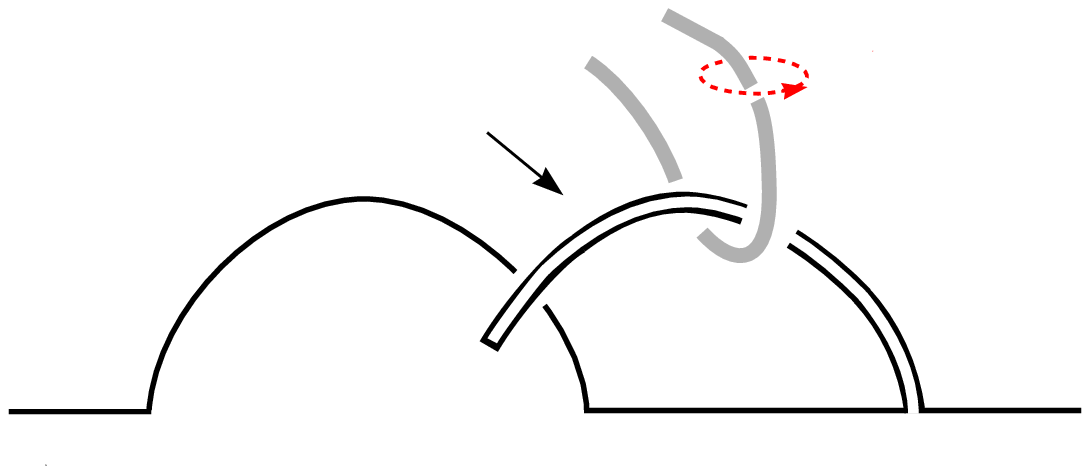}
\end{minipage}\overset{\raisebox{2pt}{\scalebox{0.8}{\text{isotopy}}}}{\Longleftrightarrow}\ \ \ \begin{minipage}{114pt}
\psfrag{y}[c]{\fs\phantom{x}$b$}
\psfrag{B}[c]{$B_1^{\text{new}}$}\psfrag{A}[c]{$B_2^{\text{new}}$}
\includegraphics[height=45pt]{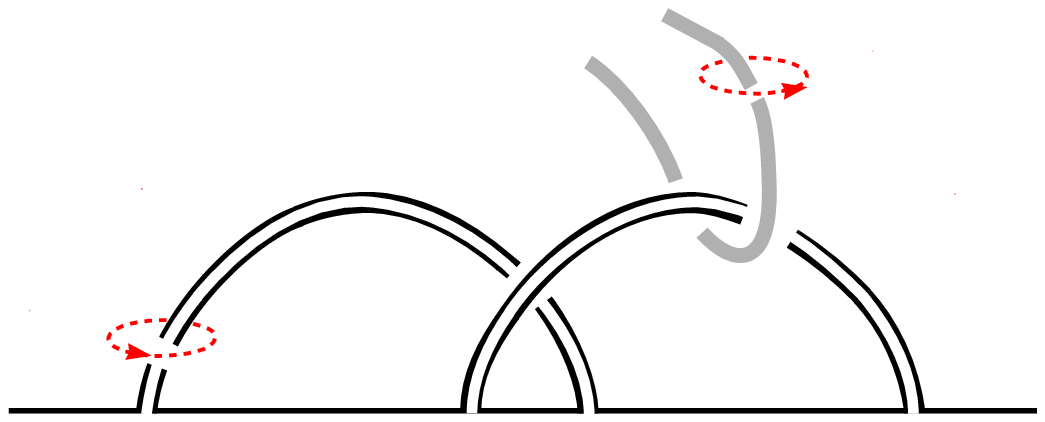}
\end{minipage}
\end{equation}

Now slide bands as follows (compare with \cite[Section 4.2.2]{KM09})

\begin{gather*}
\begin{minipage}{160pt}
\scriptsize \psfrag{a}[r]{$v_{2g-1}$}\psfrag{b}{$v_{2g}$}\psfrag{c}{$b$}\psfrag{d}{$0$}
\includegraphics[width=160pt]{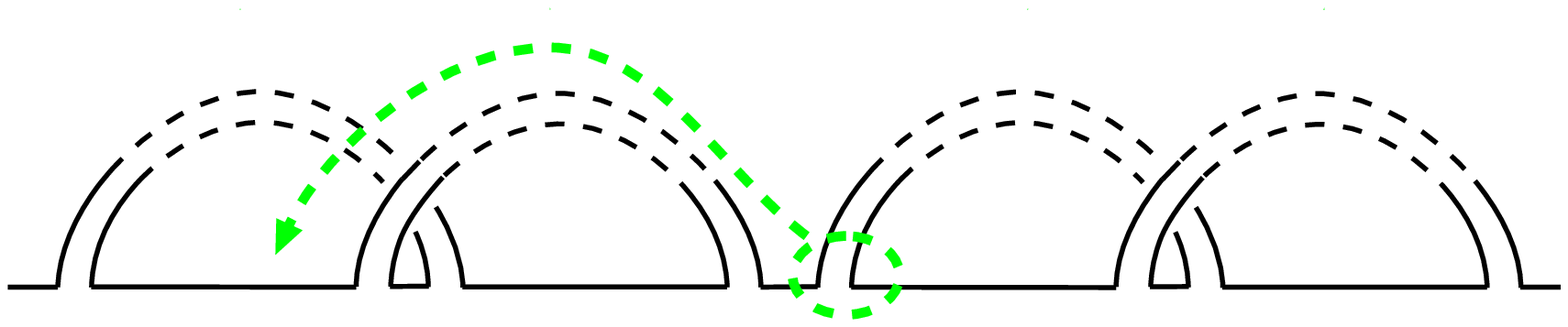}
\end{minipage}\quad \Longleftrightarrow\ \ \
\begin{minipage}{160pt}
\scriptsize \psfrag{a}[r]{$v_{2g-1}$}\psfrag{b}[c]{$b$}\psfrag{c}[c]{$v_{2g}+b$}\psfrag{d}{$0$}
\includegraphics[width=160pt]{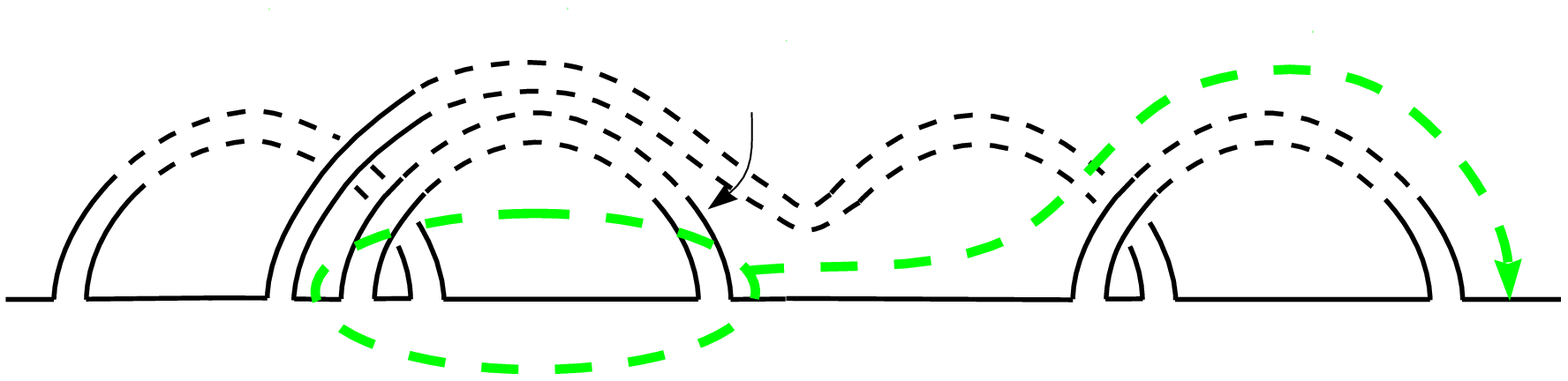}
\end{minipage}\\
\Longleftrightarrow\quad
\begin{minipage}{145pt}\scriptsize
\scriptsize
\psfrag{b}[c]{$v_{2g-1}$}\psfrag{a}{$v_{2g}+b$}
\includegraphics[width=145pt]{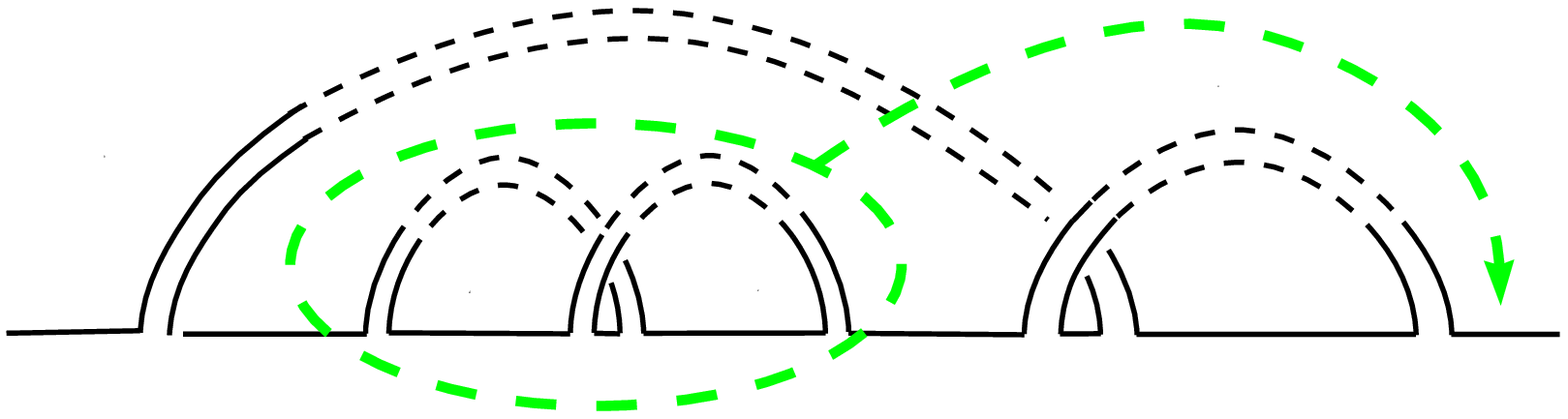}
\end{minipage}\quad\Longleftrightarrow\quad
\begin{minipage}{145pt}\scriptsize
\psfrag{a}[r]{$v_{2g-1}$}\psfrag{b}{$v_{2g}+b$}\psfrag{c}{$b$}\psfrag{d}{$v_{2g-1}$}
\includegraphics[width=145pt]{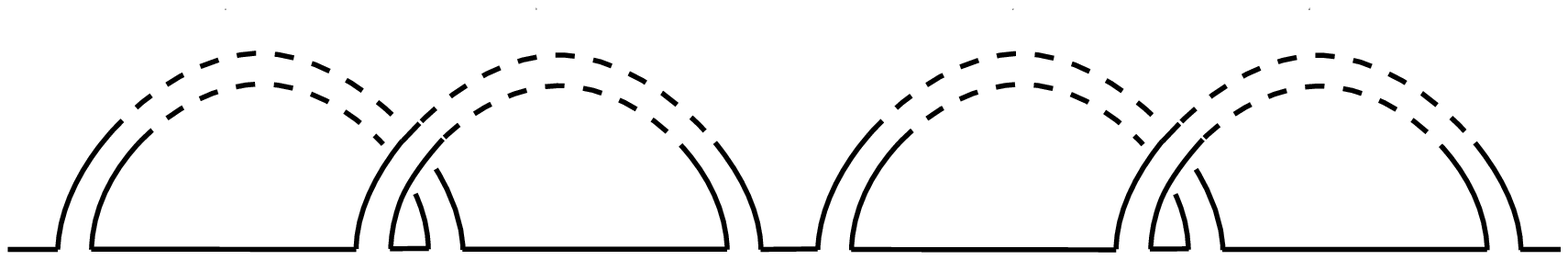}
\end{minipage}
\end{gather*}

We obtain a colouring vector $\V^{\text{new}}$ which satisfies $\V^{\text{new}}\prec \V$. By Zorn's Lemma we are finished.

\subsubsection{Step 2: Bring null-twists together}

Parameterize each band $B_i$ in a band projection of $F$ as $I\times I$. We show that for each band, up to $Y_0$-moves, all null-twists may be assumed to take place in $I\times [\frac{1}{2},1]$, while everything else (linkage, twisting, and knotting) takes place in $I\times [0,\frac{1}{2})$.

\begin{lem}\label{L:clasprho}
A leaf may be moved past a null-twist by a $Y_0$-move. See Figure \ref{F:clasprho}.
\end{lem}

\begin{proof}
Write the null-twist, between bands $B_1,\ldots,B_k$ coloured $a_1,\ldots,a_k$ correspondingly, in terms of surgery on basic claspers.
Let $C\ass A_1\cup A_2\cup E$  be a basic clasper such that $A_2$ clasps $B_1$ and $A_1$ clasps a band $B_0$ with colour $a_0$.
Moving $A_2$ past a null-twist entails performing one $\Delta_1$-move for each clasper coming from the null-twist which clasps $B_1$. Each $\Delta$-move is realized by inserting a $Y$--clasper. The collective contribution of these $Y$--claspers to $\Phi$ is
\begin{equation}a_0\wedge a_1\wedge \sum_{i=2}^k a_i= -a_0\wedge a_1\wedge a_1 =0.\end{equation}
\end{proof}

\begin{figure}
\begin{minipage}{115pt}
\psfrag{c}[c]{\footnotesize$a_0$}\psfrag{d}[c]{\footnotesize$a_1$}\psfrag{e}[c]{\footnotesize$a_2$}\psfrag{f}[c]{\footnotesize$a_k$}
\psfrag{p}[c]{\fs$2\pi$ twist}
\includegraphics[width=100pt]{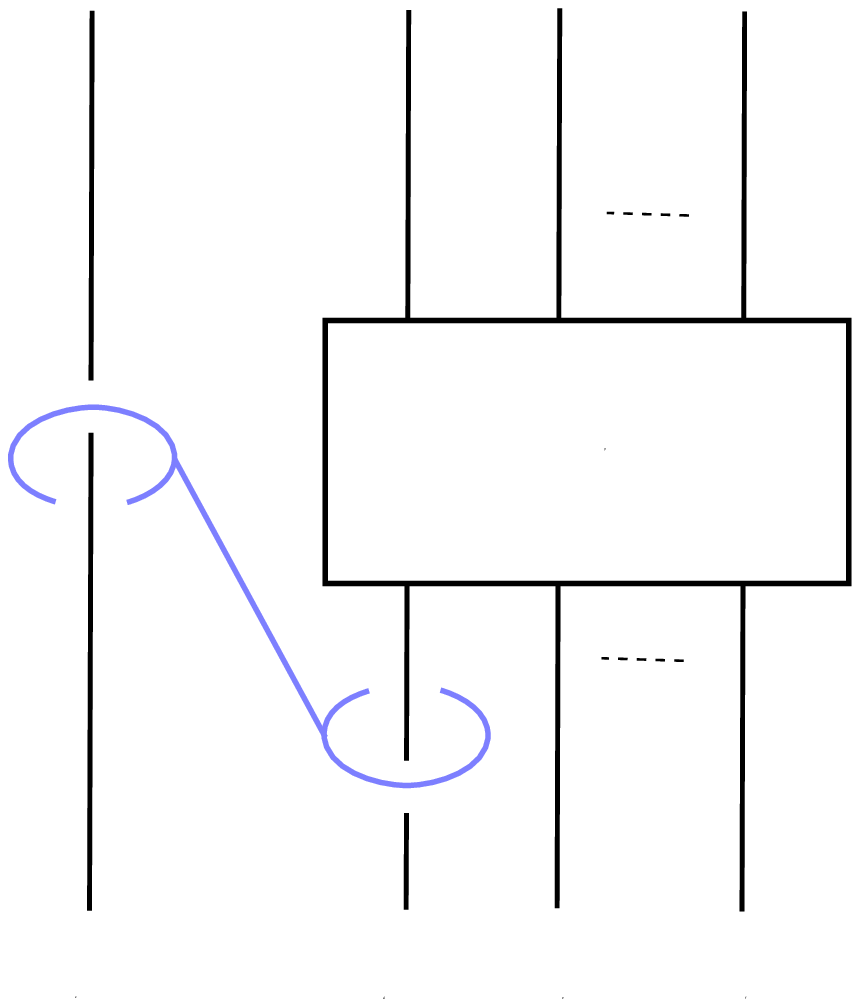}
\end{minipage}
 $\overset{\raisebox{2pt}{\scalebox{0.8}{\text{surgery}}}}{\Longleftrightarrow}$\quad \ \
\begin{minipage}{115pt}
\psfrag{c}[c]{\footnotesize$a_0$}\psfrag{d}[c]{\footnotesize$a_1$}\psfrag{e}[c]{\footnotesize$a_2$}\psfrag{f}[c]{\footnotesize$a_k$}
\psfrag{p}[c]{\fs$2\pi$ twist}
\includegraphics[width=100pt]{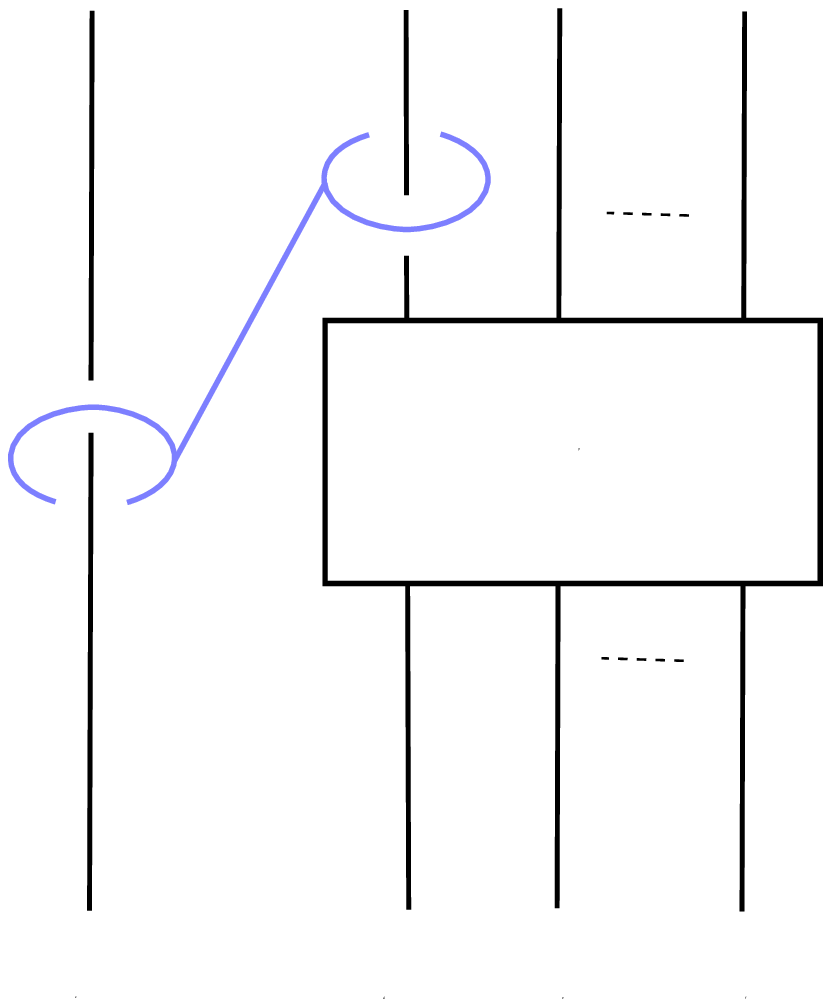}
\end{minipage}
\caption{\label{F:clasprho} A move is realized by inserting $Y$--claspers in $\ker\Phi$.}
\end{figure}

\subsubsection{Eliminate null-twists}

Having carried out the preceding steps, we arrive at a presentation of $(F_2,\bar\rho_2)$ in $(F_1,\bar\rho_1)$ by a collection of null-twists in a local picture which is a trivial braid between bands coloured by elements of $\bar{\mathcal{B}}\subset A$. The result of these null-twists is a braid in which every pair of bands has linking number zero. Our goal is to show that this braid is trivialized by $Y_0$-moves.\par

For a null-twist between bands $B_1,\ldots,B_k$ coloured $a_1,\ldots,a_k\subset \bar{\mathcal{B}}$ correspondingly, there exists a partition $P$ of  $\set{1,\ldots,k}$ such that for each $S\subseteq P$, both the sum $\sum_{i\in S}a_i$ vanishes, and also $a_i=\pm a_j$ for all $a_i,a_j\in S$. If for some null-twist $T$ this partition has more than $2$ parts, separate $T$ into smaller null-twists whose corresponding partitions have fewer parts, as in Figure \ref{F:threadsep}.

\begin{figure}
\psfrag{d}[c]{\footnotesize$S_1$}\psfrag{e}[c]{\footnotesize$S_2$}\psfrag{f}[c]{\footnotesize$S_3$}
\begin{minipage}{70pt}
\psfrag{p}[c]{\fs$2\pi$ twist}
\includegraphics[height=130pt]{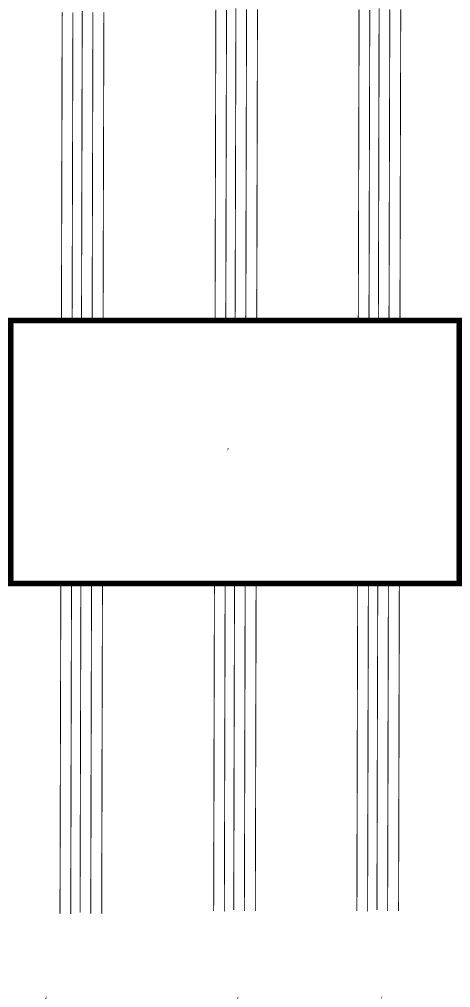}
\end{minipage}
 $\overset{\raisebox{2pt}{\scalebox{0.8}{\text{isotopy}}}}{\Longleftrightarrow}$\quad
\begin{minipage}{70pt}
\psfrag{1}[c]{\tiny$2\pi$}\psfrag{2}[c]{\tiny$2\pi$}\psfrag{3}[c]{\tiny$2\pi$}
\psfrag{a}[c]{\footnotesize$2\pi$}\psfrag{b}[c]{\footnotesize$2\pi$}\psfrag{c}[c]{\footnotesize$2\pi$}
\psfrag{p}[c]{\fs$2\pi$ twist}
\includegraphics[height=130pt]{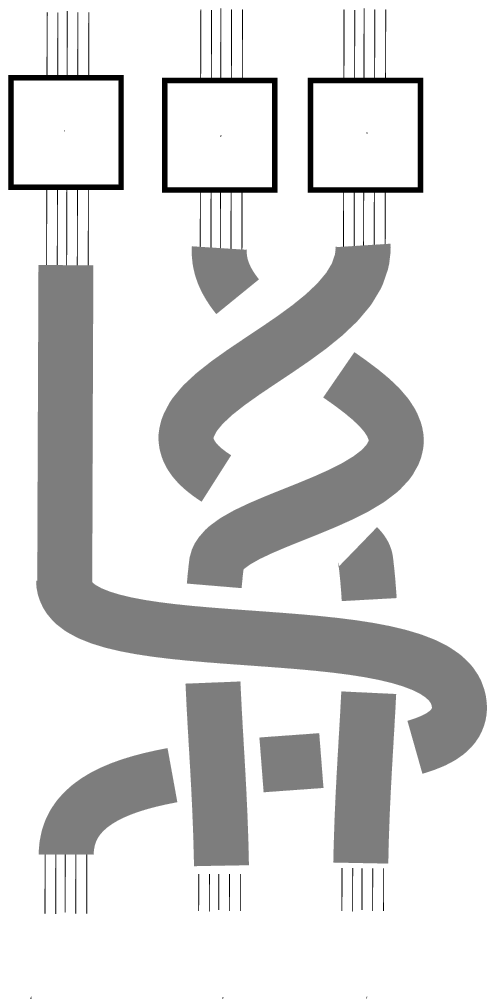}
\end{minipage}
 $\overset{\raisebox{2pt}{\scalebox{0.8}{\text{surgery}}}}{\Longleftrightarrow}$\quad\ \
\begin{minipage}{70pt}
\psfrag{1}[c]{\tiny$-2\pi$}\psfrag{2}[c]{\tiny$-2\pi$}\psfrag{3}[c]{\tiny$-2\pi$}
\psfrag{a}[c]{\footnotesize$2\pi$}\psfrag{b}[c]{\footnotesize$2\pi$}\psfrag{c}[c]{\footnotesize$2\pi$}
\psfrag{p}[c]{\fs$2\pi$ twist}
\includegraphics[height=130pt]{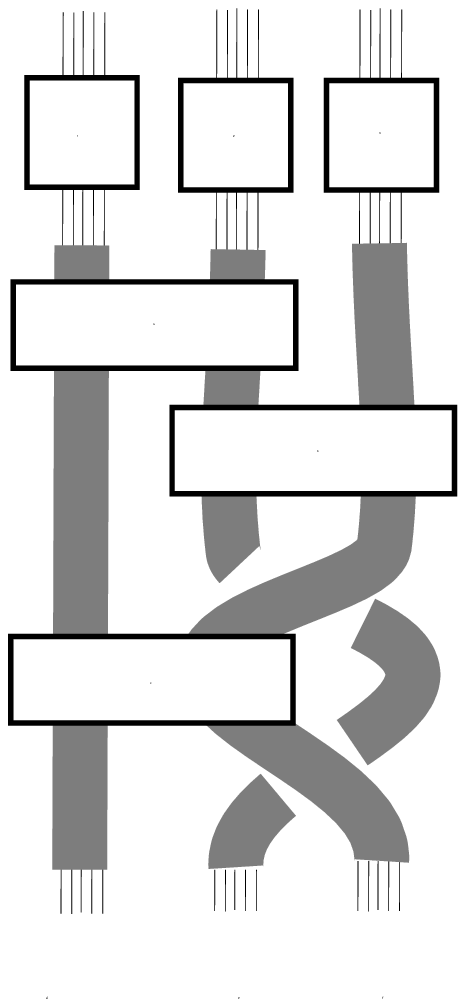}
\end{minipage}
\caption{\label{F:threadsep} Gather bands going through $T$ into bunches according to the partition $P$. Perform $Y_0$-moves between the bunches.}
\end{figure}

Choose a pair of basis elements $a,b\in\mathcal{B}\cup\{0\}$. Using Lemma \ref{L:clasprho} and the fact that partitions corresponding to null-twists now have at most two parts, perform $Y_0$-moves to create a smaller local picture in which all null-twists are between bands labeled $\pm a$ and bands labeled $\pm b$. Because the wedge of any triple in $\set{\pm a,\pm b}$ vanishes, any $Y$--clasper which we insert in this local picture will be in $\ker\Phi$. Due to the vanishing of all linking numbers between bands, by Murakami--Nakanishi all crossings between such bands cancel up to $Y_0$-moves \cite{MN89}. Repeat for each pair of basis elements, until all null-twists are between bands which share the same colour. These cancel up to $Y_0$-moves, because the wedge of any triple in $\set{a,-a}$ is zero.

\subsection{Local Moves realized by null-twists}

To prove Theorem \ref{T:clasperprop}, it remains to show that any $Y_0$-move is realized by a null-twist. We adopt the typical clasper strategy of first identifying moves between $Y$--claspers which are realized by null-twists, and then proving that these suffice to realize any $Y_0$-move.\par

\begin{lem}\label{L:(0,a,b)}
\[\lblY{0}{a}{b}{27.5}\sim_{\bar\rho}\ 0.\]
\end{lem}

\begin{proof}
Realize the $\Delta_2$ move by the following sequence of ambient isotopy and null-twists (the dotted arc is labeled $0\in A$).
\begin{equation}
\psfrag{a}[c]{$a$}\psfrag{b}[c]{$b$}\psfrag{c}[c]{$0$}
\begin{minipage}{61pt}
\includegraphics[width=61pt]{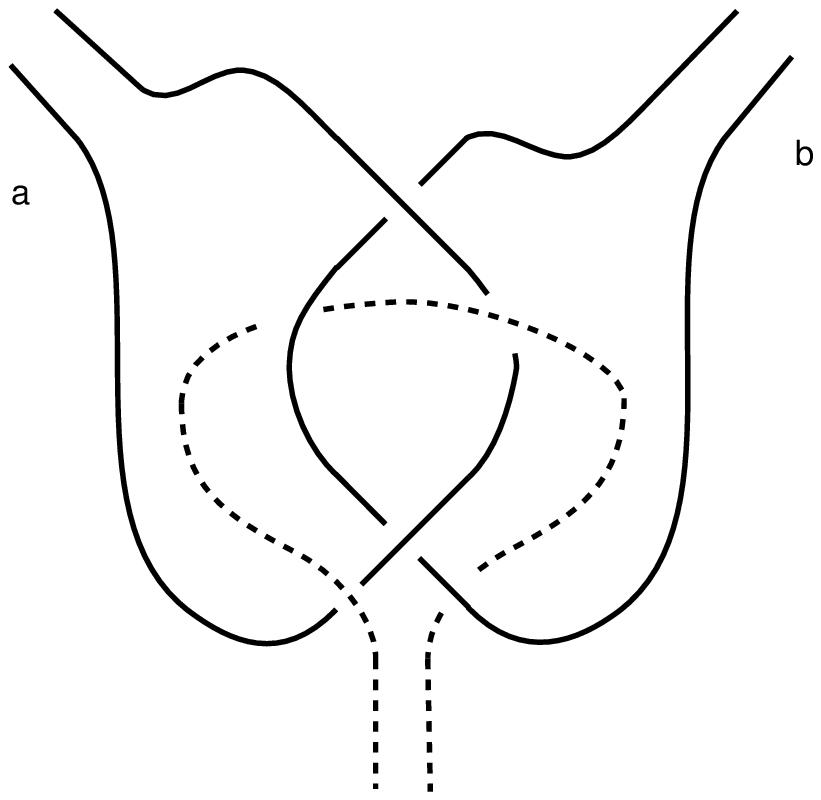}
\end{minipage}\ \ \Leftrightarrow\ \
\begin{minipage}{61pt}
\includegraphics[width=61pt]{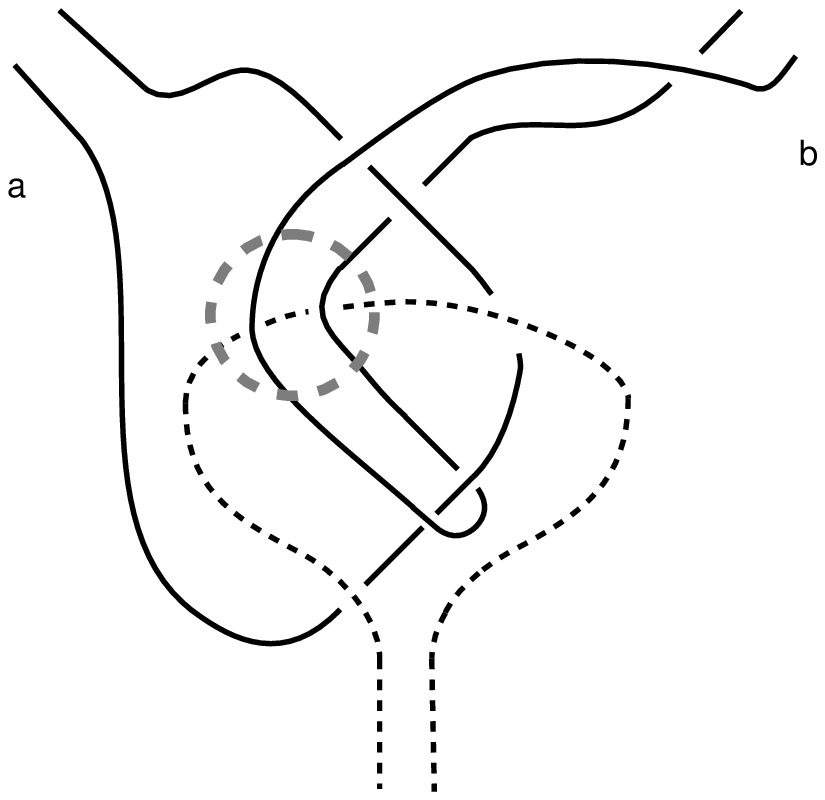}
\end{minipage}\ \ \Leftrightarrow\ \ \
\begin{minipage}{61pt}
\includegraphics[width=61pt]{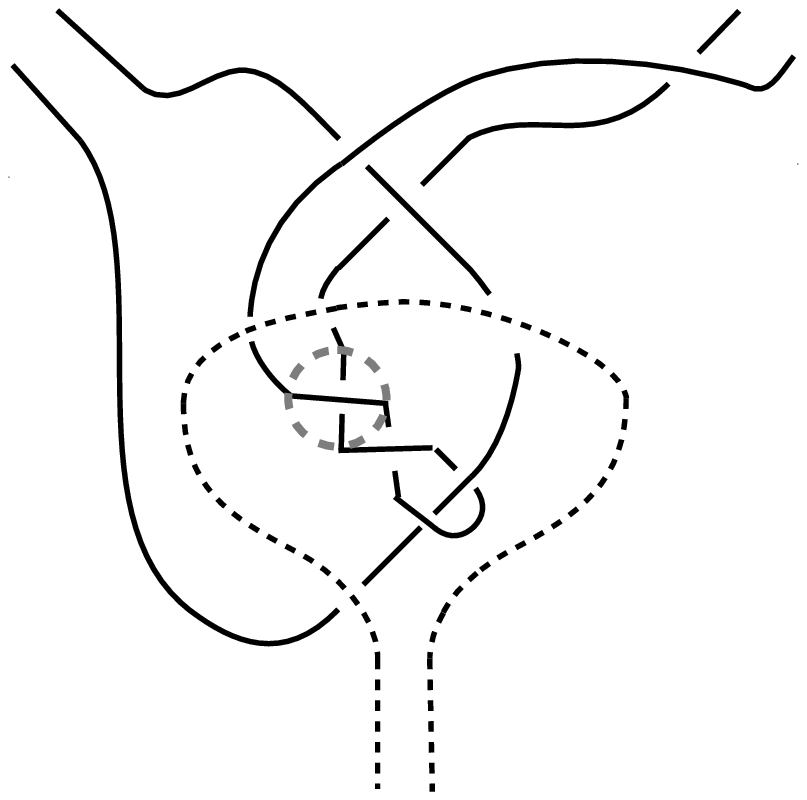}
\end{minipage}\ \ \Leftrightarrow\ \ \
\begin{minipage}{61pt}
\includegraphics[width=61pt]{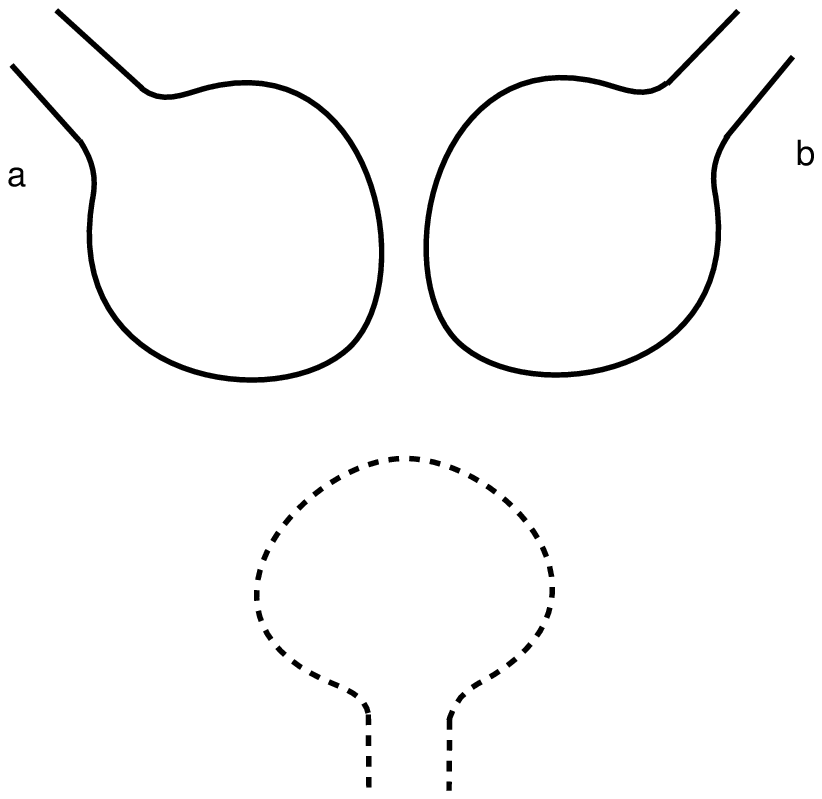}
\end{minipage}
\end{equation}
\end{proof}

\begin{lem}\label{L:(a,a,b)}
Setting $\bar{a}\ass a^{-1}$, we have $\lblY{\bar a}{a}{b}{30}\sim_{\bar\rho}\ 0$ and also $\lblY{a}{a}{b}{27.5}\sim_{\bar\rho}\ 0$.
\end{lem}

\begin{proof}
Realize the $\Delta_2$ move by the following sequence of ambient isotopy and null-twists.
\begin{equation}
\psfrag{a}[r]{$\pm a$}\psfrag{b}[c]{$a$}\psfrag{c}[c]{$b$}
\begin{minipage}{61pt}
\includegraphics[width=61pt]{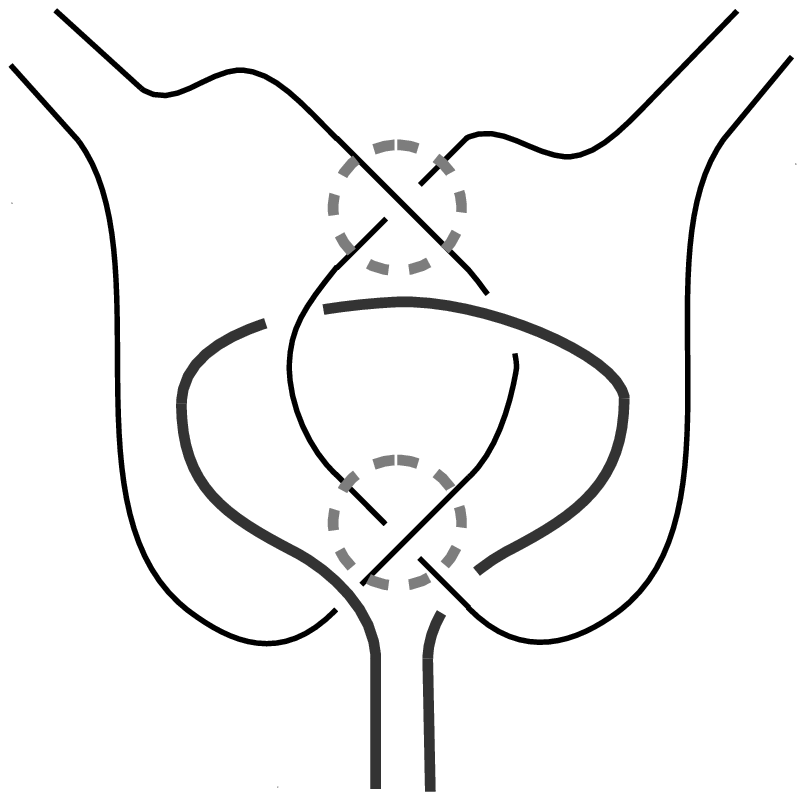}
\end{minipage}\ \ \Leftrightarrow\quad\quad
\begin{minipage}{61pt}
\includegraphics[width=61pt]{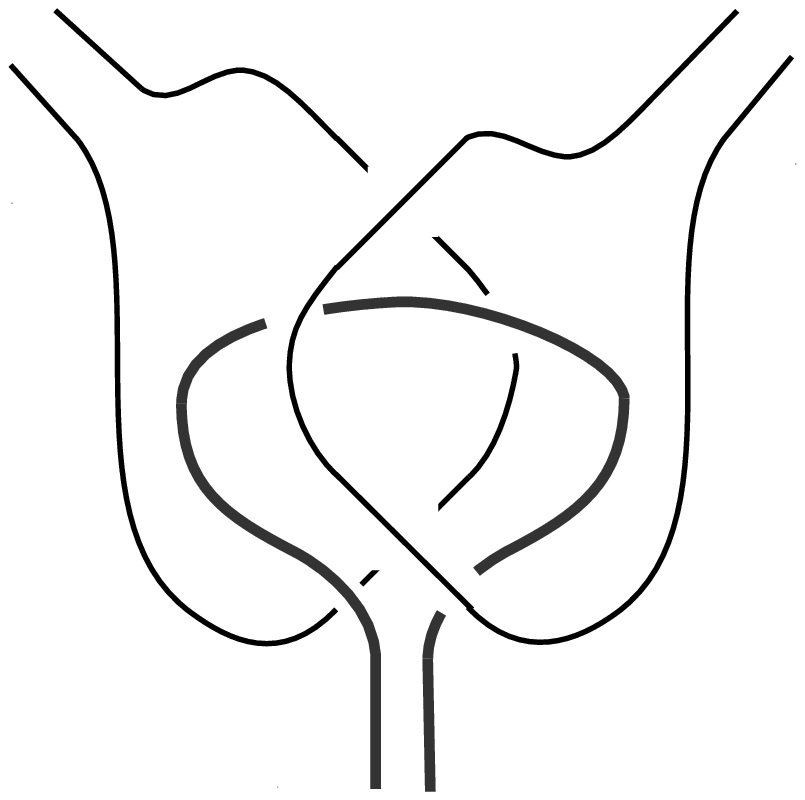}
\end{minipage}\ \ \Leftrightarrow\quad\quad\ \begin{minipage}{61pt}
\includegraphics[width=61pt]{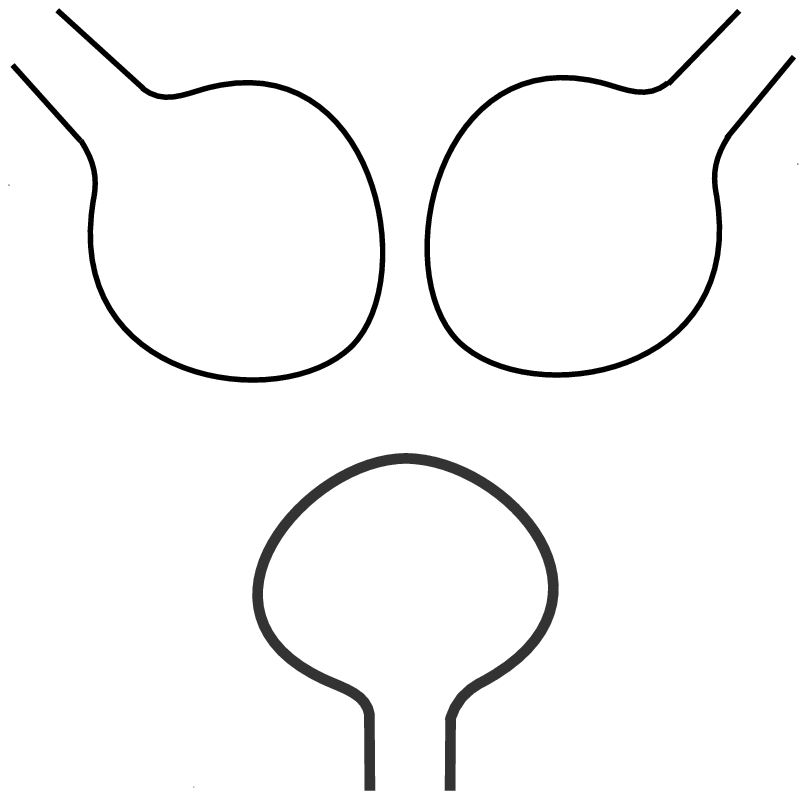}
\end{minipage}
\end{equation}
\end{proof}

\begin{lem}\label{L:clasp-pass}
The results of surgeries around the following two claspers in the complement of an $A$--coloured Seifert surface are $\bar\rho$--equivalent.

\begin{equation}
\begin{minipage}{130pt}
\includegraphics[width=130pt]{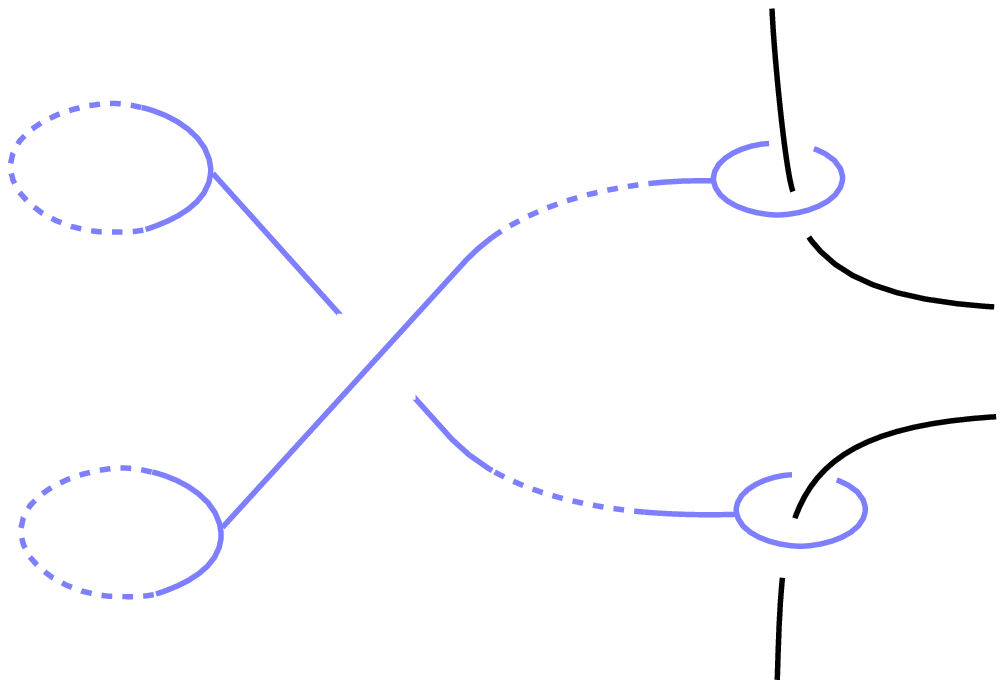}
\end{minipage}\quad\overset{\raisebox{2pt}{\scalebox{0.8}{\text{clasp-pass}}}}{\Longleftrightarrow}\quad
\begin{minipage}{130pt}
\includegraphics[width=130pt]{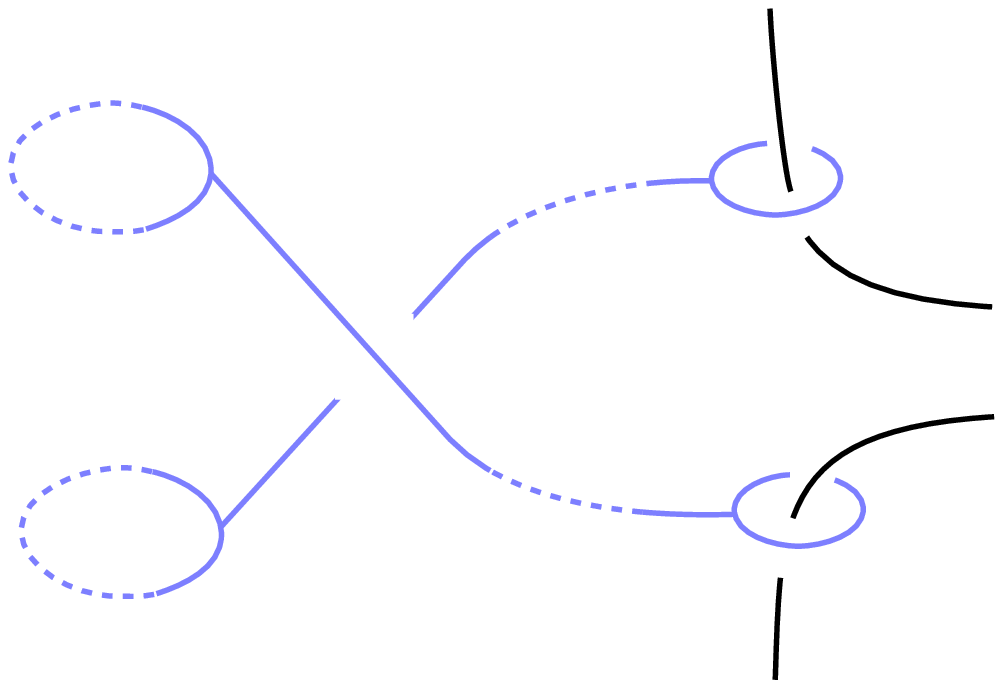}
\end{minipage}
\end{equation}
This is Habiro's \emph{clasp-pass move}.
\end{lem}

\begin{proof}
\begin{multline}
\begin{minipage}{100pt}
\includegraphics[width=100pt]{clasppass-1rn}
\end{minipage}\quad\overset{\raisebox{2pt}{\scalebox{0.8}{\text{surgery}}}}{\Longleftrightarrow}\quad
\begin{minipage}{100pt}
\includegraphics[width=100pt]{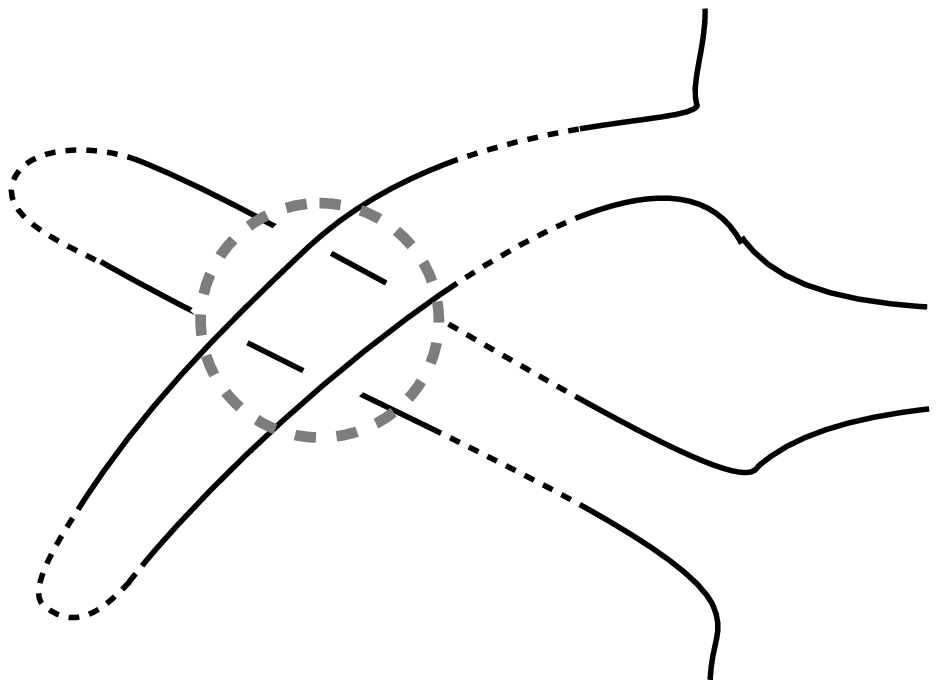}
\end{minipage}\\
\overset{\raisebox{2pt}{\scalebox{0.8}{\text{null-twist}}}}{\Longleftrightarrow}\quad
\begin{minipage}{100pt}
\includegraphics[width=100pt]{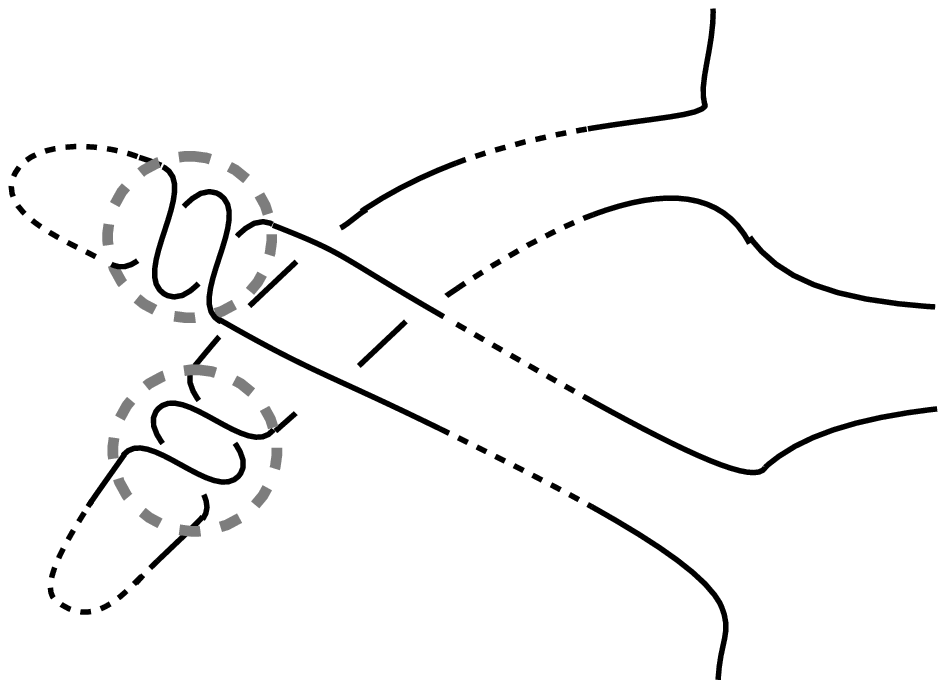}
\end{minipage}\quad
\overset{\raisebox{2pt}{\scalebox{0.8}{\text{surgery}}}}{\Longleftrightarrow}\quad
\begin{minipage}{100pt}
\includegraphics[width=100pt]{clasppass-fn}
\end{minipage}
\end{multline}
\end{proof}

\begin{cor}\label{C:ClasperFramingReduce}
A full-twist in an edge of a clasper is realized by a null-twist.
\end{cor}

\begin{lem}\label{L:YPass}
The following local move is realized by null-twists.
$$
\begin{minipage}{40pt}
\includegraphics[height=35pt]{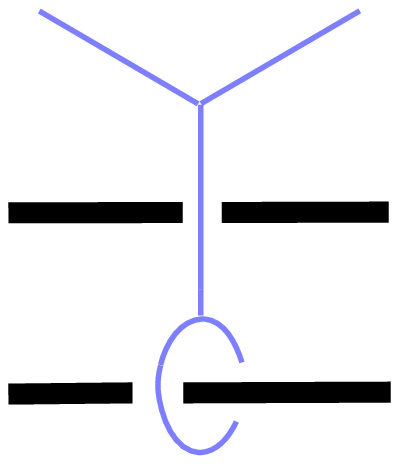}
\end{minipage}\Leftrightarrow\ \
\begin{minipage}{35pt}
\includegraphics[height=35pt]{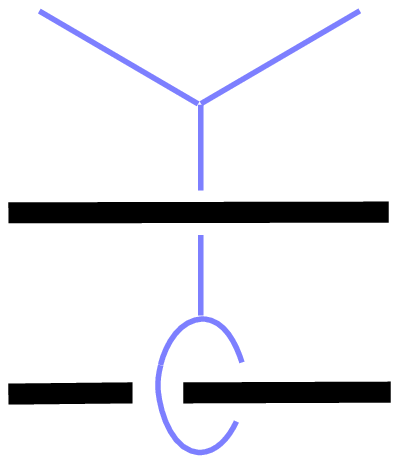}
\end{minipage}
$$
\end{lem}

\begin{proof}
\begin{equation}
\begin{minipage}{30pt}
\includegraphics[height=30pt]{YPass-1}
\end{minipage}\overset{\raisebox{0.3pt}{\scalebox{0.6}{\text{Move 9}}}}{\Leftrightarrow}
\begin{minipage}{65pt}
\includegraphics[height=30pt]{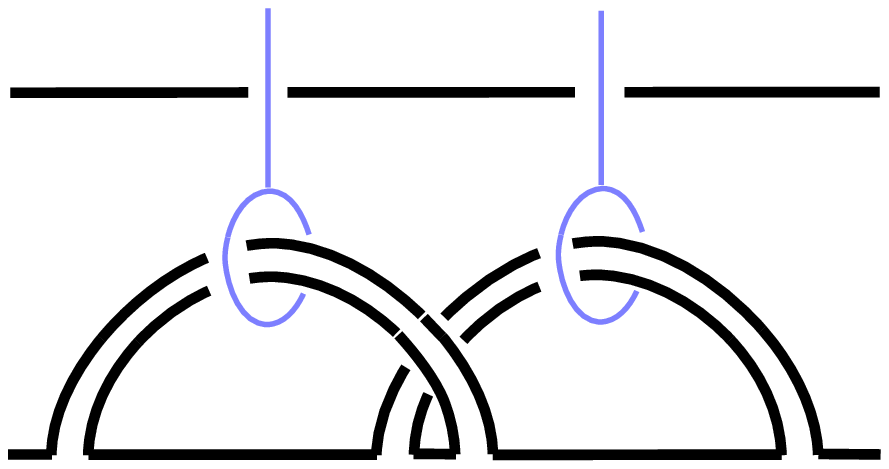}
\end{minipage}\overset{\raisebox{0.3pt}{\scalebox{0.6}{\text{Lemma {\ref{L:(0,a,b)}}}}}}{\Leftrightarrow}\
\begin{minipage}{65pt}
\includegraphics[height=30pt]{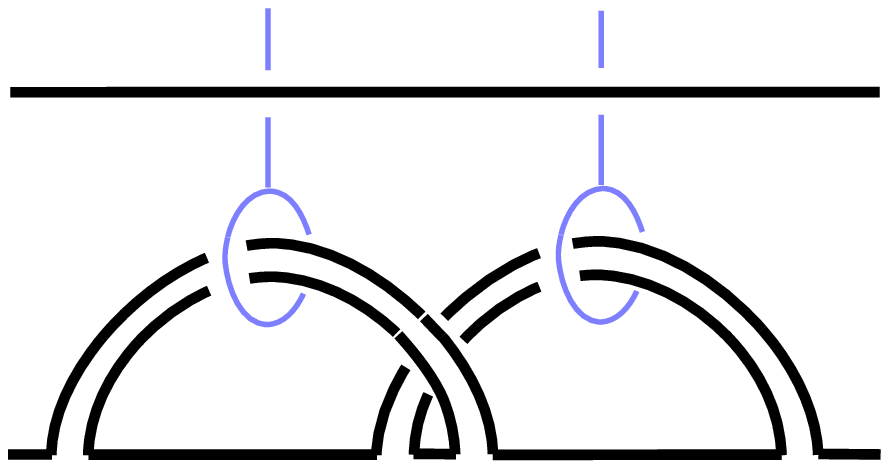}
\end{minipage}\overset{\raisebox{0.3pt}{\scalebox{0.6}{\text{Move 9}}}}{\Leftrightarrow}\
\begin{minipage}{30pt}
\includegraphics[height=30pt]{YPass-f}
\end{minipage}
\end{equation}
\end{proof}

\subsection{Leaves clasping single bands}

\subsubsection{The leaf-shepherd procedure}\label{SSS:Shepherd}

The goal of this section is to present an algorithm to generate the following output from the following input.

\begin{description}
\item[Input] A band projection of an $A$--coloured Seifert surface $(F,\bar\rho)$, together with a pair of claspers $C_{1,2}\subset E(F)$ with distinguished leaves $A^{1,2}$ each of which ring a single band, such that the colours of the bands which $C_{1,2}$ clasp are either mutually inverse or the same.
\item[Output] A clasper $C^\prime$ with distinguished leaf $A^\prime$ which clasps a single band in the same band projection of $(F,\bar\rho)$, related to $C_{1,2}$ as in Figure \ref{F:shepherd}.
\end{description}

\begin{figure}
\psfrag{1}[c]{$C_1$}\psfrag{2}[c]{$C_2$}
\begin{minipage}{75pt}
\psfrag{a}[c]{$A^1$}\psfrag{b}[c]{$A^2$}
\includegraphics[width=75pt]{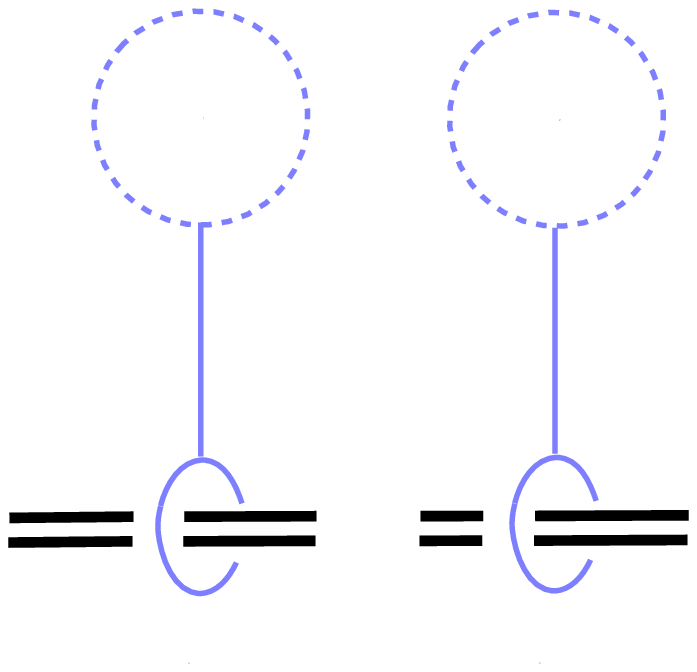}
\end{minipage}
\quad\quad\raisebox{10pt}{\begin{minipage}{23pt}\includegraphics[width=23pt]{fluffyarrow}\end{minipage}}\quad
\begin{minipage}{75pt}
\psfrag{a}[c]{$A^\prime$}
\includegraphics[width=75pt]{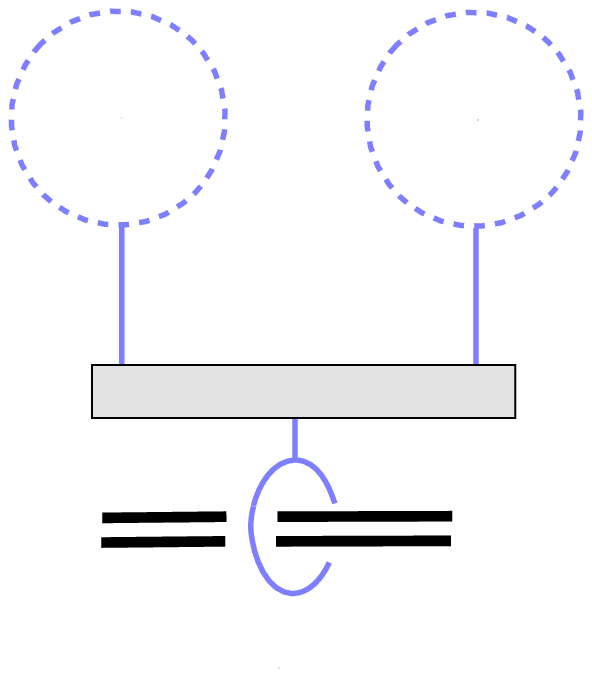}
\end{minipage}
\caption{\label{F:shepherd} The leaf shepherd procedure.}
\end{figure}

Let $B_{1,2}$ denote the bands clasped by $A^{1,2}$ correspondingly, coloured $a_{1,2}\in A$ correspondingly. Assume without the limitation of generality that $B_2$ is the left band of a $1$--handle. If $B_{1,2}$ are different, and if they are not adjacent along $D^2$, then the first step is to bring them close together. Our graphical convention above in what follows is to write the name of the leaf above its adjacent edge.

 \newcounter{stepnum}
\begin{list}{\textbf{Step \arabic{stepnum}:}}
        {\setlength{\leftmargin}{.5in}
        \setlength{\rightmargin}{.1in}
        \usecounter{stepnum}}
\item If $B_{1,2}$ are different, we find a mapping class $\tau\in\mathrm{MCG}(F)$ whose action gives a band projection for $F$ in which $C_{1,2}$ clasp the same band. Lemma \ref{L:(0,a,b)} is used to kill the excess $Y$--claspers we create along the way.

    \newcounter{casenum}[stepnum]
    \begin{list}{\textbf{Case (\roman{casenum}):}}
        {\setlength{\leftmargin}{.5in}
        \setlength{\rightmargin}{.1in}
        \usecounter{casenum}}
        \item If $B_1$ and $B_2$ are the two bands of the same handle, choose $\tau$ to be the following Dehn twist:

           \begin{equation}
                \psfrag{A}[c]{\fs$B_1$}\psfrag{B}[c]{\fs$\phantom{B}B_2$}\psfrag{a}[cb]{\fs$A^1$}\psfrag{b}[cb]{\fs$A^2$}
                \begin{minipage}{100pt}
                \includegraphics[width=100pt]{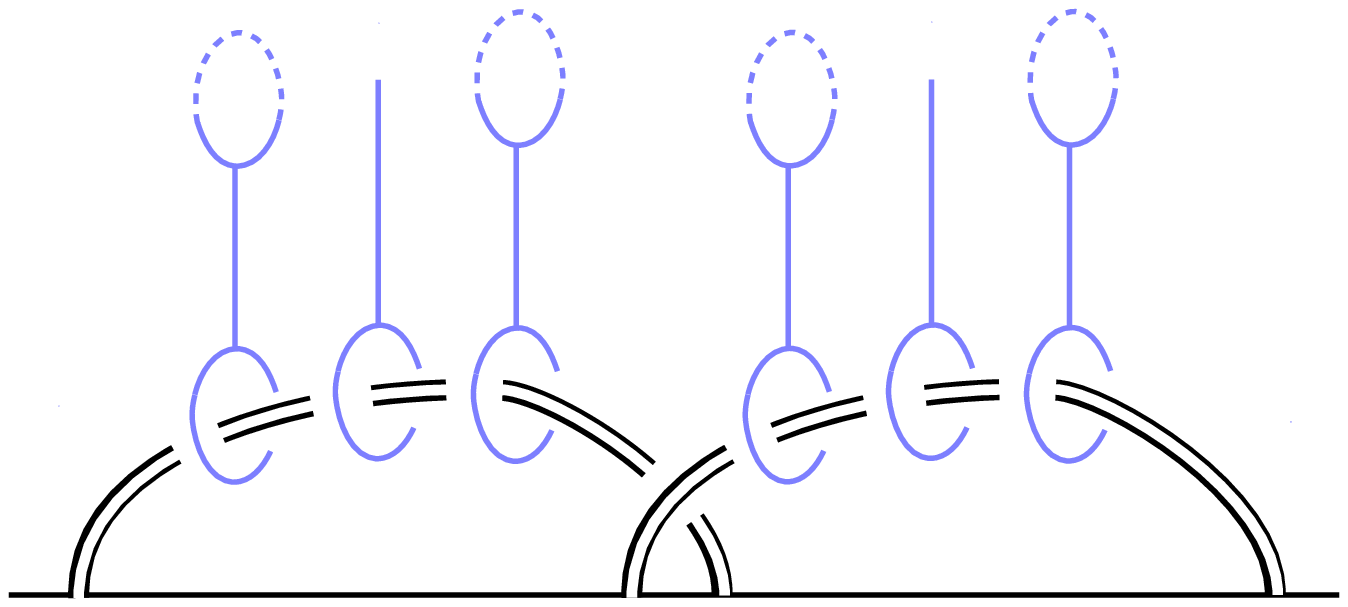}
                \end{minipage}\ \ \ \overset{\raisebox{2pt}{\scalebox{0.8}{\text{isotopy}}}}{\Longleftrightarrow}\ \ \
                \left\{\, \begin{array}{c}
                \raisebox{20pt}{\begin{minipage}{130pt}
                \includegraphics[width=100pt]{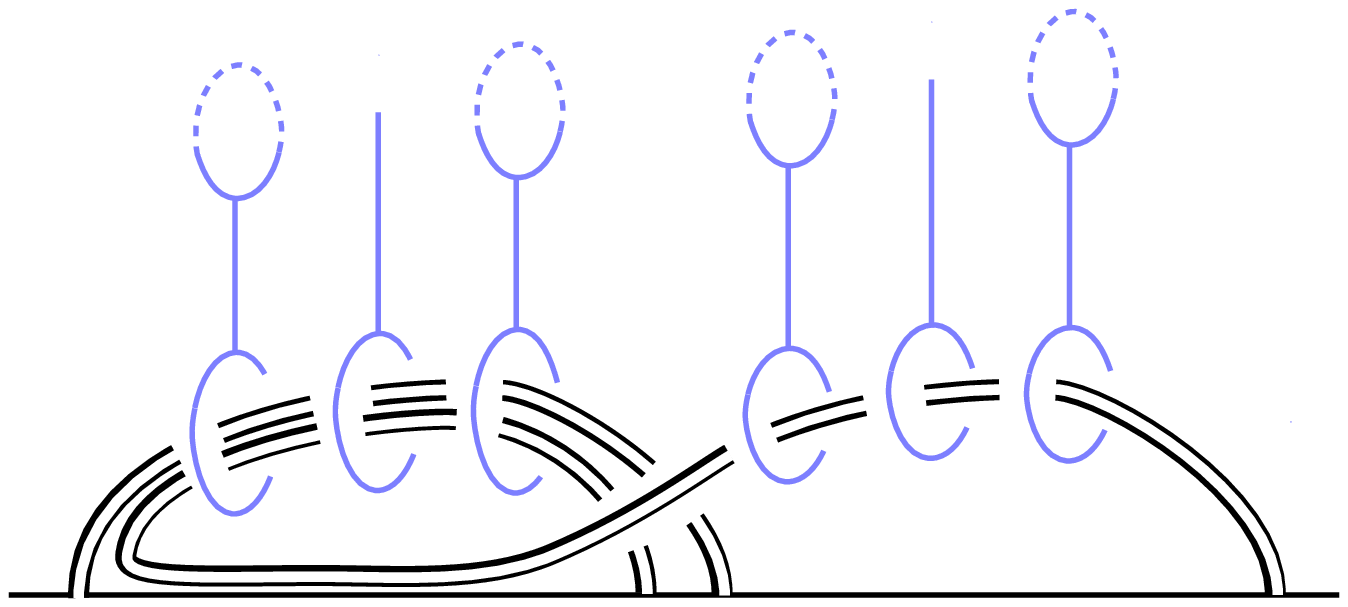}
                \end{minipage}}\raisebox{1cm}{\text{\Small if $a_1=a_2$;}}\\[0.6cm]
                \raisebox{20pt}{\begin{minipage}{130pt}
                \includegraphics[width=130pt]{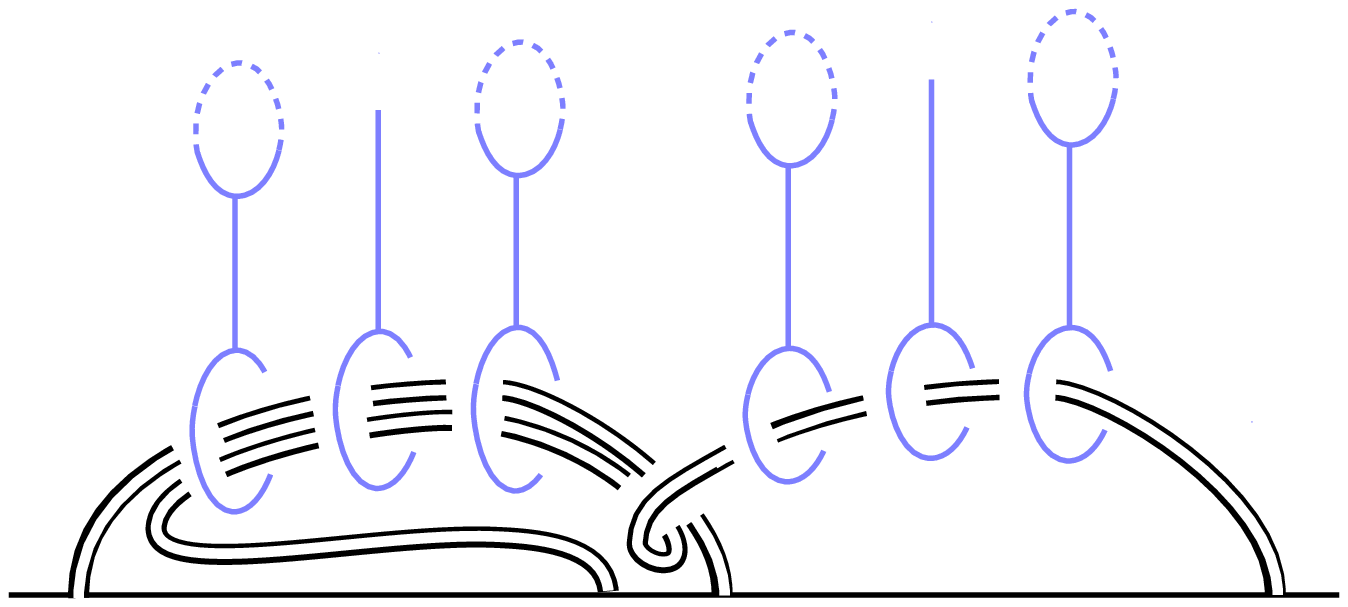}
                \end{minipage}}\ \ \raisebox{1cm}{\text{\Small if $a_1+a_2=0$.}}
                \end{array}\right.
            \end{equation}

        \item Otherwise, if $B_{1,2}$ belong to different $1$--handles, let $B$ denote the band left adjacent to the band $B_2$ clasped by $A^2$. Explicitly, if we write $D^2\cap B=\alpha\cup\beta$ and $D^2\cap B_{1,2}=\alpha_{1,2}\cap\beta_{1,2}$, then $\partial D^2$ contains a line segment of the form $x\delta\alpha_2\beta^\prime$ with $x=\alpha$ or $x=\beta^{-1}$. Repeat the following step, until $x=\alpha_1$ if $a_1=a_2$, or until $x=\beta_1^{-1}$ if $a_1=a_2^{-1}$. Slide $B$ over the $B_2$'s $1$--handle as follows:

            \scalebox{0.76}{\parbox{\textwidth}{%
            \begin{multline}
                \psfrag{A}[c]{$B$}\psfrag{B}[c]{$B_2$}\psfrag{a}[c]{\phantom{a}$A^2$}
                \begin{minipage}{138pt}
                \includegraphics[height=65pt]{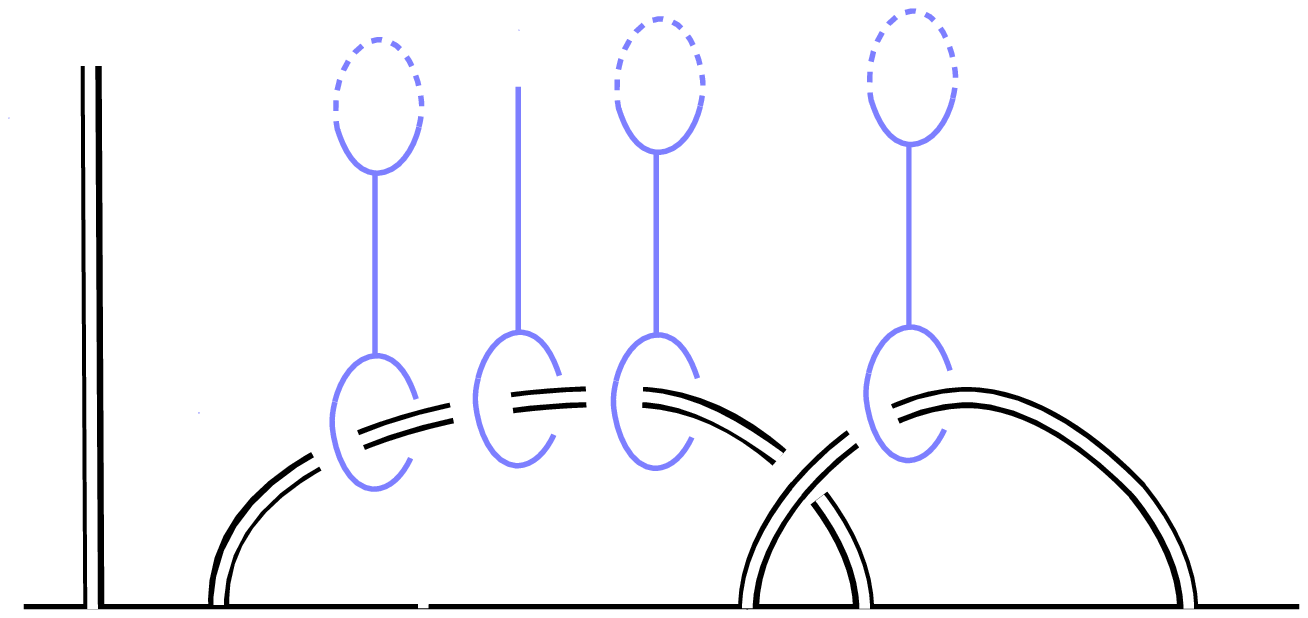}
                \end{minipage}
                \hspace{-9pt}\overset{\raisebox{2pt}{\scalebox{0.8}{\text{isotopy}}}}{\Longleftrightarrow}\
                \begin{minipage}{155pt}
                \includegraphics[height=65pt]{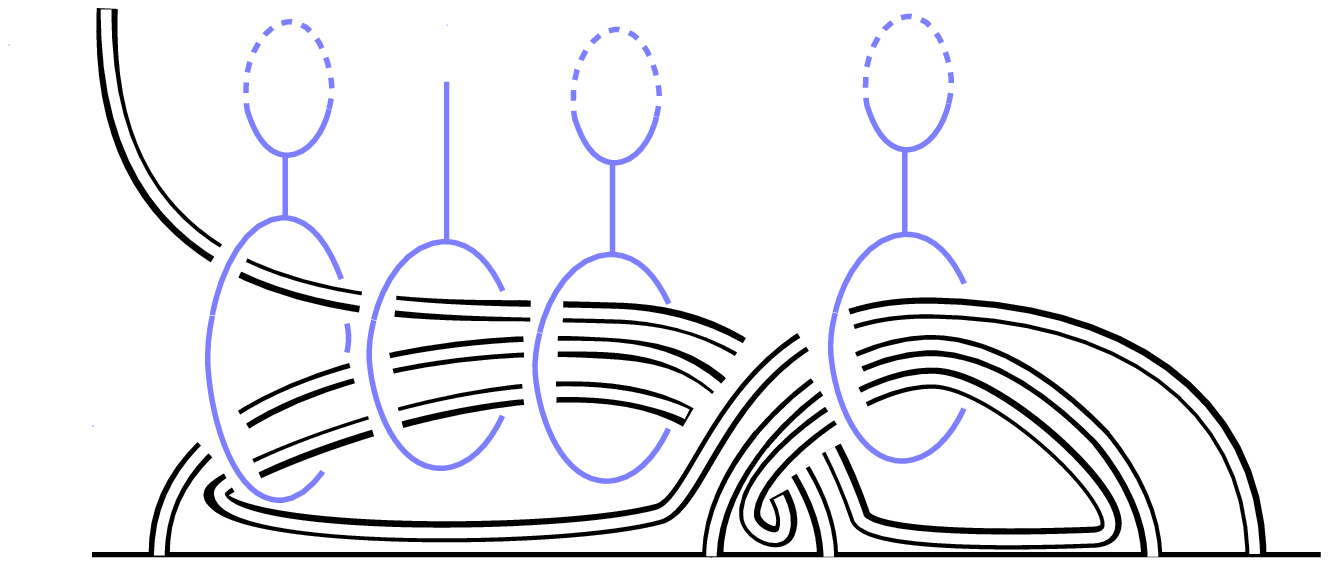}
                \end{minipage}\\[0.3cm]
                \psfrag{A}[c]{$B$}\psfrag{B}[c]{$B_2$}\psfrag{a}[c]{$\phantom{a}A^2$}
                \ \overset{\raisebox{2pt}{\scalebox{0.8}{\text{Move 8}}}}{\Longleftrightarrow}\quad
                \begin{minipage}{155pt}
                \includegraphics[height=65pt]{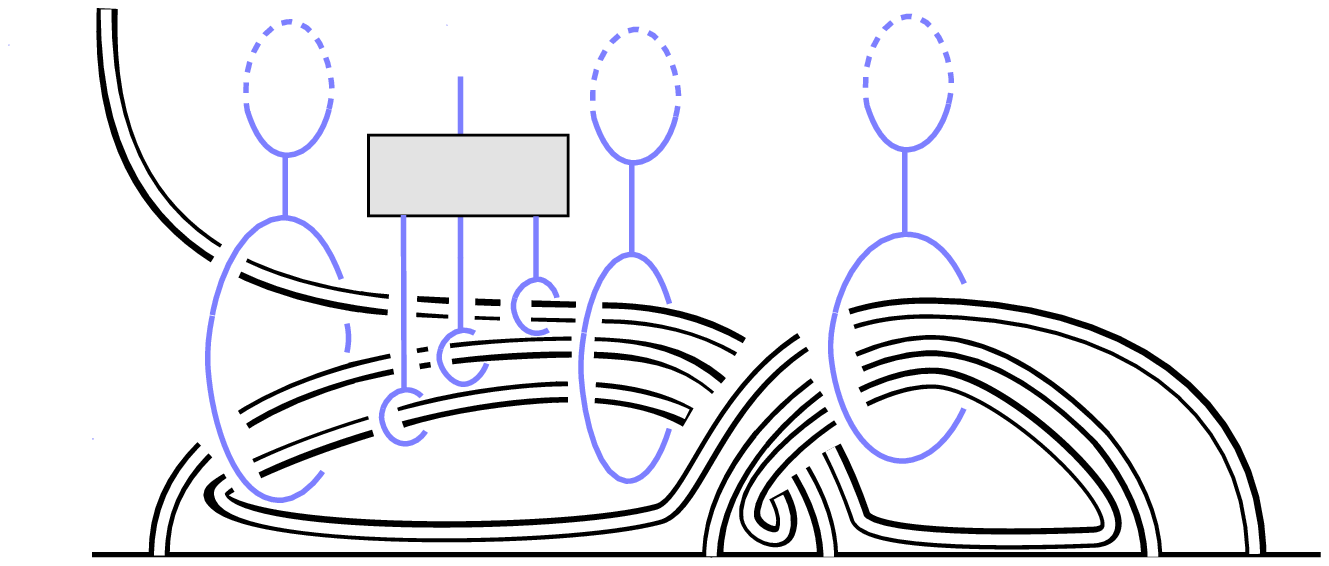}
                \end{minipage}
            \end{multline}}}

            \noindent Unzip the resulting clasper \cite[Definition 3.12]{Hab00}. Finally, when $B_1$ becomes adjacent to $B_2$ in the prescribed fashion, slide $B_2$ over $B_1$ (the diagram is of one possible configuration of the ends of the bands--- other possible configurations are handled analogously):

            \begin{equation}
                \psfrag{A}[c]{\fs$B_1$}\psfrag{B}[c]{\fs$B_2$}\psfrag{a}[cb]{\fs$A^1$}\psfrag{b}[cb]{\fs$A^2$}
                \begin{minipage}{100pt}
                \includegraphics[width=100pt]{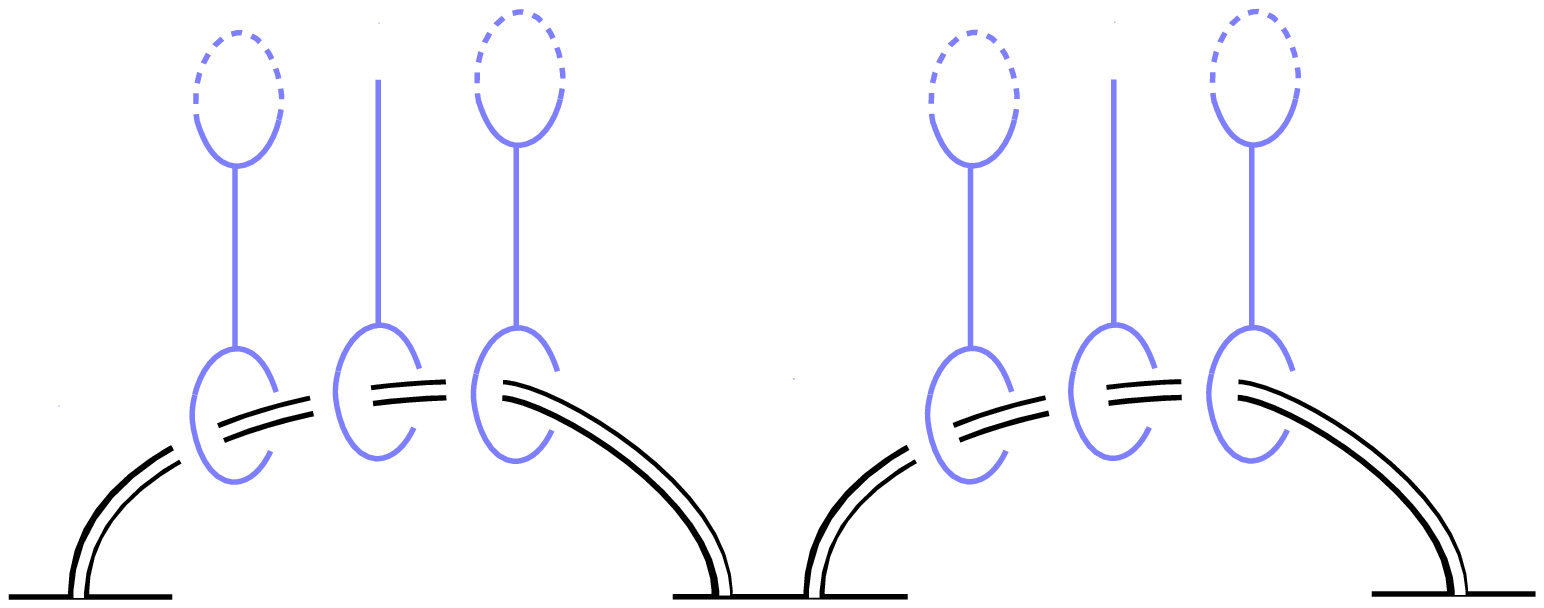}
                \end{minipage}\quad \overset{\raisebox{2pt}{\scalebox{0.8}{\text{isotopy}}}}{\Longleftrightarrow}\quad
                \begin{minipage}{130pt}
                \includegraphics[width=100pt]{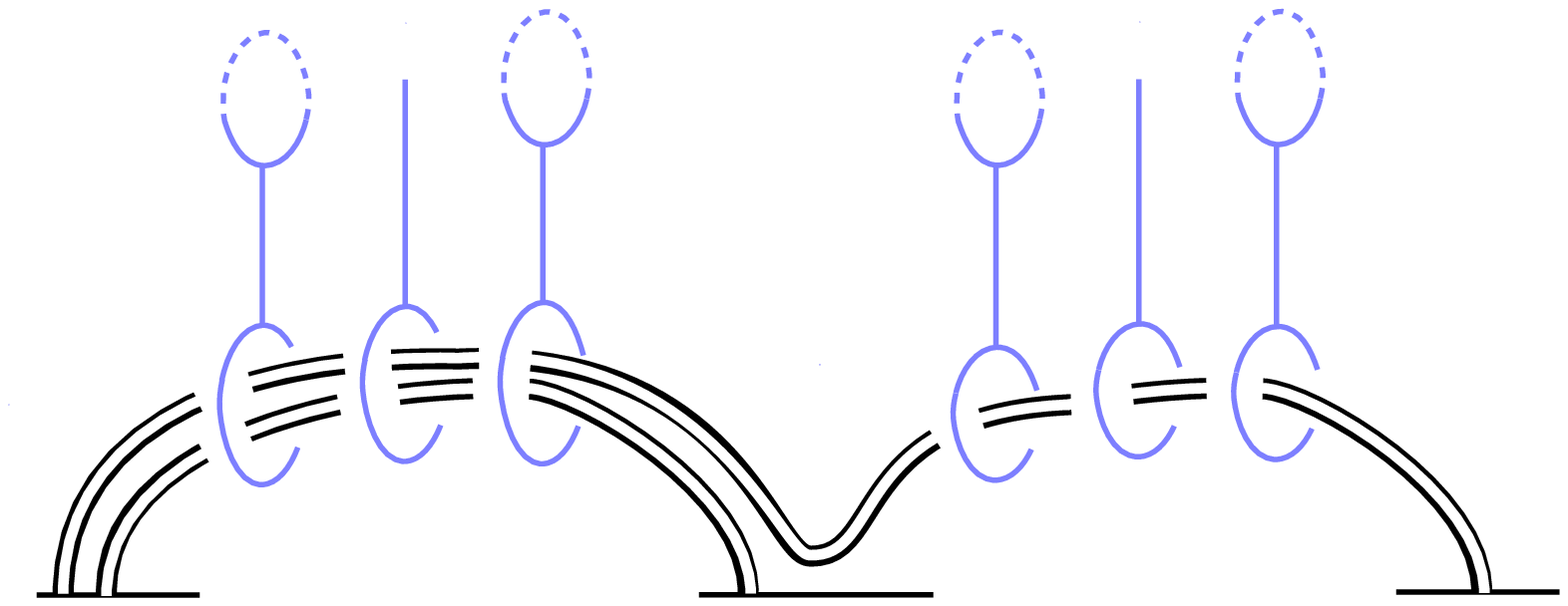}
                \end{minipage}.
            \end{equation}
    \end{list}

This slide sets the colour of $B_1$ to $0\in A$. We are left with the following local picture:

\begin{equation}
                \raisebox{10pt}{\begin{minipage}{60pt}
                \psfrag{A}[c]{\phantom{a}$A^1$}\psfrag{B}[c]{$B_1$}\psfrag{C}[c]{$B_2$}
                \includegraphics[width=55pt]{pile8-1}
                \end{minipage}}
                \quad
                \overset{\raisebox{2pt}{\scalebox{0.8}{\text{Move 8}}}}{\Longleftrightarrow}\quad
                \psfrag{A}[c]{\phantom{a}$A^{1\prime}$}\psfrag{B}[c]{$B_1$}\psfrag{C}[c]{$B_2$}\psfrag{D}[c]{\phantom{a}$A^{1\prime\prime}$}
                \ \raisebox{10pt}{\begin{minipage}{60pt}
                \includegraphics[width=60pt]{pile8-2}
                \end{minipage}}
                \quad\ \overset{\raisebox{2pt}{\scalebox{0.8}{\text{unzip}}}}{\Longleftrightarrow}\quad
                \ \raisebox{10pt}{\begin{minipage}{60pt}
                \includegraphics[width=60pt]{pile8-3}
                \end{minipage}}
\end{equation}

\noindent Delete the clasper which contains $A^{1\prime}$ using Lemma \ref{L:(0,a,b)} and rename $A^1$ as $A^\prime$. We take $\tau$ to be the mapping class corresponding to this step.

\item Now that $A^{1,2}$ clasp a common band, shepherd them together:

\begin{equation}
\begin{minipage}{61pt}
\includegraphics[width=61pt]{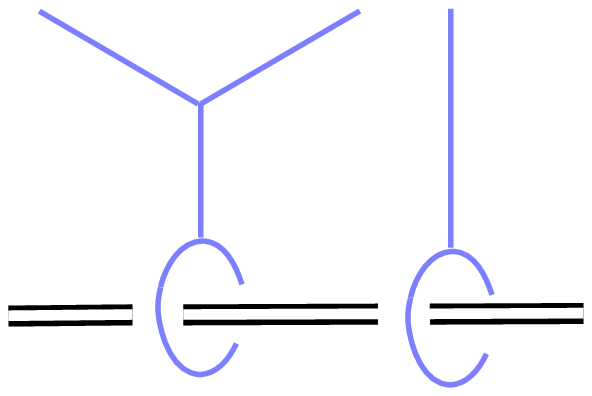}
\end{minipage}\overset{\raisebox{2pt}{\scalebox{0.6}{\text{isotopy}}}}{\Leftrightarrow}
\begin{minipage}{61pt}
\includegraphics[width=61pt]{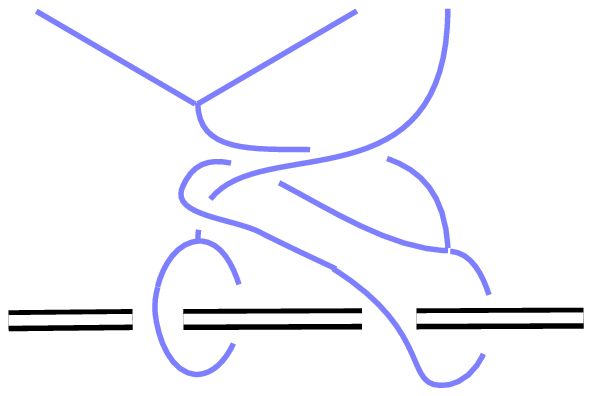}
\end{minipage}\overset{\raisebox{2pt}{\scalebox{0.6}{\text{Move 8}}}}{\Leftrightarrow}
\begin{minipage}{61pt}
\includegraphics[width=61pt]{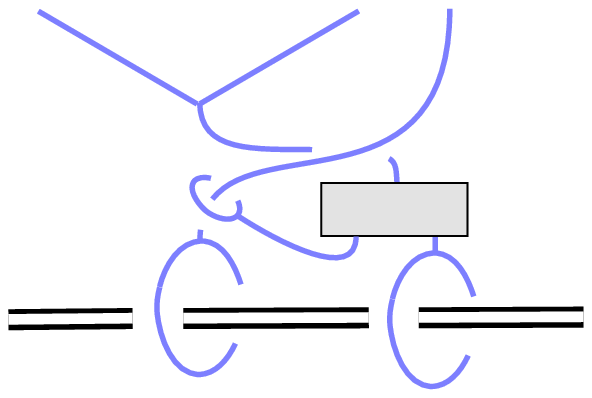}
\end{minipage}\overset{\raisebox{2pt}{\scalebox{0.6}{\text{Lemma {\ref{L:(0,a,b)}}}}}}{\Leftrightarrow}
\begin{minipage}{61pt}
\includegraphics[width=61pt]{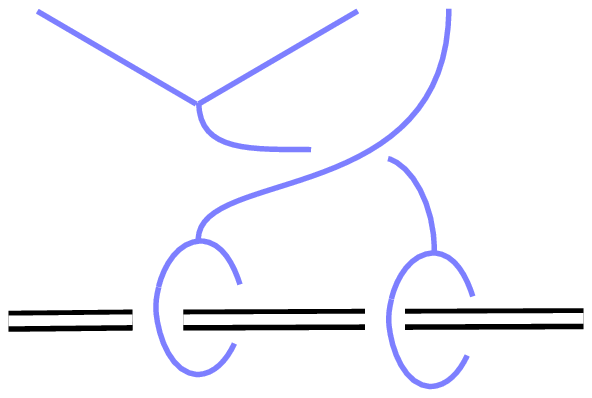}
\end{minipage}.
\end{equation}

\noindent Once $A^{1,2}$ are adjacent, create a box using Move 8.

\item Act by $\tau^{-1}$ to return to the band projection with which we started.
\end{list}

\subsubsection{Adding claspers geometrically}

\begin{lem}\label{L:monogamousclasper}
 Let $C_{1,2}\ass\, A^{1,2}_{1,2,3}\cup E^{1,2}_{1,2,3}$ be a pair of $Y$--claspers in the complement of an $A$--coloured Seifert surface $(F,\bar\rho)$ in band projection, whose leaves $A_{1,2,3}^{1,2}$ clasp single bands $B^{1,2}_{1,2,3}$ correspondingly, with $(B_1^1,B_1^2,B_2^{1,2},B_3^{1,2})$ coloured $(a,b,c,d)$ correspondingly. There exists a $Y$--clasper $C_{3}\ass\,A^{3}_{1,2,3}\cup E^3_{1,2,3}$ in the complement of $(F,\bar\rho)$ whose leaves $A_{1,2,3}^3$ clasp bands $B^3_{1,2,3}$ coloured $(a+b)$, $c$, and $d$ correspondingly.
\end{lem}

\begin{proof}
Shepherd leaves to bring together $A_{2,3}^{1,2}$ coloured $c$ and $d$ (Section \ref{SSS:Shepherd}). If any one of the edges $E_{1,2,3}^{1,2}$ crosses under an edge of another clasper, or under a band, use Lemma \ref{L:clasp-pass} or \ref{L:YPass} to change that crossing, to make  $E_{1,2,3}^{1,2}$ cross over all edges and all over bands. Untie $E_{1,2,3}^{1,2}$ (Lemma \ref{L:clasp-pass}), and remove all full twists in them (Corollary \ref{C:ClasperFramingReduce}). Push the two boxes past the trivalent vertex as follows:

\begin{equation}
\psfrag{a}[c]{$a$}\psfrag{b}[c]{$b$}\psfrag{c}[c]{$c$}\psfrag{d}[c]{$d$}
\begin{minipage}{90pt}
\includegraphics[height=61pt]{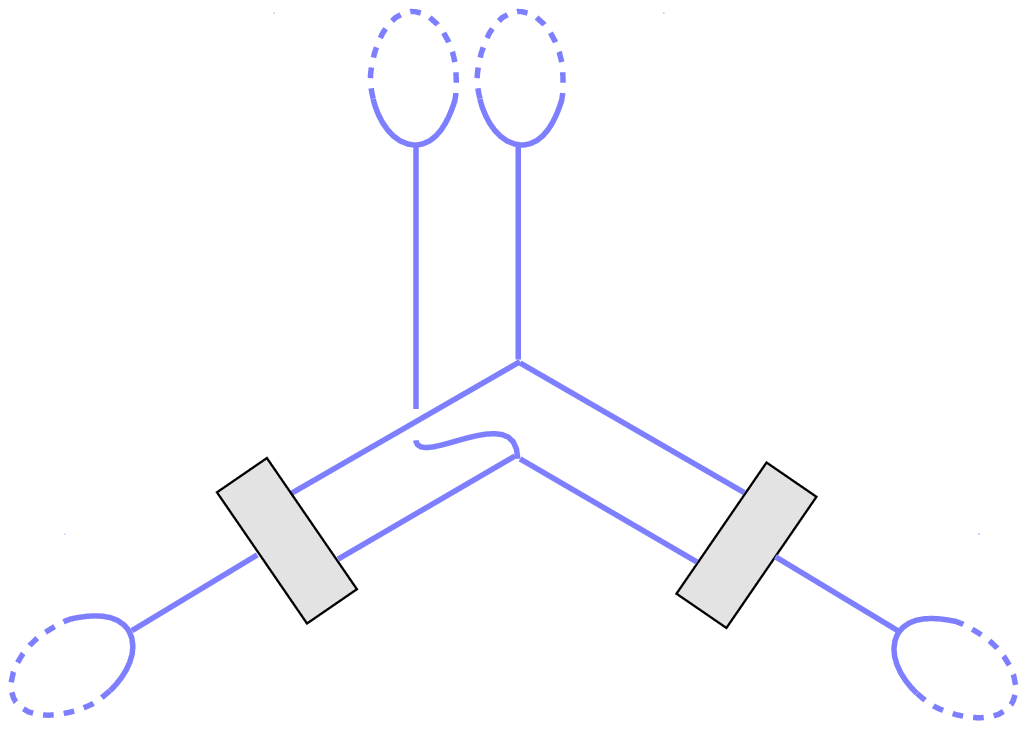}
\end{minipage}\overset{\raisebox{2pt}{\scalebox{0.6}{\text{Lemma \ref{L:(0,a,b)}}}}}{\Longleftrightarrow}
\begin{minipage}{90pt}
\includegraphics[height=61pt]{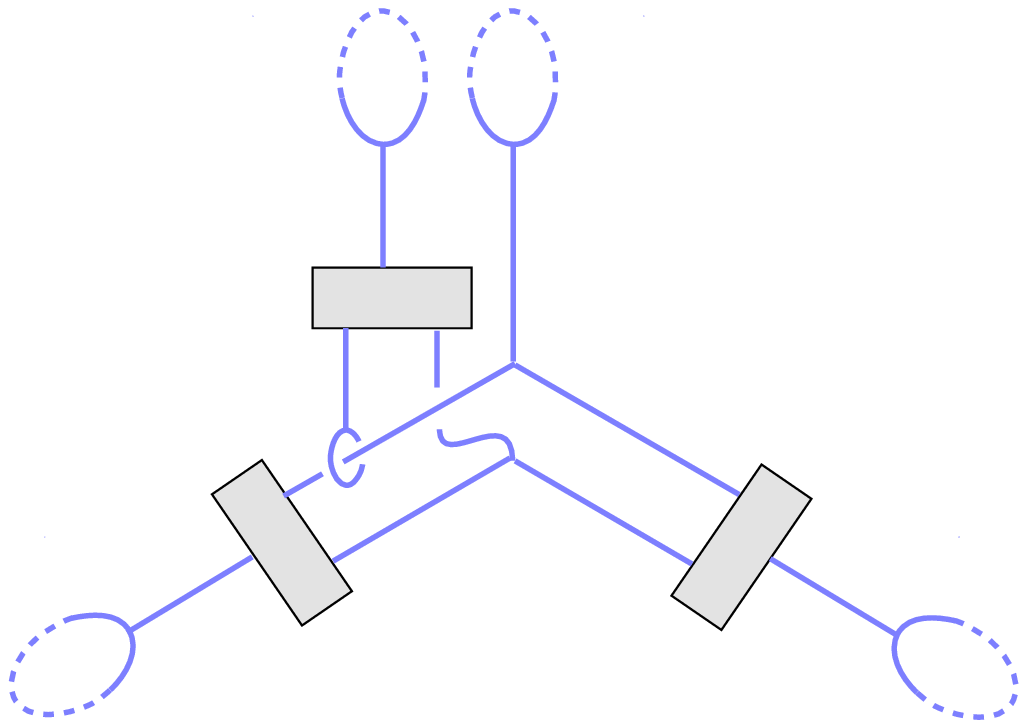}
\end{minipage}\overset{\raisebox{2pt}{\scalebox{0.6}{\text{Move 11}}}}{\Longleftrightarrow}
\begin{minipage}{61pt}
\includegraphics[height=61pt]{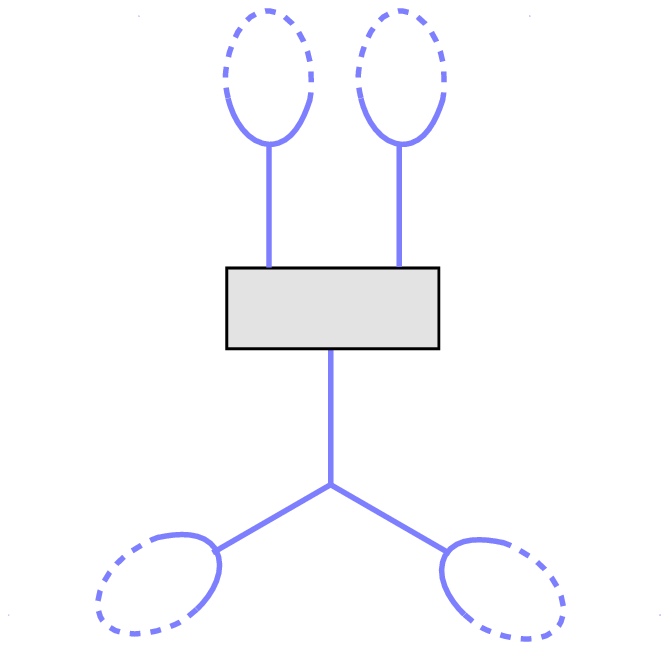}
\end{minipage}.
\end{equation}

This unites the two pairs of leaves $A_2^{1,2}$ and $A_3^{1,2}$ into single leaves which we suggestively call $A^3_2$ and $A^3_3$ correspondingly, and the two pairs of edges $E_{2}^{1,2}$ and $E_3^{1,2}$ into single edges which we suggestively call $E^3_2$ and $E^3_3$ correspondingly.\par

As in Step 1 of Section \ref{SSS:Shepherd}, bring $A_1^{1,2}$ to adjacent positions along $D^2$. Slide $A_1^2$ over $A_2^2$, and resolve as follows (we draw the procedure in the case that $A_1^{1,2}$ belong to the same handle. The remaining case is analogous):

\begin{multline}
\psfrag{A}[c]{\fs$B^1_1$}\psfrag{B}[c]{\phantom{a}\fs$B^2_1$}\psfrag{a}[cb]{\fs$A^1_1$}\psfrag{b}[cb]{\fs$A^2_1$}
                \begin{minipage}{100pt}
                \includegraphics[width=100pt]{twinY-1}
                \end{minipage}\quad\ \overset{\raisebox{2pt}{\scalebox{0.8}{\text{isotopy}}}}{\Longleftrightarrow}\quad
                \begin{minipage}{100pt}
                \includegraphics[width=100pt]{twinY-2}
                \end{minipage}\\[0.4cm]
                \psfrag{A}[c]{\fs$B^1_1$}\psfrag{B}[c]{\fs$\,B^2_1$}\psfrag{a}[cb]{\fs$A^1_1$}\psfrag{b}[cb]{\fs$A^2_1$}
                \overset{\raisebox{2pt}{\scalebox{0.8}{\text{Move 8}}}}{\Longleftrightarrow}\quad
                \begin{minipage}{100pt}
                \includegraphics[width=100pt]{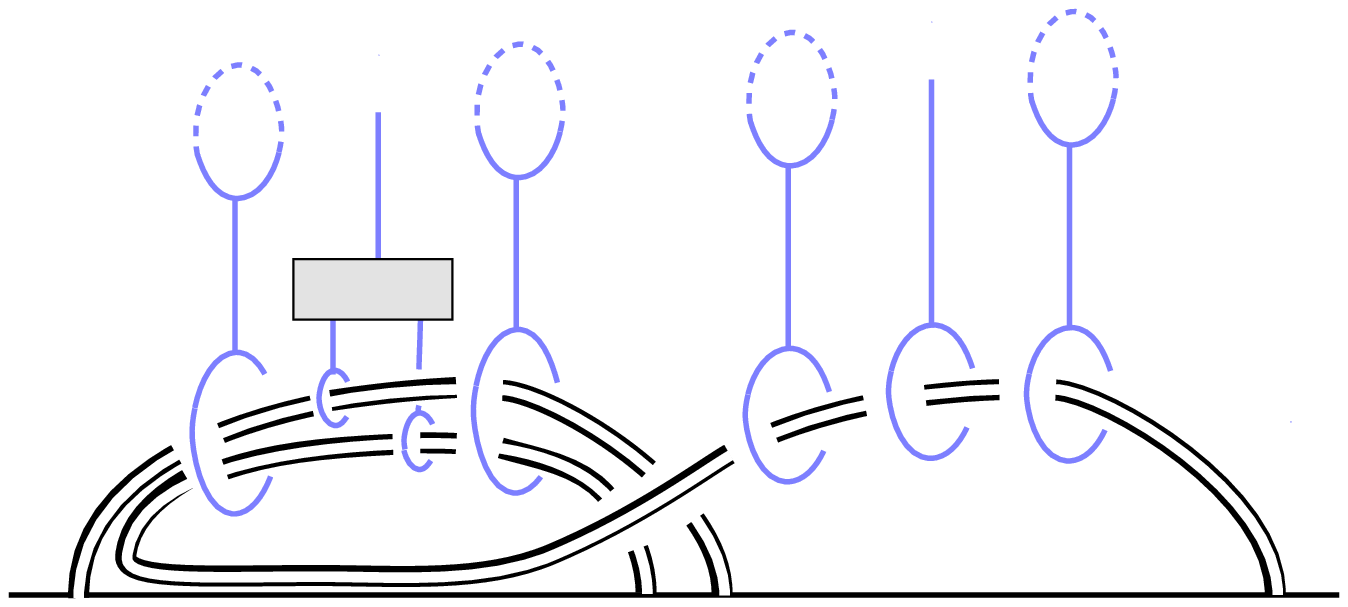}
                \end{minipage}
                \ \ \ \overset{\raisebox{2pt}{\scalebox{0.8}{\text{Lemma \ref{L:YPass}}}}}{\Longleftrightarrow}\quad
                \begin{minipage}{110pt}
                \includegraphics[width=110pt]{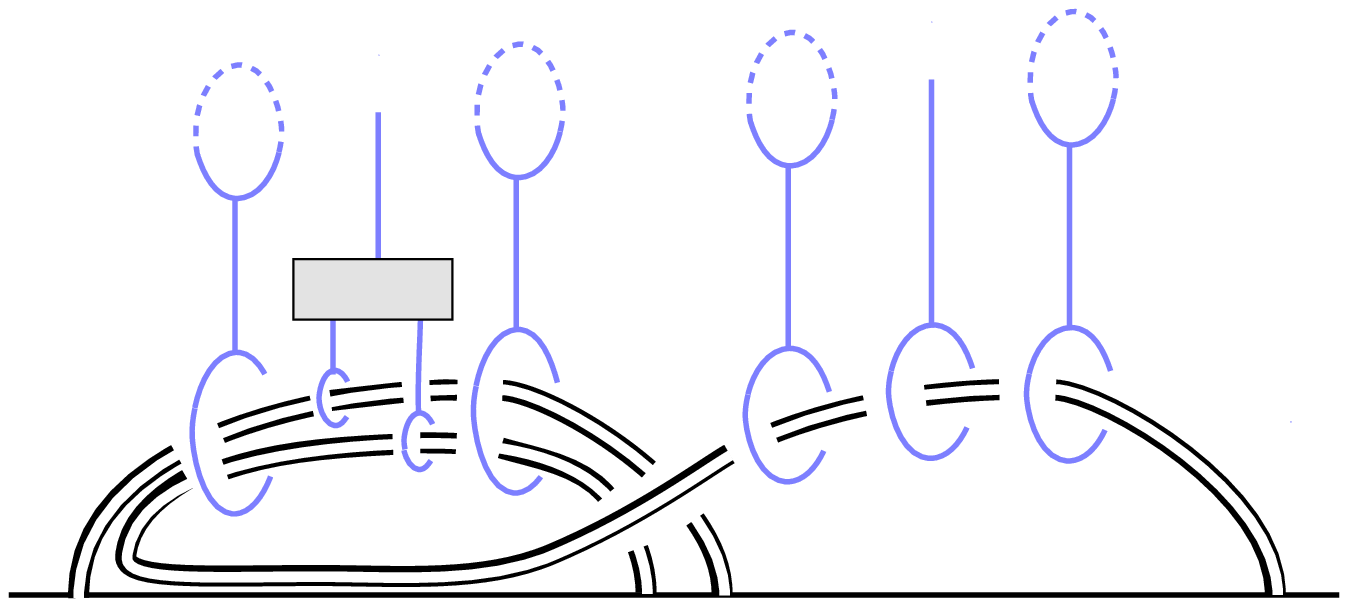}
                \end{minipage}\\[0.4cm]
                \overset{\raisebox{2pt}{\scalebox{0.8}{\text{Lemma \ref{L:(0,a,b)}}}}}{\Longleftrightarrow}\quad
                 \begin{minipage}{110pt}
                \psfrag{A}[c]{\fs$B^1_1$}\psfrag{B}[c]{\fs$\,B^2_1$}\psfrag{a}[cb]{\fs$A^1_1$}\psfrag{b}[cb]{\fs$A^2_1$}
                \includegraphics[width=110pt]{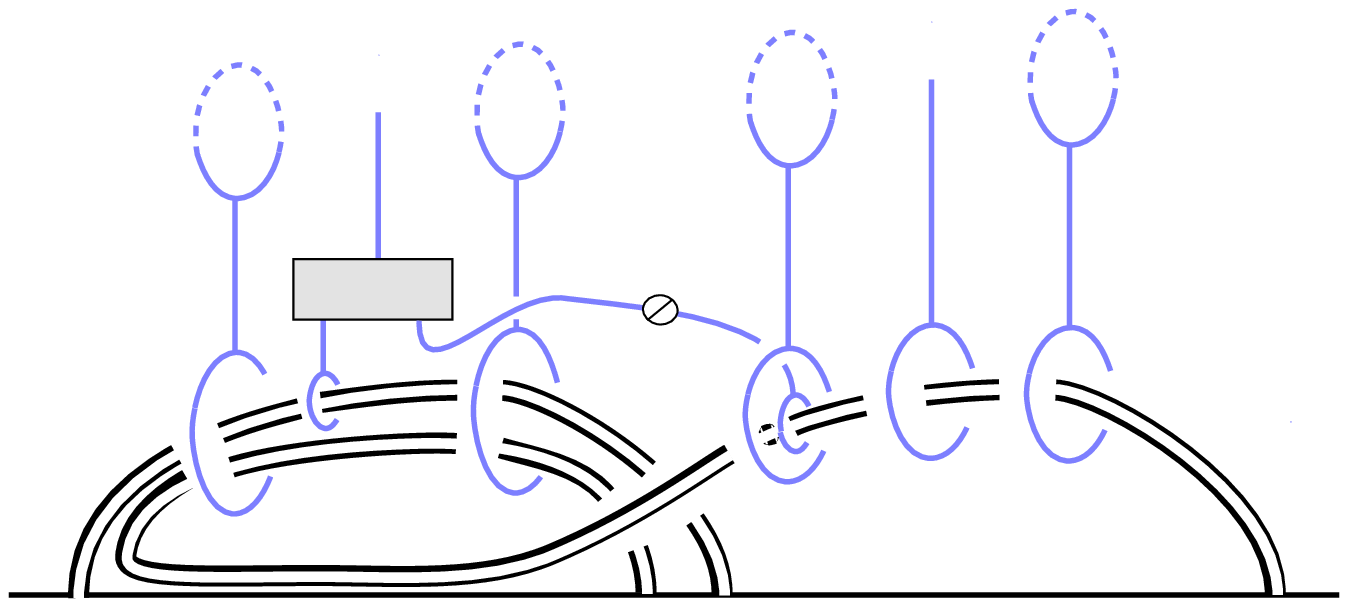}
                \end{minipage}
\end{multline}

In the above sequence, $A^1_1$ was broken up into two leaves, which we sloppily collectively called $A_1^1$. This sloppiness causes no harm because of the next step.

Shepherd $A_{1}^1$ and $A_{1}^2$ together as in Step 2 of the procedure in Section \ref{SSS:Shepherd}, and manipulate the resulting local picture as follows:

\begin{equation}
\psfrag{E}[c]{\fs$E^3_2$}\psfrag{F}[c]{\fs$E^3_3$}\psfrag{A}[c]{\fs$A_1^1$} \psfrag{B}[c]{\fs$A_1^2$}
\begin{minipage}{61pt}
\includegraphics[width=61pt]{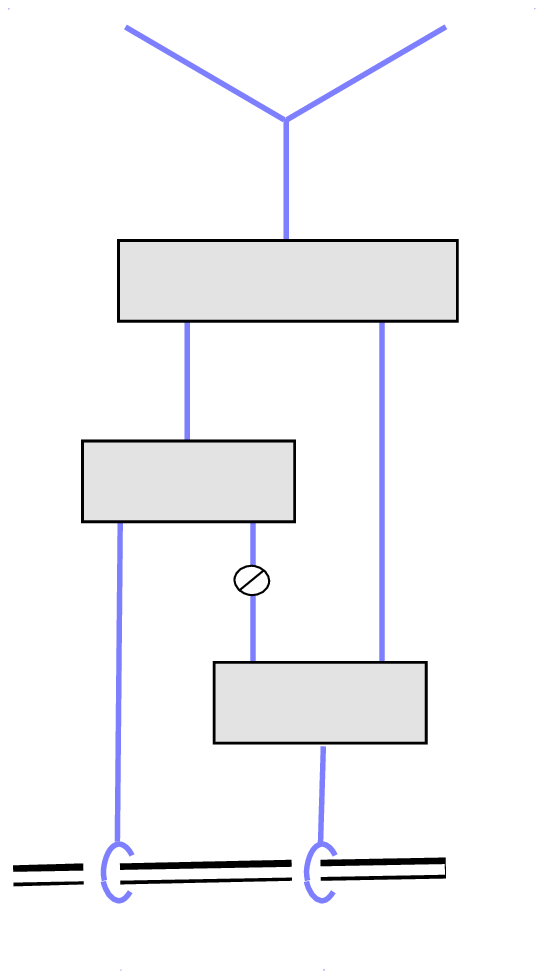}
\end{minipage}\overset{\raisebox{2pt}{\scalebox{0.6}{\text{unite-box}}}}{\Longleftrightarrow}
\begin{minipage}{61pt}
\includegraphics[width=61pt]{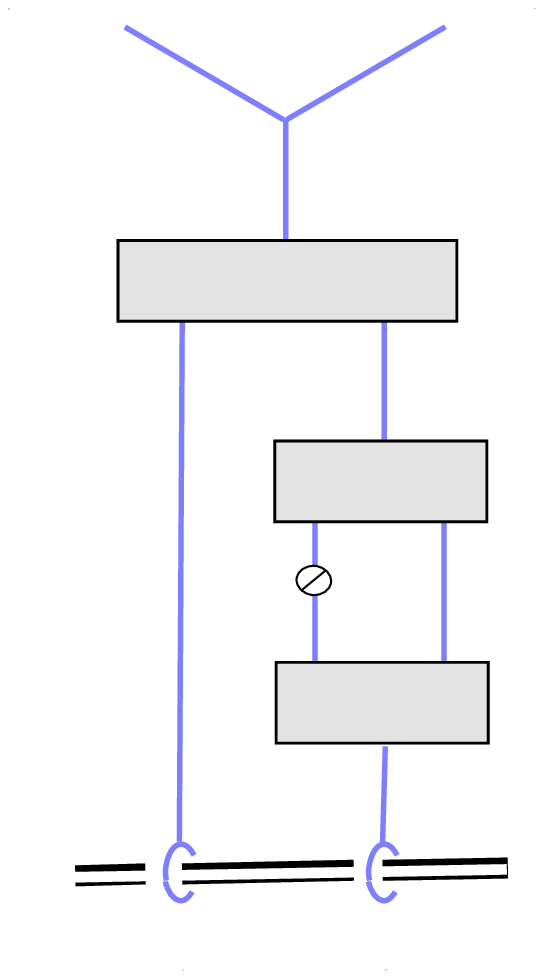}
\end{minipage}\ \ \overset{\raisebox{2pt}{\scalebox{0.6}{\text{Move 4}}}}{\Longleftrightarrow}
\begin{minipage}{61pt}
\includegraphics[width=61pt]{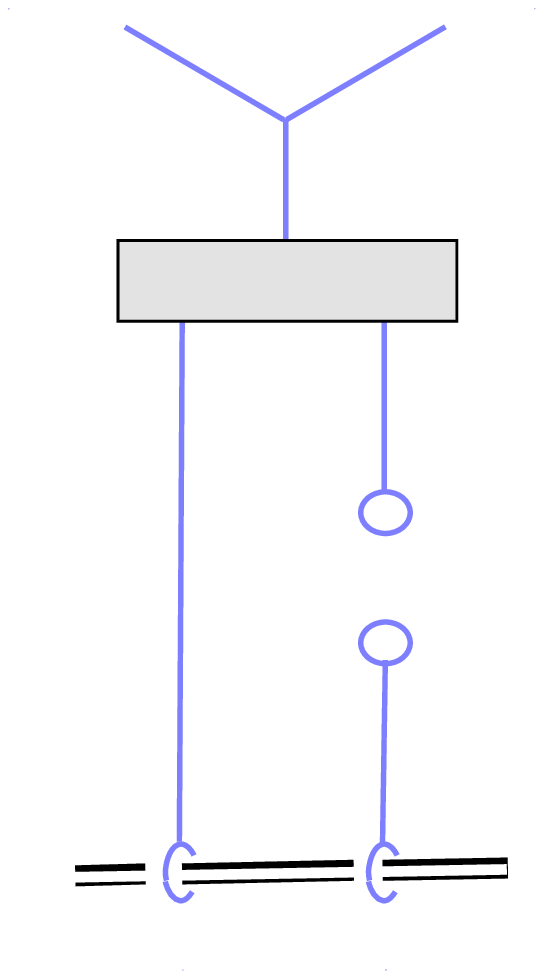}
\end{minipage}\ \ \overset{\raisebox{2pt}{\scalebox{0.6}{\text{Moves 1,3}}}}{\Longleftrightarrow}\!
\hspace{-0.6cm}\begin{minipage}{61pt}
\includegraphics[width=61pt]{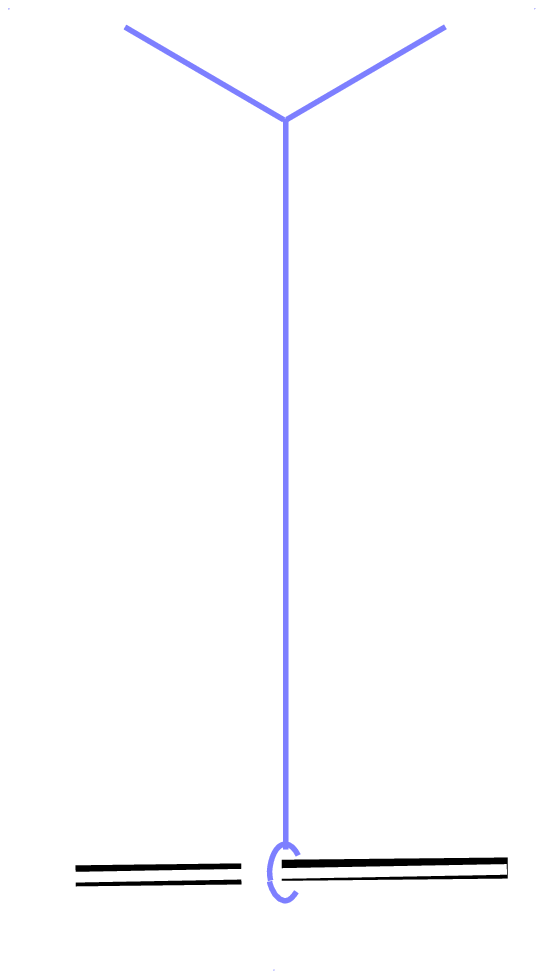}
\end{minipage}
\end{equation}

\noindent Finally, we are left with a single $Y$--clasper with three leaves: $A_1^1$, which we relabel $A^3_1$, which clasps a band coloured $a+b$, and $A_{2,3}^3$ which clasp bands coloured $c$ and $d$ respectively.
\end{proof}

\subsection{Leaves clasping multiple bands}

\begin{lem}\label{L:3wedge}
If $a\wedge b\wedge c=0\in\bigwedge^3 A$, then
\[\lblY{a}{b}{c}{27.5}\sim_{\bar\rho}\,0.\]
\end{lem}

\begin{proof}
 Consider a $Y$--clasper $C\in \lblY{a}{b}{c}{27.5}$. Shorten words (Section \ref{SSS:Shorten}) with respect to an ordered basis $\mathcal{B}$ for $A$ which contains a maximal independent subset of $S\subset \set{a,b,c}$. Use Move 8 and unzip to split $C$ into a collection of claspers, each of whose leaves clasps a single band. Each clasper $C^\prime$ in this collection which has a leaf which clasps a band coloured $d\notin S$ has a counterpart $C^{\prime\prime}$ whose corresponding leaf clasps a band coloured $-d$, and these cancel by Lemma \ref{L:monogamousclasper} combined with Lemma \ref{L:(0,a,b)}. We are left with claspers whose leaves clasp bands all of whose colours are in $S$, which cancel by Lemma \ref{L:(a,a,b)} because $S$ has cardinality at most $2$.
\end{proof}

\begin{lem}\label{L:Ydistribute}
\[\begin{minipage}{30pt}
\psfrag{a}[c]{\small$a+b$}\psfrag{b}[c]{\small$c$}\psfrag{c}[c]{\small$d$}\includegraphics[width=30pt]{labeledYz}\end{minipage}\
\sim_{\bar\rho}\lblY{a}{c}{d}{27.5}+ \lblY{b}{c}{d}{27.5}\]
\end{lem}

\begin{proof}
We show that any $A$--coloured Seifert surface $(F,\bar\rho)$ is $\bar\rho$--equivalent to any $A$--coloured Seifert surface $(F^\prime,\bar\rho^\prime)$ obtained from $(F,\bar\rho)$ through a finite sequence of $Y$--clasper surgeries, deletion of an element in $\begin{minipage}{30pt}
\psfrag{a}[c]{\small$a+b$}\psfrag{b}[c]{\small$c$}\psfrag{c}[c]{\small$d$}\includegraphics[width=30pt]{labeledYz}\end{minipage}$, and insertion of  an element in $\lblY{a}{c}{d}{27.5}+ \lblY{b}{c}{d}{27.5}$. The converse follows analogously.\par

Consider $Y$--claspers $C_{1,2}\ass\, A^{1,2}_{1,2,3}\cup E^{1,2}_{1,2,3}$  in $\lblY{a}{c}{d}{27.5}$ and in $\lblY{b}{c}{d}{27.5}$ correspondingly, such that the colours of $A^{1,2}_{2,3}$ are $c$ and $d$ correspondingly. By Lemma \ref{L:3wedge} we may assume that $c\wedge d\neq 0$. As in the proof of Lemma \ref{L:3wedge}, word shorten with respect to an ordered basis $\mathcal{B}$ for $A$ which contains a maximal independent subset of $S\subseteq \set{a,b,c,d}$, and then use Move 8 and unzip to split $A^1_2$, splitting $C_1$ into a collection of claspers each of whom has a distinguished leaf which clasps a single band. The distinguished leaf of each clasper $C^\prime$ in this collection which clasps a band coloured $x\neq c$ has a counterpart $C^{\prime\prime}$ whose corresponding leaf clasps a band coloured $-x$, and these cancel by Lemma \ref{L:monogamousclasper} combined with Lemma \ref{L:(0,a,b)}. Only one clasper $C_1^\prime$ survives, whose distinguished leaf clasps a band labeled $c$. Repeat the above procedure to replace $C_2$ by a corresponding $Y$--clasper $C_2^\prime$, and combine $C_{1,2}^\prime$ using Section \ref{SSS:Shepherd}. Repeat for $A_{3}^{1,2}$. Repeat again for $A_1^{1,2}$, except that this time $C_1$ turns into a clasper whose distinguished leaf clasps a band coloured $a$, while $C_2$ turns into either $n$ claspers whose distinguished leaf clasp single bands labeled $\pm a$ if $b=\pm na$ for $n\in\mathds{N}$, or into a single clasper whose distinguished leaf clasps a band coloured $b$ otherwise. Combine these using Lemma \ref{L:monogamousclasper} to obtain $C\in \begin{minipage}{30pt}
\psfrag{a}[c]{\small$a+b$}\psfrag{b}[c]{\small$c$}\psfrag{c}[c]{\small$d$}\includegraphics[width=30pt]{labeledYz}\end{minipage}$.
\end{proof}

\subsection{Proof of Theorem {\ref{T:clasperprop}}}\label{SS:ClasperProof}

Theorem \ref{T:clasperprop} is equivalent to the statement that two $A$--coloured Seifert surfaces sharing the same Seifert matrix are $\bar\rho$--equivalent if and only if they are related by inserting a $Y$--clasper in $\ker\Phi$ (a $Y_0$-move). This is implied by Proposition \ref{P:Y-barrho}, which we now prove.

\begin{proof}[Proof of Proposition \ref{P:Y-barrho}]
To prove that $\sim_{\bar\rho}$ is an equivalence relation we must prove that it is transitive. $\lblY{a}{b}{c}{27.5}\sim_{\bar\rho}\lblY{d}{e}{f}{28.5}\sim_{\bar\rho}\lblY{g}{h}{i}{27.5}$ implies that
\begin{equation}\lblY{d}{e}{f}{28.5}+\lblY{\bar a}{b}{c}{30}\sim_{\bar\rho} \lblY{g}{h}{i}{27.5}+\lblY{\bar a}{b}{c}{30}.\end{equation}
Adding $\lblY{a}{b}{c}{27.5}$ to both sides implies, by Lemma \ref{L:Ydistribute}, that $\lblY{d}{e}{f}{28.5}\sim_{\bar\rho}\lblY{g}{h}{i}{27.5}$ as required. By Lemma \ref{L:Ydistribute}, $\lblY{\bar a}{b}{c}{30}$ is the inverse of $\lblY{a}{b}{c}{27.5}$, making $C/\sim_{\bar\rho}$ into an abelian group. The map $\hat{\Phi}$ is surjective by Section \ref{SS:Y-nulltwist} and is injective by Lemma \ref{L:3wedge}, therefore it is an isomorphism.
\end{proof}

\section{Coloured untying invariants}\label{S:untyinginvariants}

 We construct invariants of $\rho$--equivalence classes and of $\bar\rho$--equivalence classes. In Sections \ref{S:r12} and \ref{S:A4} these will be used to bound from below the number of such classes, and to determine whether or not two given $G$--coloured knots $(K_{1,2},\rho_{1,2})$ are $\rho$--equivalent or $\bar\rho$--equivalent. In Section \ref{SS:su} we identify an analogue for $A$--coloured surfaces of the coloured untying invariant \cite[Section 6]{Mos06b}, and in Section \ref{SS:cu} we generalize the definition of the coloured untying invariant for covering spaces. The homological algebra parallels the treatment of Lannes and Latour \cite{LaLa75}, using methods in Hatcher \cite{Hat02}, and is condensed. The finitely generated abelian group $A$ is given the structure of a principal ideal ring, which by abuse of notation we also call $A$.

\subsection{An untying invariant for surfaces}\label{SS:su}

Let $(K,\rho)$ be a $G$--coloured knot. Choose a marked Seifert surface $\left(F,\set{x_1,\ldots,x_{2g}}\right)$ for $K$. By the Universal Coefficient Theorem, the colouring $\bar\rho\co H_1(E(F))\twoheadrightarrow A$ corresponds to a cohomology class $\bar\alpha\in H^1(E(F); A)$. Let $r$ be the rank of $A$ as a $\mathds{Z}$--module with presentation

\begin{equation}\label{E:split}
0\Too\mathds{Z}^r\overset{\iota}{\Too}\mathds{Z}^r  \overset{\mathrm{p}}{\Too} A \Too 0.
\end{equation}

\noindent If it happens to be the case that $A$ is of the form $\left(\mathds{Z}/n\mathds{Z}\right)^r$, then $\iota$ is represented by the matrix $nI_r$, and $p$ is the `modulo $n$' map. For $k\in\{1,2,\ldots\}$, the above maps extend by linearity:

\begin{equation}
0\Too\left(\mathds{Z}^r\right)^k\overset{\iota}{\Too}\left(\mathds{Z}^r\right)^k  \overset{\mathrm{p}}{\Too} A^k \Too 0.
\end{equation}

Short exact sequence \ref{E:split} gives rise to a long exact sequence on homology

\[
\cdots\rightarrow H_2(E(F);A))\overset{\raisebox{2pt}{\scalebox{0.8}{$\beta_2$}}}{\rightarrow}  H_1(E(F);\mathds{Z}^r)\overset{\raisebox{2pt}{\scalebox{0.8}{$\iota_\ast$}}}{\rightarrow} H_1(E(F);\mathds{Z}^r) \overset{\raisebox{2pt}{\scalebox{0.8}{$\mathrm{p}_\ast$}}}{\rightarrow} H_1(E(F); A)\rightarrow\cdots
\]

\noindent where $\beta_\ast$ is the Bockstein homomorphism on homology; and to the long exact sequence on cohomology

\[
\cdots\rightarrow H^1(E(F);\mathds{Z}^r)\overset{\iota^\ast}{\rightarrow} H^1(E(F);\mathds{Z}^r)\overset{\mathrm{p}^\ast}{\rightarrow} H^1(E(F);A)\overset{\beta^1}{\rightarrow} H^2(E(F); \mathds{Z}^r)\rightarrow\cdots
\]
\noindent where $\beta^\ast$ is the Bockstein homomorphism on cohomology. We write $[E(F)]$ for the fundamental class of $E(F)$. Define the \emph{surface untying invariant} as

\begin{equation}\label{E:su}
\mathrm{su}(F,\bar\rho)\ass \left\langle\rule{0pt}{11.5pt}  \bar\alpha\smile \beta^1\bar\alpha,[E(F)]\right\rangle\in A.
\end{equation}

\begin{prop}
The surface untying invariant is an invariant of $\bar\rho$--equivalence classes of $G$--coloured knots.
\end{prop}

\begin{proof}
Two $A$--coloured Seifert surfaces of a $G$--coloured knot $(K,\rho)$ have the same surface untying invariant, because there are related by tube equivalence, and a loop around a tube is contractible in $E(K)$.\par

The proof that the surface untying invariant is invariant under null-twists follows \cite[Proposition 17]{Mos06b}. Denote the Poincar\'{e} duality isomorphism by $D$. The Poincar\'{e} dual of $\mathrm{su}(F,\bar\rho)$ is the algebraic intersection number of $D\bar\alpha$ with $D\beta^1\bar\alpha$. A curve $L$ in $\ker\bar\rho$ vanishes in $H_1(E(F);A)$ because it vanishes in $H_1(E(F);aA)$ for each principal ideal $aA$ of $A$ (note that $aA$ is a cyclic ring). Therefore $L$ may be taken to be disjoint from $D\bar\alpha$ as an element of $H_1(E(F);A)$, and surgery on $L$ does not change $\mathrm{su}(F,\bar\rho)$.
\end{proof}

 By Alexander duality, $(\tau^+-\tau^-)$ gives rise to an isomorphism from $H_1(F;A)$ to $H_1(E(F);A)$. We denote by $\bar{a}\in H_1(F;A)$ the Alexander dual of $\bar\alpha$, which satisfies
 \[\widehat{\bar{a}}\ass\,(\tau^+-\tau^-) \bar{a} = D\,\mathrm{p}\iota\beta^1\bar\alpha.\]
 \noindent In the dihedral case, a curve representing $\bar{a}$ was called a \emph{mod $p$ characteristic knot} in \cite{CS75,CS84}. Recall the homological definition for the self-linking number, as in \cite[Chapter 77]{ST34} or in \cite[Page 18]{LaLa75}. For $g\in A$, let $\tilde{g}$ denote $\mathrm{p}^\ast(g)$, the smallest element of $\mathds{Z}^r$ for which $\mathrm{p}(\tilde{g})=g$. The surface untying invariant is seen to be the self-linking number of $\widehat{\bar{a}}$ as follows:

\begin{equation}
\left\langle\rule{0pt}{11.5pt} \widehat{\bar\alpha}\smile\beta^1\widehat{\bar\alpha}\,,[E(F)]\right\rangle= \left\langle\rule{0pt}{11.5pt} (D\mathrm{p}\iota \beta^1)^\ast\,\widehat{\bar{a}}\smile D\,\widehat{\bar{a}} \,,[E(F)]\right\rangle
= \left\langle\rule{0pt}{11.5pt} (\mathrm{p}\iota\beta^1)^\ast D\,\widehat{\bar{a}}\,,\widehat{\bar{a}}\right\rangle.
\end{equation}

 Let us calculate an explicit formula for the surface untying invariant of a $G$--coloured knot $(K,\rho)$ with surface data $(\M,\V)$ with respect to a marked Seifert surface $\left(F,\set{x_1,\ldots,x_{2g}}\right)$. Unraveling the definitions gives

\begin{equation}\label{E:su-formula1}
\mathrm{su}(F,\bar\rho)=\ \rule{0pt}{11pt}\epsilon\,\V^{\,T}(\mathrm{p}\iota\beta^1)^\ast\left(\rule{0pt}{10pt}\M\,t\cdot\V-\M^{\,T}\,\V\right),
\end{equation}

\noindent where $\epsilon\co A^{2g}\to \mathds{Z}^{2g}$ is the augmentation map. For $g\in A$, write $\tilde{g}$ for the smallest element of $\mathds{N}^r\subset\mathds{Z}^r$ for which $\mathrm{p}(\tilde{g})=g$. 
In the special case  $A\approx\left(\mathds{Z}/n\mathds{Z}\right)^r$, Formula \ref{E:su-formula1} simplifies to

\begin{equation}\label{E:su-formula2}
\mathrm{su}(F,\bar\rho)=\ \epsilon\,\V^{\,T}\frac{\M\,\widetilde{t\cdot\V}-\M^{\,T}\,\widetilde{\V}}{n}\bmod n.
\end{equation}

\subsection{An untying invariant for covering spaces}\label{SS:cu}

We set up a parallel construction to the one in Section \ref{SS:su}. Set $\Lambda\ass \mathds{Z}[\mathcal{C}_{m}]$. Denote by $l$ the rank of $A$ as a $\Lambda$-module with presentation

\begin{equation}\label{E:cplit}
0\Too\Lambda^{l}\overset{\iota}{\Too}\Lambda^l  \overset{\mathrm{p}}{\Too} A \Too 0.
\end{equation}

\noindent If it happens to be the case that $A$ is of the form $\left(\mathds{Z}/n\mathds{Z}\right)^r$, then $\iota$ is represented by the matrix $nI_l$, and $p$ assigns to each element of $\Lambda$ its $\mathcal{C}_m$ orbit, modulo $n$. For $k\in\{1,2,\ldots\}$, the above maps extend by linearity:

\begin{equation}
0\Too\left(\Lambda^l\right)^k\overset{\iota}{\Too}\left(\Lambda^l\right)^k  \overset{\mathrm{p}}{\Too} A^k \Too 0.
\end{equation}

A $G$--colouring $\rho\co \pi\twoheadrightarrow G$ of $K$ lifts to an  $A$--colouring $\tilde\rho\co H_1(C_m(K))\twoheadrightarrow A$ of its $m$--fold branched cyclic covering space $C_m(K)$, which corresponds to a cocycle $\alpha\in H^1(C_m(K);A)$ by the Universal Coefficient Theorem. The long exact sequences

\scalebox{0.92}{\parbox{\textwidth}{%
\[
\cdots\rightarrow H_2(C_m(K);A))\overset{\raisebox{2pt}{\scalebox{0.8}{$\beta_2$}}}{\rightarrow} H_1(C_m(K);\Lambda^l)\overset{\raisebox{2pt}{\scalebox{0.8}{$\iota_\ast$}}}{\rightarrow} H_1(C_m(K);\Lambda^l)\overset{\raisebox{2pt}{\scalebox{0.8}{$\mathrm{p}_\ast$}}}{\rightarrow} H_1(C_m(K); A)\rightarrow\cdots
\]}}

\noindent and

\scalebox{0.92}{\parbox{\textwidth}{%
\[
\cdots\rightarrow H^1(C_m(K);\Lambda^l)\overset{\iota^\ast}{\rightarrow} H^1(C_m(K);\Lambda^l)\overset{\mathrm{p}^\ast}{\rightarrow} H^1(C_m(K);A)\overset{\beta^1}{\rightarrow} H^2(C_m(K); \Lambda^l)\rightarrow\cdots
\]}}

\noindent are induced by short exact sequence \ref{E:cplit}. The \emph{coloured untying invariant} is defined by the formula

\begin{equation}\label{E:cu-defn}
\mathrm{cu}(K,\rho)\ass \left\langle\rule{0pt}{11.5pt}  \alpha\smile \beta^1\alpha,[C_m(K)]\right\rangle\in A.
\end{equation}

The argument of \cite[Proof of Proposition 17]{Mos06b} shows the following:

\begin{prop}
The coloured untying invariant is an invariant of $\rho$--equivalence classes of $G$--coloured knots.
\end{prop}

The coloured untying invariant is the self-linking number of $a\ass D\,\mathrm{p}\iota\beta^1\alpha$ as is seen via

\begin{equation}\label{E:cu-is-sl}
\left\langle\rule{0pt}{11.5pt} \alpha\smile\beta^1\alpha\,,[C_m(K)]\right\rangle= \left\langle\rule{0pt}{11.5pt} (D\mathrm{p}\iota \beta^1)^\ast\, a\smile D a \,,[C_m(K)]\right\rangle
= \left\langle\rule{0pt}{11.5pt} (\mathrm{p}\iota \beta^1)^\ast\, Da\,,a\right\rangle.
\end{equation}

 We would next like an explicit formula for the coloured untying invariant of a $G$--coloured knot $(K,\rho)$ with surface data $(\M,\V)$ with respect to a marked Seifert surface $\left(F,\set{x_1,\ldots,x_{2g}}\right)$. We work this out for $m>0$. It turns out that the easiest way to do this is in two stages, first by regarding the coloured untying invariant as a $\tilde\rho$--equivalence invariant by forgetting the action of $\mathcal{C}_m$ on $A$, then by obtaining an explicit formula for this invariant, and then by adding this $\mathcal{C}_m$ action back `by hand'. Note that the analogues of Equations \ref{E:cu-defn} and \ref{E:cu-is-sl} will continue to hold (with analogous proofs). Thus, having forgotten the covering transformations, $\tilde\rho$ corresponds to a cocycle $\tilde\alpha\in H^1(C_m(K);A)$, and for $\tilde{a}\ass D\,\mathrm{p}\iota\beta^1\tilde\alpha$ we have

\begin{equation}\label{E:cu-defn-2}
\widetilde{\mathrm{cu}}(C_m(K),\tilde\rho)\ass \left\langle\rule{0pt}{11.5pt} \tilde\alpha\smile \beta^1\tilde\alpha,[C_m(K)]\right\rangle=\left\langle\rule{0pt}{11.5pt} (\mathrm{p}\iota \beta^1)^\ast\, D\tilde{a}\,,\tilde{a}\right\rangle\in A,
\end{equation}

\noindent which is an invariant of $\tilde{\rho}$--equivalence classes of $G$--coloured knots.\par

 For $m>0$, push a Seifert surface $F$ for $K$ into $D^4$. The intersection form of the $m$--fold branched cyclic cover of this manifold represents the linking form of its boundary, which is $C_m(K)$. Kauffman in \cite[Proposition 5.6]{Kau74} gives the matrix representing this linking form with respect to the basis $\set{t^jx_i}_{\substack{1\leq j\leq m-2;\\1\leq i\leq 2g.}}$ as
\[
L(\M)\ass\ \left[
\begin{matrix}
\M+\M^{\,T} & \M^{\,T} & 0 & 0 &\cdots & 0 & 0 & 0\\
\M & \M+\M^{\,T} & \M^{\,T} & 0 & \cdots & 0 & 0 & 0\\
0 & \M & \M+\M^{\,T} & \M^{\,T} & \cdots & 0 & 0 & 0\\
\vdots & \vdots & \vdots & \vdots & \ddots & \vdots & \vdots & \vdots \\
0 & 0 & 0 & 0 & \cdots & 0 & \M & \M+\M^{\,T}
\end{matrix}\right],
\]

\noindent where the sign and transpose differences are due to differences between our orientation conventions and the ones used by Kauffman.
Set

\[\V_{(m)}\ass\left(\V; t\cdot \V;\ldots; t^{m-2}\cdot \V\right).\]

We obtain the explicit formula

\begin{equation}\label{E:L-Equation-1}
\widetilde{\mathrm{cu}}(C_m(K),\tilde\rho)=\ \epsilon\,\V_{(m)}^{\,T}(\iota\beta^1)^\ast
\left(L(\M)\widetilde{\V}_{(m)}\right).
\end{equation}

In the special case  $A\approx\left(\mathds{Z}/n\mathds{Z}\right)^r$, this simplifies to

\begin{equation}\label{E:L-equation-2}
\widetilde{\mathrm{cu}}(C_m(K),\tilde\rho)= \ \epsilon\,\V_{(m)}^{\,T}\frac{L(\M)\widetilde{\V}_{(m)}}{n}\bmod n.
\end{equation}

Finally, notice that the action of $\mathcal{C}_m$ on $A$ sends $\V_{(m)}$ to $t\cdot\V_{(m)}$, leaving invariant the right hand side of Equation \ref{E:L-equation-2}. Thus,

\begin{equation}\label{E:L-equation-3}
\widetilde{\mathrm{cu}}(C_m(K),\tilde\rho)= \mathrm{cu}(K,\rho).
\end{equation}

\subsection{The $S$--equivalence class of the colouring}\label{SS:sequiv}

Let $(K,\rho)$ be a $G$--coloured knot, with surface data $(\M,\V)$ with respect to a choice $\left(F,\set{x_1,\ldots,x_{2g}}\right)$ of marked Seifert surface for $K$. Let $P$ be a unimodular matrix such that
\begin{equation}
P^{\thinspace T}\thinspace \left(\M-\M^{\thinspace T}\right)\thinspace P= \left[\begin{matrix}0 & -1\\
1 & 0\end{matrix}\right]^{\oplus g}.
\end{equation}
Write $P^{-1}\V\ass \left(v_1;\,\ldots;v_{2g}\right)$, and define the \emph{$S$--equivalence class of the colouring}

\begin{equation}\label{E:s-equiv-formula}
\mathrm{s}(K,\rho)= \sum_{j=1}^{g_2}v_{2j-1}\wedge v_{2j}\in A\wedge A.
\end{equation}


As we defined $S$--equivalence for surface data, it can be defined for vectors. Two vectors $\V_{1,2}$ are said to be
\emph{$S$--equivalent} if there exist matrices $\M_{1,2}$ such that $(\M_1,\V_1)$ and $(\M_2,\V_2)$ are $S$--equivalent. The following proposition shows that the $S$--equivalence class of the colouring is a well-defined invariant of $\bar\rho$--equivalence classes of $G$--coloured knots, and it explains what it measures.

\begin{prop}\label{P:ColClass}
 Given a pair of surface data $(\M_{1,2},\V_{1,2})$, colouring vectors $\V_{1,2}$ are $S$--equivalent if and only if, for any $G$--coloured knots $(K_{1,2},\rho_{1,2})$ with surface data $(\M_{1,2},\V_{1,2})$ correspondingly, corresponding to a choice of marked Seifert surfaces for each, we have
\[
\mathrm{s}(K_{1},\rho_1)=\ \mathrm{s}(K_{2},\rho_2).
\]
\end{prop}

\begin{proof}[Proof of Proposition {\ref{P:ColClass}}]
Identify the symplectic group  $\mathrm{Sp}(2g,\mathds{Z})$ with the group of integral square matrices $P$ satisfying

\begin{equation}
P^{\,T}\left[\begin{matrix}0 & -1\\
1 & 0\end{matrix}\right]^{\oplus g} P= \left[\begin{matrix}0 & -1\\
1 & 0\end{matrix}\right]^{\oplus g}.
\end{equation}

By an argument of Rice \cite{Ric71}, two Seifert matrices $\M_{1,2}$ are $S$--equivalent if and only if there exist Seifert matrices $\M_{3,4}$ which are $S$--equivalent to $\M_{1,2}$ correspondingly such that $\M_{3,4}-\M_{3,4}^{\,T}=\left[\begin{smallmatrix}0 & -1\\ 1 & 0\end{smallmatrix}\right]^{\oplus g}$ and $\M_{3}$ is $S$--equivalent to $\M_{4}$ via a finite sequence of $\Lambda_2$-moves, and $\Lambda_1$-moves of the form $\M\mapsto P^{\,T}\,\M\thinspace P$, with $P \in \mathrm{Sp}(2g,\mathds{Z})$. We may therefore assume that $\M$ satisfies $\M-\M^{\,T}=\left[\begin{smallmatrix}0 & -1\\ 1 & 0\end{smallmatrix}\right]^{\oplus g}$ without loss of generality.

The $S$--equivalence relation on symplectic matrices induces an equivalence relation on the corresponding colouring vectors. Let $A^{2g}_{\text{full}}$ denote the set of vectors in $A^{2g}$ whose entries together generate $A$. A $\Lambda_1$-move on surface data sends a colouring vector $\V\in A^{2g}_{\text{full}}$ to a vector $P^{-1}\V$, for $P\in  \mathrm{Sp}(2g,\mathds{Z})$. A $\Lambda_2$-move sends a colouring vector $(v_{1};\,\ldots;v_{2g})\in A^{2g}_{\text{full}}$ to a colouring vector $(v_{1};\,\ldots;v_{2g};0;y)\in A^{2g+2}_{\text{full}}$ for any $y\in A$.\par

Define a map

\begin{equation}
\begin{aligned}
\varphi\co \V\in \bigcup_{g\in\mathds{N}^\ast}A^{2g}_{\text{full}}\left/\rule{0pt}{11pt}\Lambda_{1,2}\right.&\Too A\wedge A\\
(v_1;\,\ldots;v_{2g}) &\mapsto\ \ \sum_{j=1}^{g_2}v^{2}_{2j-1}\wedge
v^{2}_{2j}
\end{aligned}.
\end{equation}

\noindent We next show that $\varphi$ is well-defined. Because $a\wedge 0=0$ for any $a\in A$, the $\varphi$--image of a vector $\V\in A^{2g}_{\text{full}}$ is not changed by a $\Lambda_2$-move. To see that it is not changed by a $\Lambda_1$ either, use the fact that $\mathrm{Sp}(2g,\mathds{Z})$ is generated by

\begin{equation}
R^{\,T}\left[\begin{matrix} 0 & I_g\\
-I_g & 0
\end{matrix}\right]R,\ \ R^{\,T}\left[\begin{matrix} A & 0\\
0 & (A^{\,T})^{-1}\end{matrix}\right]R,\ \ R^{\,T}\left[\begin{matrix} I_g & 0\\
B & I_g
\end{matrix}\right]R,
\end{equation}

\noindent with $A\in \mathrm{GL}(g,\mathds{Z})$, and $B$ a symmetric integral matrix (see \textit{e.g.} \cite[Proposition A5]{Mum83}). Above, $R$ denotes the integral matrix satisfying

\begin{equation}
R^{\,T} \left[\begin{matrix}0 & -I_g\\ I_g & 0\end{matrix}\right] R=  \left[\begin{matrix}0 & -1\\ 1 & 0\end{matrix}\right]^{\oplus g}.
\end{equation}

The reader may verify directly that the $\varphi$--image of a vector $\V\in A^{2g}_{\text{full}}$ is not changed by left multiplication by any of the above basis elements.\par

Next, we construct the inverse map

\begin{equation}
\psi\co A\wedge A\Too \bigcup_{g\in\mathds{N}^\ast}A^{2g}_{\text{full}}\left/\rule{0pt}{11pt}\{\Lambda_{1,2}\}\right.
\end{equation}

Let $b_1,\ldots,b_r$ be a fixed basis for $A$, and let $X\ass
\sum_{1\leq i<j\leq r}c_{i,j}{b_{i}\wedge b_j}$ be some element of
$A\wedge A$. If $c_{1,2}>0$, we
set $v_1,\ldots,v_{2c_{1,2}-1}$ to $s_1$ and we set
$v_{2},\ldots,v_{2c_{1,2}}$ to $s_2$. If $c_{1,3}>0$, we set
$v_{2c_{1,2}+1},\ldots,v_{2c_{1,2}+2c_{1,3}-1}$ to $s_1$ and
$v_{2c_{1,2}+2},\ldots,v_{2c_{1,2}+2c_{1,3}}$ to $s_3$, and so on
lexicographically, until we finish with $c_{r-1,r}$. We conclude by
setting $v_{2C+2k-1}$ to $0$ and setting $v_{2C+2k}$ to $s_k$ for
$k=1,\ldots,r$, where $C$ denotes $\sum_{1\leq i<j\leq r}c_{i,j}$. By construction, entries in this colouring vector, whose length is $g\ass 2C+2r$, together generate $A$. For example, for $A$ generated by $s_1,s_2,s_3$, we would have $\psi(2s_1\wedge s_2)=(s_1;s_2;s_1;s_2;0;s_1;0;s_2;0;s_3)$.\par

To prove that $\psi$ is well-defined, identify $A\wedge A$ with the free commutative monoid over $A^2$ modulo moves $S_{1,2}$, where $S_1$ takes elements of the form $a\wedge (b+c)$ to elements of the form $a\wedge b + a\wedge c$, and $S_2$ takes elements of the form $a\wedge a$ to zero. We call this monoid $\mathcal{M}$.

First, for $X\in\mathcal{M}$, commutativity of $\mathcal{M}$ corresponds to a $\Lambda_1$-move on $\psi(X)$ with matrix $P=I_{2i}\oplus \left[\begin{smallmatrix}0 & 1 & 0 & 0\\
1 & 0 & 0 & 0\\ 0 & 0 & 0 & 1\\ 0 & 0 & 1 & 0\end{smallmatrix}\right]$. The effect of an $S_1$-move is replicated in $A^{2g}_{\text{full}}\left/\rule{0pt}{11pt}\Lambda_{1,2}\right.$ by first applying a $\Lambda_2$-move
\begin{equation}
\psi(X)=(v_1;\,\ldots\,;v_{2g-2}\,; a\,; (b+c))\mapsto (v_1;\,\ldots\,;v_{2g-2}\,; a\,; (b+c)\, ; 0\, ; c),
\end{equation}
\noindent and then applying a $\Lambda_1$-move with matrix
\begin{equation}
P= I_{2g-2}\oplus \left[\begin{matrix}1 & 0 & 0 & 0\\ 0 & 1 & 0 & 1\\ -1 & 0 & 1 & 0\\ 0 & 0 & 0 & 1\end{matrix}\right].
\end{equation}
\noindent The result is the vector $(v_1;\,\ldots\,;v_{2g-2}\,; a\,; b\, ; a\, ; c)$, as desired. The effect of an $S_2$-move is replicated by a $\Lambda_1$-move with matrix $P= I_{2g}+E_{2g,g-1}$ to get
\begin{equation}\psi(X)=(v_1;\,\ldots\,;v_{2g-2}\,; a\,; a)\mapsto (v_1;\,\ldots\,;v_{2g-2}\,; 0\,; a),\end{equation}
\noindent after which a $\Lambda_2$-move erases the last two entries, and we obtain $(v_1;\,\ldots\,;v_{2g-2})$ as desired.\par

We have shown that both $\varphi$ and $\psi$ are well-defined, and following through the definitions shows that $\varphi(\psi(X))=X$ for any $X\in A\wedge A$. So $\varphi$ is invertible, and is therefore an isomorphism.
\end{proof}

Because null-twists don't change the colouring vector, Proposition \ref{P:ColClass} implies the following.

\begin{cor}
The element $\mathrm{s}(K,\rho)\in A\wedge A$ is an invariant of $\bar\rho$--equivalence classes of $G$--coloured knots.
\end{cor}

\begin{rem}
The proof that $\psi$ is well-defined is an algebraic version of the band sliding arguments of \cite[Section 4.2]{KM09}.
\end{rem}

\begin{rem}
If it were necessary, we could upgrade $\mathrm{s}(K,\rho)\in A\wedge A$ to a $\hat\rho$--invariant by considering $A\wedge A$ as a $\mathcal{C}_m$-module with respect to the diagonal action of $t$.
\end{rem}

\section{Groups whose commutator subgroup has small rank}\label{S:r12}

Armed with the tools of Section \ref{S:untyinginvariants}, we are now in a position to find complete sets of base-knots for some metabelian groups $G$ of the particularly simple form $G=\mathcal{C}_m\ltimes_\phi (\mathds{Z}/n\mathds{Z})^r$ for $r\leq 2$, where the order of $\phi$ is $m$. Then $\phi$ is represented by an integer matrix $N$. In Section \ref{SS:r1} we consider the case $r=1$, and we find a complete set of base-knots for metacyclic groups for which $2(\phi^{-3}-\mathrm{id})$ is invertible. This generalizes \cite[Sections 4.2, 4.3 and 5.1]{KM09}, where the $m=2$ case is treated. In Section \ref{SS:r2}  we find complete sets of base knots for certain families of groups $G$ of the form $\mathcal{C}_m\ltimes_{\phi}(\mathds{Z}/n_1\mathds{Z}\times\mathds{Z}/n_2\mathds{Z})$.\par
The strategy is always the same. Relative bordism gives an upper bound on the number of $\bar\rho$--equivalence classes via the K\"{u}nneth Formula. To find a lower bound, choose a colouring vector to represent each $S$--equivalence class, and solve $\M\V N=\M^{\,T}\V$ (Proposition \ref{P:HNN}) for $\M$ over $A$. If an entry of $\M$ is not determined, set it to zero, if it is determined then set it to that value, and if the equation for that entry admits no solutions, then there are no $G$--coloured knots in that equivalence class. Finally, to get different values for the surface untying invariant (Equation \ref{E:su-formula2}), add `$A$--torsion' elements to $\M$. This gives a list of surface data representing non-$\bar\rho$--equivalent $G$--coloured knots, and if the length of the list equals the upper bound then we are finished. For $\rho$--equivalence, check that these $G$--coloured knots all have different coloured untying invariants using Equation \ref{E:L-equation-2}.\par

Throughout this section, for $a\in \mathds{Z}/n\mathds{Z}$, let $\tilde{a}\in\mathds{N}=\{0,1,2,\ldots\}$ denote the smallest natural number such that $a=\tilde{a}\bmod n$ unless otherwise specified.

\subsection{The $r=1$ case}\label{SS:r1}

\subsubsection{$\bar\rho$--equivalence}\label{SSS:R1M0}
 Groups of the form $\mathcal{C}_m\ltimes_{\phi}(\mathds{Z}/n\mathds{Z})$ are called \emph{metacyclic groups}. Note that, because both $\phi$ and $\phi-\mathrm{id}$ are invertible (\textit{e.g} \cite[Proposition 14.2]{BZ03}), it follows that $m,n>0$. The automorphism $\phi$ takes the form $\phi(s)=\xi s$ with respect to a fixed generator $s$ for $\mathds{Z}/n\mathds{Z}$, where $\xi^m=1\bmod n$. Both $\xi$ and $\xi-1$ are units.\par

 The relative bordism upper bound $\abs{H_3(\mathds{Z}/n\mathds{Z};\mathds{Z})}=n$ for $\bar\rho$--equivalence classes coming from Corollary \ref{C:Wallacebound} coincides with the lower bound coming from Section \ref{S:untyinginvariants}, which is given as $\abs{\mathds{Z}/n\mathds{Z}\wedge \mathds{Z}/n\mathds{Z}}\abs{\mathds{Z}/n\mathds{Z}}=1\cdot n=n$. A complete set of base-knots with respect to $\bar\rho$--equivalence are the twist knots $(T_k,\rho_{k})$ of Figure \ref{F:metacycbase}, with surface data

\begin{figure}
\psfrag{l}[c]{\Huge$\vdots$}
\begin{minipage}{2in}\centering
\psfrag{t}[c]{$t$}\psfrag{s}[c]{$ts$}\psfrag{k}[r]{\parbox{0.5in}{\tiny$kn+a$\\[0.1cm]
\ \ \
$\frac{1}{2}$--twists}\normalsize$\left\{\rule{0pt}{0.55in}\right.$}
\includegraphics[width=1.1in]{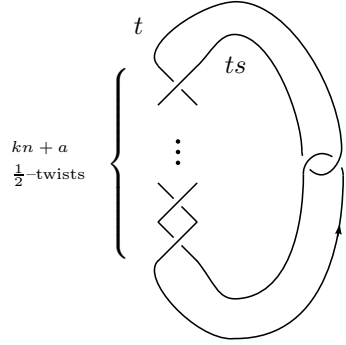}
\end{minipage}
\caption{\label{F:metacycbase} A complete set of base-knots with respect to $\bar\rho$--equivalence for a metacyclic group, with $k=1,2,\ldots,n$. These are also a complete set of base-knots with respect to $\rho$--equivalence if $2(\phi^{-3}-\mathrm{id})$ is invertible.}
\end{figure}

\begin{equation}\label{E:metacycReps}
(\M_k,\V)\ =\ \genusoneknot{a+kn}{0}{1}{1}{s}{s^{\frac{\xi}{1-\xi}}},
\end{equation}

\noindent where $a\in\mathds{N}$ is the minimal natural number such that $a\bmod n= \frac{-\xi}{(1-\xi)^2}$ and $k=1,\ldots,n$. These $\bar\rho$-equivalence classes are distinguished by
$\mathrm{su}(F_k,\bar{\rho}_{k})=k(\xi-1)$ (plug the surface data into Equation \ref{E:su-formula2}), where $F_k$ is the obvious Seifert surface for $T_k$ in the projection of Figure \ref{F:metacycbase}.

\begin{rem}
A more explicit way to establish the upper bound of $n$ for the number of $\rho$--equivalence classes would have been to apply the algorithm of \cite[Section 4]{KM09}. Given a $G$--coloured knot $(K,\rho)$, the arguments of Sections 4.2 and of 4.3.1 provide an algorithm to relate $(K,\rho)$ by an explicit sequence of null-twists to a genus $1$ knot with surface data
\begin{equation}\genusoneknot{a_{1,1}}{a_{1,2}}{a_{1,2}+1}{a_{2,2}}{s}{0}.\end{equation}
Moreover, $a_{2,2}$ can be made to vanish by null-twists, and one may add or subtract $n$ from $a_{1,2}$ and $n^2$ from $a_{1,1}$ by the arguments of \cite[Page 1382]{KM09}. Finally, Proposition \ref{P:HNN} tells us that $a_{1,1}\bmod n=0$, and that $a_{1,2}\bmod n =\frac{1}{\xi-1}$.
\end{rem}

\subsubsection{$\rho$--equivalence}\label{SSS:R1M3}

\begin{thm}\label{T:metacyclic}
 For $G$ metacyclic, the number of $\rho$--equivalence classes of $G$--coloured knots is bounded from below by the order of  $2(\phi^{-3}-\mathrm{id})$. In particular, if $2(\phi^{-3}-\mathrm{id})$ is invertible, then $\set{(T_k,\rho_{k})}_{1\leq k<n}$ is a complete set of base-knots for $G$.
\end{thm}

\begin{proof}
Two $\bar\rho$--equivalence knots are in particular $\rho$--equivalent, therefore it suffices to check that the coloured untying invariant distinguishes  $(T_k,\rho_{k})$. We do this by plugging the surface data of Equation \ref{E:metacycReps} into Equation \ref{E:L-equation-2}. Decompose this equation as

\begin{equation}
\mathrm{cu}(T_K,\rho_K)= \ \epsilon\,\V_{(m)}^{\,T}\frac{\left(L\left[\begin{matrix}kn & 0\\ 0 & 0\end{matrix}\right]+L\left[\begin{matrix}a & 0\\ 1 & 1\end{matrix}\right]\right)\widetilde{\V}_{(m)}}{n}\bmod n.
\end{equation}

The matrix $L\left[\begin{matrix}a & 0\\ 1 & 1\end{matrix}\right]$ does not depend on $k$, so it suffices to calculate

\begin{equation}
\left[\tilde\xi^0,\tilde\xi^1,\ldots,\tilde\xi^{n-2}\right]\,\frac{L[kn]}{n}\left(\xi^0;\xi^1;\ldots;\xi^{n-2}\right)=2k\sum_{i=1}^{2(n-2)}\xi^i.
\end{equation}

\noindent  with $\tilde\xi\in \mathds{N}$ the smallest number such that $\tilde{\xi}^m=1\bmod n^2$ and $\xi= \tilde{\xi}\bmod n$. \par

Because $1-\xi$ is invertible in $\mathds{Z}/n_1\mathds{Z}$, the number of distinct numbers in the set $\left\{2k\sum_{i=1}^{2n-4}\xi\right\}_{1\leq k\leq n_1}\subset \mathds{Z}/n_1\mathds{Z}$ equals the order of
$2(1-\xi)\sum_{i=1}^{2n-4}\xi^i=2(1-\xi^{2n-3})$, which is equal to the order of $2(1-\xi^{-3})$ in $\mathds{Z}/n\mathds{Z}$.
\end{proof}

This generalizes \cite[Theorem 3]{KM09} to all metacyclic groups with $2(1-\xi^{-3})$ invertible in $\mathds{Z}/n\mathds{Z}$. Thus, the simplest group $G$ for which we have not classified $G$--coloured knots up to $\rho$--equivalence is then $G=\mathcal{C}_3\ltimes_{[2]} (\mathds{Z}/7\mathds{Z})$.

\subsection{The $r=2$ case}\label{SS:r2}

In Sections \ref{SSS:R2M0D} and \ref{SSS:R2M3D} we consider groups of the form $\mathcal{C}_m\ltimes_{\phi}(\mathds{Z}/n_1\mathds{Z}\times\mathds{Z}/n_2\mathds{Z})$. Matrix notation is misleading when $n_1\neq n_2$, but we'll use it anyway, with care. Given a basis $s_{1,2}\ass\left(s_{1,2}^1,s_{1,2}^2\right)$ for $A=\mathds{Z}/n_1\mathds{Z}\times\mathds{Z}/n_2\mathds{Z}$, there is a $2\times 2$ matrix $N$ such that $\phi(s_{1,2})=s_{1,2}N$.

\subsubsection{$\bar\rho$--equivalence; $N$ is diagonalisable}\label{SSS:R2M0D}

Again $n_{1,2}>0$. Relative bordism gives an upper bound of on the number of $\bar\rho$--equivalence classes of $n_1n_2\gcd(n_1,n_2)$ by the K\"{u}nneth Formula, which simplifies to

\begin{equation}
0\to A\to H_3(A;\mathds{Z})\to \mathds{Z}/\gcd(n_1,n_2)\mathds{Z}\to 0 .
\end{equation}

The surface untying invariant will be seen to detect the `$A$' part, while the $S$--equivalence class of the colouring detects the `$\mathrm{Tor}$' part, noting for our groups that

\begin{equation}\mathrm{Tor}_1(\mathds{Z}/n_1\mathds{Z},\mathds{Z}/n_2\mathds{Z}) \approx \mathds{Z}/\gcd(n_1,n_2)\mathds{Z}\approx A\wedge A.\end{equation}

Choose a basis $s_{1,2}$ for $A$, with respect to which the matrix $N$ is of the form $\left[\begin{smallmatrix}\tilde{\xi}_1 & 0\\ 0 & \tilde{\xi}_2\end{smallmatrix}\right]$, with $\tilde\xi_{1,2}^m=1\bmod n_{1,2}^2$ correspondingly, and set $\xi_{1,2}\ass\tilde{\xi}\bmod n_{1,2}$. The $S$--equivalence classes of $G$--colourings are represented by the colouring vectors $(s_1;is_2)$ and $(s_1;0;s_2;0)$, where $i=1,\ldots,\gcd(n_1,n_2)-1$.\par

By explicitly solving $\M\thinspace\V N=\M^{\thinspace T}\thinspace\V$, we see that there exist $G$--coloured knots in the $S$--equivalence class represented by $(s_1;is_2)$ only if there exists a number $x\in \mathds{N}$ with $\frac{\xi_1}{1-\xi_1}=x\bmod n_1$ and with $\frac{1}{\xi_2-1}=x\bmod n_2$. For $n_1=n_2$, this condition would become $\xi_1=\xi_2^{-1}$, while for $n_{1,2}$ coprime it would be vacuous. The $\bar\rho$--equivalence classes for such knots are represented by $G$--coloured knots $(K_{k,l},\rho_{k,l,i})$ with surface data

\begin{equation}(\M_{k,l},\V_i)\ =\ \genusoneknot{kn_1}{x}{x+1}{ln_2}{s_1}{is_2},\end{equation}

\noindent with $x\in\mathds{N}$ being the minimal integer satisfying the above, and with $k=1,\ldots,n_1$ and $l=1,\ldots,n_2$ and $i=1,\ldots,\gcd(n_1,n_2)-1$.\par

These $\bar\rho$-equivalence classes are distinguished by the surface untying invariant. The simplest way to see this is to decompose $M$ as $\left[\begin{smallmatrix}kn_1 & 0\\ 0 & ln_{2}\end{smallmatrix}\right]+ \left[\begin{smallmatrix}0 & x\\ x+1 & 0\end{smallmatrix}\right]$ and observe that the contribution of the first summand to Equation \ref{E:su-formula2} is $\left((\xi_1-1)k, (\xi_2-1)l\right)$ which spans $A$, while the contribution of the second summand is constant.

The knots in the $S$--equivalence class represented by $(s_1;0;s_2;0)$ are represented by $G$--coloured knots $(K^\ast_{k,l},\rho^\ast_{k,l})$ with surface data

\begin{equation}(\M^\ast_{k,l},\V^\ast)\ =\ \left(\rule{0pt}{38pt}\begin{bmatrix} kn_1 & x_1 & 0 & 0\\ x_1+1 & 0 & 0 & 0\\
0 & 0 & ln_2 & x_2\\ 0 & 0 & x_2+1 & 0 \end{bmatrix}\raisebox{-5.5pt}{\huge{,}}\normalsize \begin{pmatrix} s_1\\ 0\\ s_2\\ 0 \end{pmatrix}\right),\end{equation}

\noindent where $x_{1,2}\in\mathds{N}$ are the smallest natural numbers such that
$\frac{\xi_{1,2}}{1-\xi_{1,2}}=x_{1,2}\bmod n_{1,2}$ correspondingly.\par

These $\bar\rho$-equivalence classes are distinguished by
\begin{equation}\mathrm{su}(F^\ast_{k,l},\bar\rho^\ast_{k,l})=
\left((\xi_1-1)k, (\xi_2-1)l\right).\end{equation}

\begin{rem}\label{R:AlgoFail}
The algorithm of \cite[Section 4]{KM09} would give an upper bound on the number of $\bar\rho$--equivalence classes, which would not be sharp for $\gcd(n_1,n_2)>1$. Given a $G$--coloured knot $(K,\rho)$, the arguments of \cite[Sections 4.2 and 4.3.1]{KM09} provide an algorithm to relate $(K,\rho)$ by an explicit sequence of null-twists to a knot $(K_0,\rho_0)$ of genus $\leq 2$. We can arrange for the colouring vector to take the form $(s_1;is_2)$ or $(s_1;0;s_2;0)$ by band slides \cite[Section 4.1.4 and Section 4.2.2]{KM09}, where $i=1,\ldots,\gcd(n_1,n_2)-1$. For genus $1$, write the Seifert matrix as $\left[\begin{smallmatrix}a_{1,1} & a_{1,2}\\
a_{1,2}+1 & a_{2,2}\end{smallmatrix}\right]$. Proposition \ref{P:HNN} determines the values of $a_{1,1}\bmod n_1$, of $a_{1,2}\bmod\gcd(n_1,n_2)$, and of $a_{2,2}\bmod n_2$. Moreover, one may add or subtract $n_1^2$ from $a_{1,1}$, and $\gcd(n_1,n_2)^2$ from $a_{1,2}$, and $n_2^2$ from $a_{2,2}$ by the arguments of \cite[Page 1382]{KM09}. For genus $2$ let the Seifert matrix be $\left[\begin{smallmatrix}a_{1,1} & a_{1,2} & a_{1,3} & a_{1,4}\\
a_{1,2}+1 & a_{2,2} & a_{2,3} & a_{2,4}\\
a_{1,3} & a_{2,3} & a_{3,3} & a_{3,4}\\
a_{1,4} & a_{2,4} & a_{3,4}+1 & a_{4,4}\end{smallmatrix}\right]$. We may kill $a_{2,2}$, $a_{2,4}$, and $a_{4,4}$ by \cite[Equation 4.7]{KM09}. Proposition \ref{P:HNN} determines the values of $a_{1,1}\bmod n_1$, of $a_{3,3}\bmod n_2$, and of $a_{1,3}\bmod\gcd(n_1,n_2)$. The other entries are determined on the nose, because they are determined modulo either $n_1$ or $n_2$, and we can add or subtract either $n_1$ or $n_2$ from them by \cite[Equation 4.11]{KM09}. On the other hand, all we can add or subtract from $a_{1,1}$, $a_{3,3}$, and $a_{1,3}$ is $n_1^2$, $n_2^2$, and $\gcd(n_1,n_2)^2$ correspondingly. In summary, the upper bound which the algorithm gives is $n_1n_2\gcd(n_1,n_2)^2$, which in general is not sharp.
\end{rem}

\subsubsection{$\rho$--equivalence; $N$ is diagonalisable}\label{SSS:R2M3D}

For $A$ of rank $2$ and for $N$ diagonalisable, we obtain a complete set of base knots if $2(\phi^{-3}-\mathrm{id})$ is invertible. We remark that this would hold for $A$ of any rank if $N$ were diagonalisable, with pairwise coprime diagonal entries (in this case $A\wedge A$ and $A\wedge A\wedge A$ both vanish).

\begin{thm}\label{T:R2M3Diag}\hfill
\begin{itemize}
\item For each $1\leq l_0\leq n_2$, the number of non-$\rho$--equivalent $G$--coloured knots in the set $\left\{(K_{k,l_0},\rho_{k,l_0,i})\right\}_{1\leq k=1\leq n_1}$, and also the number of non-$\rho$--equivalent $G$--coloured knots in the set $\left\{(K^\ast_{k,l_0},\rho^\ast_{k,l_0})\right\}_{1\leq k=1\leq n_1}$, are bounded from below by the order of $2(1-\xi_{1}^{-3})\in\mathds{Z}/n_1\mathds{Z}$.
\item For each $1\leq k_0\leq n_1$, the number of non-$\rho$--equivalent $G$--coloured knots in the set $\left\{(K_{k_0,l},\rho_{k_0,l,i})\right\}_{1\leq l=1\leq n_2}$, and also the number of non-$\rho$--equivalent $G$--coloured knots in the set $\left\{(K^\ast_{k_0,l},\rho^\ast_{k_0,l})\right\}_{1\leq l=1\leq n_2}$, are bounded from below by the order of $2(1-\xi_{2}^{-3})\in\mathds{Z}/n_2\mathds{Z}$.
\end{itemize}
\end{thm}

\begin{proof}
Consider the claim for $G$--coloured knots in the $S$--equivalence class represented by $(s_1;s_2)$.
To prove the first assertion, decompose $M$ as $\left[\begin{smallmatrix}kn_1 & 0\\ 0 & 0\end{smallmatrix}\right]+\left[\begin{smallmatrix}0 & x\\ x+1 & l_0n_2\end{smallmatrix}\right]$. The matrix $\left[\begin{smallmatrix}0 & x\\ x+1 & l_0n_2\end{smallmatrix}\right]$ is independent of $k$. Thus, to show that the coloured untying invariant (Equation \ref{E:L-equation-2}) distinguishes our base knots, it suffices to calculate

\begin{equation}
\left[\tilde\xi_1^0,\tilde\xi_1^1,\ldots,\tilde\xi_1^{n-2}\right]\,\frac{L\left[kn_1\right]}{n_1}\,\left(\xi_1^0;\xi_1^1;\ldots;\xi_1^{n-2}\right)=
2k\sum_{j=1}^{2(n-2)}\xi_1^j.
\end{equation}

Because $1-\xi_1$ is invertible in $\mathds{Z}/n_1\mathds{Z}$, the number of distinct numbers in the set $\left\{2k\sum_{j=1}^{2n-4}\xi_j\right\}_{1\leq k\leq n_1}\subset \mathds{Z}/n_1\mathds{Z}$ equals the order of
$2(1-\xi_1)\sum_{j=1}^{2n-4}\xi^j_1$, which is the order of $2(1-\xi_1^{-3})$ in $\mathds{Z}/n_1\mathds{Z}$.
The proof of the second assertion is analogous, as is the proof for $i>1$ and for $G$--coloured knots in the $S$--equivalence class represented by $(s_1;0;s_2;0)$.
\end{proof}

\subsubsection{$\bar\rho$--equivalence; $N$ is not diagonalisable}\label{SSS:R2M0N}

Set $n\ass n_1=n_2$. The bordism upper bound is $n^3$. Choose a basis $s_{1,2}$ for $A$ such that $\phi(s_1)=s_2$. With respect to such a basis, $N$ takes the form $\left[\begin{smallmatrix}0 & 1\\ N_{2,1} & N_{2,2}\end{smallmatrix}\right]$, such that $N^m=I_2\bmod n^2$. Because $N$ and $N-I_2$ are both invertible in $\mathds{Z}/n\mathds{Z}$, it follows that $\abs{N}=-N_{2,1}$ and $\abs{N-I_2}=N_{2,1}+N_{2,2}-1$ are both invertible modulo $n$. Let $\xi\in \mathds{Z}/n\mathds{Z}$ be an element satisfying $(1-N_{2,1}-N_{2,2})\xi=1$.\par

 Solving $\M\,\V N=\M^{\thinspace T}\thinspace\V$ shows that there exist $G$--coloured knots in the $S$--equivalence class represented by $(s_1;is_2)$ for $1\leq i<n$ only if $N_{2,1}\bmod n=-1$, and that $\rho$--equivalence classes for such knots are represented by $G$--coloured knots $(J_{k,l,i},\rho_{k,l,i})$ with surface data

\begin{equation}
(\M_{k,l,i},\V_i)\ =\ \left(\rule{0pt}{18pt}\tilde\xi\,\begin{bmatrix} \tilde{i}+kn & -1\\ 1-N^{\prime}_{2,2} & \widetilde{i^{-1}}+ln \end{bmatrix}\raisebox{-5.5pt}{\huge{,}}\normalsize \begin{pmatrix} s_1\\ is_2 \end{pmatrix}\right),
\end{equation}

\noindent where $N^{\prime}_{2,2}$ denotes the minimum integer which agrees modulo $n$ with $N_{2,2}$ for which $1-2\tilde\xi+\tilde\xi N^{\prime}_{2,2}\bmod n=0$. For this surface data

\begin{equation}
\mathrm{su}(F_{k,l,i},\bar\rho_{k,l,i})=
\left(k-N_{2,1}\tilde{i}l, (N_{2,2}-1)\tilde{i}l-k\right)\bmod n.
\end{equation}

\noindent Because  $N_{2,1}+N_{2,2}-1$ is a unit modulo $n$, the number of distinct values of the surface untying invariant for these knots is $n\sum_{j=1}^{n-1} \frac{n}{\gcd(n,j)}$. In particular, if $n$ is prime then all possible values are realized.

Knots in the $S$--equivalence class represented by $(s_1;0;s_2;0)$ are represented by $G$--coloured knots $(J^\ast_{k,l},\rho^\ast_{k,l})$ with surface data

\begin{equation}
(\M^\ast_{k,l},\V^\ast)\ =\ \left(\rule{0pt}{38pt}\begin{bmatrix} kn & \tilde{\xi}N_{2,1} & 0 & \tilde\xi\\ \tilde\xi N_{2,1}+1 & 0 & \tilde\xi & 0\\
0 & \tilde{\xi} & ln & a-1\\ \tilde\xi & 0 & a & 0 \end{bmatrix}\raisebox{-5.5pt}{\huge{,}}\normalsize \begin{pmatrix} s_1\\ 0\\ s_2\\ 0 \end{pmatrix}\right),
\end{equation}

\noindent where $a\in\mathds{N}$ is the minimal natural number congruent modulo $n$ to $\frac{\xi}{N_{2,1}}$.\par

The surface coloured untying invariant for these is

\begin{equation}\mathrm{su}(F^\ast_{k,l},\bar\rho^\ast_{k,l})=
\left(N_{2,1}l-k, k+l(N_{2,2}-1)\right)\bmod n.\end{equation}

For this $S$--equivalence class, the number of possible values of the surface untying invariant equals $n$ times the order of $(N_{2,2}-N_{2,1}+1)\bmod n$. If this number is a unit, then we have classified knots coloured by such groups up to $\bar\rho$--equivalence.

\begin{rem}
As in Remark \ref{R:AlgoFail}, the algorithm of \cite[Section 4]{KM09} gives a non-sharp upper bound of $n^4$ for the number of $\bar\rho$--equivalence classes.
\end{rem}

\subsubsection{$\rho$--equivalence; $N$ is not diagonalisable}\label{SSS:R2M3N}

Because $N$ is not diagonalisable, $m$ must be greater than $2$. We consider only the case $m=3$.

\begin{thm}\label{T:R2M3Non}
 The number of non-$\rho$--equivalent $G$--coloured knots among elements of the set $\left\{(J_{k,l,i},\rho_{k,l,i})\right\}_{1\leq k,l\leq n}$ for each $1\leq i<n$, and also the number of non-$\rho$--equivalent $G$--coloured knots in the set $\left\{(J^\ast_{k,l},\rho^\ast_{k,l})\right\}_{1\leq k,l\leq n}$, are bounded from below by the order of $6(1+N_{2,2}+N^{2}_{2,2}-N_{2,1}^2)\bmod n$.
\end{thm}

\begin{proof}
Consider the claim for $G$--coloured knots in the $S$--equivalence class represented by $(s_1;s_2)$.
As in the proof of Theorem \ref{T:R2M3Diag}, it suffices to consider the quantity

\begin{multline}
\epsilon\left(\V,\V N\right)\, L\left[\begin{matrix}kn & 0\\ 0 & ln\end{matrix}\right]\,\left(\V;\V N\right)
=\\\left(3,3\right)k+\left(N_{2,1}(1+2N_{2,1}+2N_{2,2}),2+N_{2,1}+N_{2,2}(2+2N_{2,1}+2N_{2,2})\right)l\bmod n.
\end{multline}

To see how many $\rho$--equivalence classes we can distinguish by the surface coloured untying invariant, we calculate

\[ (2+N_{2,1}+N_{2,2}(2+2N_{2,1}+2N_{2,2}))-N_{2,1}(1+2N_{2,1}+2N_{2,2})=2+2N_{2,2}+2N^{2}_{2,2}-2N_{2,1}^2.
\]
\noindent The theorem follows.\par

The proof is analogous for $G$--coloured knots in the other $S$--equivalence classes.
\end{proof}

\section{$A_4$-Coloured Knots}\label{S:A4}

To finish this paper, we go beyond the algebraic techniques of Section \ref{S:untyinginvariants},
to classify $G$--coloured knots up to $\rho$--equivalence for a specific small but interesting group.

\subsection{Setup}
The alternating group $A_4$ is the group of orientation preserving symmetries of an oriented tetrahedron. As a metabelian group it is of the form

\begin{equation}A_4 = \mathcal{C}_3\ltimes_{\phi}  \left(\mathds{Z}/2\mathds{Z}\right)^{2},\end{equation}

\noindent where the matrix associated to $\phi$ is  $N=\left[\begin{smallmatrix} 0 & 1\\
-1 & -1
\end{smallmatrix}\right]$.\par

The number of $\rho$--equivalence classes of $A_4$-coloured knots is bounded from above by the number of $\bar\rho$--equivalence classes of such knots equipped with marked Seifert surfaces, which is $8$ by the bordism upper bound of Corollary \ref{C:Wallacebound}. For the $S$--equivalence class represented by $(s_1;s_2)$, the four distinct $\bar\rho$--equivalence classes are represented by the knots in Figure \ref{F:A4s1s2}, which are denoted $3_1^l, 3_1^r, 4_1^l,$ and $4_1^r$ correspondingly.

\begin{figure}
\centering
\begin{minipage}[c]{0.2\textwidth}
\psfrag{a}[c]{$a$}\psfrag{b}[c]{$b$}
\includegraphics[width=50pt]{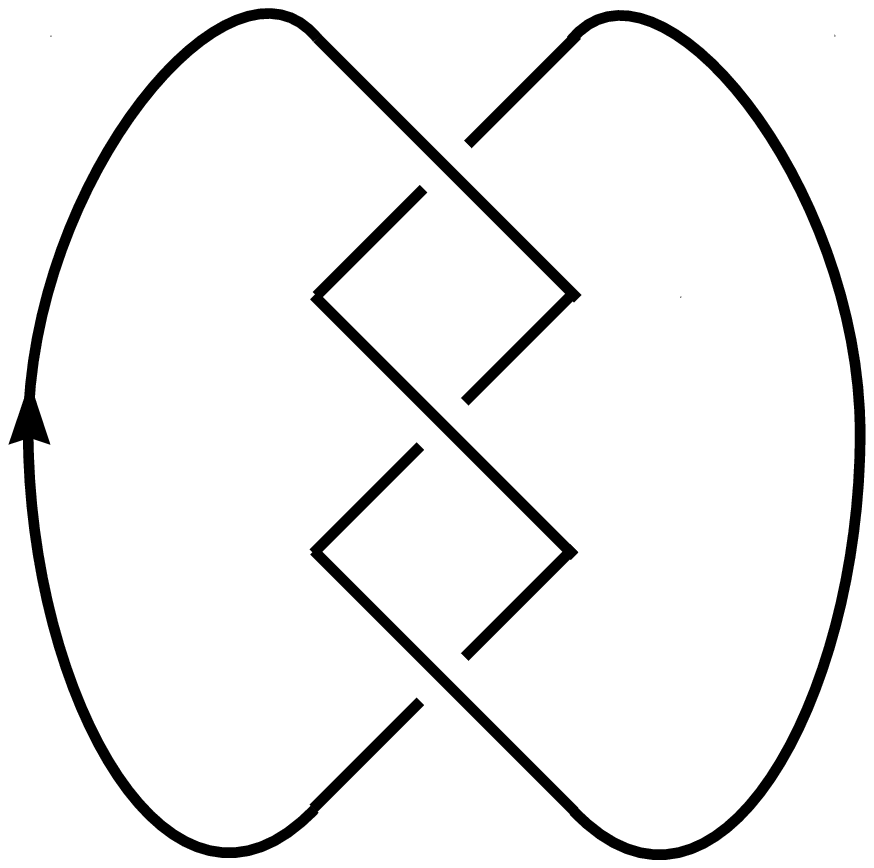}
\end{minipage}\quad
\begin{minipage}[c]{0.2\textwidth}
\psfrag{a}[c]{$a$}\psfrag{b}[c]{$b$}
\includegraphics[width=50pt]{T32r}
\end{minipage}\quad
\begin{minipage}[c]{0.2\textwidth}
\psfrag{a}[c]{$a$}\psfrag{b}[l]{$b$}
\includegraphics[width=60pt]{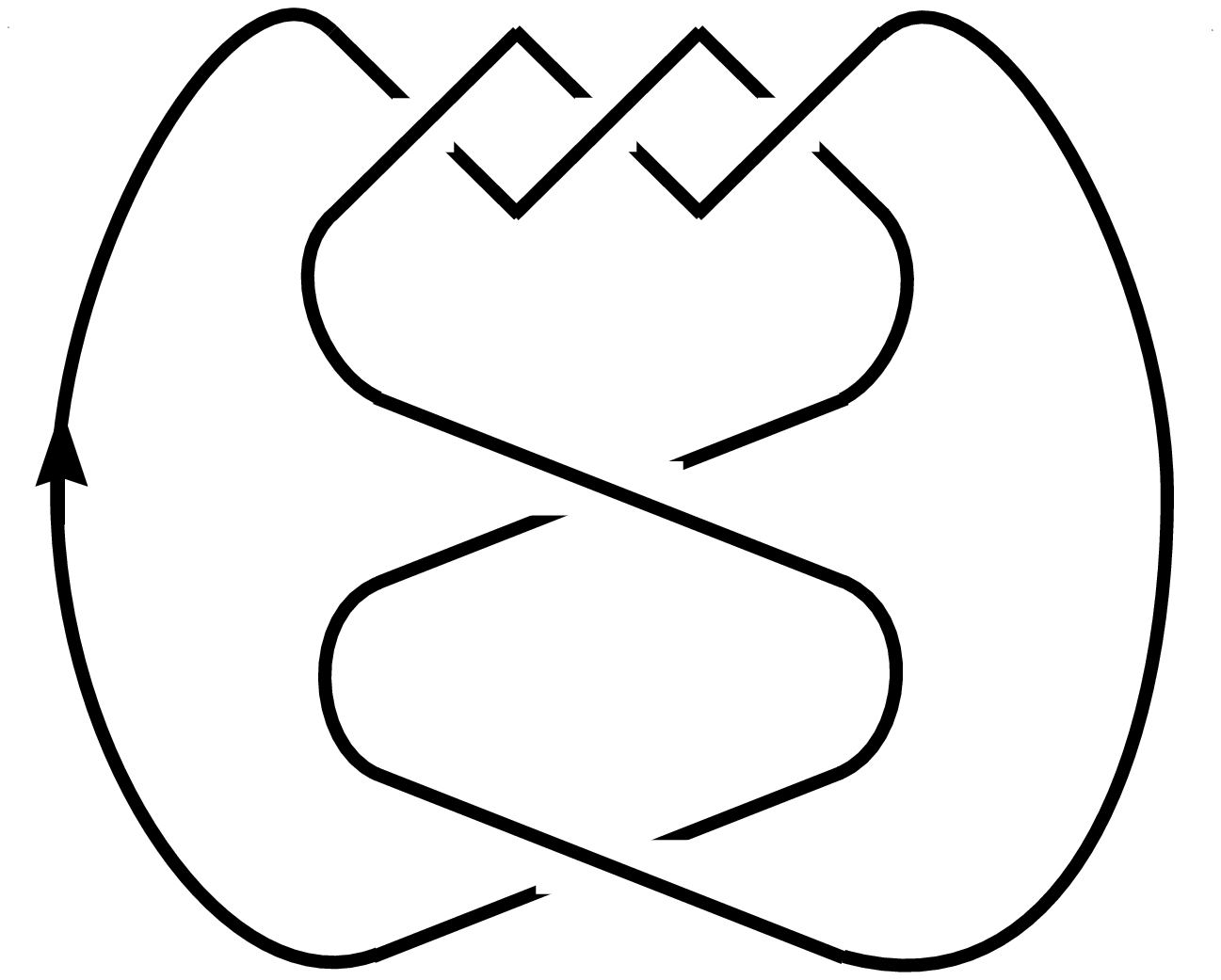}
\end{minipage}\quad\quad
\begin{minipage}[c]{0.2\textwidth}
\psfrag{a}[c]{$a$}\psfrag{b}[l]{$b$}
\includegraphics[width=60pt]{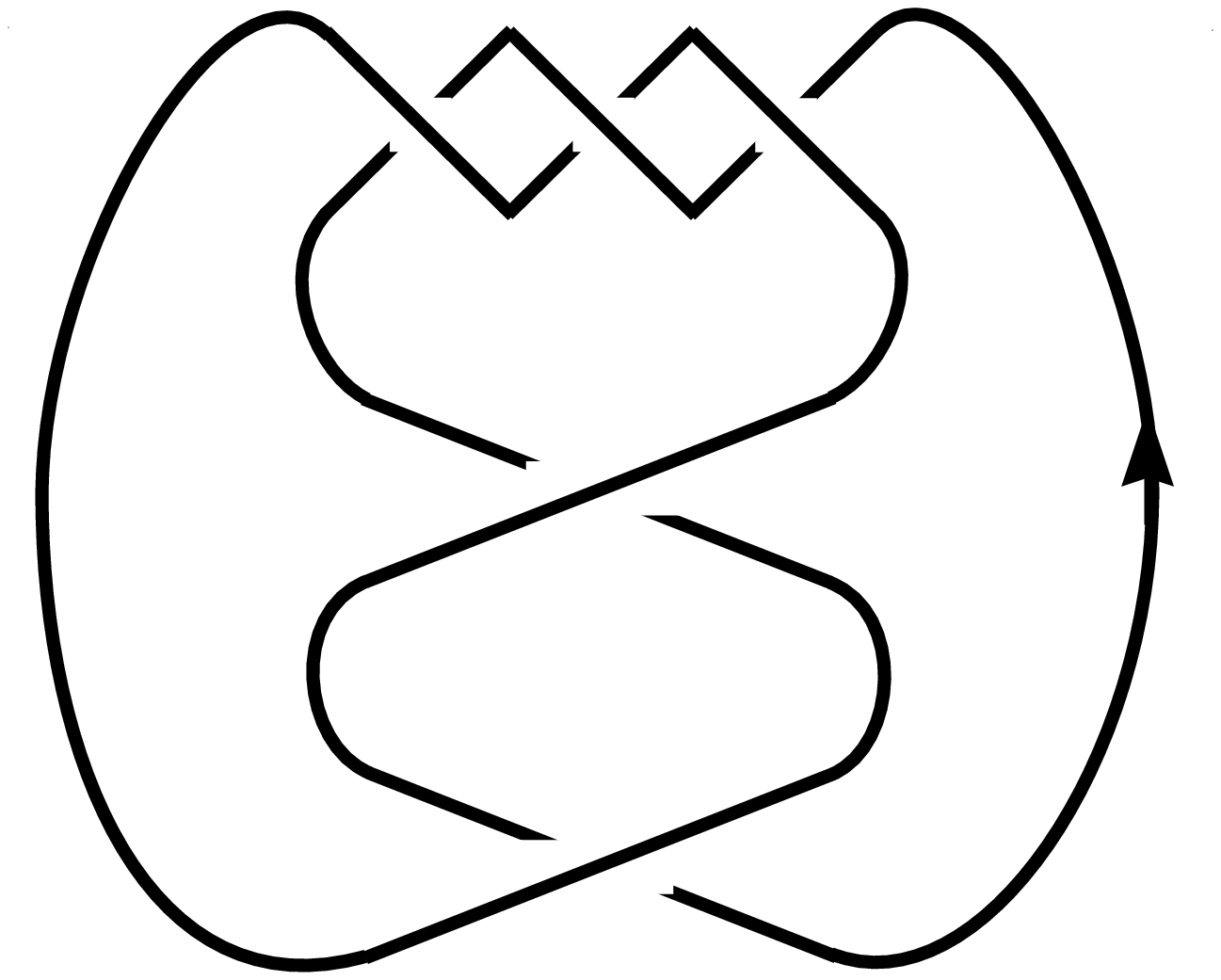}
\end{minipage}
\caption{\label{F:A4s1s2} Representatives for $\bar\rho$--equivalence classes of $A_4$-coloured knots $S$--equivalence class $s_1\wedge s_2$. Each knot diagram comes equipped with its evident Seifert surface.}
\end{figure}

We choose the colouring vector $(s_1;s_2;s_1;s_2)$ to represent the remaining $S$--equivalence class. The four distinct $\bar\rho$--equivalence classes of $A_4$-coloured knots with this colouring vector are represented by $3_1^l\Hash 3_1^l$, $3_1^l\Hash 4_1^l$, $3_1^l\Hash 4_1^r$, and $4_1^l\Hash 4_1^r$ (these are well-defined $A_4$-coloured knots by Lemma \ref{L:consumwelldef}).\par

On the other hand, the number of $\rho$--equivalence classes of $A_4$-coloured knots is bounded from below by $2$, because $3_1^l$ and $4_1^l$ are distinguished by their coloured untying invariants, which are $1$ and $s_1$ correspondingly. Notice first that $4_1^l$ and $4_1^r$ are ambient isotopic, and we therefore don't distinguish between them, and call them $4_1$ collectively. We finish this paper by showing that the lower bound of $2$ is sharp, by reducing each knot in our list to either $3_1^l$ or to $4_1$ by twist moves. Let $S$ denote the commutative semigroup of $\rho$--equivalence classes of $A_4$-coloured knots, equipped with the connect sum operation (see Section \ref{SS:A4Prelim}). Consider $\psi\co \mathcal{C}_2\to S$ which maps $0$ and $1$ to the $\rho$--equivalence classes of $3_1^l$ and of $4_1$ correspondingly.

\begin{thm}\label{T:A4Theorem}
 The map $\psi$ is a bijection. In particular, $S$ is isomorphic to a group with two elements, which are distinguished by the coloured untying invariant.
\end{thm}

\subsection{Preliminaries}\label{SS:A4Prelim}

According to our conventions, $\rho$ sends Wirtinger generators to elements of the coset $t\left(\mathds{Z}/2\mathds{Z}\right)^{2}$. To simplify notation we write its elements $\left\{t,ts_1,ts_2,ts_1s_2\right\}$ as $\left\{a,b,c,d\right\}$ correspondingly. Let $Q$ denote the conjugation quandle whose elements are $\left\{a,b,c,d\right\}$ and whose quandle operation is given by Table \ref{T:A4Table}. This table is found also in \cite[Figure 2]{HN05}.

\begin{figure}
\begin{minipage}[b]{0.4\textwidth}
\begin{center}
\begin{tabular}{c | c c c c}
$x\ast y$ & a & b & c & d\\ \hline
\rule{0pt}{13pt}a & a & d & b & c \\
b & c & b & d & a\\
c & d & a & c & b\\
d & b & c & a & d\\
\end{tabular}
\end{center}
\end{minipage}%
\quad\quad
\begin{minipage}[c]{0.2\textwidth}
\psfrag{a}[c]{$y$}\psfrag{b}[c]{$x$}\psfrag{c}[c]{$x\ast y$}
\includegraphics[width=50pt]{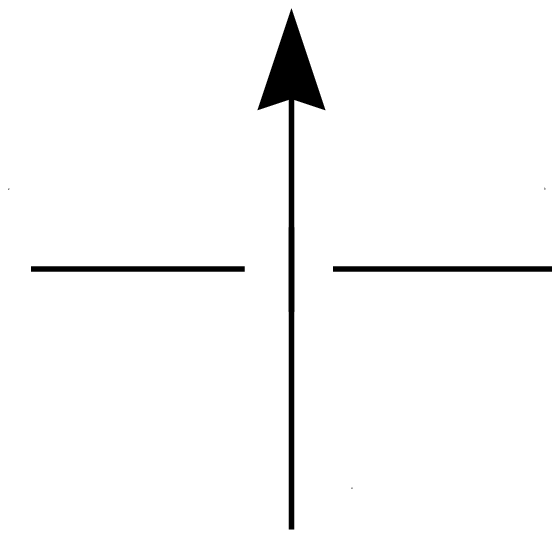}
\end{minipage}\ \
\begin{minipage}[c]{0.2\textwidth}
\psfrag{a}[c]{$y$}\psfrag{b}[l]{$x\ast y$}\psfrag{c}[c]{$x$}
\includegraphics[width=50pt]{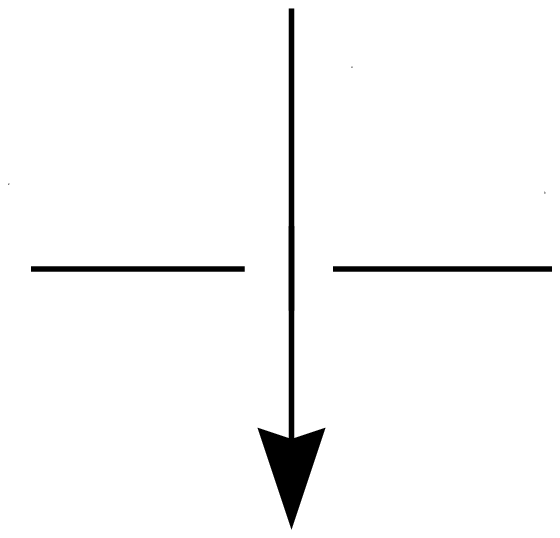}
\end{minipage}
\renewcommand{\figurename}{Table\rule{0pt}{11pt}}
\caption{\label{T:A4Table}Conjugation table for $Q$.}
\end{figure}


The connect-sum $(K_1,\rho_1)\Hash(K_2,\rho_2)$ of $A_4$-coloured knots $(K_{1,2},\rho_{1,2})$ is well-defined, and does not depend on the choice of basepoints, as proven \cite[Lemma 4]{Mos06b}. If one of the connect summands is an invertible knot (ambient isotopic to itself with the opposite orientation), and if its $A_4$-colouring is unique up to inner automorphism, then the connect sum is independent of the choice of orientations. This implies in particular the following.

\begin{lem}\label{L:consumwelldef}
If $K_{1,2}$ are connect sums of trefoil knots and of figure-eight knots, and if $\rho_{1,2}$ are their corresponding unique $A_4$-colourings, then $(K_1,\rho_1)\Hash(K_2,\rho_2)$ is independent of the orientations of $K_{1,2}$.
\end{lem}

\subsection{Proof of Theorem {\ref{T:A4Theorem}}}

We identify some $\rho$--equivalences between trefoils and figure-eight knots by explicitly finding sequences of twist moves which relate them. The notation $(K_1,\rho_1)\seq (K_2,\rho_2)$ means that $(K_{1,2},\rho_{1,2})$ are $\rho$--equivalent.

\begin{lem}\label{L:A4idents}\hfill
\begin{enumerate}
\item $3_1^l\Hash 4_1\seq 3_1^r$ and by reflection $3_1^r\Hash 4_1\seq 3_1^l$.
\item $3_1^r\Hash 3_1^r\seq 4_1$.
\item $3_1^l\Hash 3_1^r\seq 4_1$.
\end{enumerate}
\end{lem}
\begin{proof}
\hfill
\begin{enumerate}
\item
\begin{multline*}
\begin{minipage}{110pt}
\psfrag{1}[c]{$b$}\psfrag{2}[c]{$c$}\psfrag{4}[c]{$a$}\psfrag{5}[c]{$d$}
\includegraphics[width=110pt]{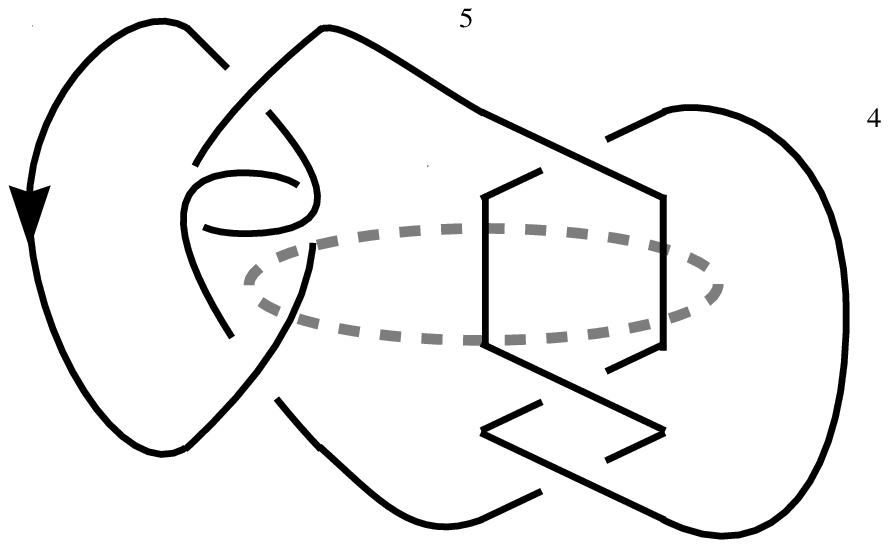}
\end{minipage}\ \overset{\raisebox{2pt}{\scalebox{0.8}{\text{twist}}}}{\Longleftrightarrow}\ \
\begin{minipage}{110pt}
\psfrag{1}[c]{$b$}\psfrag{2}[c]{$c$}\psfrag{4}[b]{$a$}\psfrag{5}[c]{$d$}
\includegraphics[width=110pt]{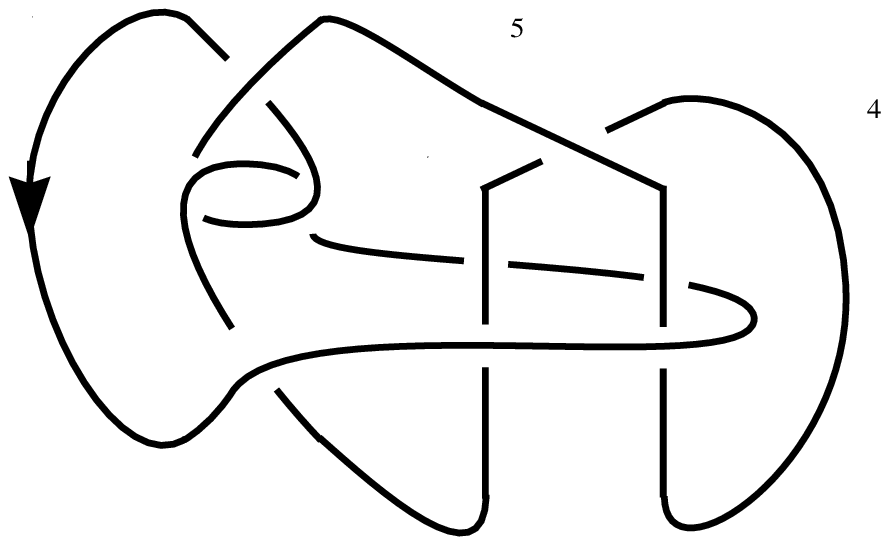}
\end{minipage}
\overset{\raisebox{2pt}{\scalebox{0.8}{\text{isotopy}}}}{\Longleftrightarrow}\ \
\begin{minipage}{65pt}
\psfrag{a}[c]{$a$}\psfrag{b}[c]{$b$}
\includegraphics[width=65pt]{T32r}
\end{minipage}
\end{multline*}

\item
\begin{multline*}
\begin{minipage}{120pt}
\psfrag{1}[c]{$c$}\psfrag{2}[c]{$b$}\psfrag{3}[c]{$d$}\psfrag{5}[c]{$b$}\psfrag{6}[c]{$c$}\psfrag{7}[c]{$a$}
\includegraphics[width=120pt]{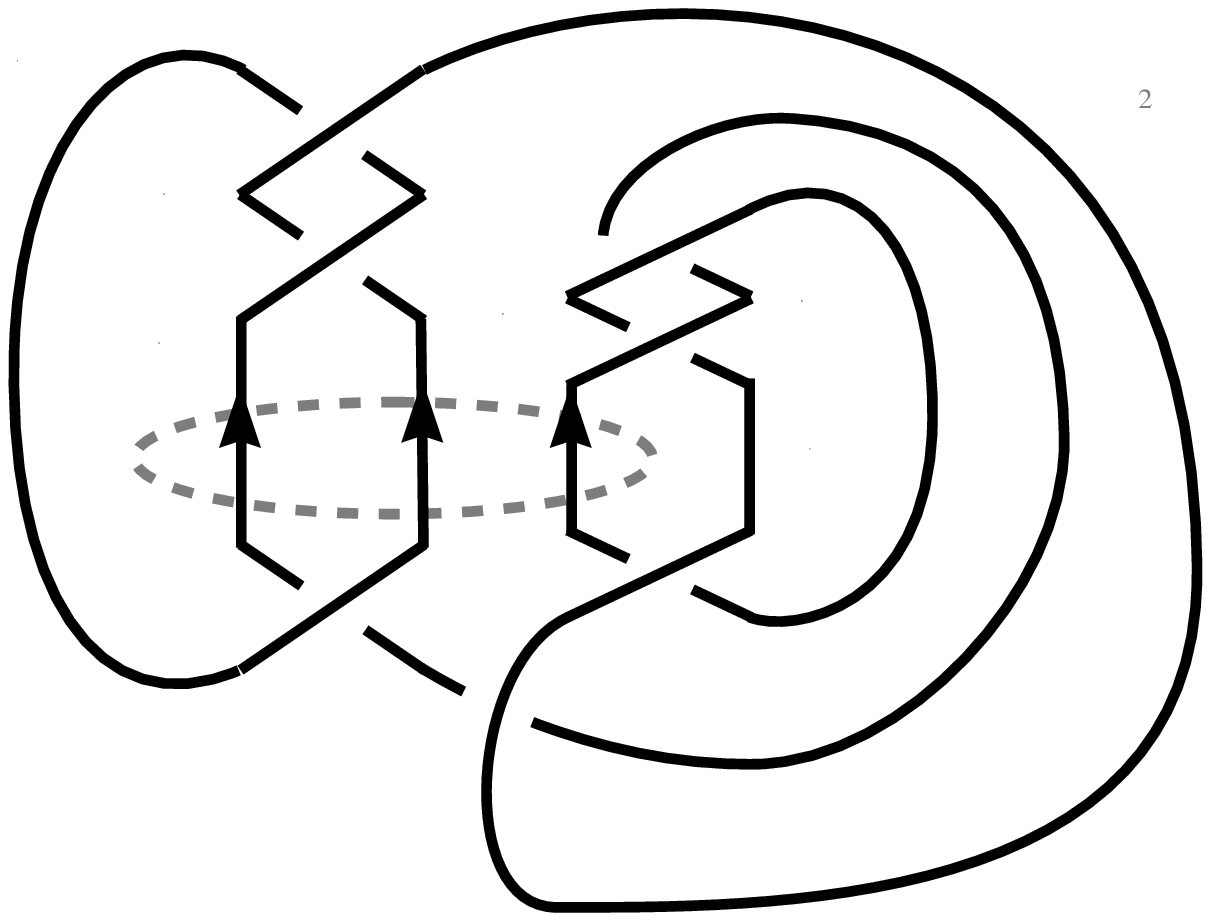}
\end{minipage}\ \ \overset{\raisebox{2pt}{\scalebox{0.8}{\text{twist}}}}{\Longleftrightarrow}\ \
\begin{minipage}{120pt}
\psfrag{1}[c]{$c$}\psfrag{2}[c]{$b$}\psfrag{3}[c]{$d$}\psfrag{5}[c]{$b$}\psfrag{6}[c]{$c$}\psfrag{7}[c]{$a$}
\includegraphics[width=120pt]{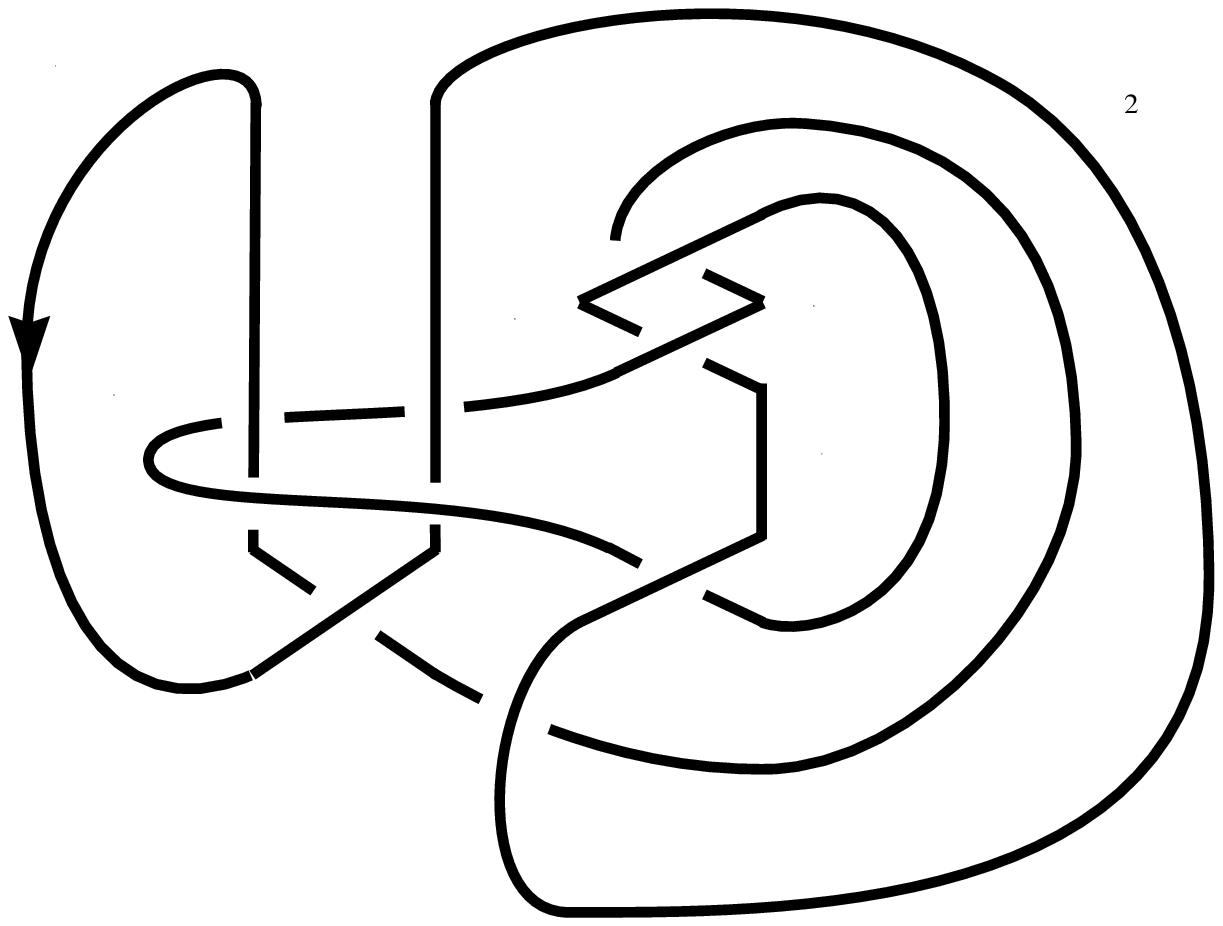}
\end{minipage}\\[0.3cm]
\ \ \overset{\raisebox{2pt}{\scalebox{0.8}{\text{isotopy}}}}{\Longleftrightarrow}\ \
\begin{minipage}{90pt}
\psfrag{1}[c]{$a$}\psfrag{2}[c]{$b$}
\includegraphics[width=90pt]{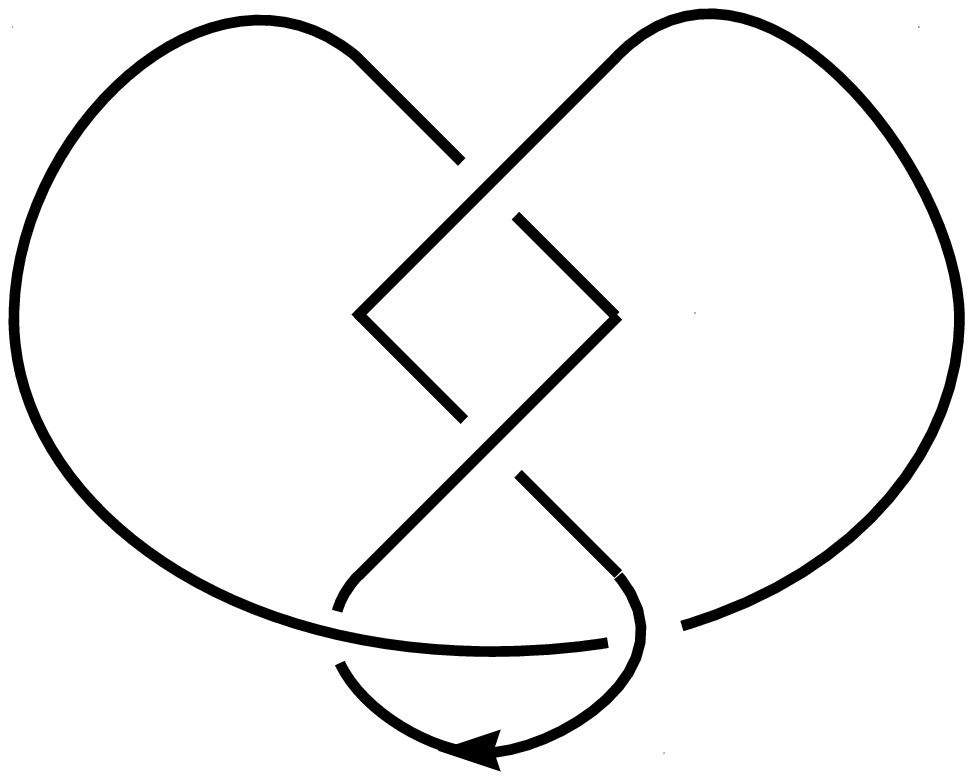}
\end{minipage}
\end{multline*}

\item
\begin{multline*}
\begin{minipage}{110pt}
\psfrag{1}[c]{$a$}\psfrag{8}[c]{$d$}\psfrag{2}[c]{$b$}\psfrag{5}[c]{$b$}\psfrag{3}[c]{$c$}\psfrag{4}[c]{$a$}
\includegraphics[width=110pt]{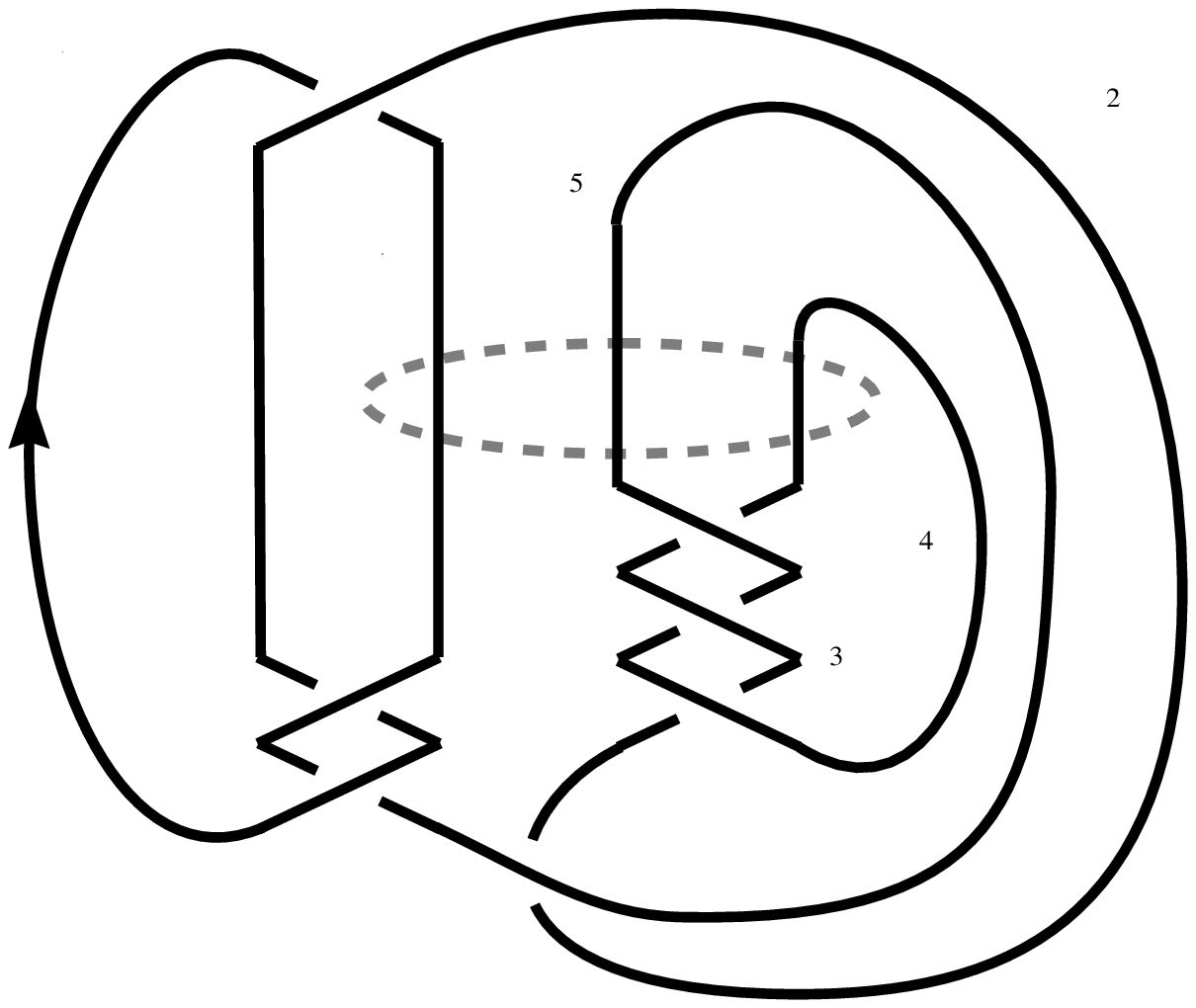}
\end{minipage}\ \ \ \overset{\raisebox{2pt}{\scalebox{0.8}{\text{twist}}}}{\Longleftrightarrow}\ \
\begin{minipage}{110pt}
\psfrag{1}[c]{$a$}\psfrag{6}[c]{$b$}\psfrag{8}[c]{$b$}\psfrag{9}[c]{$d$}\psfrag{2}[c]{$a$}
\includegraphics[width=110pt]{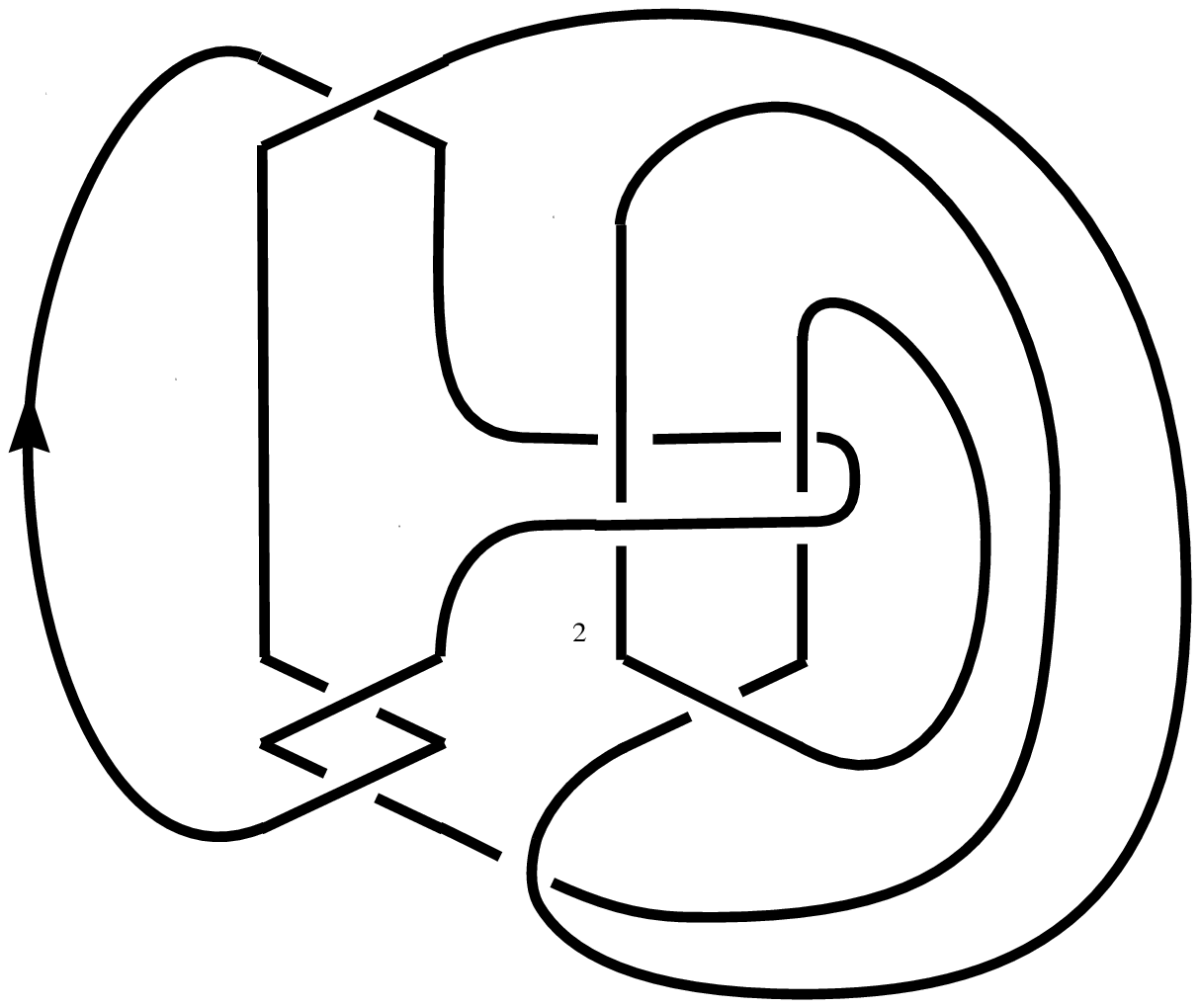}
\end{minipage}\\[0.3cm]
\overset{\raisebox{2pt}{\scalebox{0.8}{\text{isotopy}}}}{\Longleftrightarrow}\
\begin{minipage}{90pt}
\psfrag{a}[c]{$a$}\psfrag{b}[c]{$b$}
\includegraphics[width=90pt]{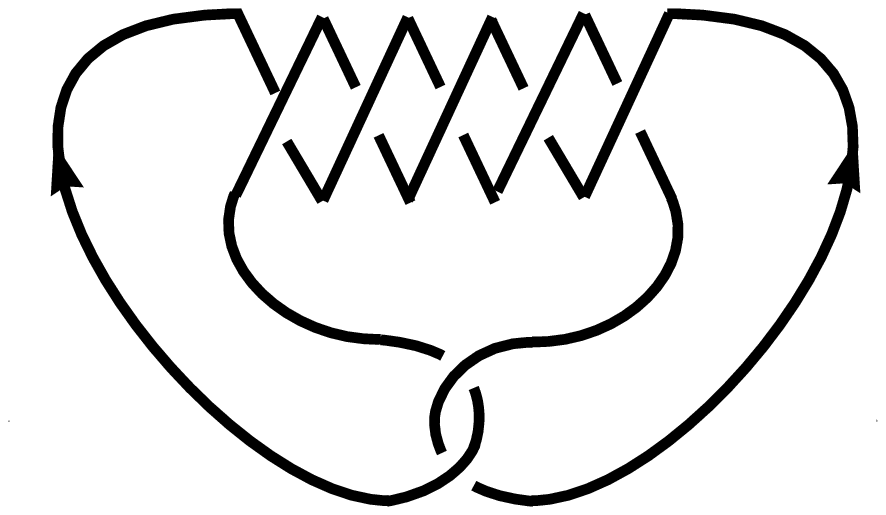}
\end{minipage}
\ \ \overset{\raisebox{2pt}{\scalebox{0.8}{{\ref{E:add4twist}}}}}{\Longleftrightarrow}\
\begin{minipage}{90pt}
\psfrag{a}[c]{$a$}\psfrag{b}[c]{$b$}
\includegraphics[width=90pt]{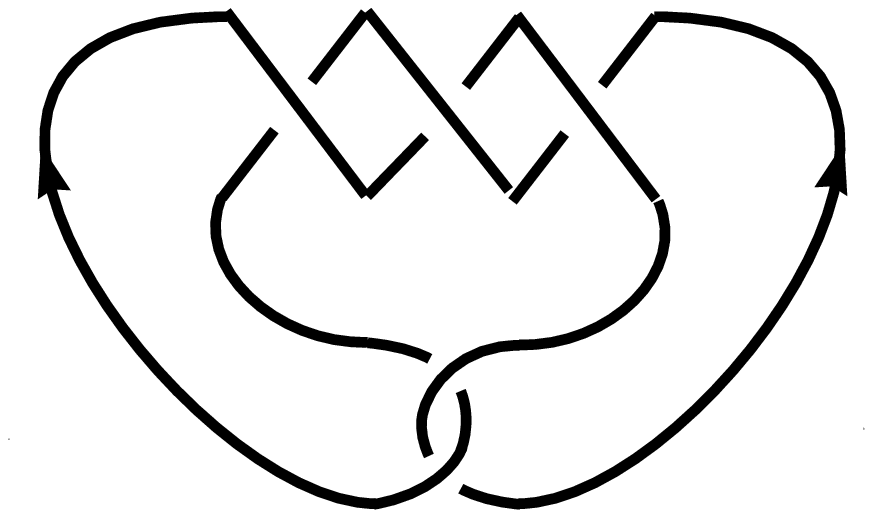}
\end{minipage}\\
\overset{\raisebox{2pt}{\scalebox{0.8}{\text{isotopy}}}}{\Longleftrightarrow}\ \  \begin{minipage}{80pt}
\psfrag{1}[c]{$a$}\psfrag{2}[c]{$b$}
\includegraphics[width=80pt]{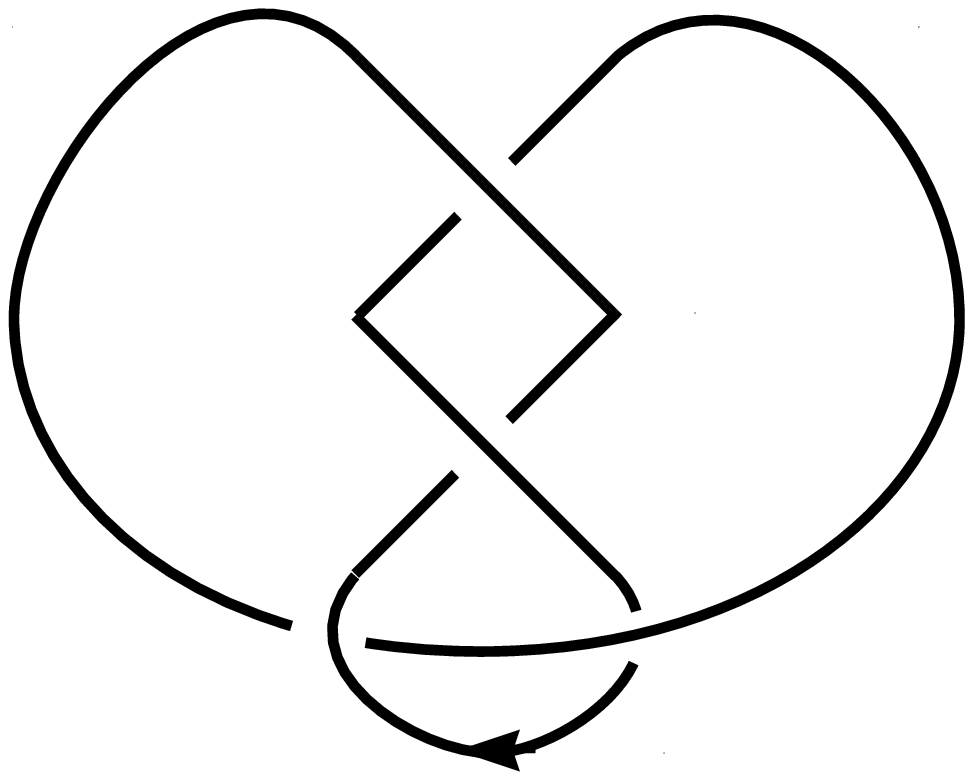}
\end{minipage}
\end{multline*}

\noindent where for the penultimate step we subtract four full twists via the following sequence of null-twists:

\begin{equation}\label{E:add4twist}
\begin{minipage}{20pt}
\includegraphics[height=120pt]{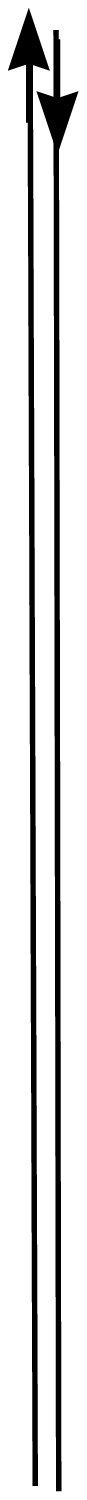}
\end{minipage}\overset{\raisebox{2pt}{\scalebox{0.8}{\text{isotopy}}}}{\Longleftrightarrow}\ \ \
\begin{minipage}{60pt}
\includegraphics[height=120pt]{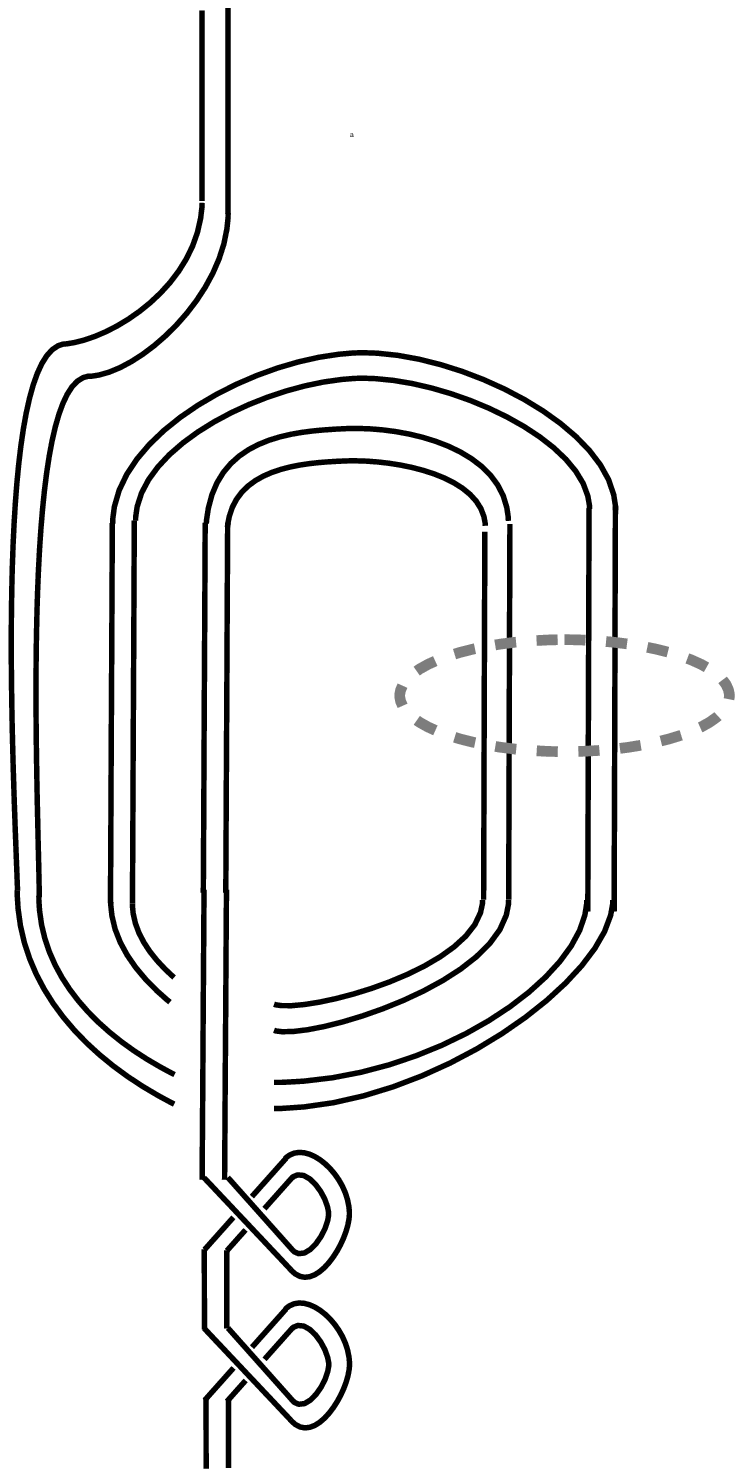}
\end{minipage}\ \overset{\raisebox{2pt}{\scalebox{0.8}{\text{null-twist}}}}{\Longleftrightarrow}\
\begin{minipage}{60pt}
\includegraphics[height=120pt]{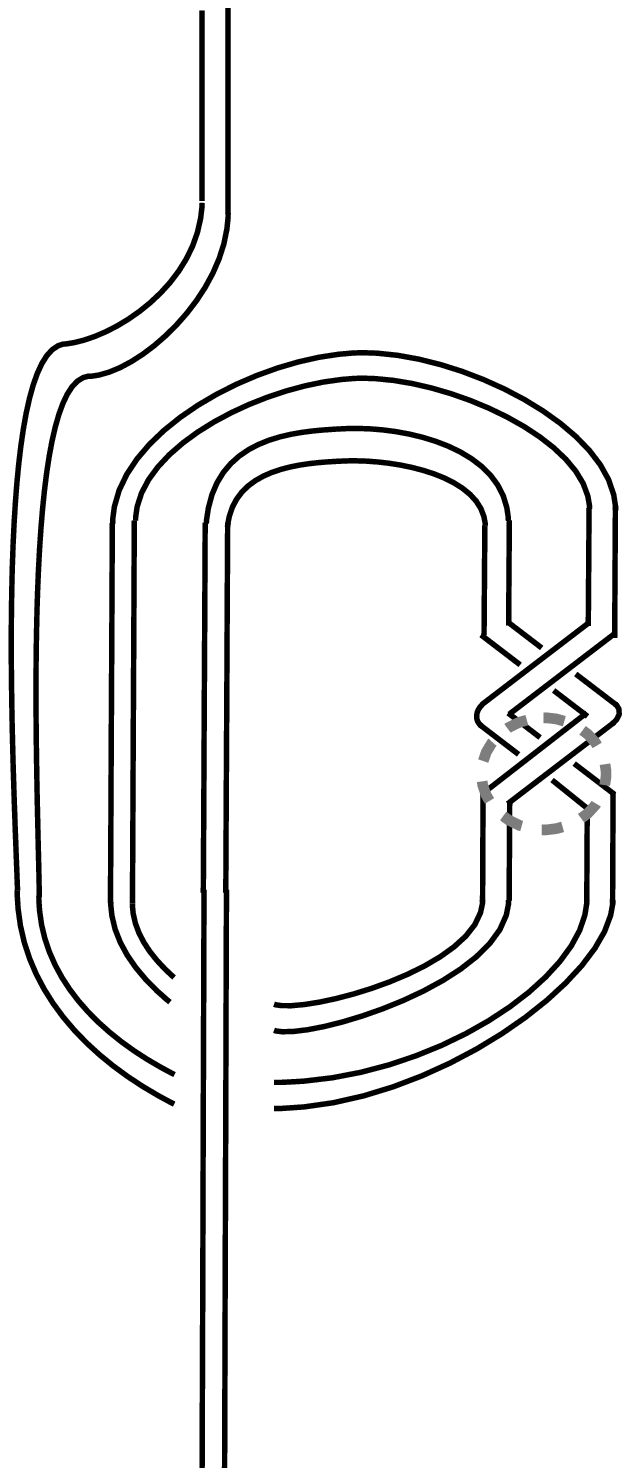}
\end{minipage}
\overset{\raisebox{2pt}{\scalebox{0.8}{\text{null-twist}}}}{\Longleftrightarrow}\ \ \
\begin{minipage}{30pt}
\includegraphics[height=120pt]{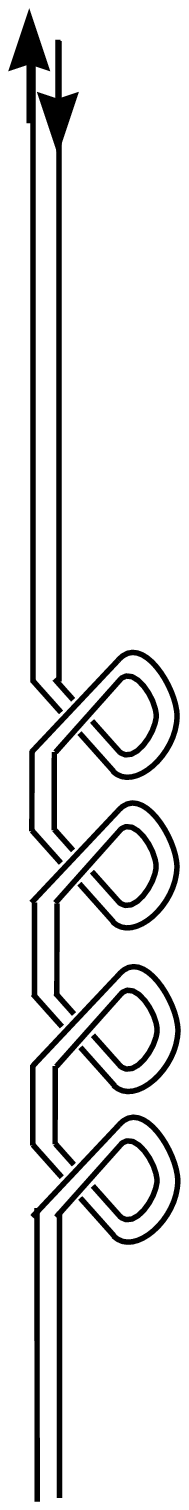}
\end{minipage}.
\end{equation}
\end{enumerate}
\end{proof}

\begin{proof}[Proof of Theorem {\ref{T:A4Theorem}}]

As a corollary to Lemma \ref{L:A4idents} we have

\begin{equation}
3_1^r\seq 3_1^l\Hash 4_1 \seq 3_1^r\Hash 3_1^l\Hash 3_1^r \seq 3_1^r\Hash 4_1\seq 3_1^l.
\end{equation}

Thus, up to $\rho$--equivalence, there is no need to distinguish between $3_1^r$ and $3_1^l$, and we may call them both $3_1$. By looking back at our list of representatives of $\bar\rho$--equivalence classes, we now know that any $A_4$ coloured knot is $\rho$--equivalent to one of $\left\{3_1,4_1,3_1\Hash 3_1,3_1\Hash 4_1,4_1\Hash 4_1\right\}$. The classes of $3_1\Hash 3_1$ and of $4_1$ are the same by Lemma \ref{L:A4idents}, as are the classes of $3_1\Hash 4_1$ and of $3_1$. Finally, the classes of $4_1\Hash 4_1$ and of $4_1$ are the same, because

\begin{equation}
4_1\Hash 4_1 \seq 4_1\Hash 3_1\Hash 3_1\seq 4_1\Hash 3_1\seq 4_1.
\end{equation}

Therefore the map $\psi$, which maps $0$ to $3_1$ and $1$ to $4_1$, is a bijection of groups where the connect-sum gives rise to the group operation on $S$.
\end{proof}

\section{Additional questions}\label{S:conclusion}

We have classified $G$--coloured knots up to $\rho$--equivalence for a large class of metabelian groups $G=\mathcal{C}_m\ltimes A$ with $\Rank(A)\leq 2$. This work raises a number of additional questions.

\begin{enumerate}
\item Classify $G$--coloured knots up to $\rho$--equivalence for a wider class of groups. The particularly interesting cases seem to be:
    \begin{itemize}
    \item For metacyclic groups with $\Ab\,G\approx \mathcal{C}_3$, we have classified $G$--coloured knots up to $\bar\rho$--equivalence. However the coloured untying invariant is trivial, so we have no lower bound on the number of $\rho$--equivalence classes.
    \item For metabelian groups $G$ with $\Rank(A)>2$, the techniques are the same but the matrices are bigger, and one must take the $Y$--obstruction into account. In general, can you determine for which groups our invariants classify $G$--coloured knots up to $\bar\rho$--equivalence?
    \item Polycyclic groups. How can our methods be iterated?
    \item The symmetric group $S^4$ and the alternating group $A^5$ are finite subgroups of $SO(3)$, and the classification of their $\rho$--equivalence classes looks interesting.
    \item It makes sense to consider the $\rho$--equivalence classification problem not only for groups, but also for more general quandle colourings.
    \end{itemize}
\item For $G$ metabelian, classify $G$--coloured links, perhaps in $3$--manifolds, up to $\rho$--equivalence. My guess is that one would need to figure out how handle with the maximal abelian covering directly, instead of using a Seifert matrix.
\item In order to apply our classification results to the construction of quantum topological invariants, the base knots have to be sufficiently `nice'. What are the conditions on $G$ for each $G$--coloured knot to be $\rho$--equivalent to:
    \begin{itemize}
    \item A knot with unknotting number 1?
    \item A fibred knot?
    \end{itemize}
\item Find a conceptual reason that different flavours of $\rho$--equivalence should coincide for some groups but not for others. Can this be detected homologically?
\end{enumerate}

\bibliographystyle{amsplain}

\end{document}